%% file: 0a_main.tex
\documentclass[reqno]{amsart}

\usepackage{0b_macros}

\begin{document}
\title[Canonical Landau-Ginzburg models for cominuscule homogeneous spaces]{Canonical Landau-Ginzburg models \\ for cominuscule homogeneous spaces}
\author{Peter Spacek}
\address{Technische Universit\"at Chemnitz, Germany}
\email{peter.spacek@math.tu-chemnitz.de}

\author{Charles Wang}
\address{University of Michigan, USA}
\email{cmwa@umich.edu}
\date{\today}

\begin{abstract}
  We present a type-independent Landau-Ginzburg (LG) model $(\mX_\can, \pot_\can)$ for any cominuscule homogeneous space $\X=\G/\P$. We give a fully combinatorial construction for our superpotential $\pot_\can$ as a sum of $n+1$ rational functions in the (generalized) Pl\"ucker coordinates on the ``Langlands dual'' minuscule homogeneous space $\cmX=\dP\backslash\dG$. {Explicitly, we define the denominators $\cD_\is$ of these rational functions using} the combinatorics of order ideals of the corresponding minuscule poset, which can be interpreted as (generalized) Young diagrams, by a process that can be described by ``moving boxes'' and hence is easily implemented. To construct the corresponding numerators, we define derivations $\deis$ on $\C[\cmX]$ that act by ``adding an appropriate box if possible'' and then we apply each $\deis$ to the corresponding $\cD_{\is}$. By studying certain Weyl orbits in the fundamental representations of $\udG$ and exploiting the existence of a certain dense algebraic torus in $\cmX$, we show that the polynomials $\cD_\is$ coincide with the generalized minors $\phiGLS[\is]$ appearing in the cluster structures for homogeneous spaces studied by Gei\ss-Leclerc-Schr\"oer in \cite{GLS_partial_flag_varieties_and_preprojective_algebras}. We then define the mirror variety $\mX_\can=\cmX\setminus D_\ac$ to be the complement of the anticanonical divisor $D_\ac = \sum_{\is}\{\cD_\is=0\}$ formed by the $\cD_{\is}$.  Moreover, we show that the LG models $(\mX_\can,\pot_\can)$ are isomorphic to the Lie-theoretic LG-models $(\mX_\Lie,\pot_\Lie)$ constructed by Rietsch in \cite{Rietsch_Mirror_Construction} and our models naturally generalize the type-dependent Pl\"ucker coordinate LG-models previously studied by various authors in \cite{Marsh_Rietsch_Grassmannians, Pech_Rietsch_Odd_Quadrics,Pech_Rietsch_Williams_Quadrics,Pech_Rietsch_Lagrangian_Grassmannians}.
\end{abstract}

\maketitle
\vspace{-.5cm}
\setcounter{tocdepth}{2}
\tableofcontents

\input{1introduction}

\subsubsection*{Acknowledgements}
The first author was supported by DFG grant SE 1114/6-1. The second author would like to thank Lauren Williams for helpful conversations.

\input{2prelims}
\input{3toric_monomials}

\input{4denominators}
\input{5numerators}
\input{6main_proof}

\bibliographystyle{amsalpha}
\bibliography{bib}
\addresseshere
\newpage
\appendix
\input{appAexceptional}
\input{appBtype_dependent_LGs}
\end{document}

%% file: 1introduction.tex
\newpage\section{Introduction}\label{sec:intro}
Let $\X=\G/\P$ be a projective variety homogeneous under the action of a simple, semi-simple, complex algebraic group $\G$ and $\P$ the parabolic subgroup of $\G$ associated to $\X$. There are various descriptions of the small quantum cohomologies of these varieties; see e.g. \cite{Fulton_on_qH_of_hom_spaces} for a survey. In this article, we will describe the quantum cohomology $qH^*(\X)$ of $\X$ using a mirror symmetry isomorphism to the \emph{Jacobi ring} $\C[\mX\times(\C^*)^\rho]/\Jac{\bq}$ of a \emph{Landau-Ginzburg (LG) model} $(\mX,\pot_\bq)$ consisting of a mirror space $\mX$ and a superpotential $\pot_\bq:\mX\rightarrow\C$ depending on $\rho=\rk(\Pic(\X))$ formal parameters $\bq=(q_1,\ldots,q_\rho)$ corresponding to the quantum parameters of $qH^*(\X)$. Rietsch constructed LG models $(\X_\Lie,\pot_\Lie)$ for homogeneous spaces in \cite{Rietsch_Mirror_Construction} by combining Peterson's presentation \cite{Peterson} of the quantum cohomology of homogeneous spaces with Lie-theoretic constructions, leading to the name \emph{Lie-theoretic mirror models}.

In the case when $\X$ is \emph{cominuscule} (i.e.~when $\P=\P_k=P_{\{k\}}$ is a maximal parabolic subgroup of $\G$ which stabilizes the highest weight vector in the cominuscule fundamental weight representation $V(\fwt[k])$; in this case we have $\rho=1$ and hence $\bq=q$ a single parameter), the geometric Satake correspondence provides a one-to-one correspondence between generators of the cohomology of $\X$ and a natural set of projective coordinates on the space $\cmX=\dP\backslash\dG$ which we will refer to as \emph{(generalized) Pl\"ucker coordinates}. Here $\dG$ and $\dP$ are the Langlands dual groups to $\G$ and $\P$ respectively, and hence we will refer to $\cmX$ as the \emph{Langlands dual homogeneous space}. The geometric Satake correspondence is especially appealing due to the fact that when $\X$ is cominuscule, then the Langlands dual $\cmX$ is \emph{minuscule} which endows the Pl\"ucker coordinates on $\cmX$ with many useful combinatorial properties.

The geometric Satake correspondence motivated the construction of Landau-Ginzburg models $(\mX_\can,\pot_\can)$ with $\mX_\can\subset\cmX$ for many families of cominuscule spaces (e.g. \cite{Marsh_Rietsch_Grassmannians,Pech_Rietsch_Lagrangian_Grassmannians, Pech_Rietsch_Odd_Quadrics,Pech_Rietsch_Williams_Quadrics}). The mirror varieties $\mX_\can$ in these constructions all arise as the complement of a distinguished anticanonical divisor $D_\ac$ inside $\cmX$, and the mirror models have correspondingly been referred to as \emph{canonical} Landau-Ginzburg models. In each case, these models have been shown to be isomorphic to Rietsch's Lie-theoretic models $(\mX_\Lie,\pot_\Lie)$. Although these constructions of LG models for so many cominuscule spaces suggests the possibility of a uniform construction, the methods and objects in each case heavily relied on type-specific representation-theoretic methods. Our main result in this paper is the type-independent construction of LG models for these cominuscule cases which are isomorphic to Rietsch's Lie-theoretic models.

In the works \cite{Spacek_Wang_exceptional, Spacek_Wang_OG}, we have been working towards generalizing the previous approaches with the goal of obtaining a uniform, type-independent construction which applies equally well to any cominuscule homogeneous space. The main difficulty lies in identifying suitable candidates for the mirror space $\X_\can$ and the superpotential $\pot_\can$. The starting point for our construction comes from certain \emph{generalized minors} (up to sign) $\phiGLS[\is]:\dG\to\C$ for $\is\in[n]$ which can be defined in terms of the projection formulas
\[
\phiGLS[\is](g)=(-1)^h\minor_{\wo\cdot\dfwt[\is],\wop\cdot\dfwt[\is]}(g^{-1}) = (-1)^h\Bigl\lan g^{-1}\bwop\cdot\hwt[\is],\lwt[\is]\Bigr\ran_{\dfwt[\is]} \qfor \is\in[n],
\]
along with $\phiGLS[0](g) = \minor_{\dfwt[k],\dfwt[k]}(g^{-1})$.
These generalized minors will be introduced in further detail in Section \ref{ssec:Coordinate_ring_of_dunimP}. Our first step will be to identify regular functions, which we denote by $\cD_{\is}$, corresponding to the $\phiGLS[\is]$ on the minuscule Langlands dual space $\cmX$. For this, we consider $\cmX$ in its natural minimal embedding $\cmX\hookrightarrow \PPluckerDualRep$, which is often called the \emph{(generalized) Pl\"ucker embedding} of $\cmX$; see equation \eqref{eq:Plucker_embedding} and the discussion leading up to it. The resulting projective coordinates on $\cmX$ are referred to as \emph{(generalized) Pl\"ucker coordinates} and in order to construct the $\cD_{\is}$, we begin by computing in Section \ref{sec:toric-monomials} the restriction of each $\phiGLS[\is]$ to an algebraic torus $\opendunim\subset\dG$, which is defined in equation \eqref{eq:opendunim_decompostion_of_u-}. These computations use order ideals of the minuscule poset $\minposet$ associated to $X$ as introduced in Section \ref{sec:min_posets}; they can be depicted as (skew) Young diagrams with each box labeled by a simple reflection, see Table \ref{eq:young_diagrams} for examples of such diagrams. Because $\phiGLS[0],\phiGLS[k]$ are defined using the representation dual to the one used to define Pl\"ucker coordinates, we directly compute
\[
\phiGLS[k]|_{\opendunim}(u_-) = p_{\minposet}|_{\opendunim}(u_-) = a_{\minposet}\qand \phiGLS[0]|_{\opendunim}(u_-) = p_{\varnothing}|_{\opendunim}(u_-) = 1.
\]
We determine the restrictions of the remaining $\phiGLS[\is]$ for $\is\in[n]\setminus\{k\}$ to $\opendunim$ 
\[
\phiGLS[\is]|_{\opendunim}(u_-) = \prod_{j=1}^{\cis} a_{\ideal_{i_j}}
\]
in Theorem \ref{thm:toric_monomial_phi}, where $\cis$ is the coefficient of the simple root $\sr_{\is}$ in the highest root $\sro$ for $G$ and $a_{\ideal_{i_j}}$ is a monomial in the toric coordinates of $\opendunim$ corresponding to a certain order ideal $\idealij$ of $\minposet$ allowing an easy combinatorial description, namely $\idealij$ the maximal Young subdiagram obtained from $\idealij[j-1]$ by adding boxes with label unequal to $s_{i_j}$ and where $\idealij[0]=\varnothing$. (See also Corollary \ref{cor:combinatorial_description_wij}.) Here, the sequence of indices $(i_j)$ is given in equation \eqref{eq:index_seq} and starts with $i_1=\DSi[k](\is)$ and $i_2=\DSi[i_1](k)$ where $\DSi:[n]\setminus\{i\}\to[n]\setminus\{i\}$ denotes the (possibly trivial) Dynkin involution of the Dynkin diagram of $\dG$ after removing node $i$.)

The reason for working with these restrictions is that order ideals of the minuscule poset $\minposet$ also give a combinatorial description of the restrictions of Pl\"ucker coordinates to $\opendunim$ (considering the coordinates as $\udP$-invariant functions on $\udG$). In Section \ref{sec:denominators}, we use the structure of order ideals to analyze the cancellation of toric terms in Pl\"ucker monomials and describe this cancellation in terms of moving boxes between Young diagrams corresponding to different Pl\"ucker coordinates. To organize these cancellations, we construct a poset $\sP_{\is}$ whose elements are all $\cis$-tuples of order ideals of $\minposet$ corresponding to Pl\"ucker monomials of degree $\cis$ and whose cover relations correspond to moving a single box between Young diagrams. We will construct $\sP_{\is}$ using the initial element $(\idealij[1],\dots, \idealij[\cis])$ consisting of the order ideals $\idealij\subideal\minposet$ mentioned above and generating the remaining elements by moving boxes in all possible ways from $\idealij[l]$ to $\idealij[m]$ for $l<m$. Letting $L_d\subset\sP_{\is}$ denote the tuples $(\ideal_1,\dots, \ideal_{\cis})$ obtained using $d$ such moves starting from $(\idealij[1],\dots, \idealij[\cis])$, we establish in Theorem \ref{thm: phi_cancellation} that $\phiGLS[\is]$ is equal on the torus $\opendunim$ to
\begin{equation}
\cD_{\is} = \sum_{d\ge 0} (-1)^d \sum_{(\ideal_1,\dots, \ideal_{\cis})\in L_d} p_{\ideal_1}\dots p_{\ideal_{\cis}},
\label{eq:intro_phiGLS}
\end{equation}
and we obtain the full equality of this expression on $\cmX$ via Remark \ref{rem:equality_of_functions_on_dunimP}. It will follow that for all but four cases, i.e. when $\cis\le 2$, $\sP_\is$ either consists of a single element or else can be identified with the Bruhat order on weight spaces in a minuscule representation; see Remark \ref{rem:poset_is_rep}.

Using the $\cD_{\is}$, we define our mirror space $\X_\can$. For convenience, we set $\cD_0=p_{\varnothing}$ and $\cD_k=p_{\minposet}$. Then letting $D_{\is} = \{\cD_{\is}=0\}$ for each $0\le \is \le n$, we set $\X_\can$ to be the complement of the divisor $D_{\ac}=\sum_{\is=0}^n D_{\is}$ formed by the union of the $D_{\is}$ inside of $\cmX$ 
\[
\X_\can = \cmX\setminus D_{\ac}.
\]
As noted below in Remark \ref{rem:anticanonical}, it will follow from Theorem \ref{thm: phi_cancellation} that $D_{\ac}$ anticanonical.

We will construct our superpotential as a sum of $n+1$ rational functions in the Pl\"ucker coordinates of $\cmX$. The denominators of these rational functions will be given by the $\cD_{\is}$ for $0\le \is\le n$. In order to determine the corresponding numerators, we define for each $j\in[n]$ the derivation $\delta_j:\mathbb{C}[\cmX]\rightarrow\mathbb{C}[\cmX]$ which acts on Pl\"ucker coordinates labeled by Young diagrams as
\begin{equation}
    \label{eq:derivation_young_diagrams}
    \delta_j(p_\ideal) = \begin{cases} p_{\ideal^{+j}}, & \text{$\ideal^{+j}$ is a valid Young diagram for $\cmX$,} \\ 0, & \textrm{otherwise,}\end{cases}
\end{equation}
where $\ideal^{+j}$ is the result of adding a box labeled $s_j$ to the Young diagram $\ideal$ if the result is a valid Young diagram. (Recall from above that Young diagrams for $\cmX$ are naturally equipped with a labeling by simple reflections in the Weyl group of $\dG$.) We extend $\de_j$ to polynomials in the Pl\"ucker coordinates as a derivation via the product rule.

For each $\is\in[n]\setminus\{k\}$, the numerator of the term of our superpotential with the denominator $\cD_{\is}$ will be given by $\deis(\cD_{\is})$, where we recall that $i_1=\DSi[k](\is)$. For the term of our superpotential with denominator $\cD_0=p_{\varnothing}$, the corresponding numerator will be given by $\Yboxdim{3pt}\delta_k(\cD_0)=\p_{\yng(1)}$. In Section \ref{sec:quantum}, we address the numerator of the ``quantum'' term of our superpotential which has denominator $\cD_k$. Our results in Section \ref{sec:denominators} show that $\cD_k$ is equal to the Pl\"ucker coordinate $p_{\minposet}$ and we show in Corollary \ref{cor:quantum_term} that the corresponding numerator is the Pl\"ucker coordinate $p_{\idealPrime}$ where $\idealPrime\subideal\minposet$ is order ideal isomorphic to the \emph{complement} of the largest order ideal containing \emph{exactly} one element labeled $s_k$. Equivalently, $\idealPrime$ is the ideal corresponding to the minimal coset representative $\wPrime=\wP(\wPPrime)^{-1}\in\cosets$, where $\wPPrime\in\cosets$ is the minimal coset representative of $\wop s_k\weylp$. Using these descriptions with $\cD_i$ as in \eqref{eq:intro_phiGLS} above, we obtain the superpotential 
\[
\Yboxdim{3pt}
\pot_\can
= \frac{\p_{\yng(1)}}{\p_\varnothing} + q\frac{\p_{\idealPrime}}{\p_{\minposet}} + \sum_{\is\in[n]\setminus\{k\}} \frac{\de_{\DSi[k](\is)}(\cD_{\is})}{\cD_{\is}}.
\]

Before stating the full extent of our results, we will illustrate our models in three well-known examples: quadrics, Lagrangian and orthogonal Grassmannians. First, we consider the Lagrangian Grassmannian $\X=\LG(4,8)=\Sp_8/\P_4=\LGC_4^\SC/\P_4$, describing Lagrangian subspaces in the standard eight-dimensional, complex symplectic vector space $\C^8$. The Langlands dual homogeneous space is $\cmX=\OG(4,9)=\dP_4\backslash\PSO_9 = \dP_4\backslash\LGB_4^\AD$. This space yields the following minuscule poset $\minposet$, containing the maximal order ideal $\idealPPrime$ with exactly one element labeled $s_4$, and with complement $\idealPrime$:
\[\Yboxdim{7pt}
\minposet = \raisebox{-10pt}{\tiny$\young(4,34,234,1234)$},
\qquad
\idealPPrime = \raisebox{-10pt}{\tiny$\young(4,3,2,1)$}
\qand
\idealPrime = \raisebox{-3pt}{\tiny$\young(4,34,234)$},
\]
where we drew the minuscule poset with the \emph{indices} $i$ instead of the \emph{labels} $s_i$, cf.~Section \ref{ssec:Combinatorics_of_minuscule_posets}. This already determines $\cD_0$ and $\cD_4$; for the remaining terms $\is\in\{1,2,3\}$, we find that the corresponding simple roots have coefficient $\cis=\llan\dfwt,\sro\rran=2$ in the highest root, meaning that $\cD_\is$ are quadratic polynomials and we need to find the indices $i_1=\DSi[k](\is)$ and $i_2=\DSi[i_1](k)$. Recall that $\DSi$ denotes the (possibly trivial) Dynkin involution of the Dynkin diagram with node $i$ removed and furthermore that the initial order ideals $\idealij$ are then constructed by setting $\idealij[0]=\varnothing$ and letting $\idealij$ be the maximal order ideal obtained by adding boxes with index unequal to $i_j$ to $\idealij[j-1]$. The posets $\sP_\is$ are then obtained by starting with this initial tuple and moving boxes from $\idealij$ to $\idealij[{j'}]$ for $j<j'$ as allowed by the labeling. We find:
\[\Yboxdim{3pt}
\begin{array}{c||c|c|c|c|l}
    \is & i_1 & \idealij[1] & i_2 & \idealij[2] & \cD_\is \\\hline
    \vspace{-1.2em}&&&&\\\hline
     1 & 3 & \raisebox{6pt}{\young(~)} & 4 & \raisebox{-2pt}{\young(~,~,~,~)\vphantom{\young(~,~,~,~,~)}} & p_{\yng(1)}p_{\yng(1,1,1,1)} - p_{\varnothing}p_{\yng(1,2,1,1)} 
     \vphantom{p_{\young(~,~,~,~,~)}}\\\hline
     2 & 2 & \raisebox{3pt}{\young(~,~~)} & 4 & \raisebox{-2pt}{\young(~,~~,~~,~~)\vphantom{\young(~,~,~,~,~)}} & p_{\yng(1,2)}p_{\yng(1,2,2,2)} - p_{\yng(1,1)}p_{\yng(1,2,3,2)} + p_{\yng(1)}p_{\yng(1,2,3,3)} - p_{\varnothing}p_{\yng(1,2,3,4)} 
     \vphantom{p_{\young(~,~,~,~,~)}}\\\hline
     3 & 1 & \raisebox{.5pt}{\young(~,~~,~~~)} & 4 & \raisebox{-2pt}{\young(~,~~,~~~,~~~)\vphantom{\young(~,~,~,~,~)}} & p_{\yng(1,2,3)}p_{\yng(1,2,3,3)} - p_{\yng(1,2,2)}p_{\yng(1,2,3,4)} 
     \vphantom{p_{\young(~,~,~,~,~)}}
\end{array}
\]
This determines the mirror variety as $\mX_\can=\cmX\setminus D_\ac$ for $D_\ac=\sum_{\is=0}^4 D_\is$ the anticanonical divisor defined by $D_\is=\{\cD_\is=0\}$. We now act by the derivations $\deis$ on each of the $\cD_\is$ with $\is\neq 4$ (where $\is=4$ is the quantum term), which in terms of order ideals adds a box with index $i_1$ if allowed by $\minposet$ and maps the Pl\"ucker coordinate to $0$ otherwise; in conclusion, we find:
\begin{align*}
    \pot_\can &= \frac{\delta_4(\cD_0)}{\cD_0} + \frac{\delta_3(\cD_1)}{\cD_1} + \frac{\delta_2(\cD_2)}{\cD_2} + \frac{\delta_1(\cD_3)}{\cD_3} + q\frac{p_{\idealPrime}}{p_{\minposet}}\\
    &= \Yboxdim{3pt}\frac{p_{\yng(1)}}{p_{\varnothing}} 
    + \frac{p_{\yng(1,1)}p_{\yng(1,1,1,1)} - p_{\varnothing}p_{\yng(1,2,2,1)}}
        {p_{\yng(1)}p_{\yng(1,1,1,1)} - p_{\varnothing}p_{\yng(1,2,1,1)}} 
    + \frac{p_{\yng(1,2,1)}p_{\yng(1,2,2,2)} - p_{\yng(1,1,1)}p_{\yng(1,2,3,2)}}
        {p_{\yng(1,2)}p_{\yng(1,2,2,2)} - p_{\yng(1,1)}p_{\yng(1,2,3,2)} + p_{\yng(1)}p_{\yng(1,2,3,3)} - p_{\varnothing}p_{\yng(1,2,3,4)}} 
    + \frac{p_{\yng(1,2,3,1)}p_{\yng(1,2,3,3)} - p_{\yng(1,2,2,1)}p_{\yng(1,2,3,4)}}
        {p_{\yng(1,2,3)}p_{\yng(1,2,3,3)} - p_{\yng(1,2,2)}p_{\yng(1,2,3,4)}} 
    + q\frac{p_{\yng(1,2,3)}}{p_{\yng(1,2,3,4)}}.
\end{align*}

We will describe the model for the eight-dimensional quadric $\X=Q_{8}=\Spin_{10}/P_1=\LGD_4^\SC/\P_1$ more succinctly. We have $\cmX=Q_8=\dP_1\backslash\PSO_{10}=\dP_1\backslash\LGD_{4}^\AD$ with $\cis=1$ for $\is\in\{1,3,4\}$ and $\cis=2$ for $\is=2$, giving us
\[
\minposet = \Yboxdim{7pt}\raisebox{-3pt}{\tiny\young(123,:421)},
\quad
\idealPPrime=\raisebox{-3pt}{\tiny\young(123,:42)}
\qand
\idealPrime=\raisebox{3.5pt}{\tiny\young(1)}.
\] With this, we obtained the following data for the potential: \[
\Yboxdim{3pt}
\begin{array}{c||c|c|c|c|l}
    \is & i_1 & \idealij[1] & i_2 & \idealij[2] & \cD_\is \\\hline
    \vspace{-1.2em}&&&&\\\hline
     2 & 2 & \raisebox{2pt}{\young(~)} & 4 & \raisebox{0pt}{\young(~~~,:~~)\vphantom{\young(~,~,~)}} & p_{\yng(1)}p_{\young(~~~,:~~)} - p_{\varnothing}p_{\young(~~~,:~~~)} 
     \vphantom{p_{\young(~,~,~)}}\\\hline
     3 & 4 & \raisebox{2pt}{\young(~~~)}\vphantom{\young(~,~)} & \text{n/a} & \text{n/a} & p_{\young(~~~)}\vphantom{p_{\young(~,~)}}\\\hline
     4 & 3 & \raisebox{0pt}{\young(~~,:~)}\vphantom{\young(~,~,~)} & \text{n/a} & \text{n/a} & p_{\young(~~,:~)}
     \vphantom{p_{\young(~,~,~)}}
\end{array}
\]
This determines $\mX_\can$ and applying the respective $\deis$ to each $\cD_\is$ yields the potential
\begin{align*}
    \pot_\can &= \frac{\deis[1](\cD_0)}{\cD_0} + q\frac{p_{\idealPrime}}{p_{\minposet}} + \frac{\delta_2(\cD_2)}{\cD_2} + \frac{\delta_4(\cD_3)}{\cD_3} + \frac{\delta_3(\cD_4)}{\cD_4}\\
    &= \Yboxdim{3pt}\frac{p_{\yng(1)}}{p_{\varnothing}} + q\frac{p_{\yng(1)}}{p_{\young(~~~,:~~~)}} + \frac{p_{\yng(2)}p_{\young(~~~,:~~)}}{p_{\yng(1)}p_{\young(~~~,:~~)} - p_{\varnothing}p_{\young(~~~,:~~~)}} + \frac{p_{\young(~~~,:~)}}{p_{\yng(3)}} + \frac{p_{\young(~~~,:~)}}{p_{\young(~~,:~)}}.
\end{align*}
These two LG models can be identified with the ones presented in \cite{Pech_Rietsch_Lagrangian_Grassmannians, Pech_Rietsch_Williams_Quadrics} up to identification of our indexing set for Pl\"ucker coordinates with theirs, cf.~Appendix \ref{app:type_dependent_LGs}.

In the previous two examples, the posets $\sP_\is$ obtained from moving boxes from $(\idealij[1],\idealij[2])$ were both chains; to illustrate that this set can in fact contain higher dimensional structures---what we call ``cubes'' (cf.~Definition \ref{df:restricted_subposet_order})---as well as to show the convenience of the combinatorics in larger examples, we will finally illustrate our model with the orthogonal Grassmannian $\OG(8,16)=\Spin_{16}/\P_8=\LGD_8^\SC/\P_8$. The Langlands dual homogeneous space is in this case an orthogonal Grassmannian itself, namely $\cmX=\OG(8,16)=\dP_8\backslash\PSO_{16}=\dP_8\backslash\LGD_8^\AD$. Here we find
\[
\Yboxdim{7pt}
\minposet = \raisebox{-21pt}{\tiny\young(8,67,568,4567,34568,234567,1234568)},
\quad
\idealPPrime = \raisebox{-21pt}{\tiny\young(8,67,56,45,34,23,12)}
\qand
\idealPrime = \raisebox{-7pt}{\tiny\young(8,67,568,4567,34568)},
\]
and we have $\cis=1$ for $\is\in\{1,7,8\}$ while $\cis=2$ for $\is\in\{2,3,4,5,6\}$, yielding the following data for the potential:
\[
\Yboxdim{3pt}
\begin{array}{c||c|c|c|c|l}
    \is & i_1 & \idealij[1] & i_2 & \idealij[2] & \cD_\is \\\hline
    \vspace{-1.2em}&&&&\\\hline
     1 & 7 & \raisebox{-6pt}{\yng(1,1,1,1,1,1,1)\vphantom{\yng(1,1,1,1,1,1,1,1)}} & \text{n/a} & \text{n/a} & 
     \raisebox{8pt}{$p_{\yng(1,1,1,1,1,1,1)}
     \vphantom{p_{\yng(1,1,1,1,1,1,1,1)}}$}
     \\\hline
     2 & 6 & \raisebox{10pt}{\yng(1)} & 8 & \raisebox{-6pt}{\yng(1,2,2,2,2,2,2)\vphantom{\yng(1,1,1,1,1,1,1,1)}} & 
     \raisebox{8pt}{$p_{\yng(1)}p_{\yng(1,2,2,2,2,2,2)} - p_{\varnothing}p_{\yng(1,2,3,2,2,2,2)}
     \vphantom{p_{\yng(1,1,1,1,1,1,1,1)}}$}\\\hline
     3 & 5 & \raisebox{7pt}{\yng(1,2)} & 7 & \raisebox{-6pt}{\yng(1,2,3,3,3,3,3)\vphantom{\yng(1,1,1,1,1,1,1,1)}} & 
     \raisebox{8pt}{$p_{\yng(1,2)}p_{\yng(1,2,3,3,3,3,3)} - p_{\yng(1,1)}p_{\yng(1,2,3,4,3,3,3)} + p_{\yng(1)}p_{\yng(1,2,3,4,4,3,3)} - p_{\varnothing}p_{\yng(1,2,3,4,5,3,3)}
     \vphantom{p_{\yng(1,1,1,1,1,1,1,1)}}$}\\\hline
     4 & 4 & \raisebox{5pt}{\yng(1,2,3)} & 8 & \raisebox{-6pt}{\yng(1,2,3,4,4,4,4)\vphantom{\yng(1,1,1,1,1,1,1,1)}}\raisebox{-9pt}{\vphantom{\yng(1,1,1,1,1,1,1,1)}} & \text{see below} \\\hline
     5 & 3 & \raisebox{2pt}{\yng(1,2,3,4)} & 7 & \raisebox{-6pt}{\yng(1,2,3,4,5,5,5)\vphantom{\yng(1,1,1,1,1,1,1,1)}} & 
     \raisebox{8pt}{$p_{\yng(1,2,3,4)}p_{\yng(1,2,3,4,5,5,5)} - p_{\yng(1,2,3,3)}p_{\yng(1,2,3,4,5,6,5)} + p_{\yng(1,2,3,2)}p_{\yng(1,2,3,4,5,6,6)} - p_{\yng(1,2,2,2)}p_{\yng(1,2,3,4,5,6,7)}
     \vphantom{p_{\yng(1,1,1,1,1,1,1,1)}}$}\\\hline
     6 & 2 & \raisebox{-1pt}{\yng(1,2,3,4,5)} & 8 & \raisebox{-6pt}{\yng(1,2,3,4,5,6,6)\vphantom{\yng(1,1,1,1,1,1,1,1)}} & 
     \raisebox{8pt}{$p_{\yng(1,2,3,4,5)}p_{\yng(1,2,3,4,5,6,6)} - p_{\yng(1,2,3,4,4)}p_{\yng(1,2,3,4,5,6,7)}
     \vphantom{p_{\yng(1,1,1,1,1,1,1,1)}}$}\\\hline
     7 & 1 & \raisebox{-3pt}{\yng(1,2,3,4,5,6)\vphantom{\yng(1,1,1,1,1,1,1)}} & \text{n/a} & \text{n/a} & 
     \raisebox{8pt}{$p_{\yng(1,2,3,4,5,6)}
     \vphantom{p_{\yng(1,1,1,1,1,1,1)}}$}
\end{array}
\]
where we have
\[
\cD_4 = \Yboxdim{3pt}
p_{\yng(1,2,3)}p_{\yng(1,2,3,4,4,4,4)} 
- p_{\yng(1,2,2)}p_{\yng(1,2,3,4,5,4,4)} 
+ p_{\yng(1,2,1)}p_{\yng(1,2,3,4,5,5,4)} 
- \Bigl(p_{\yng(1,1,1)}p_{\yng(1,2,3,4,5,6,4)} + p_{\yng(1,2)}p_{\yng(1,2,3,4,5,5,5)} \Bigr) + p_{\yng(1,1)}p_{\yng(1,2,3,4,5,6,5)} - p_{\yng(1)}p_{\yng(1,2,3,4,5,6,6)} + p_{\varnothing}p_{\yng(1,2,3,4,5,6,7)}.
\]
The terms in parentheses indicate where the poset $\sP_4$ becomes two-dimensional: we can move two boxes in the preceding pair of order ideals resulting in two terms with equal sign. To finish up the LG model, we note that we now have all data for $\mX_\can=\cmX\setminus D_\ac$, and applying the respective $\deis$ to the $\cD_\is$ yields the superpotential:
\ali{
\pot_\can &= \Yboxdim{3pt}
\frac{p_{\yng(1,2,1,1,1,1,1)}}{p_{\yng(1,1,1,1,1,1,1)}}
+\frac{
    p_{\yng(1,1)}p_{\yng(1,2,2,2,2,2,2)} 
    - p_{\varnothing}p_{\yng(1,2,3,3,2,2,2)}
    }{
    p_{\yng(1)}p_{\yng(1,2,2,2,2,2,2)} 
    - p_{\varnothing}p_{\yng(1,2,3,2,2,2,2)}} 
+ \frac{
    p_{\yng(1,2,1)}p_{\yng(1,2,3,3,3,3,3)} 
    - p_{\yng(1,1,1)}p_{\yng(1,2,3,4,3,3,3)} 
    + p_{\yng(1)}p_{\yng(1,2,3,4,4,4,3)} 
    - p_{\varnothing}p_{\yng(1,2,3,4,5,4,3)}
    }{
    p_{\yng(1,2)}p_{\yng(1,2,3,3,3,3,3)} 
    - p_{\yng(1,1)}p_{\yng(1,2,3,4,3,3,3)} 
    + p_{\yng(1)}p_{\yng(1,2,3,4,4,3,3)} 
    - p_{\varnothing}p_{\yng(1,2,3,4,5,3,3)}}
\\&\quad\Yboxdim{3pt}
+ \frac{p_{\yng(1,2,3,1)}p_{\yng(1,2,3,4,4,4,4)} 
    - p_{\yng(1,2,2,1)}p_{\yng(1,2,3,4,5,4,4)} 
    + p_{\yng(1,2,1,1)}p_{\yng(1,2,3,4,5,5,4)} 
    - p_{\yng(1,1,1,1)}p_{\yng(1,2,3,4,5,6,4)}
    }{
    p_{\yng(1,2,3)}p_{\yng(1,2,3,4,4,4,4)} 
    - p_{\yng(1,2,2)}p_{\yng(1,2,3,4,5,4,4)} 
    + p_{\yng(1,2,1)}p_{\yng(1,2,3,4,5,5,4)} 
    - p_{\yng(1,1,1)}p_{\yng(1,2,3,4,5,6,4)} 
    - p_{\yng(1,2)}p_{\yng(1,2,3,4,5,5,5)} 
    + p_{\yng(1,1)}p_{\yng(1,2,3,4,5,6,5)} 
    - p_{\yng(1)}p_{\yng(1,2,3,4,5,6,6)} 
    + p_{\varnothing}p_{\yng(1,2,3,4,5,6,7)}}
\\&\quad\Yboxdim{3pt} 
+ \frac{
    p_{\yng(1,2,3,4,1)}p_{\yng(1,2,3,4,5,5,5)} - p_{\yng(1,2,3,3,1)}p_{\yng(1,2,3,4,5,6,5)} + p_{\yng(1,2,3,2,1)}p_{\yng(1,2,3,4,5,6,6)} - p_{\yng(1,2,2,2,1)}p_{\yng(1,2,3,4,5,6,7)}
    }{
    p_{\yng(1,2,3,4)}p_{\yng(1,2,3,4,5,5,5)} 
    - p_{\yng(1,2,3,3)}p_{\yng(1,2,3,4,5,6,5)} 
    + p_{\yng(1,2,3,2)}p_{\yng(1,2,3,4,5,6,6)} 
    - p_{\yng(1,2,2,2)}p_{\yng(1,2,3,4,5,6,7)}}
+ \frac{
    p_{\yng(1,2,3,4,5,1)}p_{\yng(1,2,3,4,5,6,6)} 
    - p_{\yng(1,2,3,4,4,1)}p_{\yng(1,2,3,4,5,6,7)}
}{
    p_{\yng(1,2,3,4,5)}p_{\yng(1,2,3,4,5,6,6)}
    - p_{\yng(1,2,3,4,4)}p_{\yng(1,2,3,4,5,6,7)}
}
+ q\frac{p_{\yng(1,2,3,4,5)}}{p_{\yng(1,2,3,4,5,6,7)}}
}

We note that although our expressions for $\pot_\can$ seem to distinguish the quantum and $\cD_0$ terms, it is possible to give a much more uniform expression for our $\pot_\can$ using a derivation defined via the quantum product instead. Unfortunately, the expressions obtained using the quantum derivation have many extraneous cancellations which make $\pot_\can$ much more difficult to compute compared to our current presentation. For more details on the quantum derivation, see Remark \ref{rem:quantum derivation}. Furthermore, our current expression makes clear which term of $\pot_\can$ corresponds to which component of the divisor $D_{\ac}$, which will be useful for our comparisons with Rietsch's Lie-theoretic LG models and also was a useful feature in the computations of Rietsch and Williams in \cite{Rietsch_Williams_NO_bodies_cluster_duality_and_mirror_symmetry_for_Grassmannians}. For more details on this connection, see Remark \ref{rem:cluster}. 

With our candidate LG-models in hand, we are now ready to state our main theorem. 
\begin{thm} 
\label{thm:main_theorem}
Let $\X=\G/\P_k$ be a cominuscule homogeneous space for $\G$ a simply-connected linear algebraic group of rank $n$, and let $\cmX=\dP_k\backslash\dG$ be the associated ``Langlands dual'' homogeneous space, with the set of projective coordinates $\{\p_\ideal~|~\ideal\subideal\minposet\}$ known as \emph{(generalized) Pl\"ucker coordinates} obtained from the natural embedding $\cmX\hookrightarrow\PPluckerDualRep$ with $\ideal\subideal\minposet$ order ideals of the corresponding minuscule poset. Then:
\begin{romanize}
    \item The $\phiGLS[\is]$, for $\is\in [n]\setminus \{k\}$, are equal on $\cmX$ to the $\cD_{\is}$ as given by equation \eqref{eq:intro_phiGLS} as well as $\phiGLS[0]=\cD_0=\p_\varnothing$ and $\phiGLS[k] = \cD_k=p_{\minposet}$ on $\cmX$ and these polynomials form the set of frozen variables for the cluster structure on $\C[\cmX]$ studied in \cite{GLS_Kac_Moody_groups_and_cluster_algebras}
    \item The divisor $D_{\ac} = \sum_{\is=0}^n D_{\is}$ is anti-canonical, where $D_{\is}=\{\cD_{\is}=0\}$.
    \item For $\is\neq0,k$, the derivations $\de_j$ defined combinatorially by \eqref{eq:derivation_young_diagrams} yield homogeneous polynomials $\de_{\DSi[k](\is)}(\cD_{\is})$ in Pl\"ucker coordinates, where $\DSi[k]$ is the Dynkin involution of the Dynkin diagram with the vertex $k$ removed, and the rational functions
    \[
    \pot_0 = \frac{\delta_k(\cD_0)}{\cD_0} = \frac{p_{\Yboxdim{3pt}\yng(1)}}{p_\varnothing}
    \qand
    \pot_{\is} = \frac{\de_{\DSi[k](\is)}(\cD_{\is})}{\cD_{\is}}
    \qfor \is\neq0,k,
    \]
    satisfy the following equations on the dense algebraic torus $\dP_k\backslash\opendunim\subset\cmX$
    \[
    \pot_0=\sum_{\pb\in\minposet:~\indx(\pb)=k} a_{\pb} \qand \pot_{\is} = \sum_{\pb\in\minposet:~\indx(\pb)=\DSi[k](\is)} a_\pb
    \qfor \is\neq k,
    \]
    where $a_{\pb}$ denote the toric coordinates $\{a_\pb~|~\pb\in\minposet\}$ of $\dP_k\backslash\opendunim$.
    \item For $\is=k$, we find that $\phiGLS[k]=\cD_k=\p_{\minposet}$ on $\cmX$, and define the rational function which we refer to as the \emph{quantum term} by
    \[
    \pot_k = q\frac{p_{\idealPrime}}{p_{\minposet}} \quad\text{for a formal parameter $q\in\C^*$,}
    \]
    where $\idealPrime$ is the order ideal corresponding to $\wPrime=\wP(\wPPrime)^{-1}\in\cosets$ defined by the minimal coset representatives $\wP$ and $\wPPrime$ of the cosets $\wo\weylp$ and $\wop s_k\weylp$, with $\wo\in\weyl$ and $\wop\in\weylp=\lan s_i~|~i\neq k\ran$ the respective longest Weyl group elements. Moreover, $\pot_k$ satisfies on the torus $\dP_k\backslash\opendunim$ the equation
    \[
    \pot_k = q\frac{\sum_{\iota:\idealPrime\hookrightarrow\minposet} a_{\iota(\idealPrime)}}{a_{\minposet}},
    \]
    where $a_{\sS}=\prod_{\pb\in\sS}a_\pb$ for a subset $\sS\subset\minposet$ and the sum is over labeled order embeddings of the poset $\idealPrime$ into the minuscule poset $\minposet$.
    \item The pair $(\mX_\can,\pot_\can)$ consisting of the mirror variety and superpotential
    \[
    \mX_\can = \cmX\setminus D_{\ac} 
    \qand
    \pot_\can = \sum_{\is=0}^n \pot_{\is}
    \]
    forms an LG-model for $\X$, which we refer to as the \emph{canonical mirror model}. That is, we have the isomorphism
    \[
    qH^*(\X)_{(q)} \cong \C[\mX_\can\times\C^*_q]/\Jac{\can}.
    \] 
    Furthermore, the canonical mirror model $(\mX_\can, \pot_\can)$ is isomorphic to the Lie-theoretic mirror model $(\mX_{\Lie},\pot_{\Lie})$ defined in \cite{Rietsch_Mirror_Construction}; moreover, on the torus $\dP_k\backslash\opendunim$ the LG-model restricts to a Laurent polynomial model isomorphic to the one defined in \cite{Spacek_LP_LG_models}.
\end{romanize}
\end{thm}

In order to show that our canonical LG model agrees with Rietsch's Lie-theoretic LG model, we need to establish an isomorphism between $\mX_\can$ and $\mX_\Lie$ as well as an equality of the superpotentials $\pot_\can$ and $\pot_\Lie$ after pulling back along the isomorphism. Building off of our previous results in \cite{Spacek_Wang_exceptional}, we establish that $\mX_\Lie\cong\mX_\can$ in Theorem \ref{thm: mirror_space} by observing that both spaces are affine varieties and then comparing coordinate rings.

To compare the superpotentials, we consider the restrictions of both superpotentials to the torus $\dP_k\backslash\opendunim$. Although many of our previous results have been established on the torus $\opendunim$, the proof of Theorem \ref{thm: mirror_space} implies (see Remark \ref{rem:equality_of_functions_on_dunimP}) that the quotient map $\pi: \dG\xrightarrow{\sim}\cmX$ restricts to an isomorphism $\opendunim \cong \dP_k\backslash \opendunim$ so that restrictions to $\opendunim$ can be considered interchangeably with restrictions to $\dP_k\backslash\opendunim$. Thus, the first four parts of Theorem \ref{thm:main_theorem} follow immediately from our previous results.  

The final computation we need for this comparison is the restrictions of the numerators $\delta_{\sigma_k(\is)}(\cD_{\is})$ to the torus $\opendunim$. For these computations, we construct in Section \ref{sec:numerators} the poset $\sP_{\is}^+$ by translating the derivation $\delta_{\sigma_k(\is)}$ to elements of $\sP_{\is}$ and then using an analogous approach as in Section \ref{sec:denominators} to study the action of $\delta_{\sigma_k(\is)}$ on $\cD_{\is}$. We show in Theorem \ref{thm:derivation_cancellation} that the restriction to $\opendunim$ of $\delta_{\sigma_k(\is)}(\cD_{\is})$ is given by
\[
\delta_{\sigma_k(\is)}(\cD_{\is})|_{\opendunim} = \cD_{\is}|_{\opendunim}\cdot\left(\tsum[_{\pb\in\minposet:~\indx(\pb)=\sigma_k(\is)}] a_{\pb}\right),
\]
where the sum is over certain toric coordinates corresponding to the simple reflection $s_{\sigma_k(\is)}$. Thus, we obtain on the torus $\opendunim$ the expression
\[
\pot_{\is} = \frac{\delta_{\sigma_k(\is)}(\cD_{\is})}{\cD_{\is}} = \sum_{\pb\in\minposet:~\indx(\pb)=\sigma_k(\is)} a_{\pb}
\]
for each term of our superpotential except the quantum term, which we handle separately in Section \ref{sec:quantum}. Using this computation, we prove the final fifth part of our main Theorem \ref{thm:main_theorem} that for any value of the formal quantum parameter $q$ our canonical superpotential agrees with Rietsch's Lie-theoretic superpotential on the dense torus $\dP_k\backslash \opendunim$ and hence on the entire space $\mX_\can$. From this, we conclude that our canonical mirror models are isomorphic to Rietsch's Lie-theoretic models. Before concluding the introduction, we describe a few features of our LG models as well as interesting connections with other works.

\begin{rem}
\label{rem:cluster}
The $\cD_{\is}$ form the set of frozen variables for the cluster structure on $\cmX$ as given in \cite{GLS_partial_flag_varieties_and_preprojective_algebras} for simply-laced types and extended by \cite{demonet-skew-sym-cluster} to the general case. Note that we have introduced an extra sign and permuted the indices of $\phiGLS$ by $\DS$ compared with these references to simplify notation arising from switching from ``upper'' to ``lower'' unipotent cells. In particular, our superpotentials $\pot_\can$ are regular functions on the open cluster variety $\X_\can$. 

Thus, it would be feasible to use our LG models to attempt a type-independent generalization of the results of Rietsch and Williams relating Newton-Okounkov bodies for Grassmannians to so-called superpotential polytopes for Grassmannians in \cite{Rietsch_Williams_NO_bodies_cluster_duality_and_mirror_symmetry_for_Grassmannians}. Importantly, both the Newton-Okounkov bodies as well as the superpotential polytopes they study depend on the choice of a divisor supposed on $D_{\ac}$ and for the superpotential polytope in particular it was important in their computations to understand the correspondence between terms of the superpotential and components of $D_{\ac}$. 
\end{rem}

\begin{rem}
\label{rem:anticanonical}
For each $\is\in[n]$, each $\cD_{\is}$ is a homogeneous polynomial of degree $\cis=\llan \dfwt[i],\sro\rran$ in the Pl\"ucker coordinates and for $\is=0$, $\cD_0$ is a homogeneous polynomial of degree $1$ in the Pl\"ucker coordinates. Recall that $D_{\is} = \{\cD_{\is}=0\}$, so as a result, the degree of the divisor
\[
D_{\ac}=\sum_{\is=0}^n D_{\is}
\]
is equal to
\[
\sum_{0\le \is\le n} \deg(\phiGLS[\is]) = 1+ \sum_{\is\in[n]} \cis = 1 + \sum_{\is\in[n]} \langle \dfwt[\is],\sro\rangle.
\]
In order to compute the final sum on the right hand side, we list below the highest roots corresponding to each $\G$ with at least one cominuscule fundamental weight as well the value of the sum.
\[
\begin{array}{c|c|c}
    \G & \sro & 1 + \sum_{\is\in[n]} \langle \dfwt[i], \sro\rangle\\ \hline
    \LGA_{n-1} & \sr_1 + \dots + \sr_{n-1} & n\\ 
    \LGB_n & \sr_1 + 2\sr_2 + \dots + 2\sr_n & 2n\\
    \LGC_n & 2\sr_1 + \dots + 2\sr_{n-1} + \sr_n & 2n\\
    \LGD_n & \sr_1 + 2\sr_2 + \dots + 2\sr_{n-2} + \sr_{n-1} + \sr_n & 2n-2\\
    \LGE_6 & \sr_1 + 2\sr_2+ 2\sr_3 + 3\sr_4 + 2\sr_5 + \sr_6 & 12\\
    \LGE_7 & 2\sr_1 + 2\sr_2 + 3\sr_3 + 4\sr_4 + 3\sr_5 + 2\sr_6 + \sr_7 & 18
\end{array}
\]
We now compare the quantities computed above with the indices of the Langlands dual minuscule spaces that arise in our constructions. 
\[
\begin{array}{c|c|c|c|c|c}
\G & \P & \X=\G/\P & \cmX=\dP\backslash\dG & \mathrm{index}(\cmX) & 1 + \sum_{\is\in[n]} \langle \dfwt[i], \sro\rangle\\ \hline
\LGA_{n-1} & \textrm{any }\P_k & \Gr(k,n) & \Gr(n-k,n) & n & n\\
\LGB_n & \P_1 & Q_{2n-1} & \mathbb{P}^{2n-1} & 2n & 2n\\
\LGC_n & \P_n & \LG(n,2n) & \OG(n,2n+1) & 2n & 2n\\
\LGD_n & \P_1 & Q_{2n-2} & Q_{2n-2} & 2n-2 & 2n-2\\
\LGD_n & \P_{n-1}, \P_n & \OG(n,2n) & \OG(n,2n) & 2n-2 & 2n-2\\
\LGE_6 & \P_1, \P_6 & \mathbb{O}\mathbb{P}^2 & \mathbb{O}\mathbb{P}^2 & 12 & 12\\
\LGE_7 & \P_7 & \LGE^{\SC}_7/\P_7 & \LGE^{\SC}_7/\P_7 & 18 & 18
\end{array}
\]
Thus, in every case we obtain the equality 
\[
1 + \tsum[_{\is\in[n]}]\langle \dfwt[i], \sro\rangle = \mathrm{index}(\cmX)
\]
showing that the divisor $D_{\ac}$ is an anticanonical divisor in $\cmX$. As a result, $\mX_{\can}$ is the complement of an anticanonical divisor in $\cmX$ and for this reason, these Pl\"ucker coordinate Landau-Ginzburg models have often been referred to as canonical Landau-Ginzburg models.
\end{rem}
  
\begin{rem}\Yboxdim{3pt}
\label{rem:quantum derivation}
Since $X$ is cominuscule, the geometric Satake correspondence gives a vector space isomorphism between the cohomology ring $H^*(X)$ and the fundamental coweight representation $\dfwtrep[k]$ that we will use to define a uniform ``quantum'' derivation $\Delta$. The geometric Satake correspondence identifies the Pl\"ucker coordinate $\p_{\yng(1)}$ with the Schubert class $\scs_{\yng(1)}\in qH^*(\X)$ of a hyperplane. Now, for an arbitrary Pl\"ucker coordinate $\p$, we denote the corresponding Schubert class by $\scs\in qH^*(\X)$ and consider the quantum product $\scs_{\yng(1)}*_q\scs$ which can be expressed as a linear combination of Schubert classes. By geometric Satake, we interpret this result as a linear combination of Pl\"ucker coordinates and define $\Delta(\p)$ to be equal to this linear combination. In other words, we use geometric Satake to pull the quantum product back to the coordinate ring of $\cmX$; hence, we will simply write $\De(\p)=\p_{\yng(1)}*_q\p$ for this map. We extend $\De$ to polynomials in Pl\"ucker coordinates as a derivation.

In terms of Young diagrams, ignoring any quantum contributions, we can realize $\Delta$ as the ``full'' derivation adding any possible box, whereas $\delta_j$ acts as a ``restricted'' derivation only adding boxes with label $s_j$. Hence, at $q=0$, we can in fact obtain $\Delta$ by the equality $\Delta = \sum_j \delta_j$. For arbitrary $q$, our quantum term $q\frac{p_{\idealPrime}}{p_{\minposet}}$ and Corollary \ref{cor:quantum_term} suggest the following generalization of the rim-hook rule for Grassmannians studied by \cite{Fulton_Woodward}. The role of the rim-hook diagram in arbitrary cominuscule type is played by the (generalized) Young diagram $\idealPPrime\subideal\minposet$ corresponding to the minimal coset representative $w_Ps_kW_P$, which can equivalently be obtained as the unique subdiagram of $\minposet$ containing \emph{exactly} one element labeled $s_k$, see Example \ref{ex:list_of_minposets} for a table of these generalized rim-hooks.

\begin{conj}    
    For any cominuscule homogeneous space and (generalized) Young diagram $\ideal\subideal\minposet$ with corresponding Schubert class $\sigma_\ideal$, the quantum product of $\sigma_\ideal$ with the hyperplane class $\sigma_{\yng(1)}$ is given by
    \[\sigma_\ideal *_q \sigma_{\yng(1)} = q\sigma_{\ideal^-} + \sum_{\ideal^+} \sigma_{\ideal^+},\]
    where $\sigma_{\ideal^-}$ denotes the the Schubert class of the complement of $\idealPPrime$ in $\ideal$ if $\idealPPrime \subset\ideal$ and $0$ otherwise and the sum is over all possible ways $\ideal^+\subideal\minposet$ to add a single box to $\ideal$. 
\end{conj}

It is possible to replace the derivations $\delta_j$ in our construction of $\pot_\can$ with the uniform ``quantum'' derivation $\Delta$ and to extend the derivation to the quantum term as well. Letting $\cD = \prod_{\is=0}^n \cD_{\is}$, we obtain the expression
\[\pot_\can = \sum_{\is=0}^n \frac{\Delta(\cD_{\is})}{\cD_{\is}} = \frac{\Delta(\cD)}{\cD}.\]
We note in particular that with this interpretation, the superpotential terms corresponding to $\cD_0$ and $\cD_k$ are no longer distinguished from the others by having different formulas. Since Remark \ref{rem:anticanonical} above shows that the anticanonical divisor $D_{\ac}$ is defined exactly by the vanishing of the regular function $\cD$, we may thus interpret the superpotential $\pot_\can$ in some sense as a ``logarithmic derivative of an anticanonical divisor.''

To give an example of this construction, we consider $\X=\LG(3,6)=\Sp_6/P_3$ the Lagrangian Grassmannian of type $\LGC_3$. In this case, our Theorem \ref{thm:main_theorem} yields the superpotential
\begin{align*}
    \pot_\can &= \frac{\delta_3(\cD_0)}{\cD_0} + \frac{\delta_2(\cD_1)}{\cD_1} + \frac{\delta_1(\cD_2)}{\cD_2} + q\frac{p_{\idealPrime}}{p_{\minposet}}\\
    &= \Yboxdim{3pt}\frac{p_{\yng(1)}}{p_{\varnothing}} + \frac{p_{\yng(1,1)}p_{\yng(1,1,1)} - p_{\varnothing}p_{\yng(1,2,2)}}{p_{\yng(1)}p_{\yng(1,1,1)} - p_{\varnothing}p_{\yng(1,2,1)}} + \frac{p_{\yng(1,2,1)}p_{\yng(1,2,2)} - p_{\yng(1,1,1)}p_{\yng(1,2,3)}}{p_{\yng(1,2)}p_{\yng(1,2,2)} - p_{\yng(1,1)}p_{\yng(1,2,3)}} + q\frac{p_{\yng(1,2)}}{p_{\yng(1,2,3)}}.
\end{align*}

If we instead use the quantum derivation of each denominator $\cD_\is$, we compute
\begin{align*}\Yboxdim{3pt}
    \Delta(\phiGLS[0]) &= p_{\yng(1)}*_q p_{\varnothing} = p_{\yng(1)}\\
    \Delta(\phiGLS[1]) &= \Delta(p_{\yng(1)}p_{\yng(1,1,1)}) - \Delta(p_{\varnothing}p_{\yng(1,2,1)}) = (p_{\yng(1)}*_q p_{\yng(1)})p_{\yng(1,1,1)} + p_{\yng(1)}(p_{\yng(1)}*_q p_{\yng(1,1,1)}) - (p_{\yng(1)}*_qp_{\varnothing})p_{\yng(1,2,1)} - p_{\varnothing}(p_{\yng(1)}*_q p_{\yng(1,2,1)})\\
    &= p_{\yng(1,1)}p_{\yng(1,1,1)} + p_{\yng(1)}(p_{\yng(1,2,1)} + qp_{\varnothing}) - p_{\yng(1)}p_{\yng(1,2,1)} - p_{\varnothing}(p_{\yng(1,2,2)}+qp_{\yng(1)}) = p_{\yng(1,1)}p_{\yng(1,1,1)} - p_{\varnothing}p_{\yng(1,2,2)}\\
    \Delta(\phiGLS[2]) &= \Delta(p_{\yng(1,2)}p_{\yng(1,2,2)}) - \Delta(p_{\yng(1,1)}p_{\yng(1,2,3)}) = (p_{\yng(1)}*_q p_{\yng(1,2)})p_{\yng(1,2,2)} + p_{\yng(1,2)}(p_{\yng(1)}*_qp_{\yng(1,2,2)}) - (p_{\yng(1)}*_q p_{\yng(1,1)})p_{\yng(1,2,3)} - p_{\yng(1,1)}(p_{\yng(1)}*_q p_{\yng(1,2,3)})\\
    &= p_{\yng(1,2,1)}p_{\yng(1,2,2)} + p_{\yng(1,2)}(p_{\yng(1,2,3)}+qp_{\yng(1,1)}) - (p_{\yng(1,1,1)}+p_{\yng(1,2)})p_{\yng(1,2,3)} - p_{\yng(1,1)}qp_{\yng(1,2)} = p_{\yng(1,2,1)}p_{\yng(1,2,2)} - p_{\yng(1,1,1)}p_{\yng(1,2,3)}\\
    \Delta(\phiGLS[3]) &= \Delta(p_{\yng(1,2,3)}) = qp_{\yng(1,2)}.
\end{align*}
Thus we obtain the potential
\[\Yboxdim{3pt} 
\pot_\can = \frac{\Delta(\Phi)}{\Phi} = \frac{p_{\yng(1)}}{p_{\varnothing}} + \frac{p_{\yng(1,1)}p_{\yng(1,1,1)} - p_{\varnothing}p_{\yng(1,2,2)}}{p_{\yng(1)}p_{\yng(1,1,1)} - p_{\varnothing}p_{\yng(1,2,1)}} + \frac{p_{\yng(1,2,1)}p_{\yng(1,2,2)} - p_{\yng(1,1,1)}p_{\yng(1,2,3)}}{p_{\yng(1,2)}p_{\yng(1,2,2)} - p_{\yng(1,1)}p_{\yng(1,2,3)}} + q\frac{p_{\yng(1,2)}}{p_{\yng(1,2,3)}}
\]
which agrees exactly with our construction. While the uniform presentation of the superpotential appears much more elegant than the previous presentation, it severely complicates the expressions by introducing a large amount of extraneous cancellation as seen even in the relatively small example above. For this reason among others, we prefer in our construction to work with the derivations $\delta_j$ which yield Pl\"ucker coordinate expressions without any cancellations, even though the two expressions turn out to be equivalent. In fact, we also have an equivalence at the level of individual terms terms via $\Delta(\cD_{\is}) = \delta_{\sigma_k(\is)}(\cD_{\is})$, although this equality is false for general polynomials in the Pl\"ucker coordinates. Thus, we view the derivations $\delta_{\sigma_k(\is)}$ as parts of the full derivation $\Delta$ which are optimized for simplicity of computation by eliminating extraneous cancellations.
\end{rem}

\begin{rem}\label{rem:which_anticanonical}\Yboxdim{3pt}
Let $p_1\dots p_d$ be a monomial in Pl\"ucker coordinates of degree $d$. Then for the action of the derivation $\Delta$ defined above in Remark \ref{rem:quantum derivation}, we obtain by commutivity and associativity of the quantum product that
\[\Yboxdim{3pt}\Delta(p_1\dots p_d) = \sum_{i=1}^d p_1\dots p_{i-1}(p_{\yng(1)}*_q p_i)p_{i+1}\dots p_d = \sum_{i=1}^d p_{\yng(1)}*_q(p_1\dots p_d) = d\,p_{\yng(1)}*_q(p_1\dots p_d),\]
where all products are now interpreted to be the quantum product inside of $qH^*(X)$ via the geometric Satake correspondence as in the previous Remark \ref{rem:quantum derivation}. 
In particular, we compute for the individual terms $\pot_{\is}$ of the potential $\pot_\can$ that
\[
\pot_{\is} = \frac{\Delta(\cD_{\is})}{\cD_{\is}} = \frac{\deg(\cD_{\is})~\p_{\yng(1)}*_q \cD_{\is}}{\cD_{\is}} = \cis\,\p_{\yng(1)}
\]
and thus that if we interpret $\pot_\can$ inside of $qH^*(X)$, we obtain
\[
\pot_\can=\Bigl(1 + \tsum[_{\is\in[n]}] \cis\Bigr)\p_{\yng(1)} = \mathrm{index}(\cmX)\p_{\yng(1)}.
\]
With this interpretation, we may interpret $\pot_\can$ in some sense as an incarnation of the anticanonical class of $\cmX$ inside of $qH^*(X)$. We note that slightly different observations were made in previous works by Pech-Rietsch (\cite[\S1]{Pech_Rietsch_Lagrangian_Grassmannians} and \cite[\S9]{Pech_Rietsch_Odd_Quadrics}). Furthermore, from our perspective, the reason the expression for $\pot_\can$ is particularly simple in the case of Grassmannians $\Gr(k,n)$ has to do with the fact that every fundamental weight in type $\LGA_{n-1}$ is minuscule, so that each term of the potential is simply a ratio of Pl\"ucker coordinates rather than the general situation which is instead a ratio of higher degree polynomials in Pl\"ucker coordinates. 
\end{rem}

\begin{rem}
We may consider $\mX=\Pi(R_{w_P}^{w_0})$ as the projection of the Richardson variety $R_{w_P}^{w_0}$, which is defined to be the closure of the intersection of opposite Schubert cells $\Sigma_{w_P}$ and $\Sigma^{w_0}$ in the full flag manifold $\dB\backslash \dG$. For more details on projections of Richardson varieties, see \cite{KLS-projections-richardson}. It is known (see \cite[Lemma 5.4]{Lam_Templier_The_mirror_conjecture_for_minuscule_flag_varieties}) that the sum of the classes of those projected Richardson varieties which are hypersurfaces inside $\Pi=\Pi(R_{w_P}^{w_0})$ forms an anticanonical divisor. 

In our context, the hypersurface projected Richardson varieties are given by
\[\Pi_0=\Pi(R_{s_k w_P}^{w_0}), \Pi_1=\Pi(R_{w_P}^{s_1w_0}),\dots, \Pi_n=\Pi(R_{w_P}^{s_nw_0}).\] We consider $\Pi_{\is}$ for $\is\in[n]$ as associated to the correspondingly labeled node of the Dynkin diagram of $\dG$. We do not, however, associate such an interpretation to $\Pi_0$. Although we will not use this, we believe that each hypersurface $\Pi_{\is}$ for $\is\in[n]$ is defined by $\{\cD_{\is}=0\}$ while $\Pi_0$ is defined by $\{p_\varnothing=0\}$. 
\end{rem}

We conclude with a brief outline of this article. We start by fixing notation and conventions as well as recalling relevant Lie-theoretic and order-theoretic background in Section \ref{sec:prelims}. We then prove several key components of our main theorem in Sections \ref{sec:toric-monomials}, \ref{sec:denominators}, and \ref{sec:numerators} before assembling all the pieces into our main result in Section \ref{sec:main_proof}. In Appendix \ref{app:exceptional} we apply our construction to four exceptional cases directly and in Appendix \ref{app:type_dependent_LGs} we show how to recover LG models previously constructed with type-dependent methods using our type-independent construction. 

%% file: 2prelims.tex
\section{Preliminaries}\label{sec:prelims}
In this section we start by summarizing all conventions and previous results that we will need for our work. We begin with some Lie-theoretic conventions in Section \ref{sec:conventions} followed by the results we will be building on in Section \ref{sec:preliminary_results}, including the Lie-theoretic and Laurent polynomial LG-models of \cite{Rietsch_Mirror_Construction} and \cite{Spacek_LP_LG_models} respectively. Afterwards, we will introduce minuscule posets, their order ideals, and how their combinatorics can be used in our context in Section \ref{sec:min_posets}.

\subsection{Conventions and notation}\label{sec:conventions}
In this section we fix our notations and conventions and recall some facts from Lie theory. Unless otherwise stated, we will always use the notation and conventions fixed below. We write $[N]=\{1,2,\ldots,N\}$ for any $N\in\N$.

Let $\G$ be a simple and simply-connected complex algebraic group of rank $n$ and fix a maximal torus $\torus$ and Borel subgroup $\borelp\supset\torus$. We can write $\borelp=\torus\unip$ for a unipotent subgroup $\unip$, and we write $\unim$ for the opposite unipotent subgroup (i.e.~such that $\borelm=\torus\unim$ satisfies $\borelp\cap\borelm=\torus$). Recall that the parabolic subgroups of $\G$ containing $\borelp$ correspond one-to-one to subsets $I\subset[n]$, and that the Levi subgroup of the parabolic subgroup $\P_I$ is semisimple of Dynkin type given by \emph{removing} the vertices in $I$ from the Dynkin diagram of $\G$. In particular, the Weyl group $\weylj[I]$ of (the Levi subgroup of) $\Pj[I]$ is realized as the subgroup of the Weyl group $\weyl$ of $\G$ generated by all simple reflections \emph{not} in $I$; we write $\woj[I]\in\weylj[I]$ and similarly $\wo\in\weyl$ for the respective longest elements. Now, we note that these longest elements define (possibly trivial) Dynkin diagram involutions $\DSi[I]$ by $\woj[I]\cdot\fwt[i] = -\fwt[{\DSi[I]}(i)]$ where $\fwt[i]$ is the fundamental weight of the Levi subgroup of $\Pj[I]$ and $i\in[n]\setminus I$; we extend $\DSi[I]$ to $[n]$ by $\DSi[I](i)=i$ for $i\in I$, and write $\DS=\DSi[\varnothing]$.

The pair $(\G,\torus)$ defines the \emph{Langlands dual pair} $(\dG,\dtorus)$ that  has character and cocharacter lattices interchanged. The parabolic subgroups of $\dG$ form exactly the Langlands dual pairs to $(\Pj[I],\torus)$, so we denote these simply by $\dPj[I]$. This holds in particular for $\borelp=\Pj[{[n]}]$ the Langlands dual of which we denote $\dborelp$; together with $\dtorus$ this defines the opposite Borel subgroup $\dborelm$ and the unipotent subgroups $\dunip$ and $\dunim$. Writing $\roots$ resp.~$\droots$ for the roots resp.~coroots of $\G$ with respect to $\torus$, we note that they form the \emph{coroots} resp.~\emph{roots} of $\dG$ with respect to $\dtorus$. Similarly, for $\sroots = \{ \sr_1, \ldots, \sr_n\} \subset\roots$ the base of simple roots for $\G$ fixed by $\borelp$, the simple coroots $\sdroots= \{\sdr_1, \ldots, \sdr_n\}\subset\droots$ form a base of simple roots for $\dG$; analogously fundamental weights $\fwt$ and coweights $\dfwt$ for $\G$ are fundamental coweights and weights for $\dG$. It is in particular clear from this that the Weyl groups of $\G$ and $\dG$ are identical, so we simply write $\weyl$ for both. Finally, we remark that assuming $\G$ is simply-connected implies that $\dG$ is \emph{adjoint}.

We will also need to refer to the Lie algebras of the algebraic groups introduced above: we will write $\dg$, $\dparab_I$, $\dbp$, $\dup$, and $\dcartan$ for the Lie algebras of $\dG$, $\dP_I$, $\dborelp$, $\dunip$ and $\dtorus$, respectively. The simple Lie algebra $\dg$ has a root space decomposition (or Cartan decomposition) of the form $\dg=\dcartan\op\bigoplus_{\drt\in\droots}\dg_\rt$, and we choose a set of Chevalley generators $(\dChe_i,\dChh_i,\dChf_i)_{i\in[n]}$ for $\dg$ satisfying the Serre relations, where $\dChe_i\in\dg_{\sdr_i}$, $\dChf_i\in\dg_{-\sdr_i}$ and $\dChh_i=[\dChe_i,\dChf_i]\in\dcartan$. Note that $\dbp$ is generated by $\{\dChe_i,\dChh_i~|~i\in[n]\}$ and that $\dparab_I$ is generated by $\{\dChe_i,\dChh_i,\dChf_j~|~i,j\in[n],j\notin I\}$. The Chevalley generators define one-parameter subgroups
\[
\dx_i(a) = \exp(a\,\dChe_i)\in\dunip \qand \dy_i(a) = \exp(a\dChf_i)\in\dunim \qfor a\in\C.
\]
These subgroups in turn define representatives of $\weyl$ in $\dG$ by assigning to a simple reflection $s_i\in\weyl$ the elements
\[
\ds_i = \dx_i(+1)\,\dy_i(-1)\,\dx_i(+1) \qand \bs_i = \dx_i(-1)\,\dy_i(+1)\,\dx_i(-1),
\]
and extending this to $\dw,\bw\in\dG$ for a general word $w\in\weyl$ by a choice of reduced expression for $w$. Note that the maps $\weyl\to\dG$ given by $w\mapsto\dw$ and $w\mapsto\bw$ are \emph{not} homomorphisms: $s_i^2=e\in\weyl$ but $\ds_i^2\in\dtorus$ is nontrivial. 

Since $\dG$ is adjoint, not all fundamental representations of $\dg$ give rise to representations of $\dG$; hence, we will have to consider the \emph{universal cover} $\udG$ of $\dG$. (Note that for simply-laced Lie types, $\udG\cong\G$.) We will analogously write $\udPj[I]$ etc.~for the corresponding universal covers. Since $\udunip\cong\dunip$ and $\udunim\cong\dunim$, we will identify $\dunip$ and $\dunim$ with the universal covers in $\udG$; this yields preferred lifts of the elements $\ds_i$ and $\bs_i$ to $\udG$ as the products of the uniquely lifted one-parameter subgroups. 

To finish up, we will need a small observation regarding the representatives $\ds_i$ and $\bs_i$ in $\udG$ of the simple reflection $s_i\in\weyl$: if $\wt$ is some weight of $\udG$ mapped to $\wt-d\sdr_i$ by $s_i$ with $d\in\Z$, then for any \emph{extremal} vector $\wtvmu$ of weight $\wt$ in an irreducible representation the actions of these representatives are related by 
\begin{equation}
\ds_i\cdot\wtvmu = (-1)^d\bs_i\cdot\wtvmu = \begin{cases}
    (-1)^d d! (\dChf_i)^d\cdot\wtvmu, & \text{if~} d\ge0; \\
    d! (\dChe_i)^d\cdot\wtvmu, & \text{if~} d\le0; \\
\end{cases}
\label{eq:ds_bs_dChe_dChf_relation}
\end{equation}
see e.g.~\cite[Section 7.1]{GLS_Kac_Moody_groups_and_cluster_algebras}. By extension (i.e.~taking a reduced expression and repeated application), we find for a general element $w\in\weyl$ and extremal weight vector $\wtvmu$ that
\begin{equation}
\dw\cdot\wtvmu = (-1)^h\bw\cdot\wtvmu \qwhere h=\height\bigl(w\cdot\wt-\wt\bigr)\in\Z.
\label{eq:dw_bw_relation}
\end{equation}

\subsection{Related Landau-Ginzburg models and results}\label{sec:preliminary_results}
We now summarize the Lie-theoretic construction of LG-models in  \cite{Rietsch_Mirror_Construction} that applies to general homogeneous spaces. Afterwards, we specialize to cominuscule homogeneous spaces $\X=\G/\P_k$ and introduce Pl\"ucker coordinates on their minuscule Langlands dual homogeneous spaces $\cmX=\dP_k\backslash\dG$ in Section \ref{ssec:min_hom_spaces_and_plucker_coords}, followed by a summary of the presentation of the coordinate ring of the related unipotent cell as given in \cite{GLS_Kac_Moody_groups_and_cluster_algebras} in Section \ref{ssec:Coordinate_ring_of_dunimP}. Finally, we present the Laurent polynomial model of \cite{Spacek_LP_LG_models} in Section \ref{ssec:LP_LG_model}.

\subsubsection{Lie-theoretic Landau-Ginzburg model}\label{ssec:Lie_model}
We will now introduce the Lie-theoretic models defined in \cite{Rietsch_Mirror_Construction}, using the presentation as given in \cite{Pech_Rietsch_Lagrangian_Grassmannians}. Fix an \emph{arbirtrary} (i.e.~not necessarily cominuscule) homogeneous space $\X=\G/\P$, and denote by $\weylp=\weylj[I]\subset\weyl$ the parabolic Weyl subgroup corresponding to $\P=\Pj[I]$. Consider the subvariety
\[
\decomps = \dborelm\bwo^{-1}\cap\dunip\invdtorus\bwop\dunim ~~\subset~~\dG,
\]
sometimes called the \emph{geometric crystal associated to $(\dG,\dP)$} cf.~\cite{Lam_Templier_The_mirror_conjecture_for_minuscule_flag_varieties}, where we denoted $\invdtorus=(\dtorus)^{\weylp}\subset\dtorus$ for the locus of $\dtorus$ invariant under the action of $\weylp$ by conjugation, and $\wo\in\weyl$ and $\wop\in\weylp$ for the respective longest elements.

\begin{rem}
We followed our convention in \cite{Spacek_Wang_exceptional} by multiplying $\decomps$ as defined in \cite{Rietsch_Mirror_Construction} on the right by $\bwo^{-1}$.
\end{rem}

Note that $\decomps$ can be considered as those elements $z\in\dborelm\bwo^{-1}$ that can be decomposed as $z=u_+t\bwop u_-$ for $u_+\in\dunip$, $t\in\invdtorus$ and $u_-\in\dunim$. Furthermore, the $u_-$ can be restricted to lie in a specific unipotent cell as follows. Let
\begin{equation}
\dunimP = \dunim\cap\dborelp(\dwP)^{-1}\dborelp~~\subset~~\dunim
\label{eq:dunimP}
\end{equation}
where $\wP\in\cosets$ is the minimal coset representative of $\wo\weylp$, i.e.~the Weyl group element satisfying $\wo=\wP\wop$ with $\ell(\wo)=\ell(\wP)+\ell(\wop)$.
\begin{lem}[\cite{Pech_Rietsch_Odd_Quadrics}, Section 4]\label{lem:decomps_iso_dunimp_times_invdtorus}
Every $z\in\decomps$ is \emph{uniquely} decomposed as $z=u_+t\bwop u_-$ with $u_+\in\dunip$, $t\in\invdtorus$ and $u_-\in\dunimP$. More precisely, given $(u_-,t)\in\dunimP\times\invdtorus$, there exists a \emph{unique} $u_+\in\dunip$ such that $u_+t\bwop u_-\in\decomps$, and the map $\dunimP\times\invdtorus\to\decomps$ sending $(u_-,t)\mapsto u_+t\bwop u_-$ is an isomorphism.
\end{lem}
\begin{rem}
Although proven in a specific case in \cite{Pech_Rietsch_Odd_Quadrics}, the methods carry over to the general case so the proof will be omitted here.
\end{rem}

The \emph{superpotential} $\pot_{\decomps}:\decomps\to\C$ is now defined as
\[
\pot_{\decomps}(z) = \sum_{i=1}^n \dChedual(u_+^{-1}) + \dChfdual(u_-), \qwhere z=u_+t\bwop u_-,
\]
and where $\dChedual$ resp.~$\dChfdual$ are the unique characters mapping $\dx_i(a)$ resp.~$\dy_i(a)$ to $a$ and the other one-parameter subgroups to $0$. (Equivalently, these can be defined as follows: the Chevalley generators $\{\dChe_i~|~i\in[n]\}$ generate a basis for $\dup$, which in turns gives rise to a PBW basis for the completed universal enveloping algebra $\dCUEAp\supset\dup$; let $\dChedual$ be the element dual to $\dChe_i$ in the dual basis; the maps $\dChedual:\dunip\to\C$ are finally obtained using the embedding $\dunip\hookrightarrow\dCUEAp$.) 

The results in \cite{Rietsch_Mirror_Construction} now state that
\begin{thm}[\cite{Rietsch_Mirror_Construction}, Theorem 4.1]
There exists an isomorphism
\[
qH^*(\X)_\loc \cong \C[\decomps]/\Jac{\decomps},
\]
where the left-hand side is the small quantum cohomology of $\X$ localized at the quantum parameters, and the right-hand side is the quotient of the coordinate ring of $\decomps$ by the ideal generated by the derivatives of $\pot_{\decomps}$ along $\dunimP$.
\end{thm}

The Lie-theoretic mirror model is now obtained using an isomorphism $\RisoZ:\mX_\Lie\times\invdtorus\to\decomps$ where
\[
\mX_\Lie = \dRichard_{\wop,\wo} = \bigl(\dborelp\wop\dborelm\cap\dborelm\wo\dborelm\bigr)/\dborelm ~~\subset~~\dG/\dborelm
\]
is the open Richardson variety associated to $(\wop,\wo)$. (For a discussion of $\RisoZ$ in the current notation, see Remark 3.1 of \cite{Spacek_LP_LG_models}.) Writing $\pot_\Lie=\RisoZ^*(\pot_{\decomps})$, we obtain
\begin{equation}
qH^*(\X)_\loc \cong \C[\mX_\Lie\times\invdtorus]/\Jac{\Lie},
\label{eq:Lie_LG_model_iso}
\end{equation}
where $\Jac{\Lie}$ is the ideal generated by the derivatives along $\mX_\Lie$. In other words, $(\mX_\Lie,\pot_\Lie)$ forms a \emph{Landau-Ginzburg model} for $\X$. 

We now specialize to the case that $\X=\G/\P$ for $\P=\P_k=\P_{\{k\}}$ a maximal parabolic subgroup (i.e.~when $\X$ has Picard rank $\rho=1$). In this case, $\dG$ being an adjoint linear algebraic group implies that the simple coroot $\sdr_k$ gives an isomorphism $\invdtorus\overset{\sim}{\longrightarrow}\C^*$. This isomorphism allows us to work more directly with a formal parameter $q\in\C^*$ instead of invariant torus elements. This approach is indeed justified as the mirror symmetry isomorphism in \eqref{eq:Lie_LG_model_iso} identifies $\invdtorus$ with the quantum parameter space of the quantum cohomology, and $\sdr_k$ effectively splits this identification in this specialized case.

This allows us to move the focus away from $\decomps$ to $\dunimP$. First, Lemma \ref{lem:decomps_iso_dunimp_times_invdtorus} can be restated as
\begin{cor}\label{cor:unique_u+_for_u-_for_fixed_q}
    For any fixed $q\in\C^*$ let $t\in\invdtorus$ be such that $\sdr_k(t)=q$. Then, every $u_-\in\dunimP$ allows a unique $u_+\in\dunip$ such that $u_+t\bwop u_-\in\decomps$.
\end{cor}
Similarly, the superpotential $\pot_{\decomps}$ can be expressed as:
\begin{df}\label{df:Lie_potential_on_dunimP}
    For any fixed $q\in\C^*$, let $\pot_{\dunimP}:\dunimP\to\C$ be the function given by
    \[
    \pot_{\dunimP}(u_-) = \sum_{i=1}^n \dChedual(u_+^{-1}) + \dChfdual(u_-),
    \]
    where $u_+\in\dunip$ is the unique element determined by $q\in\C^*$ and $u_-\in\dunimP$ of Corollary \ref{cor:unique_u+_for_u-_for_fixed_q}.
\end{df}
Note the shift from working over $\decomps$ to working over $\dunimP$ for fixed $q\in\C^*$ makes explicit that we do not think of $\pot_{\dunimP}$ as a morphism $\dunimP\times\C^*\to\C$ but rather as a $\C^*$-family of morphisms $\dunimP\to\C$; this is in line with the fact that the Jacobi ideals $\Jac{\decomps}$ and $\Jac{\Lie}$ are only generated by derivatives along $\dunimP$ and $\mX_\Lie$ respectively.

\subsubsection{Minuscule homogeneous spaces and their Pl\"ucker coordinates}\label{ssec:min_hom_spaces_and_plucker_coords}
The Lie-theoretic mirror models introduced in the previous subsection hold for any projective homogeneous space. However, in the case of \emph{cominuscule} homogeneous spaces, these models can be reformulated using natural coordinates on the mirror varieties. From now on, we will restrict to this case; that is, we assume that $\X$ can be written as $\X=\G/\P$ for $\P=\P_k=\Pj[\{k\}]$ a maximal parabolic subgroup such that the fundamental weight $\fwt[k]$ of $\G$ is \emph{cominuscule}. Recall that a fundamental weight $\fwt[k]$ is \emph{cominuscule} if $\llan\dfwt[k],\rt\rran\in\{-1,0,+1\}$ for any root $\rt\in\roots$, where $\llan\cdot,\cdot\rran$ denotes the dual coupling between the character and cocharacter lattices; in this case, $\dfwt[k]$ is called a \emph{minuscule} coweight and $\X$ is called a \emph{cominuscule homogeneous space} or a \emph{cominuscule (generalized) flag variety}. Note that cominusculity of $\fwt[k]$ is equivalent to $\sdr_k$ having coefficient $1$ in the highest coroot $\sdro$, i.e.~$\llan\sdro,\fwt[k]\rran=1$. 

Cominusculity of $\fwt[k]$ can also be considered as minusculity of the weight $\dfwt[k]$ of $\dG$. Hence, the homogeneous space $\cmX=\dP_k\backslash\dG$ is often called \emph{minuscule}. However, since $\dG$ is adjoint, the representation with highest weight $\dfwt[k]$ of $\dg$ does not necessarily induce a representation of $\dG$. Hence, we will often consider $\cmX=\udP_k\backslash\udG$ instead (note that $\dP_k\backslash\dG\cong\udP_k\backslash\udG$, allowing us to identify these descriptions of $\cmX$). Finally, for simplicity, we often call an index (i.e.~vertex of the Dynkin diagram) $i\in[n]$  \emph{(co-) minuscule} if the corresponding fundamental weight is (co-) minuscule; the context will clarify whether we consider a fundamental weight of $\G$ or $\dG$. In Table \ref{tab:CominusculeSpaces} we have listed all minuscule and cominuscule weights, together with the associated homogeneous spaces.

\input{table_of_com_hom_spaces.tex}

Minuscule homogeneous spaces allow for the definition of a natural set of projective coordinates called \emph{(generalized) Pl\"ucker coordinates}. The fact that $\dfwt[k]$ is minuscule implies that the fundamental representations $V(\dfwt[k])$ and $\PluckerRep$ of $\udG$ are \emph{minuscule}. (Note that these are dual representations; depending on the type, they can also be isomorphic.) A representation is minuscule if the Weyl group acts transitively on its weights.

We will consider the lowest weight of $\PluckerRep$, namely $-\dfwt[k]$. Clearly, the stabilizer of $\weyl$ acting on $-\dfwt[k]$ is $\weylp=\weylj[k]$; hence, writing $\cosets$ for the set of minimal representatives of left cosets of $\weylp$ in $\weyl$, we find that $\cosets\cdot(-\dfwt[k])$ forms the set of weights of the representation $\PluckerRep$. In particular, choosing a lowest weight vector $\lwtv\in\PluckerRep$, the set of vectors $\{\wtv{w}=\dw\cdot\lwtv~|~w\in\cosets\}$ forms a weight basis for the representation. The actions of $\dw$ and $\bw$ for $w\in\weyl$ on this weight basis can be described explicitly; see e.g.~\cite[Section 3]{Green_Reps_from_polytopes} or \cite[Section 3]{Geck_Minuscule_weights_and_Chevalley_groups}, or for the summarized results phrased in the current notation see \cite[Section 4]{Spacek_LP_LG_models}. We will in particular need the following result:
\begin{lem}[\cite{Spacek_LP_LG_models}, Corollary 4.3 and Lemma 4.4(ii) and (iii)] 
\label{lem: Spacek_LPLG} 
Let $\hwtv\in\dfwtrep[k']$ be the highest vector in a minuscule representation of a linear algebraic group $H$, and let $W'\subset W$ be the set of minimal coset representatives of the Weyl subgroup $W_{k'}=\lan s_i~|~i\neq k'\ran\subset W$, then:

\begin{romanize}
    \item If the minimal coset representative $w\in W'$ has reduced expression $w=s_{j_1}\cdots s_{j_{\ell(w)}}$, then its representatives act on the highest weight vector by $\bw\cdot\hwtv = \dChf_{j_1}\cdots\dChf_{j_{\ell(s)}}\cdot\hwtv$ and $\dw\cdot\hwtv=(-1)^{\ell(w)}\dChf_{j_1}\cdots\dChf_{j_{\ell(w)}}\cdot\hwtv$.

    \item Conversely, if $\dChf_{j_1}\cdots\dChf_{j_\ell}\cdot\hwtv$ is nonzero of weight $\wt$, then $s_{j_1}\cdots s_{j_\ell}\in W'$ is a reduced expression for the (unique) element mapping the highest weight $\dfwt[k']$ to $\wt$.
\end{romanize}\end{lem}

We now take the basis $\{\wtv{w}^*~|~w\in\cosets\}$ for $\PluckerDualRep$ dual to the weight basis just constructed; this forms a weight basis as well with $\lwtv^*$ of highest weight $\dfwt[k]$ under \emph{right}-action of $\udG$. In particular, $\udP_k$ stabilizes the line $\C\lwtv^*$, yielding an embedding 
\begin{equation}
\cmX\hookrightarrow\PPluckerDualRep:\quad\dP_kg\mapsto\C\lwtv^*\cdot\dP_kg,
\label{eq:Plucker_embedding}    
\end{equation}
where we used that $\cmX=\udP_k\backslash\udG$; this embedding is called the \emph{(generalized) Pl\"ucker embedding}.
\begin{df}\label{df:Plucker_coords}
The projective coordinates $(\p_w~|~w\in\cosets)$ on $\cmX$ defined by
\[
\p_w(g) = \lwtv^*(g\cdot\wtv{w}) \qfor g\in\udG
\]
are called \emph{(generalized) Pl\"ucker coordinates}. We will refer to the maps $\p_w:\udG\to\C$ as Pl\"ucker coordinates as well. We will frequently suppress the arguments of these coordinates if it is clear on which domain we consider them.
\end{df}

\subsubsection{The coordinate ring of \texorpdfstring{$\dunimP$}{the unipotent cell} for cominuscule \texorpdfstring{$\G/\P$}{G/P}}\label{ssec:Coordinate_ring_of_dunimP}
The isomorphism $\decomps\cong\dunimP\times\invdtorus$ discussed in Lemma \ref{lem:decomps_iso_dunimp_times_invdtorus} will turn up again to relate the Lie-theoretic mirror model to our type-independent canonical mirror models. In particular, it will allow us to use a description of the coordinate ring of $\dunimP$. In \cite{GLS_Kac_Moody_groups_and_cluster_algebras}, the coordinate rings of unipotent cells are described as being a polynomial ring generated by a ``dual PBW basis'' localized at certain \emph{generalized minors} in the sense of \cite{Fomin_Zelevinsky_Double_Bruhat_Cells_and_Total_Positivity}.  However, \cite{GLS_Kac_Moody_groups_and_cluster_algebras} works with ``upper'' unipotent groups, i.e.~those that have Lie algebras contained in the positive root spaces, whereas $\dunimP$ is ``lower'' unipotent; we will use the transposition anti-isomorphism to translate between the two conventions. Recall from \eqref{eq:dunimP} that $\dunimP = \dunim\cap\dborelp(\dwP)^{-1}\dborelp$, which is mapped under the transposition map to $\dunipP = \dunip\cap\dborelm\dwP\dborelm$ (recall that $\dw^T=\dw^{-1}$). Proposition 8.5 of \cite{GLS_Kac_Moody_groups_and_cluster_algebras} now applies to $\dunipP$ directly.

Instead of presenting the construction of the dual PBW basis for $\dunipP$, we will use the type-independent characterization in terms of Pl\"ucker coordinates given in Proposition 3.8 of  \cite{Spacek_Wang_exceptional}: Consider the set $\wPpdroots\subset\pdroots$ of positive roots mapped to negative roots by $\wP$ (equivalently, this is the set of positive roots with coefficient $1$ in front of $\sdr_k$ in their decomposition into simple roots; see \cite[Lemma 3.5]{Spacek_Wang_exceptional}), then for every $\drt\in\wPpdroots$ there exists a unique $w_{\drt}\in\cosets$ such that $w_{\drt}(-\dfwt[k])=-\dfwt[k]+\drt$; each of these minimal coset representatives $w_{\drt}$ determines a Pl\"ucker coordinate $\p_{\drt}=\p_{w_{\drt}}$ on $\dunim$, so writing $\p_{\drt}^T:u_+\mapsto\p_{\drt}(u_+^T)$, we obtain a set of ``transposed'' Pl\"ucker coordinates $\{\p_{\drt}^T ~|~ \drt\in\wPpdroots\}$ on $\dunipP$; by Proposition 8.5 of \cite{GLS_Kac_Moody_groups_and_cluster_algebras} and Proposition 3.8 of \cite{Spacek_Wang_exceptional}, this set forms a generating set for the coordinate ring of $\dunipP$.

In contrast to the dual PBW basis, the generalized minors appearing in the coordinate ring of $\dunipP$ can be stated concisely as:
\[
\minor_{\dfwt,\wPinv(\dfwt)}(u_+) = \Bigl\lan u_+\dwPinv\cdot\hwt,~\hwt\Bigr\ran_{\dfwt} \qfor i\in[n],
\]
where $\hwt$ is a choice of highest weight vector in the fundamental representation $V(\dfwt)$, and $\lan v,\hwt\ran_{\dfwt}$ denotes the coefficient in front of $\hwt$ in the decomposition of $v$ with respect to a weight basis of the representation. (Note that this is independent of the choice of basis.)

Now, to translate these results to $\dunimP$, we pull these functions back along the transposition anti-isomorphism, yielding the generating set of Pl\"ucker coordinates $\{\p_{\drt}~|~\drt\in\wPpdroots\}$ and the transposed generalized minors $u_-\mapsto\minor_{\dfwt,\wPinv(\dfwt)}(u_-^T)$ for $i\in[n]$. We will translate each of these generalized minors calculated in $V(\dfwt)$ into one calculated in $V(\dfwt[\DS(i)])$ to remove the transpose maps.

For this, we first note that that evaluating an element $f\in V^*$ in some dual representation at an element $g\cdot v\in V$ is the same as evaluating the element $g^T\cdot f\in V^*$ at $v\in V$. Subsequently, we observe that the dual representation $V(\dfwt)^*$ is a \emph{right}-representation, which is mapped under the inversion anti-isomorphism to the \emph{left}-representation $V(\dfwt[\DS(i)])$. Note, however, that projection to the highest weight vector $\hwt$ is a \emph{lowest} weight element of the dual representation (of weight $-\dfwt$). In conclusion, we obtain the relations:
\begin{equation}
\Bigl\lan g\cdot\hwt,~\hwt\Bigr\ran_{\dfwt} = \Bigl\lan \hwt,~g^T\cdot\hwt\Bigr\ran_{\dfwt} = \Bigl\lan g^{-T}\cdot\lwt[\DS(i)],~\lwt[\DS(i)]\Bigr\ran_{\dfwt[\DS(i)]}
\label{eq:flipping_generalized_minors}
\end{equation}

Applying this to the generalized minor in question (and recalling that $\dw^T=\dw^{-1}$ for any Weyl group element $w\in\weyl$), we find:
\[
\minor_{\dfwt,\wPinv(\dfwt)}(u_-^T) = \Bigl\lan u_-^T\dwPinv\cdot\hwt,~\hwt\Bigr\ran_{\dfwt} = \Bigl\lan u_-^{-1}\dwPinv\cdot\lwt[\DS(i)],~\lwt[\DS(i)]\Bigr\ran_{\dfwt[\DS(i)]}.
\]
Next, we note that $\dwPinv=\bw$ for $w=\wPinv$, and hence by \eqref{eq:dw_bw_relation} we find that $\dwPinv$ acts as $(-1)^h\dw=(-1)^h\bwPinv$ on the lowest weight vector (of weight $-\dfwt$), where 
\begin{equation}
h=\height\bigl(-\wPinv\cdot\dfwt+\dfwt\bigr).
\label{eq:gen_minor_minus_sign}
\end{equation}
Combining this with the fact that $\bwo=\bwP\bwop$ maps the highest weight vector to the lowest weight vector, we find that
\ali{
\minor_{\dfwt,\wPinv(\dfwt)}(u_-^T) 
&= (-1)^h\Bigl\lan u_-^{-1}\bwPinv\bwP\bwop\cdot\lwt[\DS(i)],~ \bwo\cdot\hwt[\DS(i)]\Bigr\ran_{\dfwt[\DS(i)]} \\
&= (-1)^h\Bigl\lan u_-^{-1}\bwop\cdot\lwt[\DS(i)],~ \bwo\cdot\hwt[\DS(i)]\Bigr\ran_{\dfwt[\DS(i)]}.
}
We recognize the right-hand side as the desired generalized minor to conclude
\[
\minor_{\dfwt,\wPinv(\dfwt)}(u_-^T) = (-1)^h\minor_{\wo\cdot\dfwt[j],\wop\cdot\dfwt[j]}(u_-^{-1}) \qfor j=\DS(i),
\]
and where $h$ is given in \eqref{eq:gen_minor_minus_sign}. As $\cdot^{-1}:\dunim\to\dunim$ is an anti-automorphism, the generalized minor is now truly taken on lower unipotent elements; however, in general $\cdot^{-1}$ does \emph{not} map $\dunimP$ to itself (this only holds if $\wPinv=\wP$, i.e.~when the Dynkin involution $\DS$ is trivial).

In conclusion, we obtain
\begin{prop}[\cite{GLS_Kac_Moody_groups_and_cluster_algebras}, Prop.~8.5 \& \cite{Spacek_Wang_exceptional}, Lem.~3.5]\label{prop:coord_ring_dunimP}
The coordinate ring of $\dunimP$ is given by
\[
\C[\dunimP] = \C\bigl[\p_{\drt}~\big|~\drt\in\wPpdroots\bigr]\bigl[\phiGLS[j]^{-1}~\big|~j\in[n]\bigr]
\]
where we recall that $\p_{\drt}=\p_{w_{\drt}}$ for $w_{\drt}\in\cosets$ the unique element sending $-\dfwt[k]$ to $-\dfwt[k]+\drt$, with $\wPpdroots$ the set positive roots with coefficient $+1$ in front of $\sdr_k$ when written as a linear combination of simple roots, and where
\[
\phiGLS[j](u_-)=(-1)^h\minor_{\wo\cdot\dfwt[j],\wop\cdot\dfwt[j]}(u_-^{-1}) = (-1)^h\Bigl\lan u_-^{-1}\bwop\cdot\hwt[j],~\lwt[j]\Bigr\ran_{\dfwt[j]}
\]
is a generalized minor as defined in \cite{Fomin_Zelevinsky_Double_Bruhat_Cells_and_Total_Positivity}.
\end{prop}
We have kept the coefficient $(-1)^h$ in $\phiGLS[j]$ to cancel out the additional minus signs arising from the inversion $u_-^{-1}$. This will be discussed further in Remarks \ref{rem:generalized_minors_sign_prelims} and \ref{rem:generalized_minors_sign}, but is immaterial for the presentation of the coordinate ring. One of the intermediate results in our construction of type-independent LG-models is the derivation of Pl\"ucker coordinate expressions for these generalized minors.

\subsubsection{Laurent polynomial expressions}\label{ssec:LP_LG_model}
We will show that the type-independent canonical mirror models are Landau-Ginzburg models by showing that they are isomorphic to these Lie-theoretic mirror models, taking advantage of the fact that the latter have type-independent descriptions when restricted to an open, dense algebraic torus defined as
\[
\opendecomps = \dborelm\bwo^{-1}\cap\dunip\invdtorus\bwop\opendunim ~~\subset~~\decomps,
\]
where $\opendunim\subset\dunimP$ (cf.~\cite{Pech_Rietsch_Odd_Quadrics}, Section 5.2, or \cite{Spacek_LP_LG_models}, Lemma 5.2) is the algebraic torus of elements $u_-\in\dunimP$ that allow a decomposition in the form
\begin{equation}
u_- = \dy_{r_1}(a_1)\cdots \dy_{r_\ellwP}(a_\ellwP) \qwith a_i\in\C^*
\label{eq:opendunim_decompostion_of_u-}
\end{equation}
and with the indices $(r_j)_{j\in[\ellwP]}$ forming a reduced expression $s_{r_\ellwP}\cdots s_{r_1}=\wP$. (Note that the sequence is reversed in the two expressions. Moreover, we have reversed the labeling of the reduced expression for $\wP$ compared with \cite{Spacek_LP_LG_models} and \cite{Spacek_Wang_exceptional}.) The torus $\opendunim$ does \emph{not} depend on which reduced expression of $\wP$ the sequence $(r_j)$ generates: due to the full commutativity of $\wP$ by \cite{Stembridge_Fully_commutative_elements_of_Coxeter_groups}, it suffices to consider two reduced expressions differing by the reversal of a pair of simple reflections, but it is easy to see (cf.~\cite{Rietsch_Mirror_Construction}, proof of Prop.~7.2) that $\dy_i(a)\dy_j(a')=\dy_j(a')\dy_i(a)$ if $s_is_j=s_js_i$. Hence, we will temporarily consider some \emph{fixed} reduced expression $\wP=s_{r_\ell}\cdots s_{r_1}$ to avoid confusion, even though the choice is immaterial. However, in Section \ref{sec:min_posets}, we will introduce a more convenient and combinatorial notation for these toric coordinates $\{a_i\}_{i\in{\ellwP}}$, which allows us to refer to specific coordinates without the need of a fixed reduced expression for $\wP$.

In addition to the minimal coset representative $\wP\in\cosets$ of $\wo\weylp$ for $\wo\in\weyl$ the longest element, we will also need the minimal coset representative $\wPPrime\in\cosets$ of $\wop s_k\weylp$ for $\wop\in\weylp$ the longest element. Together, these two elements define $\wPrime=\wP(\wPPrime)^{-1}\in\weyl$ of length $\ellwPrime$; denote by $\wPrimeSubExp=\{(i_j)_{j\in[\ellwPrime]}~|~s_{r_{i_{\ellwPrime}}}\cdots s_{r_{i_1}}=\wPrime\}$ the set of subexpressions for $\wPrime$ inside the \emph{fixed} reduced expression for $\wP$. (Note that we do \emph{not} fix a reduced expression for $\wPrime$.) Then we have the following result:
\begin{thm}[\cite{Spacek_LP_LG_models}, Thm 5.7]
\label{thm:LP_LG_model_original}
For cominuscule $\X=\G/\P_k$ and with the notation above, the restriction $\pot_{\opendecomps}$ of $\pot_{\decomps}$ to $\opendecomps$ is given in terms of toric coordinates as
\[
\pot_{\opendecomps}(z) = \Bigl(\tsum[_{i\in[\ellwP]}]a_i\Bigr) + q\frac{\sum_{(i_j)\in\wPrimeSubExp} \prod_{j\in[\ellwPrime]} a_{i_j}}{\prod_{i\in[\ellwP]} a_i}, \qwhere z=u_+t\bwop u_-
\]
with $u_-=\dy_{r_1}(a_1)\cdots\dy_{r_\ell}(a_\ell)\in\opendunim$ and $q=\sdr_k(t)\in\C^*$.
\end{thm}
\begin{rem}\label{rem:wPrime_is_min_coset_rep}
Although not necessary for this theorem, we prove in Lemma \ref{lem:wJprime_min_coset_rep} that $\wPrime=\wP(\wPPrime)^{-1}\in\weyl$ is a minimal coset representative; see Section 9 of the arXiv version of \cite{Spacek_LP_LG_models} for expressions for $\wPrime$ in all cases.
\end{rem}

Analogously to Section \ref{ssec:Lie_model}, we will simplify considerations by pulling back the superpotential to $\dunimP$, where the restriction to the torus $\opendunim\subset\dunimP$ gives us a Laurent polynomial potential, namely:
\begin{cor}
For $\X=\G/\P_k$ cominuscule, and any fixed $q\in\C^*$, the superpotential $\pot_{\dunimP}$ restricts on the dense algebraic torus $\opendunim$ to
\[
\pot_{\opendunim}(u_-) = \Bigl(\tsum[_{i\in[\ellwP]}]a_i\Bigr) + q\frac{\sum_{(i_j)\in\wPrimeSubExp} \prod_{j\in[\ellwPrime]} a_{i_j}}{\prod_{i\in[\ellwP]} a_i}, 
\]
where $u_-=\dy_{r_1}(a_1)\cdots\dy_{r_\ell}(a_\ell)\in\opendunim$ and $\wPrimeSubExp$ is defined above.
\end{cor}
Just like in Section \ref{ssec:Lie_model}, note that we do not think of $\pot_{\opendunim}$ as a morphism $\opendunim\times\C^*\to\C$ but rather as a $\C^*$-family of morphisms $\opendunim\to\C$.

In Corollary 8.12 of \cite{Spacek_LP_LG_models}, a combinatorial method is described to obtain the elements of $\wPrimeSubExp$ using a specific quiver defined in \cite{CMP_Quantum_cohomology_of_minuscule_homogeneous_spaces}; it turns out that these combinatorics are equivalent to the well-known combinatorics of \emph{minuscule posets}, so we will take some time to introduce these and present the above results in this language.

\subsection{Minuscule posets and local toric expressions}\label{sec:min_posets}
As we have alluded to previously, minuscule posets will provide a very effective combinatorial tool for our models, both for describing toric expressions, as well as for formulating the Pl\"ucker coordinate superpotential. Hence, we will spend some time introducing them in this section. We start with the basic order-theoretic definitions and results in Section \ref{ssec:Combinatorics_of_minuscule_posets}, and then use them in Section \ref{ssec:min_posets_and_LP_LG_model} to rephrase the Laurent polynomial LG-model in terms of embeddings of order ideals in minuscule posets. Finally, we use these embeddings to describe toric expressions of Pl\"ucker coordinates and generalized minors in \ref{ssec:min_posets_and_Pluckers}

\subsubsection{The combinatorics of minuscule posets}\label{ssec:Combinatorics_of_minuscule_posets}
It might seem that fixing a specific reduced expression for the minimal coset representative $\wP\in\cosets$ of the longest element $\wo\in\weyl$ is unnatural; indeed, by \cite{Stembridge_Fully_commutative_elements_of_Coxeter_groups} all reduced expressions can be obtained from a given one by commuting simple reflections, and the algebraic torus $\opendunim$ does not change under such modifications. Using the combinatorics of minuscule posets will allow us to let go of this fixed expression. Hence, we will introduce the relevant order-theoretic concepts here.

Firstly, recall that an \emph{order ideal} (or \emph{lower set}) is a subset $\sI$ of a poset (i.e.~partially ordered set) $\sP$ such that any element $x\in\sP$ such that $x\le y$ for some $y\in\sI$ satisfies $x\in\sI$. The order ideal generated by a subset $S\subset \sP$ is all $x\in\sP$ such that $x\le s$ for some $s\in S$. We will write this as $\sI(S)$. The \emph{join} $x\vee y\in\sP$ of $x,y\in\sP$ (if it exists) is the least upper bound of $x$ and $y$ in $\sP$; equivalently, it is the minimal element $x\vee y\in\sP$ such that the order ideal generated by it contains the union of the ideals generated by $x$ and by $y$. An element of $\sP$ is called \emph{join-irreducible} if it \emph{cannot} be written as the join of two other elements of $\sP$. Dually, the \emph{meet} $x\wedge y\in\sP$ of $x,y\in\sP$ (if it exists) is the greatest lower bound of $x$ and $y$ in $\sP$; equivalently, it is the largest element such that the ideal generated by it is contained in the intersection of the ideals generated by $x$ and by $y$. If every pair of elements of a (finite) poset $\sP$ has both a meet and join, then $\sP$ is called a \emph{lattice}. Moreover, if in a lattice the meet and join \emph{distribute} over each other (i.e.~$x\wedge(y\vee z) = (x\wedge y)\vee(x\wedge z)$), then $\sP$ is called a \emph{distributive lattice}. 

Distributive lattices appear naturally in the context of minuscule representations: Equipping the weights of a minuscule representation with the partial order given by $\wtu[1]>\wtu[2]$ if and only if $\wtu[1]-\wtu[2]$ is a positive linear combination of simple roots, we have the following result:
\begin{thm}[{\cite{proctor_bruhat_lattices}, Proposition 3 and Theorem 4}]
The poset of weights of the minuscule representation $\PluckerRep$ is isomorphic to the set $\cosets$ with the (left) Bruhat order under the bijection $w\mapsto w(-\dfwt[k])$. Moreover, this poset forms a finite distributive lattice.
\end{thm}
Recall that the (left) Bruhat order on $\cosets$ (or on $\weyl$) is the transitive closure of the covering relation $w\gtrdot w'$ if $w=s_iw'$ with $\ell(w)=\ell(w')+1$. Because of the previous result, we will also refer to the order on the weights of the minuscule representation as the Bruhat order. The structure of finite distributive lattices is in turn determined by their join-irreducible elements:
\begin{thm}[{\cite{birkhoff_representation}, Section 9}]
Let $\sL$ be a finite distributive lattice and $\sP$ the induced subposet of join-irreducible elements, then $\sL$ is isomorphic to the poset of order ideals of $\sP$ ordered by inclusion.
\end{thm}

It is easy to see when a weight for a finite-dimensional representation is join-irreducible: when there is exactly one simple reflection lowering the weight. Hence, each element of the poset of join-irreducible weights (i.e.~each box of the skew Young diagram) corresponds to a simple reflection---the one that lowers that weight. Because of this, we will label the boxes of the minuscule poset by the simple reflection $s_i$ lowering that join-irreducible weight; we write $\labl(\pb)=s_i$ if the box $\pb$ has label $s_i$. Note that with this labeling, the minuscule poset in fact coincides with the \emph{heap} associated to $\wP\in\cosets$ as defined in \cite{Stembridge_Fully_commutative_elements_of_Coxeter_groups}, and hence every order ideal coincides with the heap of the corresponding minimal coset representative. As both combinatorial objects are equivalent, we will favor order ideals of the minuscule poset over heaps of minimal coset representatives in this paper.

Sometimes we will need the index of the simple reflection labeling a given element of the minuscule poset instead of the reflection itself; in these cases we will write $\indx(\pb)=i$ for the index corresponding to the simple reflection $\labl(\pb)=s_i$.
 
The two theorems above tell us that we can describe all weights of a minuscule representation by order ideals of the poset of join-irreducible weights; such a poset is called a \emph{minuscule poset}. Minuscule posets simplify notation greatly as they can be drawn two-dimensionally: every element covers at most two elements and is covered by at most two elements. These posets are therefore usually drawn as skew Young diagrams with the top-left box representing the minimal element and a box drawn to the right or below a given box indicating that the given box is covered by the other. (Note that there exists a unique minimal element as there exists a unique weight covering the lowest weight in a fundamental weight representation, hence in particular in a minuscule representation.)

We will write $\minposet$ for the minuscule poset obtained from the fundamental representation $\PluckerRep$ of $\udG$, where $\cmX=\udP_k\backslash\udG$ is minuscule. (We suppress the ``$^\vee$'' since this data is ultimately fixed by $\P=\P_k\subset\G$ defining $\X$.) We write $\ideal\subideal\minposet$ if a subposet $\ideal\subset\minposet$ is an order ideal, and write $\ell(\ideal)$ for the number of boxes (i.e.~elements) in the poset $\ideal$. We hence obtain two ways to refer to the weights of the minuscule representation: by order ideals $\ideal\subideal\minposet$ of the associated minuscule poset $\minposet$, and by minimal coset representatives $w\in\cosets$. Note that these are compatible in the sense that adding a box labeled $s_i$ to a given order ideal $\ideal\subideal\minposet$ corresponds to multiplying the corresponding word $w\in\cosets$ by the simple reflection $s_i$ on the left. We see from this directly that a \emph{linear extension} of $\ideal\subideal\minposet$ (i.e.~completing the partial order to a total order) gives a \emph{reduced expression} of $w\in\cosets$. In particular, the number of boxes in the order ideal equals the length of the corresponding word, i.e.~$\ell(\ideal)=\ell(w)$. 

\begin{rem}\label{rem:order_ideals_to_min_coset_reps}
In fact, the correspondence between ideals of $\minposet$ and minimal coset representatives in $\cosets$ giving the same weight vector can be made direct: Given an ideal $\ideal\subideal\minposet$, choose a linear extension of the partial order and let $w\in\cosets$ be formed by multiplying the labels $s_i$ of the elements of $\ideal$ with the greatest element on the left and decreasing to the right; which we will denote as
\[
\labelprod{\ideal} = \prod_{\pb\in\ideal}\labl(\pb)\in\cosets.
\]
Conversely, given a minimal coset representative $w\in\cosets$, take a reduced expression $w=s_{i_1}\cdots s_{i_r}$, and let $\ideal\subideal\minposet$ be the poset with elements labeled by $s_{i_j}$ and with partial order the transitive closure of the covering relations $s_{i_j}\gtrdot s_{i_{j'}}$ if $j<j'$, $(s_{i_j}s_{i_{j'}})^2\neq e$ and $\forall j''$ with $j<j''<j'$, $s_{i_{j''}}$ is distinct from and commutes with $s_{i_j}$ and $s_{i_{j'}}$. (Compare the quiver defined in Definition 2.1 of \cite{CMP_Quantum_cohomology_of_minuscule_homogeneous_spaces}.) Due to the full-commutativity of elements of $\cosets$, these maps are independent of the choice of linear extension of the partial order and the reduced expression.

In particular, we defined the minimal coset representative $\wPPrime$ for the coset $w_Ps_kW_P$. The corresponding order ideal $\idealPPrime$ is the maximal order ideal of $\minposet$ containing exactly one occurrence of $s_k$. We will show in Corollary \ref{cor:quantum_term} that $\wPrime = \wP(\wPPrime)^{-1}$ is also a minimal coset representative. Its corresponding order ideal $\idealPrime$ will be isomorphic (respecting the labeling) to the complement of $\idealPPrime$ in $\minposet$.
\end{rem}

\input{example_min_reps_and_min_posets.tex}

The correspondence between elements of $\cosets$ and order ideals of $\minposet$ gives us a second way of denoting the natural weight basis of $\PluckerRep$: $\{v_\ideal~|~\ideal\subideal\minposet\}$ instead of $\{v_w~|~w\in\cosets\}$. We will favor the notation using minimal coset representatives except in explicit examples where the order ideal notation is more convenient. As for the lowest weight vector of $\PluckerRep$, we will favor the notation $\lwtv$ over either $\wtv{\varnothing}$ (where $\varnothing\subideal\minposet$) and $\wtv{e}$ (where $e\in\cosets$). This alternative notation also gives rise to a similar alternative for Pl\"ucker coordinates:
\begin{df}[Cf.~Definition \ref{df:Plucker_coords}]\label{df:Plucker_coords_posets}
The \emph{Pl\"ucker coordinates} $(\p_\ideal~|~\ideal\subideal\minposet)$ on $\cmX$ are defined by the maps 
\[
\p_\ideal(g) = \lwtv^*(g\cdot\wtv{\ideal}) \qfor g\in\udG
\]
(which are also called \emph{Pl\"ucker coordinates} by abuse of notation). Note that $\p_\ideal=\p_w$ where $\ideal\subideal\minposet$ and $w\in\cosets$ correspond to each other under Remark \ref{rem:order_ideals_to_min_coset_reps}. We will frequently suppress the arguments of these coordinates if it is clear on which domain we consider them.
\end{df}

We will end the introduction to minuscule posets with an overview of all the possible posets:
\begin{ex}\label{ex:list_of_minposets}
Comparing with Table \ref{tab:CominusculeSpaces}, we see that there are five infinite families of minuscule homogeneous spaces and two sporadic ones: one family each in types $\LGA_n$ (Grassmannians), $\LGB_n$ (maximal orthogonal Grassmannians in odd dimensions) and $\LGC_n$ (odd dimensional projective spaces); and two families in type $\LGD_n$ (even dimensional quadrics and maximal orthogonal Grassmannians in even dimensions; $k=n-1$ and $k=n$ give isomorphic families); as well as the minuscule varieties $\LGE_6/\P_6\cong\LGE_6/\P_1$ and $\LGE_7/\P_7$. Below we have given characteristic examples of each of the corresponding minuscule posets along with the corresponding ``rim-hook'' order ideals $\idealPPrime$ corresponding to the minimal coset representative $w''$ for $w_Ps_kW_P$ (red labels). \begin{equation*}
\label{eq:young_diagrams}
\Yboxdim{7pt}
\begin{array}{c}
\begin{array}{c||c|c|c|c|c|c|c}
\raisebox{40pt}{$\minposet$}&
\raisebox{40.5pt}{\scriptsize\young(\rth\rfo\rfi\rsi,\rtw345,\ro234)} &
\raisebox{27pt}{\scriptsize\young(\rfi,\rfo5,\rth45,\rtw345,\ro2345)}& 
\raisebox{27pt}{\scriptsize\young(\ro\rtw\rth\rfo\rfi,::::\rfo,::::\rth,::::\rtw,::::1)}& 
\raisebox{27pt}{\scriptsize\young(\ro\rtw\rth\rfo\rfi,:::\rsi\rfo,::::\rth,::::\rtw,::::1)}& 
\raisebox{27pt}{\scriptsize\young(\rsi,\rfo\rfi,\rth\rfo6,\rtw\rth45,\ro\rtw346)}& 
\raisebox{33.5pt}{\scriptsize\young(\rsi\rfi\rfo\rth\ro,::\rtw\rfo\rth,:::\rfi\rfo\rtw,:::65431)}& 
\raisebox{0pt}{\scriptsize\young(\rse\rsi\rfi\rfo\rth\ro,:::\rtw\rfo\rth,::::\rfi\rfo\rtw,::::\rsi\rfi\rfo\rth\ro,::::76543,:::::::24,::::::::5,::::::::6,::::::::7)}
\vspace{-2.8em}\\
\text{type} & A_6 & B_5 & C_5 & D_6 & D_6 & E_6 & E_7 \\
k           & 3   & 5   & 1   & 1   & 6   & 6   & 7
\end{array}
\end{array}
\end{equation*}
Extending these to the general case should be straightforward. Note that choosing $k=n-1$ in type $\LGD_n$ simply switches the labels $n-1$ and $n$ along the main diagonal of the staircase partition, and that choosing $k=1$ in type $\LGE_6$ yields the poset obtained by reversing the order (or diagramatically by rotating the poset $180^\circ$).
\end{ex}

\subsubsection{Minuscule posets and the Laurent polynomial Landau-Ginzburg model}\label{ssec:min_posets_and_LP_LG_model}
As Remark \ref{rem:order_ideals_to_min_coset_reps} indicated, the minuscule poset $\minposet$ lets us deal with the simple reflections in the reduced expressions of $\wP$ without fixing a specific one. In Section \ref{ssec:LP_LG_model} we introduced the algebraic torus $\opendunim\subset\dunimP$ and discussed how the toric coordinates do not depend on the fixed reduced expression of $\wP$, despite being associated to a specific simple reflection in one of the expressions. Combining these two observations, we can uniquely associate a toric coordinate to each box of $\minposet$; the easiest way to realize this is to (temporarily) fix a reduced expression for $\wP$ to fix the decomposition of elements in the torus $\opendunim$ as in equation \eqref{eq:opendunim_decompostion_of_u-} as well as a linear extension of the partial order on $\minposet$, and then associating the $i$th toric coordinate to the $i$th box of $\minposet$ starting from the smallest element and going up. It should be clear from our earlier discussions that this association does not depend on the fixed reduced expression of $\wP$, and that the toric coordinate associated to a box labeled $s_i$ is the argument of the one-parameter subgroup $\dy_i$ for the \emph{same} index $i\in[n]$, in other words we can decompose $u_-\in\opendunim$ as
\begin{equation}
    u_- = \prod_{\pb\in\minposet} \dy_{\indx(\pb)}(a_{\pb}),
    \label{eq:toric_coords_labeled_by_boxes}
\end{equation}
with the factors ordered by increasing elements in $\minposet$ from left to right, and where factors corresponding to incomparable elements can be ordered arbitrarily.

Associating a toric coordinate $a_\pb$ to each of the boxes $\pb$ of the minuscule poset $\minposet$ will simplify notation greatly. First of all, given a subset $\sS\subset\minposet$, we will write
\begin{equation}
a_\sS = \tprod[_{\pb\in\sS}]a_\pb.
\label{eq:toric_term_from_poset}
\end{equation}
This will be particularly useful in combination with \emph{embeddings} of subposets into $\minposet$; here, an embedding is a \emph{label-preserving}, order-preserving morphism from one poset into another. In other words, an embedding $\iota:\ideal\hookrightarrow\minposet$ maps every box $\pb\in\ideal$ with a given label to a box $\iota(\pb)\in\minposet$ with $\labl(\pb)=\labl(\iota(\pb))$ such that $\iota(\pb)<\iota(\pb')$ whenever $\pb<\pb'$. Now, if $\ideal\subideal\minposet$ corresponds to some minimal coset representative $w\in\cosets$, then an embedding $\iota:\ideal\to\minposet$ corresponds to a realization of $w$ as a subword of $\wP$ for any fixed reduced expression of $\wP$, i.e.~a subset of simple reflections in the expression for $\wP$ forming a reduced expression for $w$. Hence, we find the following reformulation of the toric expression of $\pot_{\decomps}$ restricted to $\opendecomps$: 

\begin{thm}[\cite{Spacek_LP_LG_models}, Thm 5.7; pulled back to $\opendunim$ and rephrased using minuscule posets]\label{thm:LP_LG_model_posets}
For cominuscule $\X=\G/\P_k$, the restriction $\pot_{\opendunim}$ of $\pot_{\dunimP}$ to $\opendunim$ is given in terms of toric coordinates for any fixed $q\in\C^*$ as
\[
\pot_{\opendunim}(u_-) = \Bigl(\tsum[_{\pb\in\minposet}]a_\pb\Bigr) + q\frac{\sum_{\iota:\idealPrime\hookrightarrow\minposet} a_{\iota(\idealPrime)}}{a_{\minposet}},
\]
with $u_-\in\opendunim$ decomposed as in \eqref{eq:toric_coords_labeled_by_boxes}. In this equation, $\idealPrime\subideal\minposet$ (cf.~Remark \ref{rem:wPrime_is_min_coset_rep} and Corollary \ref{cor:quantum_term}) is the order ideal corresponding to $\wPrime=\wP(\wPPrime)^{-1}$ for $\wPPrime\in\cosets$ the minimal coset representative of $\wop s_k\weylp$, and the sum is over all embeddings $\idealPrime\hookrightarrow\minposet$.
\end{thm}

The first sum over all toric coordinates can be split into $n$ parts corresponding to the labels of the boxes the coordinates are associated to, as follows:
\begin{equation}
\pot_{\opendunim}(u_-) = \tsum[_{i\in[n]}]\Bigl(\tsum[_{\pb\in\minposet:\,
\indx(\pb)=i
}]a_{\pb}\Bigr) + q\frac{\sum_{\iota:\idealPrime\hookrightarrow\minposet} a_{\iota(\idealPrime)}}{a_{\minposet}}.
\label{eq:toric_expression_pot_Lie_2}
\end{equation}
Looking at the proof of \cite[Theorem 5.7]{Spacek_LP_LG_models} in detail, we notice that each of the terms in \eqref{eq:toric_expression_pot_Lie_2} corresponds to one of the terms in the Lie-theoretic mirror model as follows (as phrased in Definition \ref{df:Lie_potential_on_dunimP}):
\begin{equation}
    \begin{split}
        \dChfdual(u_-)|_{\opendunim} &= \tsum[_{\pb\in\minposet:~\indx(\pb)=i
}]a_{\pb} \qfor i\in[n], \\
\dChedual[_k](u_+)|_{\opendunim} &= q\frac{\sum_{\iota:\idealPrime\hookrightarrow\minposet} a_{\iota(\idealPrime)}}{a_{\minposet}}, \\
\dChedual(u_+)|_{\opendunim} &= 0 \qfor i\neq k,
    \end{split}
    \label{eq:toric_expression_pot_Lie_3}
\end{equation}
for any fixed $q\in\C^*$ and with $u_+\in\dunip$ fixed by $u_-\in\dunimP$ and $q\in\C^*$ as in Corollary \ref{cor:unique_u+_for_u-_for_fixed_q}.

\subsubsection{Local toric expressions for Pl\"ucker coordinates and generalized minors}\label{ssec:min_posets_and_Pluckers}
We will show that the type-independent canonical mirror models are isomorphic to the Lie-theoretic ones by showing that the superpotentials $\pot_\can$ restrict to the same toric expression as in \eqref{eq:toric_expression_pot_Lie_2} on the torus $\opendunim$. For this, we need convenient descriptions of the Pl\"ucker coordinates and generalized minors on this torus as well. These will be provided by Algorithms 3.11 and 3.12 of \cite{Spacek_Wang_exceptional}, which we briefly summarize and rephrase using order ideals of minuscule posets below. For the full details, see the corresponding statements in \cite{Spacek_Wang_exceptional}.

We start with Algorithm 3.12 of  \cite{Spacek_Wang_exceptional}, as it is the simplest case, giving a procedure to obtain the toric expressions of $\p_w(u_-)$ for $u_-=\dy_{r_1}(a_1)\cdots\dy_{r_\ellwP}(a_{\ellwP})\in\opendunim$. It can be summarized as finding all subwords of the fixed reduced expression for $\wP=s_{r_{\ellwP}}\cdots s_{r_1}$ forming reduced expressions for $w\in\cosets$ (not necessarily the same expression), and taking the corresponding toric coordinates. With the current notation in terms of ideals of $\minposet$ and with \eqref{eq:toric_term_from_poset} in mind, the result of the procedure can be given in closed form as
\begin{equation}
\p_\ideal(u_-) = \tsum[_{\iota:\ideal\hookrightarrow\minposet}] a_{\iota(\ideal)} \qfor u_-=\dy_{r_1}(a_1)\cdots\dy_{r_\ellwP}(a_{\ellwP})\in\opendunim.
\label{eq:toric_expression_for_Plucker_coord_in_posets}
\end{equation}
Note that $\p_\ideal$ is a homogeneous polynomial of degree $\ell(\ideal)$, i.e.~the number of boxes in the ideal $\ideal\subideal\minposet$. The reason for this simplification is that in a minuscule fundamental representation, all powers larger than $2$ of Chevalley generators act trivially so in particular all nontrivial actions of Chevalley generators are equivalent to actions of simple reflections. 

On the other hand, in non-minuscule fundamental representations, powers up to $4$ of the Chevalley generators may act nontrivially, and consequently the actions of Chevalley generators are no longer necessarily equivalent to the actions of simple reflections. Furthermore, since $\wP\cdot\dfwt[\is] \neq -\dfwt[\DS(\is)]$ for $\is\neq k$, it may occur that nontrivial generalized minors have trivial restrictions to $\opendunim$. Due to these difficulties, Algorithm 3.11 of \cite{Spacek_Wang_exceptional} must be slightly more complicated than Algorithm 3.12. However, since we consider restrictions to $u_-\in\opendunim$, we can still use minuscule posets after enhancing them with positive integer weights. Enhancing order ideals with weights is straightforward:
\begin{df}
A \emph{weighted subset} $(\sS,\weight)\subset\minposet$ is a subset $\sS\subset\minposet$ together with a map $\weight:\sS\to\N$ which we call its \emph{weight}.
\end{df}
It is somewhat more complicated to generalize embeddings of order ideals in the minuscule poset $\minposet$ to ``weighted embeddings''; the rough idea is to replace each element $\pb\in\ideal$ with $\weight(\pb)$ incomparable elements which behave exactly the same with respect to the ordering as $\pb$:
\begin{df}
Given a weighted order ideal $(\ideal,\weight)$, define the poset $\weight\cdot\ideal$ by taking $\weight(\pb)$ copies $\pb_1,\ldots,\pb_{\weight(\pb)}$ of every element $\pb\in\ideal$ with order relations $\pb_i<\pb_j'$ when $\pb<\pb'$. (In particular, all copies $\pb_i$ are incomparable.) A \emph{weighted embedding} $\iota:(\ideal,\weight)\hookrightarrow\minposet$ is a labeled, order morphism $\iota:\weight\cdot\ideal\to\minposet$. The image is a subset $\iota(\ideal)\subset\minposet$ which we give the weight $\weight(\pb)=|\iota^{-1}(\pb)|$ for every $\pb\in\iota(\ideal)$.
\end{df}

For a generalized minor $\minor_{w_1\cdot\dfwt[\is],w_2\cdot\dfwt[\is]}(u_-)$ with $u_-\in\opendunim$, Algorithm 3.11 can be summarized as taking the toric polynomial corresponding to all \emph{weighted} subwords $(\underline{s},\weight)$, i.e.~non-reduced subwords $\underline{s}=(s_{r_{j_1}},\ldots,s_{r_{j_m}})$ of the fixed reduced expression for  $\wP$ with a positive integer $c_{j_i}=\weight(s_{r_{j_i}})$, such that  $(\dChf_{r_{j_1}})^{c_{j_1}}\cdots(\dChf_{r_{j_m}})^{c_{j_m}}\bw_2\cdot\hwt[\is]=\bw_1\cdot\hwt[\is]$. For our generalized minors $\phiGLS[\is]$ in particular, since $\wP(\wop\cdot\dfwt[\is]) = \wo\cdot\dfwt[\is]$, we are guaranteed to have nontrivial restrictions to $\opendunim$, and we are able to use weighted order ideals and embeddings to express these restrictions. It is clear how to associate a weighted subset $(\sS,\weight)\subset\minposet$ to a given weighted subword of $\wP$: let $\sS$ be the set of elements that correspond to simple reflections with non-zero weights, and let $\weight:\sS\to\N$ be the corresponding weight.

The result of Algorithm 3.11 can be given in closed form for the generalized minors $\phiGLS[\is]$ with $\is\in[n]$ as:
\[
\phiGLS[\is](u_-) = \sum_{(\ideal,\weight)} \left(\sum_{\iota:(\ideal,\weight)\hookrightarrow \minposet} a_{\iota(\ideal)}^{\weight}\right),
\]
where we generalized \eqref{eq:toric_term_from_poset} to weighted subsets $(\sS,\weight)\subset\minposet$ as
\begin{equation}
a_\sS^\weight = \prod_{\pb\in\sS} a_{\pb}^{\weight(\pb)}
\label{eq:toric_exp_from_weighted_poset}
\end{equation}
and where the first sum over $(\ideal,\weight)$ is over all weighted order ideals $(\ideal,\weight)$ corresponding to weighted subwords $(w,\weight)$ of $\wP$ satisfying the condition 
\[
(\dChf_{r_{j_1}})^{c_{j_1}}\cdots(\dChf_{r_{j_m}})^{c_{j_m}}\bwop\cdot\hwt[\is] = \lwt[\is].
\]
Finding such weighted order ideal is the most complicated part of the algorithm, and in Section \ref{sec:toric-monomials} we will determine the \emph{unique} weighted order ideal satisfying this condition for each $\phiGLS[\is]$ type-independently.
\begin{rem}\label{rem:generalized_minors_sign_prelims}
In the above discussion we ignored the sign from the definition of $\phiGLS[\is](u_-)$ for $u_-\in\opendunim$ as $(-1)^h\minor_{\wo\cdot\dfwt[\is],\wop\cdot\dfwt[\is]}(u_-^{-1})$ since taking the inverse of $u_-$ introduces minus signs to each of the arguments of the one-parameter subgroups $\dy_i$: in fact, the sign $(-1)^h$ will exactly cancel out the minus sign arising from this inversion. In Section \ref{sec:toric-monomials}, we will show this explicitly; see Remark \ref{rem:generalized_minors_sign}.
\end{rem}
Finally, since we have rephrased all other results in terms of (weighted) order ideals of the minuscule poset, we will rephrase Proposition \ref{prop:coord_ring_dunimP} using these order ideals as well. For this, we only need the observation that a minimal coset representative $w\in\cosets$ sends $-\dfwt[k]$ to $-\dfwt[k]+\sdr$ for $\sdr\in\wPpdroots$ if and only if the simple reflection $s_k$ appears \emph{exactly once} in its reduced expression, so that the corresponding order ideal $\ideal\subideal\minposet$ will have \emph{exactly one} box with label $s_k$. The largest such minimal coset representative is $\wPPrime\in\cosets$ the representative of $\wop s_k\weylp$ and every other such representative is a subword of $\wPPrime$. In other words, all order ideals corresponding to such subwords are order ideals of $\idealPPrime$, and we conclude that:

\begin{prop}[\cite{GLS_Kac_Moody_groups_and_cluster_algebras}, Prop.~8.5 \& \cite{Spacek_Wang_exceptional}, Lem.~3.5]\label{prop:coord_ring_dunimP_in_posets}
The coordinate ring of the affine variety $\dunimP$ is given by
\[
\C[\dunimP] = \C\bigl[\p_{\ideal}~\big|~\ideal\subideal\idealPPrime\bigr] \bigl[\phiGLS[j]^{-1}~\big|~j\in[n]\bigr]
\]
where we recall that $\idealPPrime\subideal\minposet$ corresponds to the minimal coset representative $\wPPrime\in\cosets$ of $\wop s_k\weylp$ and therefore is the maximal order ideal containing \emph{exactly} one box labeled $s_k$, and where
\[
\phiGLS[j](u_-)=(-1)^h\minor_{\wo\cdot\dfwt[j],\wop\cdot\dfwt[j]}(u_-^{-1}) = (-1)^h\Bigl\lan u_-^{-1}\bwop\cdot\hwt[j],~\lwt[j]\Bigr\ran_{\dfwt[j]}
\]
is a generalized minor as defined in \cite{Fomin_Zelevinsky_Double_Bruhat_Cells_and_Total_Positivity}.
\end{prop}

%% file: table_of_com_hom_spaces.tex
\begin{table}[b!h]%
\[
\begin{array}{ccc|c|lcc}
\multicolumn{3}{c|}{\text{type of $\G$ and weight $\fwt$}} & \text{(co-) minuscule} & \text{variety $\G/\Pj[i]$} & \dim & \indx \vphantom{\dfrac MM}\\\hline
A_{n-1} & 
\begin{tikzpicture}[baseline = -0.75ex, scale = 0.625]
\coordinate (1) at (0,0);
\coordinate (2) at (0.625,0);
\coordinate (3) at (1.25,0);
\coordinate (3+) at (1.75,0);
\coordinate (n-) at (2.25,0);
\coordinate (n) at (2.75,0);
\draw[thick] (1)--(3+);
\draw[thick, dotted] (3+)--(n-);
\draw[thick] (n-)--(n);
\draw[black,fill=black] (1) circle (3pt); 
\draw[black,fill=black] (2) circle (3pt);
\draw[black,fill=black] (3) circle (3pt);
\draw[black,fill=black] (n) circle (3pt);
\end{tikzpicture}
&\text{any }k &\text{both}& \Gr(k,n) & k(n-k) & n  \vphantom{\dfrac MM}\\
B_n & 
\begin{tikzpicture}[baseline = -0.75ex, scale = 0.625]
\coordinate (1) at (0,0);
\coordinate (2) at (0.625,0);
\coordinate (2+) at (1.125,0);
\coordinate (n-1-) at (1.625,0);
\coordinate (n-1) at (2.125,0);
\coordinate (n) at (3.125,0);
\draw[thick] (1)--(2+);
\draw[thick, dotted] (2+)--(n-1-);
\draw[thick] (n-1-)--(n-);
\draw[thick,double,double distance = 1pt] (n-1)--(n);
\draw[black,fill=black] (1) circle (3pt); 
\draw[black,fill=white] (2) circle (3pt);
\draw[black,fill=white] (n-1) circle (3pt);
\draw[black,fill=white] (n) circle (3pt);
\node at (2.625,0) {$\boldsymbol{>}$};
\end{tikzpicture}
& 1 & \text{cominuscule} & Q_{2n-1} & 2n-1 & 2n-1    \vphantom{\dfrac MM}\\
B_n & 
\begin{tikzpicture}[baseline = -0.75ex, scale = 0.625]
\coordinate (1) at (0,0);
\coordinate (2) at (0.625,0);
\coordinate (2+) at (1.125,0);
\coordinate (n-1-) at (1.625,0);
\coordinate (n-1) at (2.125,0);
\coordinate (n) at (3.125,0);
\draw[thick] (1)--(2+);
\draw[thick, dotted] (2+)--(n-1-);
\draw[thick] (n-1-)--(n-);
\draw[thick,double,double distance = 1pt] (n-1)--(n);
\draw[black,fill=white] (1) circle (3pt); 
\draw[black,fill=white] (2) circle (3pt);
\draw[black,fill=white] (n-1) circle (3pt);
\draw[black,fill=black] (n) circle (3pt);
\node at (2.625,0) {$\boldsymbol{>}$};
\end{tikzpicture}
& n &\text{minuscule} & \OG(n,2n+1) &\frac12n(n+1) & 2n   \vphantom{\dfrac MM}\\
C_n & 
\begin{tikzpicture}[baseline = -0.75ex, scale = 0.625]
\coordinate (1) at (0,0);
\coordinate (2) at (0.625,0);
\coordinate (2+) at (1.125,0);
\coordinate (n-1-) at (1.625,0);
\coordinate (n-1) at (2.125,0);
\coordinate (n) at (3.125,0);
\draw[thick] (1)--(2+);
\draw[thick, dotted] (2+)--(n-1-);
\draw[thick] (n-1-)--(n-);
\draw[thick,double,double distance = 1pt] (n-1)--(n);
\draw[black,fill=black] (1) circle (3pt); 
\draw[black,fill=white] (2) circle (3pt);
\draw[black,fill=white] (n-1) circle (3pt);
\draw[black,fill=white] (n) circle (3pt);
\node at (2.625,0) {$\boldsymbol{<}$};
\end{tikzpicture}
& 1 & \text{minuscule} & \CP^{2n-1} & 2n-1 & 2n    \vphantom{\dfrac MM}\\
C_n & 
\begin{tikzpicture}[baseline = -0.75ex, scale = 0.625]
\coordinate (1) at (0,0);
\coordinate (2) at (0.625,0);
\coordinate (2+) at (1.125,0);
\coordinate (n-1-) at (1.625,0);
\coordinate (n-1) at (2.125,0);
\coordinate (n) at (3.125,0);
\draw[thick] (1)--(2+);
\draw[thick, dotted] (2+)--(n-1-);
\draw[thick] (n-1-)--(n-);
\draw[thick,double,double distance = 1pt] (n-1)--(n);
\draw[black,fill=white] (1) circle (3pt); 
\draw[black,fill=white] (2) circle (3pt);
\draw[black,fill=white] (n-1) circle (3pt);
\draw[black,fill=black] (n) circle (3pt);
\node at (2.625,0) {$\boldsymbol{<}$};
\end{tikzpicture}
& n & \text{cominuscule} & \LG(n,2n) &\frac12n(n+1) & n+1    \vphantom{\dfrac MM}\\
D_n & 
\begin{tikzpicture}[baseline = -0.75ex, scale = 0.625]
\coordinate (1) at (0,0);
\coordinate (2) at (0.625,0);
\coordinate (2+) at (1.125,0);
\coordinate (n-2-) at (1.625,0);
\coordinate (n-2) at (2.125,0);
\coordinate (n-1) at (2.75,0.2);
\coordinate (n) at (2.75,-0.2);
\draw[thick] (1)--(2)--(2+);
\draw[thick, dotted] (2+)--(n-2-);
\draw[thick] (n-2-)--(n-2)--(n-1);
\draw[thick] (n-2)--(n);
\draw[black,fill=black] (1) circle (3pt); 
\draw[black,fill=white] (2) circle (3pt);
\draw[black,fill=white] (n-2) circle (3pt);
\draw[black,fill=white] (n-1) circle (3pt);
\draw[black,fill=white] (n) circle (3pt);
\end{tikzpicture}
& 1 & \text{both} & Q_{2n-2} &2n-2 & 2n-2    \vphantom{\dfrac MM}\\
D_n & 
\begin{tikzpicture}[baseline = -0.75ex, scale = 0.625]
\coordinate (1) at (0,0);
\coordinate (2) at (0.625,0);
\coordinate (2+) at (1.125,0);
\coordinate (n-2-) at (1.625,0);
\coordinate (n-2) at (2.125,0);
\coordinate (n-1) at (2.75,0.2);
\coordinate (n) at (2.75,-0.2);
\draw[thick] (1)--(2)--(2+);
\draw[thick, dotted] (2+)--(n-2-);
\draw[thick] (n-2-)--(n-2)--(n-1);
\draw[thick] (n-2)--(n);
\draw[black,fill=white] (1) circle (3pt); 
\draw[black,fill=white] (2) circle (3pt);
\draw[black,fill=white] (n-2) circle (3pt);
\draw[black,fill=black] (n-1) circle (3pt);
\draw[black,fill=black] (n) circle (3pt);
\end{tikzpicture}
& n-1\text{ or }n & \text{both} & \OG(n,2n) &\frac12n(n-1) & 2n-2     \vphantom{\dfrac MM}\\
E_6 & 
\begin{tikzpicture}[baseline = 0.25ex, scale = 0.625]
\coordinate (1) at (0,0);
\coordinate (3) at (0.625,0);
\coordinate (4) at (1.25,0);
\coordinate (5) at (1.875,0);
\coordinate (6) at (2.5,0);
\coordinate (2) at (1.25,0.625);
\draw[thick] (1)--(3)--(4)--(5)--(6);
\draw[thick] (4)--(2);
\draw[black,fill=black] (1) circle (3pt); 
\draw[black,fill=white] (2) circle (3pt);
\draw[black,fill=white] (3) circle (3pt);
\draw[black,fill=white] (4) circle (3pt);
\draw[black,fill=white] (5) circle (3pt);
\draw[black,fill=black] (6) circle (3pt);
\end{tikzpicture}
& 1\text{ or }6 & \text{both} & \OP^2 = \LGE^\SC_6/P_6 & 16 & 12    \vphantom{\dfrac MM}\\
E_7 & 
\begin{tikzpicture}[baseline = 0.25ex, scale = 0.625]
\coordinate (1) at (0,0);
\coordinate (3) at (0.625,0);
\coordinate (4) at (1.25,0);
\coordinate (5) at (1.875,0);
\coordinate (6) at (2.5,0);
\coordinate (7) at (3.125,0);
\coordinate (2) at (1.25,0.625);
\draw[thick] (1)--(3)--(4)--(5)--(6)--(7);
\draw[thick] (4)--(2);
\draw[black,fill=white] (1) circle (3pt); 
\draw[black,fill=white] (2) circle (3pt);
\draw[black,fill=white] (3) circle (3pt);
\draw[black,fill=white] (4) circle (3pt);
\draw[black,fill=white] (5) circle (3pt);
\draw[black,fill=white] (6) circle (3pt);
\draw[black,fill=black] (7) circle (3pt);
\end{tikzpicture}
& 7 & \text{both} & \LGE^\SC_7/P_7 & 27 & 18   \vphantom{\dfrac M{\frac MM}} \\\hline
\end{array}
\]
\caption{Table listing for each type of $\G$ the fundamental weights $\fwt$ that are minuscule, cominuscule or both, the associated homogeneous spaces and their dimensions and indexes. The varieties appearing are: \\
    $Q_n$, the quadric of dimension $n$; \\
    $\Gr(k,n)$, the Grassmannian of $k$-dimensional subspaces in $\C^n$; \\
    $\LG(n,2n)$, \raisebox{-.65em}{\parbox{.71\columnwidth}{the Lagrangian Grassmannian of maximal isotropic subspaces with respect to the standard symplectic form;\vspace{1pt}}} \\
    $\OG(n,2n)$ \raisebox{-1.3em}{\parbox{.715\columnwidth}{and $\OG(n,2n+1)$, one of the two isomorphic connected components of the orthogonal Grassmannians of maximal isotropic subspaces with respect to the standard quadratic form;}} \\
    $\OP^2=\LGE^\SC_6/P_6$, the Cayley plane; \\
    $\LGE^\SC_7/\P_7$, the Freudenthal variety. \\
Note that the two varieties that are only minuscule are redundant: \\
The type-$B_n$ minuscule variety $\OG(n,2n+1)$ is isomorphic to the variety $\OG(n+1,2n+2)$ which is both minuscule and cominuscule as a type-$D_{n+1}$ homogeneous space; \\
The type-$C_n$ minuscule variety $\CP^{2n-1}$ is the same as $\Gr(1,2n)$, which is both minuscule and cominuscule as a type-$A_{2n-1}$ homogeneous space. \\
Adapted from \cite{CMP_Quantum_cohomology_of_minuscule_homogeneous_spaces} and \cite{Spacek_LP_LG_models}.}
\label{tab:CominusculeSpaces}
\end{table}

%% file: example_min_reps_and_min_posets.tex
\begin{ex}
We will illustrate the correspondence between weights of the minuscule representation and order ideals of the minuscule poset by considering the example of the maximal orthogonal Grassmannian $\OG(5,10)=\Spin(10)/\P_5=\LGD_5^{\SC}/\P_5$, as discussed in Section 4 of \cite{Spacek_Wang_OG} (note the shift in $n$ used there).

First, consider the weights of the representation $\PluckerRep[4]$ of the universal cover $\udG=\Spin(10)$ of $\dG=\PSO(10)$, which is the \emph{spin representation}: We have drawn the Hasse diagram of the poset of weights of $\PluckerRep[4]$ with respect to the Bruhat order below. Here the highest weight, $\dfwt[4]$, is drawn on the left whereas the lowest weight, $-\dfwt[5]$, is drawn on the right; the Bruhat covering relations $s_iw(-\dfwt[5])\gtrdot w(-\dfwt[5])$ are denoted by edges labeled $i$.
\[\Yboxdim{3pt}
\begin{tikzpicture}[rotate=90,scale=.5,style=very thick,baseline=0.25em]
    \draw 
	(0,-7)--(0,-5)--(-2,-3)--(0,-1)--(-1,0)--(0,1)
        (0,-5)--(1,-4)--(0,-3)--(2,-1)
	(-1,-4)--(0,-3)--(-1,-2) 
	(0,1)--(2,-1) (1,-2)--(0,-1)--(1,0)
        (0,1)--(0,3)
    ;
    \draw[black, fill=black] 
	(0,-6) circle (.15)
	(0,-5) circle (.15)
	(1,-4) circle (.15)
	(-1,-4) circle (.15)
	(-2,-3) circle (.15)
	(1,-2) circle (.15)
	(2,-1) circle (.15)
	(-1,0) circle (.15)
        (0,2) circle (.15)
        (0,3) circle (.15)
    ;
    \draw[black, fill=white] 
	(0,-7) circle (.15)
 	(0,-3) circle (.15)
	(-1,-2) circle (.15)
 	(0,-1) circle (.15)
	(1,0) circle (.15)
	(0,1) circle (.15)
    ;
	\node at (0,-6.5)[below=-2pt]{\scriptsize 5};
	\node at (0,-5.5)[below=-2pt]{\scriptsize 3};
	\node at (0.5,-4.5)[below left=-3pt]{\scriptsize 4};
	\node at (-0.5,-4.5)[above left=-3pt]{\scriptsize 2};
	\node at (0.5,-3.5)[below right=-3pt]{\scriptsize 2};
	\node at (-0.5,-3.5)[below left=-3pt]{\scriptsize 4};
	\node at (-1.5,-3.5)[above left=-3pt]{\scriptsize 1};
	\node at (0.5,-2.5)[below left=-3pt]{\scriptsize 3};
	\node at (-0.5,-2.5)[below right=-3pt]{\scriptsize 1};
	\node at (-1.5,-2.5)[above right=-3pt]{\scriptsize 4};
	\node at (1.5,-1.5)[below left=-3pt]{\scriptsize 5};
	\node at (0.5,-1.5)[below right=-3pt]{\scriptsize 1};
	\node at (-0.5,-1.5)[above right=-3pt]{\scriptsize 3};
	\node at (1.5,-0.5)[below right=-3pt]{\scriptsize 1};
	\node at (0.5,-0.5)[below left=-3pt]{\scriptsize 5};
	\node at (-0.5,0.5)[above right=-3pt]{\scriptsize 5};
    \node at (-0.5,-0.5)[above left=-3pt]{\scriptsize 2};
    \node at (0.5,0.5)[below right=-3pt]{\scriptsize 2};
    \node at (0,1.5)[below=-2pt]{\scriptsize 3};
    \node at (0,2.5)[below=-2pt]{\scriptsize 4};
	\node at (0,-7)[below=2pt]{\scriptsize$\lwtv$};
	\node at (0,-7)[above=2pt]{\tiny$\varnothing$};
	\node at (0,-6)[above=2pt]{\tiny$\yng(1)$};
	\node at (0,-5)[above=2pt]{\tiny$\yng(1,1)$}; 
 \node at (1,-4)[above=2pt]{\tiny$\yng(1,2)$};
	\node at (-1,-4)[below=2pt]{\tiny$\yng(1,1,1)$};
	\node at (0,-3)[above=2pt]{\tiny$\yng(1,2,1)$};
	\node at (-2,-3)[below=2pt]{\tiny$\yng(1,1,1,1)$};
	\node at (1,-2)[above=2pt]{\tiny$\yng(1,2,2)$};
	\node at (-1,-2)[below=2pt]{\tiny$\yng(1,2,1,1)$};
	\node at (2,-1)[above=2pt]{\tiny$\yng(1,2,3)$};
	\node at (0,-1)[below=2pt]{\tiny$\yng(1,2,2,1)$};
	\node at (1,0)[above=2pt]{\tiny$\yng(1,2,3,1)$};
\node at (-1,0)[below=2pt]{\tiny$\yng(1,2,2,2)$};
\node at (0,1)[above=2pt]{\tiny$\yng(1,2,3,2)$};
\node at (0,2)[above=2pt]{\tiny$\yng(1,2,3,3)$};
\node at (0,3)[above=2pt]{\tiny$\yng(1,2,3,4)$};
\end{tikzpicture}
\qquad\qquad \Yboxdim{7pt}
\minposet=\raisebox{-1.1em}{\scriptsize\young(5,34,235,1234)}
\]
The weights that are join-irreducible with respect to the Bruhat order are denoted by black vertices; each of these corresponds to a box in the minuscule poset $\minposet$ drawn to the right of the Hasse diagram and is labeled with the $s_i$ that decreases the weight (we have written $i$ instead of $s_i$ in the poset). Note that the minuscule poset is drawn with the smallest element top-left. The correspondence between order ideals of $\minposet$ and weights of the representation is indicated in the figure as well (note that the minimal element is drawn top-left).

Finally, we illustrate the direct correspondence between ideals and minimal coset representatives: Consider the ideal $\Yboxdim{3pt}\raisebox{.5pt}{\yng(1,2,1)}\subideal\raisebox{-2.25pt}{\yng(1,2,3,4)}$ which has partial order $\pb_{s_2},\pb_{s_4}>\pb_{s_3}>\pb_{s_5}$. (Here we write $\pb_{s_i}$ for a box satisfying $\labl(\pb_{s_i})=s_i$.) It has two linear extensions $\pb_{s_2}>\pb_{s_4}>\pb_{s_3}>\pb_{s_5}$ and $\pb_{s_4}>\pb_{s_2}>\pb_{s_3}>\pb_{s_5}$ both giving the same Weyl element $s_2s_4s_3s_5=s_4s_2s_3s_5\in\cosets$ since the simple reflections $s_2$ and $s_4$ commute in type $\LGD_5$. Conversely, considering the element $s_2s_4s_3s_5=s_4s_2s_3s_5\in\cosets$, we get the partial order $\pb_{s_2},\pb_{s_4}>\pb_{s_3}>\pb_{s_5}$ from either reduced expression, i.e.~the poset $\Yboxdim{3pt}\raisebox{.5pt}{\yng(1,2,1)}\subideal\raisebox{-2.25pt}{\yng(1,2,3,4)}$.
\end{ex}

%% file: 3toric_monomials.tex
\section{Calculating toric expressions for the generalized minors} \label{sec:toric-monomials}
Recall from Section \ref{ssec:Coordinate_ring_of_dunimP} that 
\[
\phiGLS[\is](u_-)=(-1)^h\minor_{\wo(\dfwt[\is]),\wop(\dfwt[\is])}(u_-^{-1}) = (-1)^h\Bigl\lan u_-^{-1}\bwop\cdot\hwt[\is],\lwt[\is]\Bigr\ran_{\dfwt[\is]}.
\]
Let us first consider the case $\is=k$, so that $\bwop$ acts trivially on the highest weight vector while $\bwP$ maps the highest to the lowest weight vector, so that we find using \eqref{eq:flipping_generalized_minors} that
\[
(-1)^h\phiGLS[k](u_-) 
= \Bigl\lan u_-^{-1}\cdot\hwt[k],\lwt[k]\Bigr\ran_{\dfwt[k]} 
= \Bigl\lan (\bwP)^{T}u_-^{-1}\cdot\hwt[k],\hwt[k]\Bigr\ran_{\dfwt[k]} 
= \bigl\lan (\bwP)^{T}u_-^{T}\cdot \lwtv,\lwtv\bigr\ran,
\]
with the last minor taken in the Pl\"ucker representation $\PluckerRep$. Now, using the dual action, we find that
\[
(-1)^h\phiGLS[k](u_-) 
= \bigl\lan \lwtv,u_-\bwP\cdot\lwtv\bigr\ran = (-1)^\ellwP \bigl\lan \lwtv,u_- \cdot\wtv{\minposet}\bigr\ran = (-1)^\ellwP\bigl\lan u_- \cdot\wtv{\minposet}, \lwtv\bigr\ran = (-1)^\ellwP\p_{\minposet}(u_-),
\]
where we used \eqref{eq:dw_bw_relation} in the second equation. However, since the Pl\"ucker representation is minuscule, we know that the height $h$ as defined in \eqref{eq:gen_minor_minus_sign} equals $\ellwP=\ell(\wP)$ since every simple reflection of $\wP\in\cosets$ acts as subtraction of a single simple root (cf.~Lemma \ref{lem: Spacek_LPLG}), so we conclude that
\begin{equation}
    \phiGLS[k](u_-) = \p_{\minposet}(u_-)
\label{eq:phi_k_and_p_minposet}
\end{equation}
on all of $\dunimP$. Next, in the case $\is=0$, we have
\[
\phiGLS[0](u_-) = \minor_{\dfwt[k],\dfwt[k]}(u_-^{-1}) = \llan u_-^{-1}\hwt[k],\hwt[k]\rran_{\dfwt[k]} = 1,
\]
where the final equality follows by noting that $u_-^{-1}=\dy_{r_\ell}(-a_{\ell})\cdots\dy_{r_1}(-a_1)$ acts on $\hwt[k]$ by strictly lowering the weight (with a nontrivial polynomial in $a_i$ as coefficient) or leaving it unchanged (with coefficient $1$). Since $p_{\varnothing}(u_-)=v_0^*(u_-v_0)=1$ holds similarly (recall that $v_\varnothing=v_0$), we obtain the equality 
\begin{equation}
    \phiGLS[0](u_-)=p_\varnothing(u_-)
    \label{eq:phi_0}
\end{equation}
on all of $\dunimP$.

For the remainder of this section, we fix $\is\in[n]\setminus\{k\}$ and let $u_-\in\opendunim$. Our goal is to express $\phiGLS[\is](u_-)$ as a polynomial in the toric coordinates of $\opendunim$ as defined in equation \eqref{eq:opendunim_decompostion_of_u-}. To find these expressions, we will need to consider the various fundamental representations of $\udG$. However, we will manage to avoid most of the type-dependent representation theory by focusing on the actions of the Weyl groups on the weight spaces. 

Considering $\opendunim\subset\dunim$ inside the completed universal enveloping algebra $\CUEAm$ of $\dunim$ and thus the series expansions
\[
\dy_{r_j}(a_j)=\exp(a\dChf_{r_j}) = 1+a_j\dChf_{r_j} + a_j^2\dChf_{r_j} + \dots
\]
for each factor $\dy_{r_j}(a_j)$ of $u_-$, we will need to find the coefficients in $u_-^{-1}$ of all terms $\dChf_{i_1}\cdots\dChf_{i_j}$ that map $\bwop\cdot\hwt[\is]$ to $\bwo\cdot\hwt[\is]$. In other words, we can compute the $\phiGLS[\is]$ combinatorially using strictly decreasing paths in the Hasse diagram of weights of the representation $\dfwtrep[\is]$ with the Bruhat order. Fortunately, we have the following result:
\begin{prop}[{Corollary of \cite[Thm 1.10]{Fomin_Zelevinsky_Double_Bruhat_Cells_and_Total_Positivity}}]\label{prop:gen_minors_are_toric_monomials}
The restrictions $\phiGLS[\is]|_{\opendunim}$ of the generalized minors $\phiGLS(u_-)=(-1)^h\minor_{\wop(\dfwt[\is]),\wo(\dfwt[\is])}(u_-^{-1})$ to $\opendunim$ are monomials in the toric coordinates.
\end{prop} 
In particular, exactly one term in the expansion of $u_-^{-1}$ will contribute a nonzero term to the toric expression of $\phiGLS[\is]$, as we now explain. By definition, if $\llan\wt,\sr_i\rran=c$ with $\wt$ extremal, then we have that $\dy_i(a)$ acts on a weight vector $\wtvmu$ of weight $\wt$ as
\begin{align}
\dy_i(a)\cdot\wtvmu &= \wtvmu + a\dChf_i\cdot\wtvmu + \ldots + \tfrac1{(c-1)!}a^{c-1}(\dChf_i)^{c-1}\cdot\wtvmu + \tfrac1{c!}a^c(\dChf_i)^c\cdot\wtvmu \notag\\
&= \wtvmu + a\dChf_i\cdot\wtvmu + \ldots + \tfrac1{(c-1)!}a^{c-1}(\dChf_i)^{c-1}\cdot\wtvmu + a^c\,\bs_i\cdot\wtvmu,
\label{eq:dy_action_on_extramal_wt_vector}
\end{align}
cf.~equation \eqref{eq:ds_bs_dChe_dChf_relation}. Applying this to $u_-^{-1}=\dy_{r_{\ellwP}}(-a_{\ellwP})\cdots\dy_{r_1}(-a_1)$ acting on the extremal weight vector $\bwop\cdot\hwt[\is]$ and noting that $\wP=s_{r_{\ellwP}}\cdots s_{r_1}$, we expand the action as
\begin{equation}
u_-^{-1}(\bwop\cdot\hwt[\is]) = \bwop\cdot\hwt[\is] + \ldots + (\text{constant})\,\bwP\bwop\cdot\hwt[\is],
\label{eq:u_minus_inverse_action_intermediate}
\end{equation}
for some constant depending on the toric coordinates. On the other hand, since $\wo=\wP\wop$, we know that $\bwP$ sends $\bwop\cdot\hwt[\is]$ to 
\[
\bwP\bwop\cdot\hwt[\is]=\bwo\cdot\hwt[\is]=\lwt[\is],
\]
the lowest weight vector. Thus, Proposition \ref{prop:gen_minors_are_toric_monomials} tells us that only the final term of \eqref{eq:u_minus_inverse_action_intermediate} can contribute to the torus expansion of $\phiGLS[\is]$. In particular, we only need to determine the constant in front of it, but by equation \eqref{eq:dy_action_on_extramal_wt_vector} applied to $u_-^{-1}=\dy_{r_{\ellwP}}(-a_{\ellwP})\cdots\dy_{r_1}(-a_1)$, this constant and hence $\phiGLS|_{\opendunim}(u_-)=(-1)^h\minor_{\wop(\dfwt[\is]),\wo(\dfwt[\is])}(u_-^{-1})$ must be given by
\[
\phiGLS[\is]|_{\opendunim}(u_-)=(-1)^h\prod_{j=1}^\ell (-a_j)^{\llan s_{r_j}\cdots s_{r_1}\wop\cdot\dfwt[\is],\sr_{r_j}\rran},
\]
where we fixed a reduced expression $\wP = s_{r_\ell}\dots s_{r_1}$ to label the toric coordinates. 

Note however that the exponents of the toric coordinates above should not depend on this reduced expression but only on the position of their associated element in the minuscule poset $\minposet$. To find such a simplification, we recall that two commuting simple reflections $s_i$ and $s_j$ have corresponding Cartan integer $\llan\sdr_j,\sr_i\rran=0$ so in particular
\[
\llan s_is_j\cdot\wt,\sr_i\rran = \llan s_i\cdot\wt,\sr_i\rran - \llan\wt,\sr_j\rran\llan\sdr_j,\sr_i\rran = \llan s_i\cdot\wt,\sr_i\rran.
\]
Applying this repeatedly to the product $s_{r_j}\cdots s_{r_1}$, we see that only those simple reflections (from the right to the left) that do not commute with their preceding simple reflections (i.e.~those to the left) will influence the result of the coefficient $\llan s_{r_j}\cdots s_{r_1}\wop\cdot\dfwt[\is],\sr_{r_j}\rran$. However, these simple reflections are exactly the ones corresponding to the elements of the order ideal $\sI(\pb)\subideal\minposet$ with $\pb\in\minposet$ corresponding to $s_{r_j}$. (In particular, $\labl(\pb_{r_j})=s_{r_j}$ and $\indx(\pb)=r_j$.)

Now, recalling from Remark \ref{rem:order_ideals_to_min_coset_reps} that $\labelprod{\ideal} = \prod_{\pb\in\ideal} \labl(\pb)\in\cosets$ as the increasing (from right to left) ordered product of the labels (i.e.~simple reflections) of elements in the order ideal $\ideal\subideal\minposet$ equals the minimal coset representative corresponding to the ideal, the above discussion allows us to simplify the coefficient
\[
\llan s_{r_j}\cdots s_{r_1}\wop\cdot\dfwt[\is],\sr_{r_j}\rran = \llan \labelprod{\sI(\pb)}\,\wop\cdot\dfwt[\is],\sr_{\indx(\pb)}\rran,
\]
removing the need to take a reduced expression for $\wP$ or a total order on $\minposet$. As a result, we can simplify the exponents of the toric coordinates and label the toric coordinates according to elements of $\minposet$ in the expression for $\phiGLS[\is]$ as
\[
\phiGLS[\is]|_{\opendunim}(u_-)=(-1)^h\prod_{\pb\in\minposet} (-a_{\pb})^{\llan \labelprod{\sI(\pb)}\,\wop\cdot\dfwt[\is],\sr_{\indx(\pb)}\rran}.
\]
Now, we observe that this expression is written as the product of toric coordinates taken to a certain power, which we recall was how we defined the toric monomial associated to a \emph{weighted} order ideal in \eqref{eq:toric_exp_from_weighted_poset}. So, we can summarize this expression even further by setting
\[
\weight:\begin{array}{c}
     \minposet\longrightarrow\{0\}\cup\N \\
     \pb\mapsto\llan \labelprod{\sI(\pb)}\,\wop\cdot\dfwt[\is],\sr_{\indx(\pb)}\rran
\end{array}
\]
and considering the weighted subset of $\minposet$ given by the subset $\weight^{-1}(\N)\subset\minposet$ with weight function $\weight$, yielding:
\begin{equation}
\phiGLS[\is]|_{\opendunim}(u_-)=(-1)^h(-a)_{\weight^{-1}(\N)}^{\weight},
\label{eq:phiGLS_intermediate_expression}
\end{equation}
using the notation defined in \eqref{eq:toric_exp_from_weighted_poset}, applied to the toric coordinates $(-a_\pb~|~\pb\in\minposet)$.

Hence, the calculation of the restriction of $\phiGLS$ to $\opendunim$ reduces to determining the extremal path in the fundamental representation $\dfwtrep[\is]$ taken by $\bwP$ from the extremal weight vector $\bwop\cdot\hwt[\is]$ to the lowest weight vector, and determining how many times each simple reflection $\bs_{r_i}$ in $\bwP$ acts as $\dChf_{r_i}$. In order to do this, we will first study the actions of the Weyl group elements $\woj$, which were defined in Section \ref{sec:conventions}. Recall that for $j\in [n]$, $\woj$ is the longest element of the parabolic Weyl subgroup $\weylj=\lan s_i~|~i\neq j\ran\subset\weyl$ and $\DSi[j]:[n]\setminus\{j\}\to[n]\setminus\{j\}$ is the corresponding (possibly trivial) Dynkin subdiagram involution defined by $\woj(\dfwt[i])=-\dfwt[{\DSi[j](i)}]$ and extended by $\DSi[j](j)=j$ to the full Dynkin diagram of $\weyl$. Note that for $k\in[n]$ such that $\P=\P_k$, we have $\weylj[k]=\weylp$, $\woj[k]=\wop$ and $\DSi[k]=\DSp$. We begin with a lemma regarding the action of these elements on the fundamental co-weights.
\begin{lem}\label{lem:weyl_action}
Let $i\neq j\in[n]$. Then there is some positive constant $c$ such that
\[
\woj(\dfwt) = - \dfwt[{\DSi[j](i)}] + c\dfwt[j].
\]
\end{lem}
\begin{proof}
Consider the subgroup $\dGj\subset\dG$ corresponding to the Dynkin diagram of $\dG$ with the $j$th vertex removed. (More precisely, let $\dGj\subset\dP_j$ be the Levi subgroup of the parabolic subgroup corresponding to the $j$th vertex.) Write $\sdfwt$ for the fundamental weights of $\dGj$ considered as a group in its own right, then the inclusion $\dGj\subset\dG$ gives a natural projection $\pij$ of weight lattices satisfying $\pij(\dfwt)=\sdfwt$ and $\pij(\dfwt[j])=0$. As $\woj\in\weylj$ is the longest element, it maps $\sdfwt$ to $-\sdfwt[{\DSi[j](i)}]$; pulling this back along  $\pij$, we find that $\woj(\dfwt)=-\dfwt[{\DSi[j](i)}]+c\dfwt[j]$ for some $c$. 

The coefficient of $\dfwt[j]$ is only affected by the actions of $s_j$ or $s_m$ which does not commute with $s_j$. Furthermore, the action of $s_j$ decreases the coefficient of $\dfwt[j]$ while the action of $s_m$ not commuting with $s_j$ increases the coefficient of $\dfwt[j]$. Since $\woj$ doesn't allow the action of $s_j$, then $c\ge 0$. If $c=0$, then $\woj(\dfwt[i])$ is the lowest weight of $\dfwtrep[i]$, meaning that $\woj \weylj[i]=w^i\weylj[i]$, where $w^i=\wo\weylj[i]$ is the minimal coset representative of $\wo$ in $\weylj[i]$. In particular, this implies that $w^i$ has no occurrence of $s_j$, meaning that $s_j$ acts trivially on $\dfwtrep[i]$. However, this would imply that the lie subalgebra of $\dg$ generated by $\dChf_i, \dChe_i$ lies in the kernel of the lie algebra morphism $\dg \rightarrow \mathrm{End}(\dfwtrep[i])$, contradicting the fact that $\dg$ is a simple lie algebra.
\end{proof}
The constant $c$ of Lemma \ref{lem:weyl_action} is particularly easy to find when either $i$ or $j$ corresponds to a minuscule weight of $\dG$.
\begin{lem}\label{lem:minuscule_action}
For $k'$ minuscule and $i\in[n]\setminus\{k'\}$, we have:
\begin{align}
\woj[i](\dfwt[k']) &= -\dfwt[\DSi(k')]+\dfwt,
\label{eq:action_on_minuscule_wt} \\
\woj[k'](\dfwt[i]) &= -\dfwt[{\DSi[k'](i)}]+c_i\dfwt[k'],  \qwhere c_i=\llan\dfwt,\sro\rran.
\label{eq:minuscule_action_on_wt}
\end{align}
Note that $c_i$ is also the coefficient of the simple root $\sr_i$ in the highest root $\sro$.
\end{lem}
\begin{proof}
Both equations are relatively easy consequences of Lemma \ref{lem:weyl_action} together with the observation that
\[
\llan-\dfwt[j]+c\dfwt,\sr_i\rran = c
\]
for any $i\neq j$ and the fact that $\DSi(j)\neq i$. Equation \eqref{eq:action_on_minuscule_wt} follows from the well-known fact that in a minuscule fundamental weight representation all weights have $\llan\wt,\sr_i\rran\in\{-1,0,+1\}$. For equation \eqref{eq:minuscule_action_on_wt}, the observation above becomes:
\[
c = \llan-\dfwt[{\DSi[k'](i)}] + c\dfwt[k'],\sr_{k'}\rran = \llan\woj[k'](\dfwt),\sr_{k'}\rran = \llan\dfwt,\woj[k'](\sr_{k'})\rran.
\]
(Recall that $\woj[k']^T=\woj[k']^{-1}=\woj[k']$ since it is the longest element of $\weylj[k']$.) Hence, the claim reduces to showing that $\woj[k'](\sr_{k'})=\sro$ for $k'$ minuscule. 

For this, we note that $k'$ minuscule means that the set $\proots$ of positive roots can be split into two disjoint parts: the set $\prootsj[k']=\{\sum_{i=0}^nd_i\sr_i\in\pdroots~|~d_{k'}=0\}$ and the set $\wojproots[k']=\{\sum_{i=0}^nd_i\sr_i\in\proots~|~d_i=1\}$. Note that these sets are equivalently characterized by $\woj[k'](\prootsj[k']) = -\prootsj[k']$ and $\woj[k'](\wojproots[k']) = \wojproots[k']$. In particular, $\woj[k'](\sr_{k'})\in\wojproots[k']$.

Now, suppose that $\woj[k'](\sr_{k'})\neq\sro$, then there must be a $j\neq k'$ such that $\rt=\woj[k'](\sr_{k'})+\sr_j$ is a root. Clearly, $\rt\in\wojproots[k']$ and hence $\woj[k'](\rt)$ as well. On the other hand, $\sr_j\in\prootsj[k']$, so that $\woj[k'](\sr_j)=-\rtb\in-\prootsj[k']$ is a negative root with coefficient $0$ in front of $\sr_{k'}$. Hence, we find that $\proots\ni\woj[k'](\rt)=\woj[k']\bigl(\woj[k'](\sr_{k'})+\sr_j)=\sr_{k'}-\rtb$, but this is a linear combination of simple roots with \emph{both} positive and negative coefficients, which is impossible for a root. Hence, $\woj[k'](\sr_{k'})=\sro$.
\end{proof}

The following indices will be of particular interest for the calculation of $\phiGLS[\is]|_{\opendunim}$:
\begin{equation}
i_{-1}=\is,\quad i_0=k \qand i_j=\sigma_{i_{j-1}}(i_{j-2}) \qfor j\in\{1,\ldots,\cis+2\},
\label{eq:index_seq}
\end{equation}
where $\cis=c_\is=\llan\dfwt[\is],\sro\rran$ denotes the coefficient of $\sr_\is$ in $\sro$. 
\begin{rem}
Since we are assuming the existence of a minuscule coweight $\dfwt[k]$, $\cis$ is at most $4$: namely $\llan\dfwt[4],\sro\rran=4$ for $\dG$ of type $\LGE_7$. For the classical Lie groups, $\cis$ is at most $2$.
\end{rem}
The Weyl group elements $\woj[i_j]$ satisfy the following relations:
\begin{lem}\label{lem:special_weyl_elt_action}
With the conventions above, the elements $\woj[i_j]$ satisfy the following relations:
\begin{align}
\woj[i_0](\dfwt[i_{-1}]) &= -\dfwt[i_1]+\cis\dfwt[i_0], \qwhere \cis=\llan\dfwt[\is],\sro\rran 
\label{eq:wi0}\\
\woj[i_1](\dfwt[i_{0}]) &= -\dfwt[i_2]+\dfwt[i_1], 
\label{eq:wi1}\\
\woj[i_j](\dfwt[i_{j-1}]) &= -\dfwt[i_{j+1}]+2\dfwt[i_j], \qfor j\in\{2,\ldots,\cis\}
\label{eq:wij}\\
\woj[i_{\cis+1}](\dfwt[i_{\cis}]) &= -\dfwt[i_{c_i+2}]+\dfwt[i_{\cis+1}].
\label{eq:wic+1}
\end{align}
\end{lem}
\begin{proof}
Since $i_0=k$ is minuscule, equation \eqref{eq:wi0} is simply a restatement of equation \eqref{eq:minuscule_action_on_wt}, while \eqref{eq:wi1} is a restatement of \eqref{eq:action_on_minuscule_wt}.

The remaining equations will be checked by type. For the reader's convenience we recall the highest roots for the types allowing a minuscule fundamental weight:
\[
\begin{array}{r|l}
\LGA_n & \sro = \sr_1 + \ldots +\sr_n\\ 
\LGB_n & \sro = \sr_1 + 2\sr_2 + \ldots + 2\sr_n\\
\LGC_n & \sro = 2\sr_1 + \ldots + 2\sr_{n-1} + \sr_n\\
\LGD_n & \sro = \sr_1 + 2\sr_2 + \ldots + 2\sr_{n-2} + \sr_{n-1} + \sr_n\\
\LGE_6 & \sro = \sr_1 + 2\sr_2 + 2\sr_3 + 3\sr_4 + 2\sr_5 + \sr_6\\
\LGE_7 & \sro = 2\sr_1 + 2\sr_2 + 3\sr_3 + 4\sr_4 + 3\sr_5 + 2\sr_6 + \sr_7
\end{array}
\]
In type $\LGA_n$, $\cis=1$ for all $\is$, so every fundamental weight is minuscule. Hence, \eqref{eq:wij} is vacuous, and \eqref{eq:wic+1} follows from either of the equations of Lemma \ref{lem:minuscule_action}.

In type $\LGB_n$, only $\dfwt[n]$ is minuscule, whereas $\cis=2$ for every $\is\neq n$. Hence, taking any $\is\neq n$, we find that the sequence starting on $i_{-1}=\is$ and $i_0=n$ proceeds as:
\[
i_1=\DSi[n](\is)=n-\is,\quad i_2= \DSi[n-\is](n)=n, \quad i_3=\DSi[n](n-\is) = \is, \qand i_4=\DSi[\is](n)=n.
\]
Thus, we need to show that $\woj[n](\dfwt[n-\is])=-\dfwt[\is]+2\dfwt[n]$ and $\woj[\is](\dfwt[n])=-\dfwt[n]+\dfwt[\is]$, but these follow from Lemma \ref{lem:minuscule_action} since $n$ is minuscule.
  
In type $\LGC_n$, $\dfwt[1]$ is minuscule, while $\cis=2$ for all $i\neq1$. Thus, let $\is\neq1$, then the sequence $i_{-1}=\is$, $i_0=1$ proceeds as
\[
i_1=\DSi[1](\is)=\is,\quad i_2=\DSi[\is](1)=\is-1,\quad i_3=\DSi[\is-1](\is)=\is, \qand i_4=\DSi[\is](\is-1)=1.
\]
Hence, we need to show that $\woj[\is-1](\dfwt[\is])=-\dfwt[\is]+2\dfwt[\is-1]$ and $\woj[\is](\dfwt[\is-1])=-\dfwt[1]+\dfwt[\is]$. Now, $\woj[\is-1]\in\weylj[\is-1]$ is the longest element, and $\weylj[\is-1]$ is of type $\LGA_{\is-2}\times\LGC_{n-\is+1}$, so that $\woj[\is-1]$ can be written as the product of the the longest words of these two subgroups. The longest word of the $\LGA_{\is-2}$-subgroup clearly acts trivially on $\dfwt[\is]$ (since any $s_i$ in this subgroup has $i<\is$ so $s_i(\dfwt[\is])=\dfwt[\is]$) and we find $\woj[\is-1](\dfwt[\is])=s_\is s_{\is+1}\cdots s_n\cdots s_{\is+1}s_\is(\dfwt[\is])$ which equals $-\dfwt[\is]+2\dfwt[\is-1]$. To show that $\woj[\is](\dfwt[\is-1])=-\dfwt[1]+\dfwt[\is]$, we note that $\weylj[\is]$ is of type $\LGA_{\is-1}\times\LGC_{n-\is}$, with the $\LGC_{n-\is}$-subgroup acting trivially on $\dfwt[\is-1]$, hence $\woj[\is](\dfwt[\is-1])=s_1(s_2s_1)\cdots (s_{\is-1}\cdots s_1)(\dfwt[\is-1]) = s_1\cdots s_{\is-1}(\dfwt[\is-1])= -\dfwt[1]+\dfwt[\is]$.

In type $D_n$, $\dfwt[1]$, $\dfwt[n-1]$ and $\dfwt[n]$ are minuscule, while $\cis=2$ for $\is\neq\{1,n-1,n\}$. Hence, we will need to consider three cases depending on $k\in\{1,n-1,n\}$. First, we consider $k=1$. Let $\is\notin\{1,n-1,n\}$, then the sequence $i_{-1}=\is$, $i_0=1$ proceeds as
\[
i_1=\DSi[1](\is) =\is, \quad i_2=\DSi[\is](1)=\is-1, \quad i_3=\DSi[\is-1](\is)=\is, \qand i_4=\DSi[\is](\is-1)=1.
\]
For these, one of the two subsequent indices is always minuscule, so the statements follow from Lemma \ref{lem:minuscule_action}. Now, let $\is\in\{n-1,n\}$, then the sequence $i_{-1}=\is$, $i_0=1$ continues as $i_1\in\{n-1,n\}$ (for $n-1$ even, $i_1=\is$; for $n-1$ odd, $i_1$ is the other element of $\{n-1,n\}$) and $i_2=1$. Equation \eqref{eq:wij} is vacuous and \eqref{eq:wic+1} follows from Lemma \ref{lem:minuscule_action} since all indices are minuscule.

Secondly, we should consider $k=n-1$, but this case is the same as $k=n$ with all occurrences of $n-1$ and $n$ interchanged, so we will disregard it here.

Finally, we have the case $k=n$. Let $\is\notin\{1,n-1,n\}$, then the sequence $i_{-1}=\is$, $i_0=n$ proceeds as $i_1=\DSi[n](\is)=n-\is$, $i_2=\DSi[n-\is](n)\in\{n-1,n\}$ (depending on $n-\is$ odd or even, analogously to above), $i_3=\DSi[i_2](n-\is)=\is$ and $i_4=\DSi[\is](i_2)=n$. Again, every second index is minuscule, so the equations follow from Lemma \ref{eq:minuscule_action_on_wt}. Now, let $\is=1$, then $i_{-1}=1$, $i_0=n$ give $i_1=\DSi[n](1)=n-1$, $i_2=\DSi[n-1](n)=1$, $i_3=\DSi[1](n-1)\in\{n-1,n\}$ (depending on $n-1$ even or odd); all indices are minuscule, so the equations follow from Lemma \ref{lem:minuscule_action} (equation \eqref{eq:wij} being vacuous). Lastly, let $\is=n-1$, then $i_{-1}=n-1$ and $i_0=n$ give $i_1=\DSi[n](n-1)=1$, $i_2=\DSi[1](n)\in\{n-1,n\}$ (depending on $n-1$ odd or even) and $i_3=\DSi[i_2](1)\in\{n-1,n\}$ (depending on $n-1$ even or odd). This case also has all indices minuscule, so the equations follow from Lemma \ref{lem:minuscule_action}. This completes all the cases for type $\LGD_n$.

Types $\LGE_6$ and $\LGE_7$ can be considered more concretely. In type $\LGE_6$, both $1$ and $6$ are minuscule, but we can disregard $1$ due to the isomorphism $\LGE_6/\P_1\cong\LGE_6/\P_6$. Thus, for $k=6$, then we have following sequences:
\begin{equation}
\begin{array}{c|c||c|c|c|c|c|c|c}
\is & \cis & \hspace{-.7ex}i_{-1}\hspace{-.8ex} & i_0 & i_1 & i_2 & i_3 & i_4 & i_5 \\\hline
\vspace{-1.2em}&&&&&&&\\\hline
1 & 1 & 1 & 6 & 1 & 6 & 1 &    &  \\\hline
2 & 2 & 2 & 6 & 5 & 6 & 2 & 1 &  \\
3 & 2 & 3 & 6 & 3 & 2 & 5 & 1 & \\
5 & 2 & 5 & 6 & 2 & 1 & 3 & 1 & \\\hline
4 & 3 & 4 & 6 & 4 & 5 & 3 & 4 & 1 \\
\end{array}    
\label{eq:table_of_ij_for_Cayley}
\end{equation}
For $\is=1$, all indices in the sequence are minuscule, while for $\is=2$ and $\is=5$ every second index is minuscule, so the statement follows from Lemma \ref{lem:minuscule_action} in these cases. For $\is=3$, we need to show that $\woj[2](\dfwt[3])=-\dfwt[5]+2\dfwt[2]$ and $\woj[5](\dfwt[2])=-\dfwt[1]+\dfwt[5]$. We note that $\weylj[2]$ is of type $\LGA_5$, and hence $\woj[2]$ acts on $\dfwt[3]$ as $(s_5s_4s_3s_1)(s_6s_5s_4s_3)$; on the other hand $\weylj[5]$ is of type $\LGA_4\times\LGA_1$, so $\woj[5]$ acts on $\dfwt[2]$ as $s_1s_3s_4s_2$; in both cases it is easy to verify that the resulting weights are the ones listed above. For $\is=4$, we need to show $\woj[5](\dfwt[4])=-\dfwt[3]+2\dfwt[5]$, $\woj[3](\dfwt[5])=-\dfwt[4]+2\dfwt[3]$ and finally $\woj[4](\dfwt[3])=-\dfwt[1]+\dfwt[4]$. We have $\weylj[5]$ of type $\LGA_4\times\LGA_1$, so $\woj[5](\dfwt[4])=(s_3s_4s_2)(s_1s_3s_4)(\dfwt[4])$; $\weylj[3]$ is of type $\LGA_1\times\LGA_4$, so $\woj[3](\dfwt[5])=(s_4s_5s_6)(s_2s_4s_5)(\dfwt[5])$; and $\weylj[4]$ is of type $\LGA_2\times\LGA_1\times\LGA_2$, so $\woj[4](\dfwt[3]) = s_1s_3(\dfwt[3])$. In all cases, the calculations are straightforward.

In type $\LGE_7$ only $7$ is minuscule, so we set $k=7$ and find the following sequences:
\begin{equation}
\begin{array}{c|c||c|c|c|c|c|c|c|c}
\is & \cis & \hspace{-.7ex}i_{-1}\hspace{-.8ex} & i_0 & i_1 & i_2 & i_3 & i_4 & i_5 & i_6 \\\hline
\vspace{-1.2em}&&&&&&&\\\hline
1 & 2 & 1 & 7 & 6 & 7 & 1 & 7 &&\\
2 & 2 & 2 & 7 & 2 & 1 & 2 & 7 &&\\
6 & 2 & 6 & 7 & 1 & 7 & 6 & 7 &&\\\hline
3 & 3 & 3 & 7 & 5 & 6 & 2 & 3 & 7 &\\
5 & 3 & 5 & 7 & 3 & 2 & 6 & 5 & 7 &\\\hline
4 & 4 & 4 & 7 & 4 & 5 & 3 & 5 & 4 & 7
\end{array}
\label{eq:table_of_ij_for_Freudenthal}
\end{equation}
The sequences for $\is=1$ and $\is=6$ are covered by Lemma \ref{lem:minuscule_action}. For the others, we note that the Weyl subgroups are of the following types:
\[
\begin{array}{c|c|c|c|c|c|c}
\weylj[1]	& \weylj[2]	& \weylj[3]	& \weylj[4]	& \weylj[5]	& \weylj[6]	& \weylj[7]	\\\hline
\LGD_6 & \LGA_6 & \LGA_1\times\LGA_5 & \LGA_2\times\LGA_1\times\LGA_3 & \LGA_4\times\LGA_2 & \LGD_5\times\LGA_1 & \LGE_6
\end{array}
\]
Hence, for $\is=2$, we have that $\woj[1]$ acts on $\dfwt[2]$ as $s_2(s_7s_6s_5s_4)s_3(s_6s_5s_4)s_2(s_5s_4)s_3(s_4)s_2$, yielding $-\dfwt[3]+2\dfwt[1]$; and we have that $\woj[2]$ on $\dfwt[1]$ acts as $s_7s_6s_5s_4s_3s_1$, yielding $-\dfwt[7]+\dfwt[2]$. For $\is=3$, we find that $\woj[6]$ acts on $\dfwt[5]$ as $s_2(s_1s_3s_4)s_5(s_3s_4)s_2(s_4)s_5$, yielding $-\dfwt[2]+2\dfwt[6]$; $\woj[2]$ acts on $\dfwt[6]$ as $(s_3s_4s_5s_6s_7)(s_1s_3s_4s_5s_6)$, yielding $-\dfwt[3]+2\dfwt[2]$; finally, $\woj[3]$ acts on $\dfwt[2]$ as $s_7s_6s_5s_4s_2$, yielding $-\dfwt[7]+\dfwt[3]$. For $\is=5$, we find that $\woj[2]$ acts on $\dfwt[3]$ as $(s_6s_5s_4s_3s_1)(s_7s_6s_5s_4s_3)$, yielding $-\dfwt[6]+2\dfwt[2]$; $\woj[6]$ acts on $\dfwt[2]$ as $s_5(s_1s_3s_4)s_2(s_3s_4)s_5(s_4)s_2$, yielding $-\dfwt[5]+2\dfwt[6]$; finally, $\woj[5]$ acts on $\dfwt[6]$ as $s_7s_6$, yielding $-\dfwt[7]+\dfwt[6]$. For the last sequence, starting with $\is=4$, we find that $\woj[5]$ acts on $\dfwt[4]$ as $(s_3s_4s_2)(s_1s_3s_4)$, yielding $-\dfwt[3]+2\dfwt[5]$; $\woj[3]$ acts on $\dfwt[5]$ as $(s_5s_6s_7)(s_4s_5s_6)(s_2s_4s_5)$, yielding $-\dfwt[5]+2\dfwt[3]$; $\woj[5]$ acts on $\dfwt[3]$ as $(s_4s_3s_1)(s_2s_4s_3)$, yielding $-\dfwt[4]+2\dfwt[5]$; finally, $\woj[4]$ acts on $\dfwt[5]$ as $s_7s_6s_5$, yielding $-\dfwt[7]+\dfwt[4]$.
\end{proof}

\begin{cor}\label{cor:i2_minuscule}
    Suppose $\cis=2$. After deleting the nodes $i_1$ and $i_3$ from the Dynkin diagram associated to $\G$, the component containing $i_2$ corresponds to a Lie group $\tilde{\G}$ for which $i_2$ is a minuscule index. 
\end{cor}

\begin{proof}
    This is a corollary of the proof of the previous lemma. We tabulate below all situations in which $\cis=2$.
    \[
        \begin{array}{c||c|c|c|c|c|c|c}
            \mathrm{type}(\G) & k & \is & i_1 & i_2 & i_3 & \mathrm{type}(\tilde{\G}) & \tilde{i_2}\\ \hline
            \vspace{-1.2em}&&&&&&&\\\hline
            \LGB_n & n & \is\neq n & n-\is & n & \is & \LGB_{\min(\is,n-\is)} & \min(\is,n-\is)\\ \hline
            \LGC_n & 1 & \is\neq 1 & \is & \is-1 & \is & \LGA_{\is-1} & \is-1\\ \hline
            \LGD_n & 1 & \is\notin\{1,n-1,n\} & \is & \is-1 & \is & \LGA_{\is-1} & \is-1 \\ 
             & n & \is\notin\{1,n-1,n\} & n-\is & n-1 & \is & \LGD_{\min(\is,n-\is)} & n-1\\ \hline
            \LGE_6 & 6 & 2 & 5 & 6 & 2 & \LGA_1 & 1\\
            & & 3 & 3 & 2 & 5 & \LGA_2 & 2\\
            & & 5 & 2 & 1 & 3 & \LGA_1 & 1 \\ \hline
            \LGE_7 & 7 & 1 & 6 & 7 & 1 & \LGA_1 & 1\\
            & & 2 & 2 & 1 & 2 & \LGA_6 & 1\\
            & & 6 & 1 & 7 & 6 & \LGA_1 & 1
        \end{array}
    \]  
    Note that type $\LGA$ has no such simple roots.
\end{proof}

Now, although we have described the actions of $\woj[i_j]$ on $\dfwt[i_{j-1}]$, we have seen that only a certain subword of its reduced expression acts non-trivially. Of course, this subword is exactly a minimal coset representative of $\woj[i_j]$ with respect to $\weylj[i_{j-1}]$; we will denote these subwords with $\uij$. Hence, we obtain the characterization
\begin{equation}
\uij\in\weylj\text{ is an element of minimal length such that }\woj[i_j](\dfwt[i_{j-1}])=\uij(\dfwt[i_{j-1}]).
\label{eq:df_uij}
\end{equation} 
We further define $\wij = \uij\dots\uij[1]$ for $j\in [c_*]$, and we will now show that these $\wij$ are minimal coset representatives for $W/W_P$ by investigating the actions of $\uij$ on $\dfwt[k]$.
\begin{prop}\label{prop:uij_repeated_action}
For $j\in\{1,\ldots,\cis\}$, we have
\begin{align}
\wij(\dfwt[k]) = \uij\uij[j-1]\cdots\uij[1](\dfwt[k]) &= -\dfwt[i_{j+1}]+\dfwt[i_j],
\label{eq:uij_repeated_action_min_wt} \\
\intertext{and that each $\wij=\uij[j]\cdots\uij[1]$ is a minimal coset representative for $\weyl/\weylp$. Furthermore, for each $j\in [\cis+1]$, we have}
\uij\uij[j-1]\cdots\uij[1](\uij[0]\dfwt[\is]) &= -(\cis-j+1)\dfwt[i_{j+1}]+(\cis-j)\dfwt[i_j]
\label{eq:uij_repeated_action} 
\end{align}
with each simple reflection $s_i$ of $\uij$ acting by subtracting $(\cis-j+1)\sdr_i$ in the action of $\uij$ on $\uij[j-1]\cdots\uij[0](\dfwt[\is])$.
\end{prop}
\begin{proof}
We prove \eqref{eq:uij_repeated_action_min_wt} as well as the fact that $\uij[j]\cdots\uij[1]$ is a minimal coset representative of $\weyl/\weylp$ by induction on $j$, combining Lemma \ref{lem:special_weyl_elt_action} with the observation that $\uij(\dfwt[i_j])=\dfwt[i_j]$. For $j=1$, both statements follow from \eqref{eq:wi1} combined with $i_0=k$ and the definition of $\uij$ in \eqref{eq:df_uij}. 

So now suppose that \eqref{eq:uij_repeated_action_min_wt} holds for some $j-1\in\{1,\ldots,\cis-1\}$, then \eqref{eq:wij} holds for $j$ and we find that
\ali{
\uij[j]\bigl(\uij[j-1]\cdots\uij[0](\dfwt[k])\bigr)
 &= \uij\bigl(-\dfwt[i_{j}]+\dfwt[i_{j-1}]\bigr) = -\dfwt[i_j]+\uij(\dfwt[i_j{-1}]) \\
 &= -\dfwt[i_{j}]+(-\dfwt[i_{j+1}]+2\dfwt[i_j]) =-\dfwt[i_{j+1}]+\dfwt[i_j].
}
Furthermore, $\uij$ acts only on $-\dfwt[i_{j+1}]$ and not on $\dfwt[i_j]$. Since $\uij$ represents the minimal such action by definition, each action is nontrivial, so $\uij\cdots\uij[1]$ is a minimal coset representative of $\weyl/\weylp$. Hence, \eqref{eq:uij_repeated_action_min_wt} holds and $\uij\cdots\uij[1]$ is a minimal coset representative of $\weyl/\weylp$ for all $j\in\{1,\ldots,\cis\}$.

To conclude equation \eqref{eq:uij_repeated_action} from this, we note that $\is=i_{-1}$ so that \eqref{eq:wi0} implies $\uij[0](\dfwt[\is])=-\dfwt[i_1]+\cis\dfwt[k]$. Hence \eqref{eq:uij_repeated_action_min_wt} yields
\[
\uij[1]\uij[0](\dfwt[\is])=\uij[1](-\dfwt[i_1]+\cis\dfwt[k]) = -\dfwt[i_1]+\cis\,\uij[1](\dfwt[k]) \overset{(\ast)}{=\!=} -\cis\dfwt[i_2]+(\cis-1)\dfwt[i_1],
\]
which is \eqref{eq:uij_repeated_action} for $j=1$. Note that $\uij[1]$ only acted on the minuscule weight $\dfwt[k]$ in equality $(\ast)$, and hence each simple reflection $s_i$ in a reduced expression of $\uij[1]$ acts by subtracting $\sdr_i$ since $\uij[1]$ is the minimal coset representative; the scalar factor $\cis$ now gives the final statement for $j=1$.

For $j\in\{2,\ldots,\cis+1\}$, \eqref{eq:uij_repeated_action} and the final statement follow inductively:
\ali{
\uij\uij[j-1]\cdots\uij[0](\dfwt[\is]) &= \uij\bigl(-(\cis-j+2)\dfwt[i_{j}]+(\cis-j+1)\dfwt[i_{j-1}]\bigr) \\
&= -\dfwt[i_j]+(\cis-j+1)\,\uij(-\dfwt[i_j]+\dfwt[i_{j-1}]) \\
&=-\dfwt[i_j]+(\cis-j+1)\,\uij\uij[j-1]\cdots\uij[1](\dfwt[k]) \tag{\dag} \\
&=-\dfwt[i_j]+(\cis-j+1)\,(-\dfwt[i_{j+1}]+\dfwt[i_j])\\
&= -(\cis-j+1)\dfwt[i_{j+1}]+(\cis-j)\dfwt[i_j],
}
which is equation \eqref{eq:uij_repeated_action}. Again, note that $\uij$ only acted on an element in the Weyl orbit of the minuscule weight $\dfwt[k]$ in ($\dag$), and hence we obtain the final statement.
\end{proof}

\begin{cor}\label{cor:full_action_special_weyl_elts}
We have $\uij[\cis]\cdots\uij[0](\dfwt[\is])=-\dfwt[\DS(\is)]$ and $\uij[\cis+1]\cdots\uij[1](\dfwt[k])=-\dfwt[\DS(k)]$. In particular, we have $\wP=\uij[\cis+1]\cdots\uij[1]$.
\end{cor}
\begin{proof}
For the first equation, we note that substituting $j=\cis$ into equation \eqref{eq:uij_repeated_action} gives $\uij[\cis]\cdots\uij[0](\dfwt[k]) = -\dfwt[i_{\cis+1}]$. The statement now follows from the fact that $i_{\cis+1}=\DS(\is)$, which one can verify directly from the proof of Lemma \ref{lem:special_weyl_elt_action}. The second equation follows directly from combining \eqref{eq:uij_repeated_action_min_wt} for $j=\cis$ with \eqref{eq:wic+1}. We conclude from this directly that $\wP=\uij[\cis+1]\cdots\uij[1]$, since both sides are minimal coset representatives that map $\dfwt[k]$ to $-\dfwt[\DS(k)]=\wo(\dfwt[k])$.
\end{proof}

Since each $\wij=\uij\dots\uij[1]$ is a minimal coset representative, we may define $\idealij\subideal\minposet$ to be the order ideal corresponding to $\wij$. Note in particular that $\idealij[j-1]\subideal\idealij$ for $j\in[\cis+1]$, where we define $\idealij[0]=\varnothing$. With this construction, we associate to each $\uij$ for $j\ge 1$ the subset $\su_{i_j}=\idealij\setminus \idealij[j-1]\subset \minposet$. From the previous proposition, we see that the order ideals $\idealij$ allow an easy combinatorial description:
\begin{cor}\label{cor:combinatorial_description_wij}
    The order ideal $\idealij\subideal\minposet$ for $j\in[\cis+1]$ is the maximal ideal obtained by adding all possible $\pb\in\minposet$ with $\labl(\pb)\neq s_{i_j}$ to $\idealij[j-1]$, where $\idealij[0]=\varnothing$.
\end{cor}
Recall that we have illustrated this description already in Section \ref{sec:intro}. We now have all the information we need to calculate $\phiGLS[\is]|_{\opendunim}$ as in equation \eqref{eq:phiGLS_intermediate_expression}. 
\begin{thm}\label{thm:toric_monomial_phi}
Let the minimal coset representative $\wij=\uij\dots\uij[1]$ be defined as above with corresponding order ideals $\ideal_{i_j}\subideal\minposet$. Then we have
\[
\phiGLS[\is]|_{\opendunim} = \prod_{j=1}^{\cis} a_{\idealij}
\]
where $a_{\idealij} = \prod_{\pb\in\idealij} a_\pb$, cf.~\eqref{eq:toric_term_from_poset}.
\end{thm}
\begin{proof}
Recall from equation \eqref{eq:phiGLS_intermediate_expression} that the toric expression for $\phiGLS[\is]$ on $\opendunim$ is given by
\[
\phiGLS[\is]|_{\opendunim}(u_-)=(-1)^h(-a)_{\weight^{-1}(\N)}^{\weight},
\tag{$*$}\label{eq: phiGLS_intermediate_expression_repeated}
\]
where $\weight:\minposet\to\{0\}\cup\N$ is defined as
\[
\weight(\pb)=\llan \labelprod{\sI(\pb)}\,\wop\cdot\dfwt[\is],\sr_{\indx(\pb)}\rran,
\]
and where $\sI(\pb)$ is the order ideal generated by $\pb$, $\labelprod{\sI(\pb)}\in\cosets$ is the minimal coset representative corresponding to $\sI(\pb)$ (cf.~Remark \ref{rem:order_ideals_to_min_coset_reps}), and $\indx(\pb)$ is the index of the simple reflection corresponding to $\pb$ (i.e.~$\labl(\pb)=s_{\indx(\pb)}$). Note that $\weight(\pb)$ in fact equals the number of copies of the simple root $\sdr_{\indx(\pb)}$ the simple reflection $\labl(\pb)=s_{\indx(\pb)}$ corresponding to $\pb$ subtracts from $\labelprod{\sI(\pb)\setminus\pb}\,\wop\cdot\dfwt[\is]$.

We have shown in Corollary \ref{cor:full_action_special_weyl_elts} that $\wP=\uij[\cis+1]\cdots\uij[1]$, so that the extremal path following $\bwP$ from the weight vector of weight 
\[
\wop(\dfwt[\is])=\uij[0](\dfwt[\is])
\]
to the lowest weight vector follows the successive actions of the $\uij$ for $j\in[\cis]$. Moreover, in Proposition \ref{prop:uij_repeated_action} we showed that for each $j\in [\cis+1]$, each simple reflection $\labl(\pb)=s_{\indx(\pb)}$ corresponding to $\pb\in\su_{i_j}$ acts by subtracting $\cis-j+1$ copies of the simple root $\sdr_{\indx(\pb)}$ in the action of $\uij$ on the weight $\uij[j-1]\cdots\uij[0](\dfwt[\is])$. Hence, we find that 
\[
\weight(\pb)=\cis-j+1   \qfor   \pb\in\su_{i_j}=\ideal_{i_j}\setminus\ideal_{i_{j-1}},
\]
so in particular we have that $\weight^{-1}(\N)=\ideal_{i_{\cis}}$ and the toric monomial in equation \eqref{eq: phiGLS_intermediate_expression_repeated} can therefore be expanded as
\[
\phiGLS[\is]|_{\opendunim}(u_-)=(-1)^h\prod_{j=1}^{\cis}\prod_{\pb\in\su_{i_j}}(-a_{\pb})^{\cis-j+1}.
\]

Since we are in the special case where the weight on $\ideal_{i_{\cis}}$ itself is defined by membership in its subideals $\ideal_{i_j}$, we can rewrite this using ordinary order ideals as
\ali{
    \phiGLS[\is]|_{\opendunim}(u_-)
    &=(-1)^h\prod_{j=1}^{\cis}\left(\prod_{\pb\in \idealij\setminus\idealij[j-1]} (-a_{\pb})^{\cis-j+1}\right)
    = (-1)^h\prod_{j=1}^{\cis}\left( \prod_{\pb\in \idealij\setminus \idealij[j-1]} (-a_\pb)\right)^{\cis-j+1}\\
    &= (-1)^h\prod_{j=1}^{\cis} \prod_{m=1}^j \left(\prod_{\pb\in \idealij[m]\setminus \idealij[m-1]} (-a_\pb)\right)
    = (-1)^{h+\sum_{j=1}^{\cis} \ell(\idealij)}\prod_{j=1}^{\cis}a_{\idealij}.
}
Now, it turns out that the minus signs cancel, see Remark \ref{rem:generalized_minors_sign}, so we conclude
\[
\phiGLS[\is]|_{\opendunim} = \prod_{j=1}^{\cis}a_{\idealij},
\]
as claimed.
\end{proof}

\begin{rem}
    \label{rem:generalized_minors_sign}
    We recall that $\weight(\pb)$ is exactly the number of simple roots the simple reflection $\labl(\pb)$ removed from the weight $\labelprod{\sI(\pb)\setminus\pb}$ for \emph{all} $\pb\in\minposet$, including $\pb\in\su_{i_{\cis+1}}$ since these act trivially. We conclude that we must have removed exactly the height $h$ simple roots as defined in \eqref{eq:gen_minor_minus_sign} in total, since we have acted with all simple reflections in $\idealij[\cis+1]=\minposet$ (cf.~Corollary \ref{cor:full_action_special_weyl_elts}). Hence, we conclude that 
    \[
    h = \sum_{j=1}^{\cis} \ell(\idealij)
    \]
    This is exactly what we claimed in Section \ref{ssec:Coordinate_ring_of_dunimP} and Remark \ref{rem:generalized_minors_sign_prelims}.
\end{rem}

For notational convenience, we will define $p_{i_j}=p_{\wij}=p_{\idealij}$ to be the Pl\"ucker coordinate with weight $\wij(-\dfwt[k])=\Pi(\idealij)(-\dfwt[k])$ corresponding to the minimal coset representative $\wij$ or equivalently to the order ideal $\idealij$. The following Pl\"ucker monomials will play an important role in the next section where we compute the polynomials $\cD_{\is}$ corresponding to the $\phiGLS[\is]$.

\begin{cor}
    \label{cor:leading_plucker_for_phi_i}
    The toric monomial for $\phiGLS[\is]|_{\opendunim}$ found in Theorem \ref{thm:toric_monomial_phi} is contained in the toric expression for $(p_{i_1}p_{i_2}\dots p_{i_{\cis}})|_{\opendunim}$.  
\end{cor}

\begin{proof}
    By definition, the toric expression of $p_{i_j}$ contains toric monomials corresponding to all possible embeddings in $\minposet$ of the order ideal $\idealij$ corresponding to $\wij$. In particular the identity embedding $\id_{\idealij}:\idealij\hookrightarrow\minposet$ is a valid embedding of $\idealij$ whose corresponding toric monomial is $a_{\idealij}$. Thus, the toric monomial in Theorem \ref{thm:toric_monomial_phi} is contained in the toric expression for $(p_{i_1}p_{i_2}\dots p_{i_{\cis}})|_{\opendunim}$.
\end{proof}

%% file: 4denominators.tex
\section{Pl\"ucker coordinate denominators from the generalized minors} \label{sec:denominators}
As we found in Corollary \ref{cor:leading_plucker_for_phi_i}, the toric monomial of Proposition \ref{thm:toric_monomial_phi} is contained in the toric expression for the Pl\"ucker monomial $\p_{i_1}\p_{i_2}\cdots\p_{i_\cis}$, where we denote by $\p_{i_j}$ the Pl\"ucker coordinate corresponding to the order ideal $\idealij\subideal\minposet$ or equivalently to $\wij\in\cosets$ as defined Section~\ref{sec:toric-monomials}. However, this Pl\"ucker monomial in general also contains extraneous terms in its toric expression, and we will now show how to use the combinatorics of order ideals of $\minposet$ to cancel these out and obtain a Pl\"ucker coordinate expression, which we denote by $\cD_{\is}$, for $\phiGLS[\is]$ on $\cmX$. We start by proving the following simplifying result:
\begin{lem}
    \label{lem:monomial_plucker}
    The restriction $\p_{i_{\cis}}|_{\opendunim}$ is a monomial in the coordinates of $\opendunim$.
\end{lem}
\begin{proof}
   As discussed in Section \ref{ssec:min_posets_and_Pluckers}, we may compute the restriction of $\p_{i_{\cis}}$ to $\opendunim$ by finding all possible embeddings of the order ideal $\ideal_{i_{\cis}}$ corresponding to $\wij[\cis]$ into $\minposet$. By Proposition \ref{prop:uij_repeated_action}, the weight corresponding to $\idealij[\cis]$, $\Pi(\idealij[\cis])(-\dfwt[k])=\wij[\cis](-\dfwt[k])= \dfwt[i_{\cis+1}] - \dfwt[i_{\cis}]$, is join-irreducible (considered in the Bruhat order) since it covers exactly one element (corresponding to the action of $s_{i_{\cis+1}}$). Hence $\ideal_{\cis}$ has a unique maximal element $\pb$ and this element has label $\labl(\pb)=s_{i_{\cis+1}}$. Therefore, any embedding $\emb_{\cis}:\ideal_{\cis}\rightarrow \minposet$ must have $\emb(\pb)$ as its maximal element. However, by construction $\minposet\setminus\idealij[\cis]$ has no element with label $s_{i_{\cis+1}}$. Thus, the only embedding of $\ideal_{\cis}$ is the identity embedding, corresponding to the monomial $a_{\ideal_{i_1}}$. 
\end{proof}
As an immediate consequence, we now obtain the following corollary.
\begin{cor}
\label{cor: cis=1}
    If $\cis=1$, then $\phiGLS[\is]|_{\opendunim}=\p_{i_\cis}|_{\opendunim}$.
\end{cor}
\begin{proof}
    By Corollary \ref{cor:leading_plucker_for_phi_i}, the toric expression for $\p_{i_1}|_{\opendunim}$ contains the desired toric monomial $\phiGLS[\is]|_{\opendunim}$, and the expression contains exactly one toric monomial by Lemma \ref{lem:monomial_plucker}.
\end{proof}

On the other hand, we will see that for $\cis>1$, the $\cD_{\is}$ are never monomials. Instead, we will construct a signed sum of Pl\"ucker monomials that cancels all extraneous toric terms. To organize the cancellation, we develop a classification of the toric terms in the restriction of a Pl\"ucker monomial. For now, we only consider the case $\cis=2$; the remaining cases $\cis\in\{3,4\}$ only occur in four cases---$\cis=3$ for $\is=4$ in $\LGE_6$ and $\is=3,5$ in $\LGE_7$ and $\cis=4$ for $\is=4$ in $E_7$. Rather than handling these four cases in the generality of this section, we will instead apply our construction directly in Appendix \ref{app:exceptional}.

\begin{df}
For any pair $\ideal_1,\ideal_2\subideal\minposet$ of order ideals, we say that a filter (upward closed set) $\sB\subset\ideal_1$ can be moved to $\ideal_2$ if there exists a subposet $\sB'\subset\minposet\setminus\ideal_2$ isomorphic to $\sB$ by a label-preserving isomorphism such that $\ideal_1\sqcup\sB'$ is an order ideal. We then define
\[
\movable{\ideal_1}{\ideal_2} = \{\sB\subset\ideal_1~|~\text{$\sB$ is a minimal (w.r.t.~inclusion) movable filter of $\ideal_1$}\}.
\]
For each $\sB\in\movable{\ideal_1}{\ideal_2}$ which can be moved to $\ideal_2$, we will use $\sB'$ to denote the corresponding label-isomorphic subset of $\minposet\setminus\ideal_2$. 
\end{df}
Note that the elements of $\movable{\ideal_1}{\ideal_2}$ need not have the same size; minimality only ensures that no strict subsets of an element of $\movable{\ideal_1}{\ideal_2}$ are also in $\movable{\ideal_1}{\ideal_2}$ themselves.

\begin{ex}\label{ex:D5P5_movable}
For example, consider $\LG(5,10)=\LGD_5^{\SC}/\P_5$ so that 
\begin{equation}\Yboxdim{7pt}
\wP=(s_5)(s_3s_4s_5)(s_3s_4s_5)(s_2s_3s_4s_5)(s_1s_2s_3s_4s_5)
\qand
\minposet=\raisebox{-14pt}{\scriptsize\young(5,45,345,2345,12345)},
\label{eq:D5P5_wP_and_minposet}
\end{equation}
then the elements 
\[\Yboxdim{7pt}
\left(\ideal_1=\raisebox{-7.5pt}{\scriptsize\young(5,45,3,2)},\ideal_2=\raisebox{-14pt}{\scriptsize\young(5,45,34,2,1)}\right) 
~~\text{have}~~
M_{\ideal_1,\ideal_2}=\left\{\raisebox{0pt}{\scriptsize\young(5)},\raisebox{-3.5pt}{\scriptsize\young(3,2)}\right\}
~~\text{yielding}~~
\left(\raisebox{-7.5pt}{\scriptsize\young(5,4,3,2)},\raisebox{-14pt}{\scriptsize\young(5,45,345,2,1)}\right) 
~~\text{resp.}~~
\left(\raisebox{6pt}{\scriptsize\young(5,45)},\raisebox{-14pt}{\scriptsize\young(5,45,34,23,12)}\right),
\]
giving an alternative description of $\movable{\ideal_1}{\ideal_2}$ as the set of all minimal filters $\sB$ (i.e.~subsets containing all elements larger than the elements contained) of $\ideal_1$ that can be added to $\ideal_2$ resulting in a new order ideal. 
\end{ex}

In order to study the combinatorics of these moves, we define a poset $\sP_{\is}$ as follows. 
\begin{df}
    The elements of $\sP_\is$ will be the set of all pairs of order ideals $(\ideal_1,\ideal_2)$ that can be obtained from the pair $(\idealij[1],\idealij[2])$ by any sequence of moves going from $\idealij[1]$ to $\idealij[2]$. More precisely, $\sP_{\is}$ consists of those pairs $(\ideal_1,\ideal_2)$ such that there exists a sequence of $d\ge 0$ moves
    \[
    \sB_1\in\movable{\idealij[1]}{\idealij[2]}, \quad 
    \sB_2\in\movable{\idealij[1]\setminus \sB_1}{\wij[2]\sqcup\sB_1'}, \quad 
    \sB_3\in\movable{\idealij[1]\setminus (\sB_1\sqcup \sB_2)}{\idealij[2]\sqcup(\sB_1'\sqcup\sB_2')},\quad \dots
    \] 
    with $\ideal_1=\idealij[1]\setminus(\sB_1\sqcup\dots\sqcup \sB_d)$ and $\ideal_2 =  \idealij[2]\sqcup(\sB_1'\sqcup\dots\sqcup\sB_d')$, where each $\sB_j'$ is the corresponding label-isomorphic subset of $\minposet\setminus (\idealij[2]\sqcup (\sB_1'\sqcup\dots\sqcup\sB_{j-1}'))$ to the move $\sB_j$. We define a partial order on $\sP_\is$ by $(\ideal_1,\ideal_2)>(\ideal_3,\ideal_4)$ if $(\ideal_3,\ideal_4)$ can be obtained from $(\ideal_1,\ideal_2)$ by a sequence of moves starting from $(\ideal_1,\ideal_2)$. In particular, $(\idealij[1],\idealij[2])$ will be the unique top element of $\sP_\is$.
\end{df}

For the rest of this section we will only want to consider moves between elements of $\sP_{\is}$, so we restrict to these \emph{admissible} moves as follows.

\begin{df}
    For any $(\ideal_1,\ideal_2)\in\sP_{\is}$, we call $\sB\in\movable{\ideal_1}{\ideal_2}$ is \emph{admissible} if $(\ideal_1\setminus\sB,\ideal_2\sqcup\sB')\in\sP_{\is}$ and we denote the subset of admissible moves by $A_{\ideal_1,\ideal_2}\subset \movable{\ideal_1}{\ideal_2}$. Analogously, we say that $\sB\in\movable{\ideal_2}{\ideal_1}$ is \emph{admissible} if $(\ideal_1\sqcup\sB',\ideal_2\setminus \sB)\in\sP_{\is}$ and we again the subset of admissible moves by $A_{\ideal_2,\ideal_1}\subset\movable{\ideal_2}{\ideal_1}$.
\end{df} 
 
Note by definition that for any $(\ideal_1,\ideal_2)\in\sP_{\is}$, we have $A_{\ideal_1,\ideal_2}=\movable{\ideal_1}{\ideal_2}$ since $\sP_{\is}$ is constructed using every $\sB\in \movable{\ideal_1}{\ideal_2}$. Moreover, by Remark \ref{rem:order_ideals_to_min_coset_reps}, we can associate to every cover relation in $\sP_{\is}$ a Weyl group element, namely $\labelprod{\sB}\in\weyl$ where $\sB$ is the movable poset. However, note that $\movable{\ideal_2}{\ideal_1}$ may contain $A_{\ideal_2,\ideal_1}$ as a proper subset for $(\ideal_1,\ideal_2)\in\sP_{\is}$. Furthermore, $\sB\in A_{\ideal_1,\ideal_2}$ if and only if $\sB'\in A_{\ideal_2\sqcup\sB',\ideal_1\setminus\sB}$, so by construction of $\sP_{\is}$ we have $A_{\ideal_2,\ideal_1}\neq\varnothing$ for any $(\ideal_1,\ideal_2)\in\sP_{\is}$ unequal to the top element $(\idealij[1],\idealij[2])$. In addition, we have a further simplification when studying only the admissible moves.

\begin{rem}
    \label{rem:poset_is_rep}
    Let $L^\vee$ denote the Levi subgroup of $\dG$ corresponding to the component of the Dynkin diagram of $\dG$ which contains $i_2$ after deleting the nodes $i_1$ and $i_3$. Then we note that the poset $\sP_{\is}$ can be identified with the Bruhat order on the weight space of the fundamental representation of $L^\vee$ corresponding to the node $i_2$. In particular, Corollary \ref{cor:i2_minuscule} implies that this is a minuscule representation so each cover relation $(\ideal_1,\ideal_2)\gtrdot(\ideal_3,\ideal_4)$ of $\sP_{\is}$ corresponds to an admissible move in $A_{\ideal_1,\ideal_2}$ of size $1$. Thus for the remainder of this section we will assume that all admissible moves have size $1$.
\end{rem}

Using this simplification, we can more easily study the possible moves of $\sP_{\is}$. First, we saw in Example \ref{ex:D5P5_movable} that the Weyl group elements $s_5$ and $s_2s_3$ corresponding to the two moves of $M_{\ideal_1,\ideal_2}$ commute with each other. In fact, this commutation holds more generally for any $(\ideal_1,\ideal_2)\in \sP_{\is}$ as we now show.

\begin{lem}
    \label{lem:moves_commute}
    For any $(\ideal_1,\ideal_2)\in\sP_{\is}$, the labels of all $\pb\in A_{\ideal_1,\ideal_2}$ commute with each other and similarly the labels of all elements of $A_{\ideal_2,\ideal_1}$ commute. 
\end{lem}
\begin{proof}
If $\pb_1,\pb_2\in A_{\ideal_1,\ideal_2}$ are two distinct elements, then they must be incomparable and both must be maximal elements of $\ideal_1$. To see this, note that if they were comparable, then at least one of them would not be maximal. If either was not maximal, then removing it from $\ideal_1$ would not yield a valid order ideal. 

Thus we can choose linear extensions of the partial order on $\ideal_1$ where the final two elements are either $\pb_1>\pb_2$ or $\pb_2>\pb_1$; clearly, since such total orders correspond to reduced expressions of $\labelprod{\ideal_1}$ differing only in reversing the order of the final reflections $\labl(\pb_1)$ and $\labl(\pb_2)$, this is only possible if $\labl(\pb_1)$ and $\labl(\pb_2)$ commute. The situation for $A_{\ideal_2,\ideal_1}$ is identical.
\end{proof}

Second, we may also study the actions of moves between elements of $\sP_{\is}$ in terms of their corresponding weights. Hence, we will in the following denote for a given order ideal $\ideal_i\subideal\minposet$ the corresponding weight by
\begin{equation}
\wt_i = \labelprod{\ideal_i}\cdot(-\dfwt[k]),    
\label{eq:df_wt_corresponding_to_order_ideal}
\end{equation}
and we will similarly write $\wtij=\labelprod{\idealij}\cdot(-\dfwt[k])$. We first establish two auxiliary statements which we will then use to characterize the possible moves in $\sP_{\is}$. 

\begin{lem} 
    \label{lem: move_action}
    Let $(\ideal_1,\ideal_2)\in\sP_{\is}$, then for any $\pb\in A_{\ideal_1,\ideal_2}$ we have
    \[
    \llan\wt_1+\wt_2,\sr_{\indx(\pb)}\rran=0.
    \]
\end{lem}

\begin{proof}
    The action of a simple reflection $s_m$ on the weight $\wt_i$ is given by $s_m(\wt_i) = \wt_i - \llan\wt_i,\sr_m\rran\sdr_m$, and by minusculity we have $\llan\wt_i,\sr_m\rran \in \{-1,0,+1\}$. The three possibilities $\{-1,0,+1\}$ correspond respectively to the situations: (1) there is some $\pb$ with label $s_m$ such that $\ideal_i\sqcup\pb\subideal\minposet$ is an order ideal, (2) no element with label $s_m$ can be removed from or added to $\ideal_i$, and (3) there is some $\pb$ with label $s_m$ such that $\ideal_i\setminus\pb\subideal\minposet$ is an order ideal. 

    Now let $\pb\in A_{\ideal_1,\ideal_2}$ which is moved to $\pb'\in\minposet\setminus\ideal_2$. Then by above we compute that 
    \[
    \llan \wt_1 + \wt_2,\sr_{\indx(\pb)}\rran = \llan \wt_1,\sr_{\indx(\pb)}\rran + \llan \wt_2, \sr_{\indx(\pb)}\rran = 1 - 1 = 0,
    \]
    as claimed.
\end{proof}

\begin{lem}
    \label{lem: move_equality}
    Let $(\ideal_1,\ideal_2)\in\sP_{\is}$, then for any $m\in [n]$ we have
    \[\llan\wt_1+\wt_2,\sr_m\rran = \llan \wt_{i_1}+\wt_{i_2},\sr_m\rran = \llan \dfwt[i_3]-\dfwt[i_1],\sr_m\rran.\]
\end{lem}

\begin{proof}
    Since $\sP_{\is}$ is connected by moves, we may reduce the statement to establishing that
    \[
    \llan\wt_1+\wt_2,\sr_m\rran = \llan\wt_3+\wt_4,\sr_m\rran
    \]
    for $(\ideal_3,\ideal_4)\in\sP_{\is}$ which can be obtained from $(\ideal_1,\ideal_2)$ by a move $\pb\in A_{\ideal_1,\ideal_2}$. The first equality $\llan\wt_1+\wt_2,\sr_l\rran = \llan \wt_{i_1}+\wt_{i_2},\sr_l\rran$ will then follow by a straightforward induction on the number of moves between $(\idealij[1],\idealij[2])$ and $(\ideal_1,\ideal_2)$. We now compute
    \begin{align*}
        \llan\wt_3+\wt_4,\sr_m\rran &= \llan \wt_1 - \llan\wt_1,\sr_{\indx(\pb)}\rran\sdr_{\indx(\pb)} + \wt_2 - \llan\wt_2,\sr_{\indx(\pb)}\rran\sdr_{\indx(\pb)}, \sr_m\rran\\
        &= \llan \wt_1 + \wt_2, \sr_m\rran - \llan \wt_1 + \wt_2, \sr_{\indx(\pb)}\rran\llan\sdr_{\indx(\pb)},\sr_m\rran
        = \llan\wt_1 + \wt_2, \sr_m\rran,
    \end{align*}
    where the last equality follows from Lemma \ref{lem: move_action}. Proposition \ref{prop:uij_repeated_action} now immediately implies the remainder of the lemma.
\end{proof}

Using these lemmas, we give a characterization of moves in $\sP_{\is}$. 

\begin{lem}
\label{lem: move-symmetry}
Let $(\ideal_1,\ideal_2)\in\sP_\is$, then $\llan \wt_1, \sr_m\rran = -\llan\wt_2,\sr_m\rran$ for any $m\in [n]\setminus \{i_1,i_3\}$.
\end{lem}
\begin{proof}
    By Lemma \ref{lem: move_equality}, we have that
    \[
    \llan\wt_1+\wt_2,\sr_m\rran = \llan \dfwt[i_3]-\dfwt[i_1],\sr_m\rran.
    \]
    Since $m\notin\{i_1,i_3\}$, the right hand side is zero, and we immediately obtain the claim.
\end{proof}

For any $(\ideal_1,\ideal_2)\in \sP_{\is}$, we interpret the previous lemma as follows. For any maximal element $\pb\in\ideal_1$ which has label distinct from $s_{i_1},s_{i_3}$, if $\minposet \setminus \ideal_2\neq\varnothing$, then $\pb\in A_{\ideal_1,\ideal_2}$. Similarly, for any maximal element $\pb$ of $\ideal_2$ which has label distinct from $s_{i_1},s_{i_3}$, then $\pb\in A_{\ideal_2,\ideal_1}$. The last statement follows from the fact that the unique maximal element of $\idealij[2]$ has label $s_{i_3}$ and that every other element of $\ideal_2\setminus \idealij[2]$ was moved from $\idealij[1]$. We now use moves between order ideals to distinguish embeddings of order ideals.

\begin{df}
    Let $(\ideal_1,\ideal_2)\in\sP_{\is}$. For $\pb\in A_{\ideal_2,\ideal_1}$ which is moved to $\pb'\in\minposet\setminus\ideal_1$, we define 
    \[
        \first[\pb]{\ideal_1}{\ideal_2} = 
        \Bigl\{(\emb_1,\emb_2) 
        ~\Big|~
        \emb_i:\ideal_i\hookrightarrow\minposet,~
        \text{$\emb_1'|_{\ideal_1}=\emb_1$ and $\emb_1'(\pb')=\emb_2(\pb)$ gives valid $\emb_1':\ideal_1\sqcup\pb'\hookrightarrow\minposet$}\Bigr\}.
    \]
    Analogously for $\pb\in A_{\ideal_1,\ideal_2}$ which is moved to $\pb'\in\minposet\setminus\ideal_2$, we define
    \[
        \second[\pb]{\ideal_1}{\ideal_2} = 
        \Bigl\{(\emb_1,\emb_2) 
        ~\Big|~ \emb_i:\ideal_i\hookrightarrow\minposet,~
        \text{$\emb_2'|_{\ideal_2}=\emb_2$ and $\emb_2'(\pb')=\emb_1(\pb)$ gives valid $\emb_2':\ideal_2\sqcup\pb'\hookrightarrow\minposet$}\Bigr\}.
    \]
\end{df}

In the context of embeddings of order ideals, or equivalently terms of the toric expression of a Pl\"ucker monomial $p_{\ideal_1}p_{\ideal_2}$, the set $\first[\pb]{\ideal_1}{\ideal_2}$ can be thought of as those toric pairs $(\emb_1, \emb_2)$ which allow the removal of an element $\emb_2(\pb)$ from the image of $\emb_2:\ideal_2\rightarrow\minposet$ and simultaneously the addition of the same element $\emb_2(\pb)$ to the image of $\emb_1:\ideal_1\rightarrow\minposet$. The set $\second[\pb]{\ideal_1}{\ideal_2}$ can be thought of similarly, but with the movement in the other direction from $\emb_1$ to $\emb_2$. We will use $\first[\pb]{\ideal_1}{\ideal_2}$ and $\second[\pb]{\ideal_1}{\ideal_2}$ to refine the information of $A_{\ideal_1,\ideal_2}$ and $A_{\ideal_2,\ideal_1}$ in the sense that moves which may be allowed in general for order ideals $\ideal_1$ and $\ideal_2$ may not be allowed for particular embeddings $\emb_1:\ideal_1\rightarrow\minposet$ and $\emb_2:\ideal_2\rightarrow\minposet$. 
\begin{ex}[Continuing from Example \ref{ex:D5P5_movable}]
  Consider again $\LG(5,10)=\LGD_5^\SC/\P_5$ with $\wP$ and $\minposet$ given in \eqref{eq:D5P5_wP_and_minposet}. Then the order ideals
  \[\Yboxdim{7pt}
    \left(\ideal_1=
    \raisebox{-14pt}{\scriptsize\young(5,45,345,23,12)},
    \ideal_2=
    \raisebox{-14pt}{\scriptsize\young(5,45,345,234,123)}\right)\quad\text{have}\quad M_{\ideal_1,\ideal_2}=\left\{\raisebox{0pt}{\scriptsize\young(5)}\right\}\quad\text{giving}\quad
    \left(\ideal_1\setminus\raisebox{0pt}{\scriptsize\young(5)}=
    \raisebox{-14pt}{\scriptsize\young(5,45,34,23,12)},
    \ideal_2\sqcup\raisebox{0pt}{\scriptsize\young(5)}=
    \raisebox{-14pt}{\scriptsize\young(5,45,345,2345,123)}\right).
  \]
To determine the set $\second[\pb]{\ideal_1}{\ideal_2}$, where $\pb$ denotes the maximal element of $\ideal_1$ with label $s_5$, we need to consider the embeddings of the order ideals $\ideal_1,\ideal_2\subideal\minposet$. To distinguish order ideals from their images under embeddings, we will draw the latter on top of a light gray copy of $\minposet$. The order ideal $\ideal_1\subideal\minposet$ allows three embeddings $\emb_i$ with images given respectively by
    \[\Yboxdim{7pt}
    \emb_i:\quad
    \ideal_1=
    \raisebox{-14pt}{\scriptsize\young(5,45,345,23,12)}
    \quad\hookrightarrow\quad
    \raisebox{-14pt}{\color{white!60!gray} \scriptsize\young(5,45,345,2345,12345)}
    \hspace{-33.75pt}
    \raisebox{-14pt}{\scriptsize\young(5,45,345,23,12)}
    \hspace{14pt}
    ,
    \quad
    \raisebox{-14pt}{\color{white!60!gray} \scriptsize\young(5,45,345,2345,12345)}
    \hspace{-33.75pt}
    \raisebox{-14pt}{\scriptsize\young(5,45,34,23,12)}
    \hspace{6.4pt}
    \raisebox{-7.3pt}{\scriptsize\young(5)}
    \hspace{7pt},
    \quad
    \raisebox{-14pt}{\color{white!60!gray} \scriptsize\young(5,45,345,2345,12345)}
    \hspace{-33.75pt}
    \raisebox{-14pt}{\scriptsize\young(5,45,34,23,12)}
    \hspace{13pt}
    \raisebox{-14pt}{\scriptsize\young(5)}
    \quad\subset\quad
    \raisebox{-14pt}{\scriptsize\young(5,45,345,2345,12345)}=\minposet.
    \]
    On the other hand, the ideal $\ideal_2\subideal\minposet$ consists of the three leftmost columns of $\minposet$ and only has a single embedding, namely
    \[\Yboxdim{7pt}
    \id_{\ideal_2}:\quad
    \ideal_2=
    \raisebox{-14pt}{\scriptsize\young(5,45,345,234,123)}
    \quad\hookrightarrow\quad
    \raisebox{-14pt}{\color{white!60!gray} \scriptsize\young(5,45,345,2345,12345)}
    \hspace{-33.75pt}
    \raisebox{-14pt}{\scriptsize\young(5,45,345,234,123)}
    \hspace{14pt}
    \quad\subset\quad
    \raisebox{-14pt}{\scriptsize\young(5,45,345,2345,12345)}
    =\minposet.
    \]
    Clearly we have $\emb_1(\pb)\cap\id_{\ideal_2}(\ideal_2)\neq\varnothing$, so the map from $\ideal_2\sqcup\pb'$ to $\minposet$ sending $\ideal_2$ to $\id_{\ideal_2}(\ideal_2)\subset\minposet$ and $\pb'$ to $\emb_1(\pb)$ is \emph{not} an embedding. On the other hand, $\emb_2$ and $\emb_3$ both map $\pb$ to the complement of $\id_{\ideal_2}(\ideal_2)$, so we do get valid embeddings of $\ideal_2\sqcup\pb'$ by extending $\id_{\ideal_2}$ to $\pb'$ by $\emb_2(\pb)$ or $\emb_3(\pb)$. Hence, we find that
    \[
    \second[\pb]{\ideal_1}{\ideal_2} = \bigl\{(\emb_2,\id_{\ideal_2}),~(\emb_3,\id_{\ideal_2})\bigr\}
    \]
    
    Now, consider the resulting pair of order ideals $\ideal_3=\ideal_1\setminus\pb$ and $\ideal_4=\ideal_2\sqcup\pb'$. These order ideals of course allow the element $\pb'$ with label $s_5$ to move back from $\ideal_4$ to $\ideal_3$, and we indeed find $\movable{\ideal_4}{\ideal_3}=\Yboxdim{7pt}\{\raisebox{0pt}{\scriptsize\young(5)}\}$. The ideal $\ideal_3$ now consists of the leftmost two columns of $\minposet$, and hence has a single embedding, namely
    \[\Yboxdim{7pt}
    \id_{\ideal_3}:\quad
    \ideal_3=
    \raisebox{-14pt}{\scriptsize\young(5,45,34,23,12)}
    \quad\hookrightarrow\quad
    \raisebox{-14pt}{\color{white!60!gray} \scriptsize\young(5,45,345,2345,12345)}
    \hspace{-33.75pt}
    \raisebox{-14pt}{\scriptsize\young(5,45,34,23,12)}
    \hspace{21pt}
    \quad\subset\quad
    \raisebox{-14pt}{ \scriptsize\young(5,45,345,2345,12345)}
    =\minposet
    \]
    On the other hand, $\ideal_4\subideal\minposet$ now allows two embeddings $\emb_4$ and $\emb_5$ with images given respectively by
    \[\Yboxdim{7pt}
    \emb_i:\quad
    \ideal_4=
    \raisebox{-14pt}{\scriptsize\young(5,45,345,2345,123)}
    \quad\hookrightarrow\quad
    \raisebox{-14pt}{\color{white!60!gray} \scriptsize\young(5,45,345,2345,12345)}
    \hspace{-33.75pt}
    \raisebox{-14pt}{\scriptsize\young(5,45,345,2345,123)}
    \hspace{14pt}
    ,
    \quad
    \raisebox{-14pt}{\color{white!60!gray} \scriptsize\young(5,45,345,2345,12345)}
    \hspace{-33.75pt}
    \raisebox{-14pt}{\scriptsize\young(5,45,345,234,123)}
    \hspace{6.4pt}
    \raisebox{-14pt}{\scriptsize\young(5)}
    \quad\subset\quad
    \raisebox{-14pt}{\scriptsize\young(5,45,345,2345,12345)}=\minposet.
    \]
    Note that $\emb_4(\pb')$ and $\emb_5(\pb')$ both lie in the complement of $\id_{\ideal_3}(\ideal_3)$, so the identity embedding on $\ideal_3$ can be extended to $\ideal_3\sqcup\pb$ by defining the image of $\pb$ to be either of these elements. Hence, we find
    \[
    \first[s_5]{\ideal_3}{\ideal_4} = \bigl\{(\id_{\ideal_3},\emb_4),~(\id_{\ideal_3},\emb_5)\bigr\}.
    \]
    Observe that the sets of embeddings $\second[\pb]{\ideal_1}{\ideal_2}$ and $\first[\pb']{\ideal_3}{\ideal_4}$ are closely related: if we consider the associated toric terms in the sense of \eqref{eq:toric_term_from_poset}, we find that the pairs of embeddings in $\second[\pb]{\ideal_1}{\ideal_2}$ and $\first[\pb']{\ideal_3}{\ideal_4}$ satisfy
    \[
    a_{\emb_2(\ideal_1)}a_{\ideal_2} = a_{\ideal_3}a_{\emb_4(\ideal_4)} 
    \qand
    a_{\emb_3(\ideal_1)}a_{\ideal_2} = a_{\ideal_3}a_{\emb_5(\ideal_4)} 
    \]
    since $\emb_2(\ideal_1)\cup\ideal_2=\ideal_3\cup\emb_4(\ideal_4)$ as well as $\emb_2(\ideal_1)\cap\ideal_2=\ideal_3\cap\emb_4(\ideal_4)$ and analogously for the other toric expression. Hence, the corresponding terms in the toric expression for the Pl\"ucker coordinate expression $\p_{\ideal_1}\p_{\ideal_2}-\p_{\ideal_3}\p_{\ideal_4}$ will cancel by \eqref{eq:toric_expression_for_Plucker_coord_in_posets}.
\end{ex}

In the discussion at the end of the above example we showed that there will be a natural correspondence of toric terms in a Pl\"ucker coordinate expression obtained by moving order ideal elements, and we will address the general situation in Lemma \ref{lem: toric-bijections}. Before we handle the general situation, we have to be more precise with comparing the various toric terms, leading us to the following refinment of the partitioning of the toric terms for any $(\ideal_1,\ideal_2)\in\sP_{\is}$.
\begin{df}
    Let $(\ideal_1,\ideal_2)\in \sP_{\is}$, then for any $A_1\subset A_{\ideal_2,\ideal_1}$ and $A_2\subset A_{\ideal_1,\ideal_2}$ we define the set of pairs of embeddings
    \[
    T^{A_1, A_2}_{\ideal_1,\ideal_2}=
    \left(\bigcap_{\pb\in A_1} \first[\pb]{\ideal_1}{\ideal_2}\right)
    \cap 
    \left(\bigcap_{\pb\in A_2} \second[\pb]{\ideal_1}{\ideal_2}\right)
    \cap 
    \left(\bigcap_{\pb\in A_{\ideal_2,\ideal_1}\setminus A_1} (\first[\pb]{\ideal_1}{\ideal_2})^c\right)
    \cap 
    \left(\bigcap_{\pb\in A_{\ideal_1,\ideal_2}\setminus A_2}(\second[\pb]{\ideal_1}{\ideal_2})^c\right),
    \]
    where the complements in the definition above are taken with respect to the overall set of pairs of embeddings $\{(\emb_1:\ideal_1\hookrightarrow\minposet, \emb_2:\ideal_2\hookrightarrow\minposet)\}$. If a set $A_i=\{\pb\}$ consists of a single element, we will frequently abuse notation and simply write $\pb$ instead of $\{\pb\}$.
\end{df}

Informally speaking, membership of a pair of embeddings $(\emb_1,\emb_2)$ in $T^{A_1,A_2}_{\ideal_1,\ideal_2}$ indicates that the images of $\emb_1$ and $\emb_2$ allow moves in $A_2$ from $\emb_1(\ideal_1)$ to $\emb_2(\ideal_2)$ as well as moves in $A_1$ from $\emb_2(\ideal_2)$ to $\emb_1(\ideal_1)$, but do not allow moves in $A_{\ideal_2,\ideal_1}\setminus A_1$ or in $A_{\ideal_1,\ideal_2}\setminus A_2$. 
Note that any elements $\sB\in (\movable{\ideal_1}{\ideal_2}\cup\movable{\ideal_2}{\ideal_1})\setminus (A_{\ideal_1,\ideal_2}\cup A_{\ideal_2,\ideal_1})$, do not play a role in the definition above. The reason for this is that if a move is not admissible, and hence does not occur between elements of $\sP_{\is}$, it will not affect our cancellation argument in a meaningful way. Furthermore by construction, the sets $T^{A_1,A_2}_{\ideal_1,\ideal_2}$ and $T^{A_1',A_2'}_{\ideal_1,\ideal_2}$ for $A_1\neq A_1'$ or $A_2\neq A_2'$ are disjoint. Thus for $(\ideal_1,\ideal_2)\in\sP_{\is}$, these $T^{A_1,A_2}_{\ideal_1,\ideal_2}$ partition the toric terms of $\p_{\ideal_1}\p_{\ideal_2}$ via the decomposition
\[
\p_{\ideal_1}\p_{\ideal_2}|_{\opendunim} = \sum_{(A_1,A_2)} \sum_{(\emb_1,\emb_2)} a_{\emb_1}a_{\emb_2}
\]
where the first sum is over all pairs $(A_1,A_2)$ with $A_1\subset A_{\ideal_2,\ideal_1}$ and $A_2\subset A_{\ideal_1,\ideal_2}$ and the second sum is over all pairs $(\emb_1,\emb_1)\in T^{A_1,A_2}_{\ideal_1,\ideal_2}$. We now study this decomposition explicitly for the Pl\"ucker monomial $\p_{i_1}\p_{i_2}$ which corresponds to the top element $(\idealij[1],\idealij[2])\in \sP_{\is}$.

\begin{lem}
\label{lem: initial-plucker}
The only nonempty toric sets for the top element $(\idealij[1],\idealij[2])\in\sP_{\is}$ are
\[
T^{\varnothing,\varnothing}_{\idealij[1],\idealij[2]}
=\bigl\{(\id_{\idealij[1]},\id_{\idealij[2]})\bigr\}
\qand
T^{\varnothing, \pb}_{\idealij[1],\idealij[2]},
\]  
where $\pb$ is the maximal element of $\idealij[1]$ with label $s_{i_2}$, and these sets exhaust all toric terms of $\p_{i_1}\p_{i_2}$.
\end{lem}

\begin{proof}
    Recall from Proposition \ref{prop:uij_repeated_action} that the weights corresponding to $\p_{i_1}$ and $p_{i_2}$ are $\dfwt[i_2]-\dfwt[i_1]$ and $\dfwt[i_3]-\dfwt[i_2]$. Thus $\idealij[1]$ has a unique maximal element $\pb$ and this element has label $s_{i_2}$, and by Lemma \ref{lem: move-symmetry} we have $A_{\idealij[1],\idealij[2]}=\{\pb\}$ which is moved to $\pb'\in\minposet\setminus\idealij[2]$. Thus, any embedding $\emb_{i_1}:\ideal_{i_1}\rightarrow\minposet$ has the unique maximal element $\emb_{i_1}(\pb)$. By Lemma \ref{lem:monomial_plucker}, the unique embedding of $\ideal_{i_2}$ in $\minposet$ is the identity. Since no element of $\ideal_{i_2}\setminus \ideal_{i_1}$ is labeled by $s_{i_2}$, the identity embedding of $\ideal_{i_1}$ in $\minposet$ is the only embedding of $\ideal_{i_1}$ with image contained in $\ideal_{i_2}$, and every other embedding $\emb_{i_1}:\ideal_{i_1}\hookrightarrow\minposet$ must have its maximal element $\emb_{i_1}(\pb)\notin\ideal_{i_2}$.
    
    For the identity embeddings $(\id_{\idealij[1]},\id_{\idealij[2]})$, since $\id_{\idealij[1]}(\pb)\in \ideal_{i_2}$, extending the identity embedding of $\idealij[2]$ to $\pb'$ mapping it to $\id_{\idealij[1]}(\pb)$ is not a valid embedding of the order ideal $\ideal_{i_2}\sqcup\pb'$, so $(\id_{\idealij[1]}, \id_{\idealij[2]})\notin \second[\pb]{\idealij[1]}{\idealij[2]}$. Next, for any other pair of embeddings $(\emb_{i_1}, \id_{\ideal_{i_2}})$ with $\emb_{i_1}\neq\id_{\ideal_{i_1}}$, we have $\emb_{i_1}(\pb)\notin \ideal_{i_2}$, so mapping $\pb'$ to $\emb_{i_1}(\pb)$ is in fact a valid extension of the identity embedding of $\idealij[2]$ to $\idealij[2]\sqcup\pb'$, and thus $(\emb_{i_1},\id_{\idealij[2]})\in \second[\pb]{\idealij[1]}{\idealij[2]}$ for any non-identity embedding $\emb_{i_1}$ of $\idealij[1]$. Thus, we have 
    \[
    \second[\pb]{\idealij[1]}{\idealij[2]} = \bigl\{(\emb, \id_{\idealij[2]})~\big|~ \emb\neq\id_{\idealij[1]}\bigr\}
    \qand 
    \Bigl(\second[s_{i_2}]{\wij[1]}{\wij[2]}\Bigr)^c = \bigl\{(\id_{\idealij[1]},\id_{\idealij[2]})\bigr\}.
    \]
    
    Recall that $A_{\idealij[1],\idealij[2]}=\{\pb\}$ and $A_{\idealij[2],\idealij[1]}=\varnothing$ since $(\idealij[1],\idealij[2])$ is the top element of $\sP_{\is}$. In the case $A_1=A_2=\varnothing$, the only nontrivial set defining $T^{\varnothing,\varnothing}_{\idealij[1],\idealij[2]}$ is
    \[
    T^{\varnothing,\varnothing}_{\idealij[1],\idealij[2]}
    = \Bigl(\second[\pb]{\idealij[1]}{\idealij[2]}\Bigr)^c 
    = \bigl\{(\id_{\idealij[1]},\id_{\idealij[2]})\bigr\}.
    \]
    Next, in the case $A_1=\varnothing$ and $A_2=\{\pb\}$, the only nontrivial set defining $T^{\varnothing,\pb}_{\idealij[1],\idealij[2]}$ is
    \[
    T^{\varnothing,\pb}_{\idealij[1],\idealij[2]} = \second[\pb]{\idealij[1]}{\idealij[2]}.
    \]
    Since these two sets exhaust all the toric terms of $p_{i_1}p_{i_2}|_{\opendunim}$, and the toric sets are disjoint, the lemma follows.
\end{proof}

While the maximal element of $\sP_{\is}$ is remarkable since $T^{\varnothing,\varnothing}_{\idealij[1],\idealij[2]}\neq\varnothing$, the other toric sets $T^{\varnothing,\varnothing}_{\ideal_1,\ideal_2}$ are empty, as we now show. 

\begin{lem} 
  \label{lem: toric-partition}
  For $(\ideal_1,\ideal_2) \neq (\idealij[1],\idealij[2])\in \sP_{\is}$, the toric set $T^{\varnothing,\varnothing}_{\ideal_1,\ideal_2}$ is empty. 
\end{lem}
Lemma \ref{lem: toric-partition} says that every pair of embeddings $(\emb_1,\emb_2)$ of the order ideals $\ideal_1,\ideal_2\subideal\minposet$ allows some nonempty subset of the admissible moves $A_{\ideal_2,\ideal_1}\cup A_{\ideal_1,\ideal_2}$ between $\ideal_1$ and $\ideal_2$. 
\begin{proof}
    To prove the lemma, it is equivalent to show that for each $(\ideal_1,\ideal_2)\in\sP_{\is}$ unequal to $(\idealij[1],\idealij[2])$ and each pair of embeddings $(\emb_1:\ideal_1\hookrightarrow\minposet,\emb_2:\ideal_2\hookrightarrow \minposet)$, at least one of the two sets $A_1=\{\pb\mid (\emb_1,\emb_2)\in \first[\pb]{\ideal_1}{\ideal_2}\}$ and $A_2=\{\pb\mid (\emb_1,\emb_2)\in\second[\pb]{\ideal_1}{\ideal_2}\}$ is nonempty. Note that all $(\ideal_1,\ideal_2)\in\sP_{\is}$ unequal to $(\idealij[1],\idealij[2])$ satisfy $A_{\ideal_2,\ideal_1}\neq\varnothing$. Moreover, since $A_1\neq\varnothing$ would satisfy the desired statement, we will establish this by considering the case where $A_1=\varnothing$ and constructing an element of $A_2$.
    
    We will first make a careful selection of an element of $A_{\ideal_2,\ideal_1}$ as follows. Since $\ideal_2\neq\idealij[2]$, there must be a maximal element of $\ideal_2$ with label unequal to $s_{i_3}$. Moreover, since $\minposet\setminus\idealij[2]$ has no element with label $s_{i_3}$ by construction and the unique maximal element $\pb_{\max}\in\idealij[2]$ has label $s_{i_3}$, we must also have $\emb_2(\pb_{\max})=\pb_{\max}$. Hence, any maximal element of $\ideal_2$ can never become smaller in $\emb_2(\ideal_2)$ than an element with label $s_{i_3}$. Since maximal elements of $\emb_2(\ideal_2)$ are images of maximal elements of $\ideal_2$, there must thus be an element $\pb_2\in\ideal_2$ with $\labl(\pb_2)\neq s_{i_3}$ that is maximal in $\ideal_2$ and whose image $\emb_2(\pb_2)\in\emb_2(\ideal_2)$ maximal as well. By Lemma \ref{lem: move-symmetry} and the discussion following it, we have $\pb_2\in A_{\ideal_2,\ideal_1}$. We fix such an element $\pb_2\in\ideal_2$ which moves to $\pb_2'\in\minposet\setminus\ideal_1$ for the remainder of this proof.
    
    Since $\pb_2\notin A_1=\varnothing$, extending $\emb_1:\ideal_1\hookrightarrow\minposet$ to $\ideal_1\sqcup\pb_2'$ by mapping $\pb_2'$ to $\emb_2(\pb_2)$ does \emph{not} yield a valid embedding. This means that there must be a maximal element $\pb_1\in \ideal_1$, whose label we write as $s_m=\labl(\pb_1)$, such that $\pb_1 < \pb_2'$ and $\emb_1(\pb_1) > \emb_2(\pb_2)$. We will now show that $\pb_1$ is an element of $A_2$.
    
    First, since $s_m$ is the label of an element of $\ideal_1\subideal\idealij[1]$, we have $s_m\neq s_{i_1}$. Furthermore, we have $\emb_1(\pb_1)\in\minposet\setminus\ideal_2$ since $\emb_1(\pb_1) > \emb_2(\pb_2) \ge \pb_2$ is larger than a maximal element of $\ideal_2$, so $s_m\neq s_{i_3}$ as well. Lemma \ref{lem: move-symmetry} then implies that $\pb_1$ is a move in $A_{\ideal_1,\ideal_2}$ which moves to $\pb_1'\in\minposet\setminus\ideal_2$. 
    
    Finally, we show that extending $\emb_2:\ideal_2\hookrightarrow\minposet$ to $\ideal_2\sqcup\pb_1'$ by mapping $\pb_1'$ to $\emb_1(\pb_1)$ is a valid embedding. If it is not, then there must be some $\pb_3\in\ideal_2$ such that $\emb_2(\pb_3) > \emb_1(\pb_1)$. However, this yields the chain of inequalities
    \[
    \emb_2(\pb_3) > \emb_1(\pb_1) > \emb_2(\pb_2),
    \]
    contradicting the maximality of $\emb_2(\pb_2)$ in $\emb_2(\ideal_2)$. Thus we have $\pb_1\in A_2$.
\end{proof}

Now that we have constructed and studied a way to classify embeddings of $(\ideal_1,\ideal_2)\in\sP_{\is}$, we will need one final definition to organize the cancellation.

\begin{df}\label{df:restricted_subposet_order}
    For a subset $\sC\subset\sP_\is$, we define the partial order that $a>b$ in $\sC$ if and only if there exists a sequence $(a=c_0, c_1,\dots, b=c_n)$ of elements of $\sC$ such that for each $i$, $c_i \gtrdot c_{i+1}$ is a cover relation in $\sP$. We call $\sC\subset\sP_{\is}$ an \emph{$m$-cube} if $\sC$ equipped with this partial order is isomorphic to the finite boolean lattice of subsets of an $m$-element set.
\end{df}

The reason for this terminology is that an $m$-cube in $\sP_\is$ can be visualized in the Hasse diagram of $\sP_\is$ as the $1$-skeleton (vertices and edges) of the $m$-cube $[0,1]^m$. 

\begin{rem}
    Note that the above definition does not use the natural induced poset structure on $\sC$ as a subset of $\sP_{\is}$. For example, if $\sP_{\is}$ were the poset $a>b>c$, the natural induced partial order on $\{a,c\}$ would be $a>c$ while the partial order described above on $\{a,c\}$ would be the two element antichain with $a$ and $c$ incomparable. 
    
    As a result, the poset $\sP_{\is}$ may generally allow many embeddings of boolean lattices which we do not call $m$-cubes by the definition above. For example, we have two order embeddings of the Boolean lattice in the following five element poset (drawn as Hasse diagrams with the largest element on the left where the filled-in circles are the image of the Boolean lattice), however the partial order on these subsets as given in Definition \ref{df:restricted_subposet_order} yields the two posets on the right which are clearly not isomorphic to the Boolean lattice:
\[
\begin{tikzpicture}[rotate=90,scale=.5,style=very thick,baseline=-.25em]
    \draw 
	(0,0)--(1,1)--(0,2)--(-1,1)--(0,0)
    ;
    \draw[black, fill=black] 
	(0,0) circle (.15)
	(1,1) circle (.15)
	(0,2) circle (.15)
	(-1,1) circle (.15)
    ;
\end{tikzpicture}
\quad\hookrightarrow\quad
\begin{tikzpicture}[rotate=90,scale=.5,style=very thick,baseline=-.25em]
    \draw 
	(0,0)--(1,1.5)--(0,3)--(-1,2)--(-1,1)--(0,0)
    ;
    \draw[black, fill=black] 
	(0,0) circle (.15)
	(1,1.5) circle (.15)
	(0,3) circle (.15)
	(-1,1) circle (.15)
    ;
    \draw[black, fill=white]
        (-1,2) circle (.15)
    ;
\end{tikzpicture}
\quad,\quad
\begin{tikzpicture}[rotate=90,scale=.5,style=very thick,baseline=-.25em]
    \draw 
	(0,0)--(1,1.5)--(0,3)--(-1,2)--(-1,1)--(0,0)
    ;
    \draw[black, fill=black] 
	(0,0) circle (.15)
	(1,1.5) circle (.15)
	(0,3) circle (.15)
	(-1,2) circle (.15)
    ;
    \draw[black, fill=white]
        (-1,1) circle (.15)
    ;
\end{tikzpicture}
\quad\rightsquigarrow\quad
\begin{tikzpicture}[rotate=90,scale=.5,style=very thick,baseline=-.25em]
    \draw 
	(-1,1)--(0,0)--(1,1)--(0,2)
    ;
    \draw[black, fill=black] 
	(0,0) circle (.15)
	(1,1) circle (.15)
	(0,2) circle (.15)
	(-1,1) circle (.15)
    ;
\end{tikzpicture}
\quad,\quad
\begin{tikzpicture}[rotate=90,scale=.5,style=very thick,baseline=-.25em]
    \draw 
	(0,0)--(1,1)--(0,2)--(-1,1)
    ;
    \draw[black, fill=black] 
	(0,0) circle (.15)
	(1,1) circle (.15)
	(0,2) circle (.15)
	(-1,1) circle (.15)
    ;
\end{tikzpicture}
\]
    Hence, the two subposets are not $2$-cubes as defined in Definition \ref{df:restricted_subposet_order}. One way to think of the partial order on $\sC$ is that it is the one whose Hasse diagram is obtained by removing vertices and edges from the Hasse diagram of $\sP_{\is}$ if the vertex is not in $\sC$ or the edge has an endpoint not in $\sC$. 
\end{rem}

These cubes will be the main tool for organizing our cancellation argument. Recall that every cover relation in $\sP_{\is}$ corresponds to an admissible move between two elements of $\sP_{\is}$, which we may also consider as a simple reflection via its label. Since cover relations of a cube $\sC$ are cover relations in $\sP_{\is}$, we also obtain the same correspondence between cover relations in $\sC$ and labels describing admissible moves between elements of $\sC$. Thus, we may associate to each cube $\sC$ the set $S_{\sC}$ of all labels corresponding to all cover relations between elements in $\sC$. We next establish some elementary facts about $S_{\sC}$, which we will refer to as the \emph{set of labels} of $\sC$.

\begin{lem}
    \label{lem:mcube-moves}
    Let $\sC\subset\sP_{\is}$ be an $m$-cube and $S_{\sC}$ its set of labels. Then $|S_{\sC}|=m$ and all elements of $S_{\sC}$ commute with each other.
\end{lem}

\begin{proof}
    Let $(\ideal_{t_1},\ideal_{t_2})$ denote the top element of $\sC$. We claim that 
    \[
    S_{\sC}=\{\labl(\pb)\mid \pb\in A_{\ideal_{t_1},\ideal_{t_2}},  (\ideal_{t_1}\setminus\pb,\ideal_{t_2}\sqcup\pb')\in \sC\}.
    \]
    The inclusion ``$\supset$'' is straightforward: if $\pb\in A_{\ideal_{t_1},\ideal_{t_2}}$ and $(\ideal_{t_1}\setminus\pb, \ideal_{t_2}\sqcup\pb')\in \sC$, then $\pb$ corresponds to a cover relation of $\sC$ and hence $\labl(\pb)\in S_{\sC}$ by definition. For the reverse inclusion, we will first show that if there is a cover relation $(\ideal_1,\ideal_2)\gtrdot(\ideal_3,\ideal_4)$ in $\sC$ with corresponding admissible move $\pb_*\in A_{\ideal_1,\ideal_2}$ which moves to $\pb_*'\in\minposet\setminus\ideal_2$, then 
    \[\{\pb\in A_{\ideal_1,\ideal_2}\mid (\ideal_1\setminus\pb,\ideal_2\sqcup\pb')\in \sC\} = \{\pb\in A_{\ideal_3,\ideal_4}\mid (\ideal_3\setminus\pb,\ideal_4\sqcup\pb')\in \sC\} \cup \{\pb_*\}.\]
    
    First, let $\pb\in A_{\ideal_1,\ideal_2}$ with $(\ideal_1\setminus\pb,\ideal_2\sqcup\pb')\in \sC$ and $\pb\neq \pb_*$ (if $\pb=\pb_*$ then $\pb_*\in \{\pb_*\}$ and the claim holds). Let $\ideal_3'=\ideal_1\setminus\pb$ and $\ideal_4'=\ideal_2\sqcup\pb'$. Since the union and intersection of order ideals are also order ideals, both $\ideal_5=\ideal_3\cap\ideal_3'$ and $\ideal_6=\ideal_4\cup\ideal_4'$ are order ideals. By this construction, $\ideal_5$ can be obtained from $\ideal_3$ by removing $\pb$ and from $\ideal_3'$ by removing $\pb_*$, while $\ideal_6$ can be obtained from $\ideal_4$ by adding $\pb'$ and from $\ideal_4'$ by adding $\pb_*'$, so in particular $(\ideal_5,\ideal_6)$ is covered by both $(\ideal_3,\ideal_4)$ and $(\ideal_3',\ideal_4')$ in $\sP_{\is}$ and is the unique such element. Then by definition of a cube, $\sC$ must contain $(\ideal_5,\ideal_6) = (\ideal_3,\ideal_4) \wedge (\ideal_3',\ideal_4')$. Thus, we find $\pb\in A_{\ideal_3,\ideal_4}$ corresponding to $(\ideal_3,\ideal_4)\gtrdot (\ideal_5=\ideal_3\setminus\pb,\ideal_6=\ideal_4\sqcup\pb')$ in $\sC$.

    Conversely, let $\pb\in A_{\ideal_3,\ideal_4}$ with $(\ideal_3\setminus\pb,\ideal_4\sqcup\pb')\in\sC$ and furthermore let $\ideal_5=\ideal_3\setminus\pb=\ideal_1\setminus\{\pb,\pb_*\}$ and $\ideal_6=\ideal_4\sqcup\pb'=\ideal_2\sqcup\{\pb',\pb_*'\}$ so $(\ideal_5,\ideal_6)\in \sC$. Since $\sC$ is a cube, and we have the chain $(\ideal_1,\ideal_2)\gtrdot(\ideal_3,\ideal_4)\gtrdot(\ideal_5,\ideal_6)$, there must be another element $(\ideal_3',\ideal_4')\in\sC$ such that $(\ideal_1,\ideal_2)\gtrdot(\ideal_3',\ideal_4')\gtrdot(\ideal_5,\ideal_6)$ where $(\ideal_3',\ideal_4')$ and $(\ideal_3,\ideal_4)$ are incomparable. This implies that $\ideal_1 \supset \ideal_3' \supset \ideal_5$ and $\ideal_2 \subset \ideal_4' \subset \ideal_6$. By construction, $\ideal_5$ is obtained from $\ideal_1$ by removing two elements $\pb,\pb_*$. Since $\ideal_3$ is obtained by removing $\pb_*$ from $\ideal_1$, the only possibility for $\ideal_3'$ is that it must be obtained by removing $\pb$ from $\ideal_1$. Similarly, since $\ideal_6$ is obtained from $\ideal_2$ by adding $\pb',\pb_*'$ and $\ideal_4$ is obtained by adding $\pb_*'$ to $\ideal_2$, the only possibility for $\ideal_4'$ is that it must be obtained by adding $\pb'$ to $\ideal_2$. Thus, we must have $(\ideal_1,\ideal_2)\gtrdot(\ideal_3'=\ideal_1\setminus\pb,\ideal_4'=\ideal_2\sqcup\pb')$ in $\sC$ with corresponding move $\pb\in A_{\ideal_1,\ideal_2}$. 

    With this result, we can now establish $S_{\sC} \subset \{\labl(\pb)\mid \pb\in A_{\ideal_{t_1},\ideal_{t_2}}\mid (\ideal_{t_1}\setminus\pb,\ideal_{t_2}\sqcup\pb')\in \sC\}$. Let $\labl(\pb)\in S_{\sC}$, so that there is a cover relation $(\ideal_1,\ideal_2)\gtrdot(\ideal_3,\ideal_4)\in \sC$ with corresponding move $\pb\in A_{\ideal_1,\ideal_2}$. Since $\sC$ is a cube, there must also be a chain $(\ideal_{t_1},\ideal_{t_2})\gtrdot \dots \gtrdot (\ideal_1,\ideal_2)$ in $\sC$. Applying the previously established equality for each of these inclusions, we immediately conclude that $\labl(\pb)\in \{\labl(\pb)\in A_{\ideal_{t_1},\ideal_{t_2}}\mid (\ideal_{t_1}\setminus\pb,\ideal_{t_2}\sqcup\pb')\in \sC\}$. Since $\sC$ is an $m$-cube, it contains exactly $m$ cover relations of the form $(\ideal_{t_1},\ideal_{t_2})\gtrdot(\ideal_1,\ideal_2)$ so that the set $\{\labl(\pb)\in A_{\ideal_{t_1},\ideal_{t_2}}\mid (\ideal_{t_1}\setminus\pb,\ideal_{t_2}\sqcup\pb')\in\sC\}$ must correspondingly have exactly $m$ elements. Finally, recall by Lemma \ref{lem:moves_commute} that the labels of elements of $A_{\ideal_{t_1},\ideal_{t_2}}$ all commute when considered as Weyl group elements.
\end{proof}

Letting $\labl(A_{\ideal_1,\ideal_2})$ denote the labels $\{\labl(\pb)\mid\pb\in A_{\ideal_1,\ideal_2}\}$ corresponding to elements of $A_{\ideal_1,\ideal_2}$, we obtain the following statement as an immediate corollary of the proof of Lemma \ref{lem:mcube-moves}. 
\begin{cor}
    \label{cor:mcube-subsets}
    Let $\sC\subset\sP_{\is}$ be an $m$-cube and $S_{\sC}$ its set of labels. Then the map which sends $(\ideal_1,\ideal_2)\in\sC\mapsto \labl(A_{\ideal_1,\ideal_2})\cap S_{\sC}$ identifies $\sC$ with the Boolean lattice of subsets of $S_{\sC}$. In particular, each element $(\ideal_1,\ideal_2)\in \sC$ admits $m$ moves with labels in $S_{\sC}$, consisting of moves with labels in $\labl(A_{\ideal_1,\ideal_2})\cap S_{\sC}$ which go down in the partial order on $\sC$ and moves with labels in $\labl(A_{\ideal_2,\ideal_1})\cap S_{\sC}$ which go up in the partial order on $\sC$.
\end{cor}
Our definition of cube captures the situation where a pair of embeddings $(\emb_1,\emb_2)$ of the order ideals $(\ideal_1,\ideal_2)\in\sP_{\is}$ is an element of $T^{A_1,A_2}_{\ideal_1,\ideal_2}$ and the set consisting of all elements of $\sP_{\is}$ obtainable from $(\ideal_1,\ideal_2)$ using any moves with labels in $S=\{\labl(\pb)\mid \pb\in A_1\cup A_2\}$ forms an $|S|$-cube. We now show that this characterization is unique.

\begin{lem}
    \label{lem: unique-mcube}
    If two $m$-cubes $\sC_1, \sC_2\subset \sP_{\is}$ have nonempty intersection and have the same set $S$ of labels, then $\sC_1=\sC_2$. 
\end{lem}
\begin{proof}
    Let $(\ideal_1,\ideal_2)\in \sC_1\cap \sC_2$. By Corollary \ref{cor:mcube-subsets}, for any subset $S'\subset S$, the result of applying the moves with corresponding labels in $S'$ to $(\ideal_1,\ideal_2)$ must be another element of $\sC_1\cap\sC_2$, which we denote by $S'\cdot (\ideal_1,\ideal_2)$. (Note that this element does not depend on the order of applying the moves due to the commutativity of the elements of $S$.) Furthermore, if $S_1'\neq S_2'$, then $S_1'\cdot(\ideal_1,\ideal_2)\neq S_2'\cdot(\ideal_1,\ideal_2)$. Thus, both $\sC_1$ and $\sC_2$ must contain the subset
    \[\{S'\cdot (\ideal_1,\ideal_2)\mid S'\subset S\}\]
    which contains $2^m$ elements. However, since any $m$-cube has exactly $2^m$ elements, we must have that $\sC_1=\sC_2$ are both equal to the set above.
\end{proof}
These are the main structural properties that we will need for $m$-cubes to organize our cancellation of toric terms. We now present our main technical result in this section about the uniformity of toric terms in an $m$-cube. To refer to pairs of embeddings yielding the same toric terms, we make the following definition:

\begin{df}
    \label{df:toric_equality}
    Consider two elements $(\ideal_1,\ideal_2),(\ideal_3,\ideal_4)\in\sP_\is$ and corresponding embeddings $\emb_i:\ideal_i\hookrightarrow\minposet$. We say that $(\emb_1,\emb_2)$ and $(\emb_3,\emb_4)$ are \emph{torically equal}, written $(\emb_1,\emb_2)\overset{t}{=}(\emb_3,\emb_4)$, if
    \[
    a_{\emb_1}a_{\emb_2} = a_{\emb_3}a_{\emb_4},
    \]
    where we write $a_{\emb_i} = a_{\emb_i(\ideal_i)}$, which is in turn defined in equation \eqref{eq:toric_term_from_poset}.
\end{df} 
Note that the case that $a_{\emb_1}\neq a_{\emb_3}$ or $a_{\emb_2}\neq a_{\emb_4}$ is not excluded, as long as the products agree. To simplify notation in the remainder of this section, we will use $A_{\ideal_1,\ideal_2}(S)$ to denote the set $\{\pb\in A_{\ideal_1,\ideal_2}\mid \labl(\pb)\in S\}$ of moves in $A_{\ideal_1,\ideal_2}$ whose labels are contained in the set of labels $S$. 

\begin{lem}
  \label{lem: toric-bijections}
    Let $\sC\subset \sP_\is$ be an $m$-cube with the set $S=S_{\sC}$ of labels. For any distinct pairs $(\ideal_1,\ideal_2),(\ideal_3,\ideal_4)\in \sC$, let $A_{12}=A_{\ideal_1,\ideal_2}(S)$, $A_{21}=A_{\ideal_2,\ideal_1}(S)$ and similarly $A_{34}=A_{\ideal_3,\ideal_4}(S)$, $A_{43}=A_{\ideal_4,\ideal_3}(S)$. There is a bijection
    \[
    \kappa_{A_{21},A_{12}}^{A_{43},A_{34}}:T^{A_{21},A_{12}}_{\ideal_1,\ideal_2} \overset{\sim}{\longrightarrow} T^{A_{43},A_{34}}_{\ideal_3,\ideal_4}
    \]
    such that for any $(\emb_1,\emb_2)\in T^{A_{21},A_{12}}_{\ideal_1,\ideal_2}$, we have $(\emb_1,\emb_2) \overset{t}{=} \kappa_{A_{21},A_{12}}^{A_{43},A_{34}}(\emb_1,\emb_2)$.
\end{lem}
\begin{proof}
    Since $\sC$ is connected by moves corresponding to labels in $S$, we may reduce to the case when $(\ideal_3,\ideal_4)$ can be obtained from $(\ideal_1,\ideal_2)$ by a single move with label in $S$. The full statement of the lemma then follows by a straightforward induction. Thus, we consider the two cases $(\ideal_3,\ideal_4) \gtrdot (\ideal_1,\ideal_2)$ and $(\ideal_3,\ideal_4) \lessdot (\ideal_1,\ideal_2)$. To simplify notation, we will write $T_{12}=T^{A_{21},A_{12}}_{\ideal_1,\ideal_2}$ and $T_{34}=T^{A_{43},A_{34}}_{\ideal_3,\ideal_4}$, while writing $\kappa_{12}^{34}=\kappa_{A_{21},A_{12}}^{A_{43},A_{34}}$ for the remainder of this proof.

    In the first case, $(\ideal_3,\ideal_4) \gtrdot (\ideal_1,\ideal_2)$ corresponding to a move $\pb_3\in A_{34}$ which moves to $\pb_3'$, we have $(\ideal_1,\ideal_2) = (\ideal_3\setminus\pb_3,\ideal_2=\ideal_4\sqcup\pb_3')$. In this case, we define $\kappa_{12}^{34}$ as follows. Given a pair of embeddings $(\emb_1,\emb_2)\in T_{12}$ of $(\ideal_1,\ideal_2)$, we define the embeddings $(\emb_3,\emb_4)$ of $(\ideal_3,\ideal_4)$ by setting
    \[
    \emb_3|_{\ideal_1} = \emb_1 \qand \emb_3(\pb_3)=\emb_2(\pb_3'), \qwhile \emb_4 = \emb_2|_{\ideal_4}.
    \tag{$\dag$}
    \label{aux_eq:def_kappa_case_1}
    \]
    We then set $\kappa_{12}^{34}(\emb_1,\emb_2) = (\emb_3, \emb_4)$. The fact that this defines a map $\kappa_{12}^{34}:T_{12}\rightarrow T_{34}$ follows immediately from the definitions of $A_{\ideal_2,\ideal_1}$ and $\first[\pb_3]{\ideal_1}{\ideal_2}$ along with Corollary \ref{cor:mcube-subsets}.

    In the second case, $(\ideal_3,\ideal_4) \lessdot (\ideal_1,\ideal_2)$ corresponding to a move $\pb_1\in A_{12}$ which moves to $\pb_1'$, we have $(\ideal_3,\ideal_4)=(\ideal_1\setminus\pb_1,\ideal_4=\ideal_2\sqcup\pb_1')$. In this case, we define $\kappa_{12}^{34}$ as follows. Given a pair of embeddings $(\emb_1,\emb_2)\in T_{12}$ of $(\ideal_1,\ideal_2)$, we define the embeddings $(\emb_3,\emb_4)$ of $(\ideal_3,\ideal_4)$ by setting
    \[
    \emb_3 = \emb_1|_{\ideal_3}, \qwhile \emb_4|_{\ideal_2}=\emb_2 \qand \emb_4(\pb_1')=\emb_1(\pb_1).
    \tag{$\ddag$}
    \label{aux_eq:def_kappa_case_2}
    \]
    We then set $\kappa_{12}^{34}(\emb_1,\emb_2) = (\emb_3, \emb_4)$. Similarly to before, the fact that $\kappa_{12}^{34}$ is well-defined follows immediately from the definitions of $A_{\ideal_1,\ideal_2}$ and $\second[\pb_1]{\ideal_1}{\ideal_2}$ along with Corollary \ref{cor:mcube-subsets}.

    We now show in both cases that $\kappa_{34}^{12}\circ\kappa_{12}^{34}=\mathrm{id}$. Fix any pair $(\emb_1,\emb_2)\in T_{12}$ and write $(\emb_3,\emb_4)=\kappa_{12}^{34}(\emb_1,\emb_2)$ and $(\emb_1',\emb_2')=\kappa_{34}^{12}\circ\kappa_{12}^{34}(\emb_1,\emb_2)$. First, suppose that $(\ideal_3,\ideal_4) \gtrdot (\ideal_1,\ideal_2)$, then  
    \[
    \emb_1' \overset{\text{\eqref{aux_eq:def_kappa_case_2}}}{=\!=} \emb_3|_{\ideal_1}\overset{\text{\eqref{aux_eq:def_kappa_case_1}}}{=\!=}\emb_1, 
    \qwhile 
    \emb_2'|_{\ideal_4} \overset{\text{\eqref{aux_eq:def_kappa_case_2}}}{=\!=} \emb_4 \overset{\text{\eqref{aux_eq:def_kappa_case_1}}}{=\!=} \emb_2|_{\ideal_4}
    \qand
    \emb_2'(\pb_3') \overset{\text{\eqref{aux_eq:def_kappa_case_2}}}{=\!=} \emb_3(\pb_3) \overset{\text{\eqref{aux_eq:def_kappa_case_1}}}{=\!=} \emb_2(\pb_3')
    \]
    where $\pb_3'\in\ideal_2$ and $\pb_3\in\ideal_3$ are as above. In the second case, $(\ideal_3,\ideal_4) \lessdot (\ideal_1,\ideal_2)$, we instead find
    \[
    \emb_1'|_{\ideal_3} \overset{\text{\eqref{aux_eq:def_kappa_case_1}}}{=\!=} \emb_3 \overset{\text{\eqref{aux_eq:def_kappa_case_2}}}{=\!=} \emb_1|_{\ideal_3}
    \qand
    \emb_1'(\pb_1) \overset{\text{\eqref{aux_eq:def_kappa_case_1}}}{=\!=} \emb_4(\pb_1') \overset{\text{\eqref{aux_eq:def_kappa_case_2}}}{=\!=} \emb_1(\pb_1),
    \qwhile
    \emb_2' \overset{\text{\eqref{aux_eq:def_kappa_case_1}}}{=\!=} \emb_4|_{\ideal_4} \overset{\text{\eqref{aux_eq:def_kappa_case_2}}}{=\!=} \emb_2,
    \]
    with $\pb_1\in\ideal_1$ and $\pb_1'\in\ideal_4$ as above.    
    
    We now prove the toric equality in both cases. As before, let $(\emb_1,\emb_2)\in T_{12}$, write $(\emb_3,\emb_4)=\kappa_{12}^{34}$, and consider the two cases. First, suppose $(\ideal_3,\ideal_4) \gtrdot (\ideal_1,\ideal_2)$, so $\ideal_3=\ideal_1\sqcup\pb_3$ and $\ideal_2=\ideal_4\sqcup\pb_3'$ as above, then
    \[
    a_{\emb_1}a_{\emb_2} = a_{\emb_1(\ideal_1)}a_{\emb_2(\pb_3')}a_{\emb_2(\ideal_4)}
    \overset{\text{\eqref{aux_eq:def_kappa_case_1}}}{=\!=}     a_{\emb_3(\ideal_1)}a_{\emb_3(\pb_3)}a_{\emb_4(\ideal_4)}
    = a_{\emb_3}a_{\emb_4}.
    \]
    Second, suppose $(\ideal_3,\ideal_4) \lessdot (\ideal_1,\ideal_2)$, so $\ideal_1=\ideal_3\sqcup\pb_1$ and $\ideal_4=\ideal_2\sqcup\pb_1'$ as above, then
    \[
    a_{\emb_1}a_{\emb_2} = a_{\emb_1(\ideal_3)}a_{\emb_1(\pb_1)}a_{\emb_2(\ideal_2)} \overset{\text{\eqref{aux_eq:def_kappa_case_2}}}{=\!=} a_{\emb_3(\ideal_3)}a_{\emb_4(\pb_1')}a_{\emb_4(\ideal_2)}=a_{\emb_3}a_{\emb_4}
    \]
    completing the proof of the statement.
\end{proof}

We are finally ready to prove the main theorem of this section. For this, we define $L_d\subset \sP_{\is}$ to be the set of pairs $(\ideal_1,\ideal_2)\in \sP_{\is}$ which can be obtained from the top element $(\idealij[1],\idealij[2])\in\sP_{\is}$ by a minimal length sequence of $d\ge 0$ moves 
\[
\pb_1\in A_{\idealij[1],\idealij[2]}, \quad \dots, \quad
\pb_d\in A_{\idealij[1]\setminus\{\pb_1,\dots, \pb_{d-1}\}, \idealij[2]\sqcup\{\pb_1',\dots, \pb_{d-1}'\}}
\] 
from $\idealij[1]$ to $\idealij[2]$. (Here, minimality is considered among all sequences of moves between $(\idealij[1],\idealij[2])$ and $(\ideal_1,\ideal_2)$ in $\sP_{\is}$.) If there are no such pairs in $\sP_{\is}$, then we define $L_d=\varnothing$. In particular, since $\sP_{\is}$ is finite, then $L_d=\varnothing$ for all $d$ sufficiently large.

\begin{thm} 
\label{thm: phi_cancellation}
    The generalized minor $\phiGLS[\is]$ is equal on $\dunimP$ to the polynomial $\cD_{\is}$ obtained as the alternating sum over the sets $L_d\subset\sP_{\is}$
    \[\cD_{\is} = \sum_{d\ge 0} (-1)^d\sum_{(\ideal_1,\ideal_2)\in L_d} p_{\ideal_1}p_{\ideal_2}.\]
\end{thm}

\begin{proof}
    We will prove this statement by computing the restrictions of both sides to the open, dense torus $\opendunim \subset \dunimP$ and verifying that they are equal. The desired result then follows from the irreducibility of $\dunimP$. We computed the restriction $\phiGLS[\is]|_{\opendunim}$ in Theorem \ref{thm:toric_monomial_phi}, so we now focus on the claimed Pl\"ucker coordinate expression. We begin by rewriting the restriction of the Pl\"ucker coordinate expression using the toric sets $T^{A_1,A_2}_{\ideal_1,\ideal_2}$
    \[\cD_{\is}|_{\opendunim} = \sum_{d\ge 0} (-1)^d\sum_{(\ideal_1,\ideal_2)\in L_d} p_{\ideal_1}p_{\ideal_2}|_{\opendunim} = \sum_{d\ge 0}(-1)^d\sum_{(\ideal_1,\ideal_2)\in L_d}\sum_{A_1, A_2}\sum_{(\emb_1,\emb_2)\in T^{A_1,A_2}_{\ideal_1,\ideal_2}} a_{\emb_1} a_{\emb_2}\]
    where the third sum on the rightmost side is over all subsets $A_1\subset A_{\ideal_2,\ideal_1}$ and $A_2\subset A_{\ideal_1,\ideal_2}$. For each $(\ideal_1,\ideal_2)\in \sP_{\is}$ and each pair $A_1\subset A_{\ideal_2,\ideal_1}$ and $A_2\subset A_{\ideal_1,\ideal_2}$, there is a unique $(|A_1|+|A_2|)$-cube with the set $\{\labl(\pb)\mid\pb\in A_1\cup A_2\}$ of labels containing $(\ideal_1,\ideal_2)$ by Lemma \ref{lem: unique-mcube}.
    
    For any cube $\sC$, let $d_{\sC}$ denote the minimum number of moves in $\sP_{\is}$ between the top element $(\idealij[1],\idealij[2])$ of $\sP_{\is}$ and the top element of $\sC$ and let $S_{\sC}$ denote the set of labels of $\sC$. Then we may rewrite the sum above by instead summing over all cubes contained in $\sP_{\is}$. Before rewriting it like this, we first note that by Lemma \ref{lem: toric-bijections}, we have for all cubes $\sC$ and all pairs $(\ideal_1,\ideal_2), (\ideal_3,\ideal_4)\in \sC$ that 
    \[\sum_{(\emb_1,\emb_2)\in T^{A_1,A_2}_{\ideal_1,\ideal_2}} a_{\emb_1}a_{\emb_2} = \sum_{(\emb_3,\emb_4)\in T^{A_3,A_4}_{\ideal_3,\ideal_4}} a_{\emb_3}a_{\emb_4},\]
    where $A_1=A_{\ideal_2,\ideal_1}(S_{\sC}), A_2=A_{\ideal_1,\ideal_2}(S_{\sC})$ and $A_3=A_{\ideal_4,\ideal_3}(S_{\sC}), A_4 = A_{\ideal_3,\ideal_4}(S_{\sC})$. Thus the sum of monomials corresponding to pairs of embeddings in the toric set $T^{A_{\ideal_2,\ideal_1}(S_{\sC}),A_{\ideal_1,\ideal_2}(S_{\sC})}_{\ideal_1,\ideal_2}$ is independent of the particular element $(\ideal_1,\ideal_2)\in\sC$, so for each cube $\sC$ we will simply use $T_{\sC}^{S_{\sC}}$ to denote any one of the torically equivalent toric sets $T^{A_{\ideal_2,\ideal_1}(S_{\sC}),A_{\ideal_1,\ideal_2}(S_{\sC})}_{\ideal_1,\ideal_2}$. We now rewrite the sum as
    \[\cD_{\is}|_{\opendunim}=\sum_{\sC\subset \sP_{\is}} \sum_{(\ideal_1,\ideal_2)\in \sC}(-1)^{d_{\sC}+|A_{\ideal_2,\ideal_1}(S_{\sC})|}\sum_{(\emb_1,\emb_2)\in T_{\sC}^{S_{\sC}}} a_{\emb_1} a_{\emb_2}.\]
    
    Furthermore, by Corollary \ref{cor:mcube-subsets}, $|A_{\ideal_2,\ideal_1}(S_{\sC})|$ is equal to the number of cover relations of the form $(\ideal_3,\ideal_4) \gtrdot (\ideal_1,\ideal_2)$ in $\sC$ (i.e.~where $(\ideal_1,\ideal_2)$ is the lesser element) and for a fixed $j$, there are $\binom{|S_{\sC}|}{j}$ elements $(\ideal_1,\ideal_2)\in\sC$ with $|A_{\ideal_2,\ideal_1}(S_{\sC})|=j$. Thus we may further rewrite the previous sum over all cubes in $\sP_{\is}$ as
    \begin{align*}
        \cD_{\is}|_{\opendunim} &= \sum_{\sC\subset \sP_{\is}} (-1)^{d_{\sC}} \sum_{j=0}^{|S_{\sC}|} (-1)^j\binom{|S_{\sC}|}{j}\sum_{(\emb_1,\emb_2)\in T_{\sC}^{S_{\sC}}} a_{\emb_1} a_{\emb_2}\\
        &= \sum_{\sC\subset \sP_{\is}} (-1)^{d_{\sC}} \sum_{(\emb_1,\emb_2)\in T_{\sC}^{S_{\sC}}} a_{\emb_1} a_{\emb_2}\sum_{j=0}^{|S_{\sC}|} (-1)^j\binom{|S_{\sC}|}{j}.
    \end{align*}
    The alternating sum of coefficients is equal to
    \[
    \sum_{j=0}^{S_{\sC}} (-1)^j \binom{|S_{\sC}|}{j} = (1+(-1))^{|S_{\sC}|} = \begin{cases} 0 & |S_{\sC}|>0\\ 1 & |S_{\sC}|=0\end{cases}.
    \]
    Thus, the only contributions to the overall sum come from $0$-cubes with $S_{\sC}=\varnothing$. By Lemmas \ref{lem: initial-plucker} and \ref{lem: toric-partition}, the only nonempty $0$-cube is $\sC_0 = \{(\idealij[1],\idealij[2])\}$ with $T^{\varnothing,\varnothing}_{\idealij[1],\idealij[2]} = \{(\id_{\ideal_{i_1}},\id_{\ideal_{i_2}})\}$, so the right-hand side above reduces to
    \[
    \cD_{\is}|_{\opendunim}=a_{\ideal_{i_1}}a_{\ideal_{i_2}} = \phiGLS[\is]|_{\opendunim},
    \]
    as claimed.
\end{proof}

%% file: 5numerators.tex
\section{Using derivations to obtain toric expressions for the numerators   }\label{sec:numerators}

In this section, we will define for each $\is\in[n]$ a derivation $\deis$ on the coordinate ring $\C[\cmX]$, where we recall from Section \ref{sec:toric-monomials} that $i_1=\DSi[k](\is)$ for $\is\in[n]\setminus\{k\}$ and $i_1=\DSi[k](k)=k$ for $\is=k$. As we have mentioned in Section \ref{sec:intro}, this derivation will play a central role in defining the canonical potentials, and hence we will need to investigate its action specifically on monomials appearing in $\cD_\is$. To simplify this investigation, we will translate the derivation $\deis$ on $\C[\cmX]$ into a map on $\sP_\is$ and investigate the induced poset structure of its image. More precisely, our aim will be showing that the torus expansion of $\deis(\cD_\is)$ restricted to $\opendunim$ is closely related to that of $\cD_\is$. For each $\is\in [n]$, we define the derivation $\deis: \C[\cmX]\rightarrow \C[\cmX]$ as follows. First, for a Pl\"ucker coordinate $\p_\ideal\in\C[\cmX]$ labeled by $\ideal\subideal\minposet$, the derivation $\deis$ acts either by adding the minimal element $\pb$ of $\minposet\setminus\ideal$ with label $s_{i_1}$ if $\ideal\sqcup\pb$ is a valid order ideal or trivially otherwise
\[
\deis(p_{\ideal}) = \begin{cases} 
\p_{\ideal\sqcup\{\pb\}}, & \text{if $\ideal\sqcup\{\pb\}\subideal\minposet$,}\\ 
0, & \textrm{otherwise.}\end{cases}
\]
We extend $\deis$ to polynomials in the Pl\"ucker coordinates as a derivation. With this definition, it is immediately clear that: 
\begin{cor}
    $\Yboxdim{3pt}\deis[k](\cD_0) = \deis[k](p_\varnothing) = p_{\yng(1)}$.
\end{cor}
For the remainder of this section, we will work with $\is\neq0,k$ and study the action of $\deis$ on $\cD_{\is}$. Recall from \eqref{eq:df_wt_corresponding_to_order_ideal} that we defined $\wt_i$ to be the weight $\Pi(\ideal_i)(-\dfwt[k])$ corresponding to the order ideal $\ideal_i$. First, we will make the following easy observation regarding the action of $\deis$:

\begin{lem}
    \label{lem: derivation_plucker_coordinate}
    Let $p_\ideal\in\C[\cmX]$ be a Pl\"ucker coordinate on $\cmX$, then $\deis$ acts nontrivially on $p_\ideal$ if and only if $\llan\wt,\sr_{i_1}\rran=-1$.
\end{lem}
\begin{proof}
    Since our Pl\"ucker coordinates come from a minuscule representation, we must have $\llan\wt,\sr_{i_1}\rran\in \{-1,0,+1\}$ and in particular we have $\llan\wt,\sr_{i_1}\rran=-1$ if and only if $s_{i_1}$ acts by adding $\sdr_{i_1}$ to $\wt$, which in turn is equivalent to $\ell(s_{i_1}\labelprod{\ideal})=\ell(\labelprod{\ideal})+1$, where $\labelprod{\ideal}\in\cosets$ corresponds to $\ideal\subideal\minposet$ as in Remark \ref{rem:order_ideals_to_min_coset_reps}. This last statement is equivalent to the existence of a $\pb\in\minposet\setminus\ideal$ with label $s_{i_1}$ such that $\ideal\sqcup\pb$ is a valid order ideal.
\end{proof}

Thus, we will frequently be interested in the coefficient $\llan\wt,\sr_{i_1}\rran$ when investigating the action of $\delta_{i_1}$ on the polynomial $\cD_{\is}$. This action is straightforward when $\cis=1$:
\begin{lem}
    \label{lem: numerator_cis=1}
    If $\cis=1$ and $\is\neq k$, then the restriction to $\opendunim$ of $\delta_{i_1}(\cD_{\is})$ is given by
    \[
    \delta_{i_1}(\cD_{\is})|_{\opendunim}=\left(\tsum[_{\pb:\,\indx(\pb) = i_1}] a_{\pb}\right)a_{\ideal_{i_1}},
    \]
    where the sum is over all elements $\pb\in\minposet$ with label $s_{i_1}$. 
\end{lem}
\begin{proof}
Since $\cis=1$, by Corollary \ref{cor: cis=1}, we have $\cD_{\is}=p_{\idealij[1]}$, where we recall from Proposition \ref{prop:uij_repeated_action} that $p_{\idealij[1]}$ corresponds to the weight $\dfwt[i_2]-\dfwt[i_1]$. Thus, by Lemma \ref{lem: derivation_plucker_coordinate}, there is $\pb\in\minposet\setminus\idealij[1]$ with label $s_{i_1}$ such that $\delta_{i_1}(\cD_{\is}) = \delta_{i_1}(p_{\idealij[1]}) = p_{\idealij[1]\sqcup\pb}$. Since $\idealij[1]$ has no element with label $s_{i_1}$, $\pb$ must be the minimal element of $\minposet$ with label $s_{i_1}$.

Furthermore, any embedding $\emb:\ideal_{i_1}\sqcup\pb\rightarrow\minposet$ must restrict to the identity on $\ideal_{i_1}\subideal \ideal_{i_1}\sqcup\pb$ by Lemma \ref{lem:monomial_plucker}, while $\emb(\pb)$ may be any element of $\minposet$ with label $s_{i_1}$ since $\pb$ is the minimal such element of $\minposet$. Thus in the case $\cis=1$, we immediately obtain the restriction
\[
\delta_{i_1}(\cD_{\is})|_{\opendunim} = \sum_{\emb:\,\ideal_{i_1}\sqcup\pb\hookrightarrow\minposet} a_\emb = \sum_{\emb:\,\ideal_{i_1}\sqcup\pb\hookrightarrow\minposet} a_{\emb(\pb)} a_{\ideal_{i_1}} = \left(\sum_{\pb}a_\pb\right)a_{\ideal_{i_1}}\]
where the final sum is over all $\pb\in\minposet$ with label $s_{i_1}$.
\end{proof}

Studying the action of $\deis$ on the polynomial $\cD_\is$ for $\cis>1$ will be more complicated, but the first monomial can be considered directly. Recall that $p_{i_j}$ refers to the Pl\"ucker coordinate $p_{\idealij}$.
\begin{lem}
\label{lem:leading-term-derivation}
    For any $\is\in[n]$, the action of $\delta_{i_1}$ on $p_{i_j}$ is nontrivial if and only if $j=1$. In particular, 
    \[
    \delta_{i_1}(p_{i_1}\dots p_{i_\cis}) = p_{\idealij[1]\sqcup\pb}p_{i_2}\dots p_{i_{\cis}},
    \]
    where $\pb$ is the minimal element of $\minposet$ with label $s_{i_1}$.
\end{lem}
\begin{proof}
    By Proposition \ref{prop:uij_repeated_action}, we have that $p_{i_j}$ corresponds to the weight $\dfwt[i_{j+1}] - \dfwt[i_j]$. Since these are weights in a minuscule representation, then by definition of the derivation action, we have that $\delta_{i_1}$ acts nontrivially on $p_{i_j}$ if and only if $i_j=i_1$, in which case there is $\pb\in\minposet\setminus\idealij[1]$ with label $s_{i_1}$ such that $\idealij[w]\sqcup\pb$ is an order ideal. Furthermore, $\pb$ is the minimal element of $\minposet$ with label $s_{i_1}$ since $\idealij[1]$ has no element with label $s_{i_1}$. By the proof of Lemma \ref{lem:special_weyl_elt_action}, we have that $i_1\neq i_2,\dots, i_{\cis}$, so the derivation action acts nontrivially only on $p_{i_1}$ by Lemma \ref{lem: derivation_plucker_coordinate}. 
\end{proof}
Although our construction applies to the general case, it is again more straightforward to apply our construction directly to the four exceptional cases with $\cis>2$. Hence, as in the previous section, we will restrict ourselves to the case $\cis=2$ from now on and defer four exceptional cases to Appendix \ref{app:exceptional}.

With this restriction, we will now establish the more general statement that $\delta_{i_1}$ acts nontrivially on at most one Pl\"ucker coordinate of each monomial $p_{\ideal_1}p_{\ideal_2}$ corresponding to an element $(\ideal_1,\ideal_2)\in \sP_{\is}$. 

\begin{lem}
    \label{lem:derivation-one-term}
    For any $\is\in[n]$ and any $(\ideal_1,\ideal_2)\in \sP_{\is}$, the derivation $\deis$ acts nontrivially on at most one factor of the quadratic monomial $p_{\ideal_1}p_{\ideal_2}$. 
\end{lem}

\begin{proof} 
    Because the Pl\"ucker coordinates $p_\ideal$ that we consider come from a minuscule representation, the corresponding weights $\wt$ satisfy that $\llan\wt,\sr_m\rran\in \{-1,0,+1\}$ for any $m$. By Lemma \ref{lem: derivation_plucker_coordinate}, $\deis$ acts on $p_\ideal$ nontrivially if and only if $\llan\wt,\sr_{i_1}\rran=-1$ so it suffices to show that at most one of the Pl\"ucker coordinates $p_{\ideal_1}, p_{\ideal_2}$ satisfies this property for any $(\ideal_1,\ideal_2)\in\sP_{\is}$. 

    By Lemma \ref{lem:leading-term-derivation}, we have
    \[
    \llan\wt_{i_1},\sr_{i_1}\rran=\llan\dfwt[i_2]-\dfwt[i_1],\sr_{i_1}\rran=-1
    \qand 
    \llan\wt_{i_2},\sr_{i_1}\rran=\llan\dfwt[i_3]-\dfwt[i_2],\sr_{i_1}\rran\in\{0,+1\}.
    \]
    Therefore, we have $\llan\wt_{i_1} + \wt_{i_2}, \sr_{i_1}\rran =\llan \dfwt[i_3]-\dfwt[i_1],\sr_{i_1}\rran\in \{-1,0\}$ when $i_1\neq i_3$ respectively $i_1=i_3$. Now by Lemma \ref{lem: move_equality}, letting $\wt_i$ denote the weight corresponding to $\ideal_i$, we have
    \[
    \llan\wt_1 + \wt_2,\sr_{i_1}\rran = \llan \dfwt[i_3]-\dfwt[i_1],\sr_{i_1}\rran \in \{-1,0\}.
    \]
    Since this is a sum of the two integers $\llan\wt_1,\sr_{i_1}\rran, \llan \wt_2,\sr_{i_1}\rran\in \{-1,0,1\}$, we cannot have that both of these integers are equal to $-1$, so by Lemma \ref{lem: derivation_plucker_coordinate}, $\deis$ can act nontrivially on at most one of the Pl\"ucker coordinates $p_{\ideal_1},p_{\ideal_2}$ for any $(\ideal_1,\ideal_2)\in\sP_{\is}$. 
\end{proof}

In this proof we used arguments about the coefficient $\llan\wt,\sr_{i_1}\rran$ where $\wt=\Pi(\ideal)\cdot(-\dfwt[k])$. These coefficients will be a useful tool for our upcoming arguments as well, and for notational convenience we will sometimes use the tuple $(c_{\ideal_1},c_{\ideal_2})$ to refer to these coefficients; i.e.
\begin{equation}
    c_{\ideal_1} = \llan \Pi(\ideal_1)\cdot(-\dfwt[k]),\sr_{i_1}\rran \qand c_{\ideal_2} = \llan \Pi(\ideal_2)\cdot(-\dfwt[k]),\sr_{i_1}\rran.
\label{eq:df_pair_of_coefs}
\end{equation}

Using Lemma \ref{lem:derivation-one-term}, we will now translate the action of $\deis$ to elements $(\ideal_1,\ideal_2)\in\sP_{\is}$. 
\begin{df}
The map $\deis:\sP_\is\to(\cosets)^2\cup\{0\}$ is given by
\[
\deis(\ideal_1,\ideal_2) = \begin{cases}
    (\ideal_1\sqcup\pb_1,\ideal_2) & \text{if $\deis(\p_{\ideal_1}\p_{\ideal_2})=\p_{\ideal_1\sqcup\pb_1}\p_{\ideal_2}$,} \\
    (\ideal_1,\ideal_2\sqcup\pb_2) & \text{if $\deis(\p_{\ideal_1}\p_{\ideal_2})=\p_{\ideal_1}\p_{\ideal_2\sqcup\pb_2}$,} \\
    0             & \text{else,}
\end{cases}
\]
where $\pb_1$ and $\pb_2$ are elements of $\minposet\setminus\ideal_1$ and $\minposet\setminus\ideal_2$ with label $s_{i_1}$ such that $\ideal_1\sqcup\pb_1$ and $\ideal_2\sqcup\pb_2$ are order ideals respectively. When $\deis(\ideal_1,\ideal_2)\neq0$, we will often refer to its image as $(\ideal_1^+,\ideal_2^+) = \deis(\ideal_1,\ideal_2)\neq0$.   
\end{df}
In the remainder of this section, we will mainly be interested in the image of this map $\deis$.

\begin{df}
    Let $\sP_\is^+$ denote the poset obtained from $\sP_\is$ by applying $\deis$ as follows: The elements of $\sP_{\is}^+$ will be the set $\{\deis(\ideal_1,\ideal_2)~|~\deis(\ideal_1,\ideal_2)\neq 0\}$ and the partial order $>_+$ will be the one inherited from $\sP_{\is}$:
    \[
    \deis(\ideal_1,\ideal_2) >_+ \deis(\ideal_3,\ideal_4) \quad\Leftrightarrow\quad (\ideal_1,\ideal_2)>(\ideal_3,\ideal_4).
    \]
    We also identify $\delta_{i_1}$ with the (partial) function $\delta_{i_1}:\sP_{\is}\dashrightarrow \sP_{\is}^+$ obtained by ignoring those elements of $\sP_{\is}$ on which $\delta_{i_1}$ acts trivially. 
\end{df}

Lemma \ref{lem:derivation-one-term} says that for any pair $(\ideal_1,\ideal_2)\in\sP_\is$, $\delta_{i_1}$ acts nontrivially on at most one element of the pair $(\ideal_1,\ideal_2)$. As a result, the poset $\sP_{\is}^+$ may contain strictly fewer elements than $\sP_{\is}$ and this case will require special care. We first identify exactly when this case will occur.

\begin{lem}
    \label{lem:unique_non_commuting}
    The order ideal $\idealij[1]$ contains exactly one element whose label $s_j$ does not commute with $s_{i_1}$.
\end{lem}

\begin{proof}
    Recall that $\labelprod{\idealij}=\wij$. We will prove this statement by computing the coefficient $\llan\wij[1](-\dfwt[k]),\sr_{i_1}\rran$ in two ways, where we recall that $\wij[1]$ has no occurrence of the simple reflection $s_{i_1}$. First, Lemmas \ref{lem:leading-term-derivation} and \ref{lem: derivation_plucker_coordinate} immediately imply that $\llan\wij[1](-\dfwt[k]),\sr_{i_1}\rran=-1$.

    Second, we consider the action of $\wij[1]$ on $-\dfwt[k]$. Since $\wij[1]$ is a minimal coset representative, any simple reflection $s_{m_j}$ in any reduced expression for $\wij[1]$ acts by adding $\sdr_{m_j}$, and we find
    \[
    \llan\wij[1](-\dfwt[k]),\sr_{i_1}\rran 
    = \llan\dfwt[k]+\tsum[_{j=1}^{\ell(\wij[1])}] \sdr_{m_j},\sr_{i_1}\rran 
    = \sum_{j=1}^{\ell(\wij[1])} \llan\sdr_{m_j},\sr_{i_1}\rran.
    \]
    Equating these results yields
    \[-1 = \sum_{j=1}^{\ell(\wij[1])} \llan \sdr_{m_j}, \sr_{i_1}\rran.\]
    Since $\wij[1]$ has no occurrences of $s_{i_1}$, then the right hand side above is a sum of nonpositive integers. Thus, there must be exactly one $j$ such that $\llan \sdr_{m_j}, \sr_{i_1}\rran = -1$. The lemma follows immediately since each simple reflection in $\wij[1]$ corresponds to a unique element of $\idealij[1]$.
\end{proof}

We can now describe when $\deis$ acts trivially on $(\ideal_1,\ideal_2)\in\sP_{\is}$. 

\begin{lem}
    \label{lem: i1=i3}
    For $(\ideal_1,\ideal_2)\in\sP_{\is}$, $\deis$ acts trivially on $(\ideal_1,\ideal_2)\in\sP_{\is}$ if and only if $i_1=i_3$ and there is a sequence of moves from the top element $(\idealij[1],\idealij[2])\in\sP_{\is}$ to $(\ideal_1,\ideal_2)$
    \[
    \pb_1\in A_{\idealij[1],\idealij[2]},~
    \pb_2\in A_{\idealij[1]\setminus\pb_1,\idealij[2]\sqcup\pb_1'},~
    \ldots,~
    \pb_d\in A_{\idealij[1]\setminus\{\pb_1,\dots, \pb_{d-1}\},\ideal_2\sqcup\{\pb_1',\dots,\pb_{d-1}'\}}\]
    such that $\ideal_1=\idealij[1]\setminus\{\pb_1,\dots, \pb_d\}$ and $\ideal_2=\idealij[2]\sqcup\{\pb_1',\dots, \pb_d'\}$, and the label of one of the $\pb_j$ does not commute with $s_{i_1}$. 
\end{lem}
\begin{proof}
    Suppose first that $\delta_{i_1}(p_{\ideal_1}p_{\ideal_2})=0$. Then by Lemma \ref{lem: derivation_plucker_coordinate}, it must be the case that the coefficients $(c_{\ideal_1},c_{\ideal_2})$ as defined in equation \eqref{eq:df_pair_of_coefs} are $0$ or $+1$. On the other hand, by Lemma \ref{lem:leading-term-derivation} the coefficients 
    \[
    c_{i_1}=\llan\wt_{i_1},\sr_{i_1}\rran = \llan \dfwt[i_2]-\dfwt[i_1],\sr_{i_1}\rran 
    \qand
    c_{i_2}=\llan\wt_{i_2},\sr_{i_1}\rran = \llan \dfwt[i_3]-\dfwt[i_2],\sr_{i_1}\rran
    \]
    satisfy $(c_{i_1},c_{i_2})\in\{(-1,0),~(-1,+1)\}$. Thus, since the coefficients $(c_{\ideal_1},c_{\ideal_2})$ differ from $(c_{i_1},c_{i_2})$, any sequence of moves $\pb_1,\dots, \pb_d$ between $(\idealij[1],\idealij[2])$ and $(\ideal_1,\ideal_2)$ must contain one move $\pb_j$ whose label affects these coefficients. 
    
    By Lemma \ref{lem: move_action}, $\labl(\pb_j)$ can only affect the values of these coefficients if $\labl(\pb_j)=s_{i_1}$ or if $\labl(\pb_j)$ does not commute with $s_{i_1}$. Since $\idealij[1]$ contains no element with label $s_{i_1}$, we conclude that $\labl(\pb_j)$ does not commute with $s_{i_1}$. Furthermore, since $c_{i_2}-1\in\{0,+1\}$, we must have had $c_{i_2}=+1$. Since the weight corresponding to $p_{i_2}$ is $\dfwt[i_3]-\dfwt[i_2]$ by Proposition \ref{prop:uij_repeated_action}, the equality $c_{i_2}=+1$ can only occur if if $i_3=i_1$. 

    Conversely, suppose that $i_1=i_3$ and that there is a sequence of moves $\pb_1,\dots, \pb_d$ as in the statement such that one of the $\labl(\pb_j)$ does not commute with $s_{i_1}$. Since $i_1=i_3$, we have $(c_{i_1},c_{i_2})=(-1,1)$. Then Lemma \ref{lem: move_action} implies that $(c_{\ideal_1},c_{\ideal_2})=(0,0)$, so by Lemma \ref{lem: derivation_plucker_coordinate}, the derivation acts trivially on the Pl\"ucker monomial $p_{\ideal_1}p_{\ideal_2}$. 
\end{proof}

The following corollary is an immediate consequence of Lemma \ref{lem: i1=i3}.

\begin{cor}
    \label{cor:Pi*+}
    As posets, $\sP_{\is}\cong \sP_{\is}^+$ if and only if $i_1\neq i_3$. If $i_1=i_3$, then $\sP_{\is}^+$ is isomorphic to the connected (by cover relations) component of $\sP_{\is}$ that contains the top element $(\idealij[1],\idealij[2])$ after deleting all cover relations corresponding to moves $\pb$ whose labels do not commute with $s_{i_1}$.
\end{cor}

As we did in the previous section, we now partition the toric terms for any $(\ideal_1^+,\ideal_2^+)\in\sP_{\is}^+$. While we could define these partitions from scratch, it will turn out to be much more convenient to define them in terms of our existing sets $T^{A_1,A_2}_{\ideal_1,\ideal_2}$ as follows. For each pair of embeddings $(\emb_1^+,\emb_2^+)$ of $(\ideal_1^+,\ideal_2^+)$, $\emb_i^+:\ideal_i^+\hookrightarrow\minposet$, we obtain a pair of embeddings $(\emb_1^+|_{\ideal_1}, \emb_2^+|_{\ideal_2})$ of $(\ideal_1,\ideal_2)$ by restriction. 

\begin{df}
    For each $(\ideal_1^+,\ideal_2^+)=\delta_{i_1}(\ideal_1,\ideal_2)\in\sP_{\is}^+$, $A_1\subset A_{\ideal_2,\ideal_1}$, and $A_2\subset A_{\ideal_1,\ideal_2}$, we define the set
    \[
    T^{A_1,A_2,+}_{\ideal_1,\ideal_2}=\Bigl\{(\emb_1^+:\ideal_1^+\hookrightarrow\minposet, \emb_2^+:\ideal_2^+\hookrightarrow\minposet)~\Big|~(\emb_1^+|_{\ideal_1}, \emb_2^+|_{\ideal_2})\in T^{A_1,A_2}_{\ideal_1,\ideal_2}\Bigr\}.
    \]
\end{df}

Since the restrictions are well-defined for every pair of embeddings $(\emb_1^+,\emb_2^+)$, these sets partition the toric terms of $p_{\ideal_1^+}p_{\ideal_2^+}$ and we will use them to organize the cancellation of terms in $\delta_{i_1}(\cD_{\is})|_{\opendunim}$. We will be able to use the same strategy as we used in Section \ref{sec:denominators}, but we first need a few preparatory lemmas. The first statement we will need for the cancellation is an immediate corollary of the analogous statements in Section \ref{sec:denominators}.

\begin{cor}
\label{cor: toric_sets+}
    Let $(\ideal_1,\ideal_2)\in\sP_{\is}$ such that $\delta_{i_1}(\ideal_1,\ideal_2)\neq 0$. Then the sets $T^{\varnothing,\varnothing,+}_{\ideal_1,\ideal_2}$ are empty unless $(\ideal_1,\ideal_2)=(\idealij[1],\idealij[2])$ is the top element of $\sP_{\is}$, in which case we have
    \[
    T^{\varnothing,\varnothing,+}_{\idealij[1],\idealij[2]} = \bigl\{ (\emb_1^+,\id_{\idealij[2]})~\big|~\emb_1^+|_{\idealij[1]}=\id_{\idealij[1]} \bigr\}.
    \]
    Moreover, for every $\pb_{i_1}\in\minposet$ with $\labl(\pb_{i_1})=s_{i_1}$ there exists a (necessarily unique) embedding $\emb_1^+$ with $(\emb_1^+,\id_{\idealij[2]})\in T^{\varnothing,\varnothing,+}_{\idealij[1],\idealij[2]}$ such that $\emb_1^+(\pb)=\pb_{i_1}$, where $\pb\in\minposet\setminus\idealij[1]$ is such that $\idealij[1]^+=\idealij[1]\sqcup\{\pb\}$.
\end{cor}
\begin{proof}
    The first part of the statement follows immediately from Lemmas \ref{lem: initial-plucker} and \ref{lem: toric-partition}, because only the toric set $T^{\varnothing,\varnothing}_{\idealij[1],\idealij[2]}$ is nonempty among all the toric sets $T^{\varnothing,\varnothing}_{\ideal_1,\ideal_2}$. 
    
    For the second part, Lemma \ref{lem:leading-term-derivation} gives $(\idealij[1]^+,\idealij[2]^+)=\delta_{i_1}(\idealij[1],\idealij[2])=(\idealij[1]\sqcup\pb, \idealij[2])$, where $\pb$ is the minimal element of $\minposet$ with label $s_{i_1}$. Now let $(\emb_{i_1}^+, \emb_{i_2}^+)\in T^{\varnothing,\varnothing,+}_{\idealij[1],\idealij[2]}$, so that by Lemma \ref{lem: initial-plucker} we have $\emb_{i_1}^+|_{\ideal_{i_1}} = \id_{\ideal_{i_1}}$ and $\emb_{i_2}^+|_{\ideal_{i_2}} = \id_{\ideal_{i_2}}$ since $(\id_{\ideal_{i_1}},\id_{\ideal_{i_2}})$ is the only element in $T_{\idealij[1],\idealij[2]}^{\varnothing,\varnothing}$. Since $\idealij[2]^+ = \idealij[2]$, we thus also have $\emb_{i_2}^+ = \id_{\ideal_{i_2}}$. 
    
    Finally, note that for \emph{every} $\pb_*\in\minposet$ with label $s_{i_1}$ the assignment $\pb\mapsto\pb_*$ defines a valid embedding of $\idealij[1]^+$ since $\emb_{i_1}^+|_{\idealij[1]}=\id_{\idealij[1]}$ places no restrictions on the image of $\pb\in\idealij[1]^+$ under $\emb_{i_1}^+$.
\end{proof}

Cubes are defined the same way for $\sP_{\is}^+$ as they were for $\sP_{\is}$. However, it will be more convenient for us to study cubes in $\sP_{\is}^+$ by considering cubes $\sC\subset \sP_{\is}$ and applying $\delta_{i_1}$ to $\sC$. Because $\deis$ may act trivially on some elements of $\sC$, we first verify that the image $\deis(\sC)$ is a cube in $\sP_{\is}^+$. 

\begin{lem}
    \label{lem: deriv_cube}
    For any $m$-cube $\sC\subset\sP_{\is}$ with set $S_{\sC}$ of labels, the non-zero image $\deis(\sC)\subset \sP_{\is}^+$ of $\sC$ under $\deis$ is one of the following:
    \begin{enumerate}
        \item an $m$-cube if $i_1\neq i_3$ or if $i_1=i_3$ and every minimal length sequence $\pb_1,\pb_2,\dots$ of moves between $(\idealij[1],\idealij[2])$ and the bottom element of $\sC$ consists of moves whose labels commute with $s_{i_1}$,
        \item an $(m-1)$-cube if $i_1=i_3$ and $S_{\sC}$ contains an label not commuting with $s_{i_1}$,
        \item or empty if $i_1=i_3$ and there is a minimal length sequence $\pb_1,\pb_2,\dots$ of moves between $(\idealij[1],\idealij[2])$ and the top element of $\sC$ which contains a move whose label does not commute with $s_{i_1}$.
    \end{enumerate} 
\end{lem}

\begin{proof}
    We consider first the cases that $i_1\neq i_3$ or the cases that $i_1=i_3$ but every minimal length sequence 
    \[
    \pb_1\in A_{\idealij[1],\idealij[2]}, ~
    \pb_2\in A_{\idealij[1]\setminus\pb_1,\idealij[2]\sqcup\pb_1'}, ~
    \ldots
    \]
    between $(\idealij[1],\idealij[2])$ and the bottom element of $\sC$ consists only of moves $\pb_j$ whose labels commute with $s_{i_1}$. In these cases, Lemma \ref{lem: i1=i3} implies that $\deis$ acts nontrivially on each element of $\sC$, and thus $\deis(\sC)$ is an $m$-cube in $\sP_{\is}^+$ since every cover relation between elements of $\sC$ is also a cover relation between the corresponding elements of $\deis(\sC)$ by definition of $\sP_{\is}^+$. 

    We now consider the case where $i_1=i_3$ and there is a minimal length sequence
    \[
    \pb_1\in A_{\idealij[1],\idealij[2]},~
    \pb_2\in A_{\idealij[1]\setminus\pb_1,\idealij[2]\sqcup\pb_1'},~
    \ldots
    \]
    between $(\idealij[1],\idealij[2])$ and the bottom element of $\sC$ such that the label of some $\pb_j$, $s_{m_j}=\labl(\pb_j)$ does not commute with $s_{i_1}$. Letting $S_{\sC}$ denote the set of labels of the cube $\sC$, we consider the cases that $s_{m_j}\in S_{\sC}$ or $s_{m_j}\notin S_{\sC}$. If $s_{m_j}\notin S_{\sC}$, then by Lemma \ref{lem: Spacek_LPLG} there must be a sequence of moves between $(\idealij[1],\idealij[2])$ and the top element of $\sC$ containing $\pb_j$, which by Lemma \ref{lem: i1=i3} implies that $\delta_{i_1}$ acts trivially on every element of $\sC$ and hence that $\deis(\sC)=\varnothing$. 
    
    Finally suppose $s_{m_j}\in S_{\sC}$. Since $\sC$ is an $m$-cube, we may identify it with the set of subsets of $S_{\sC}$ ordered by inclusion by Corollary \ref{cor:mcube-subsets}. Explicitly, we identify an element $(\ideal_1,\ideal_2)\in \sC$ with the labels $\labl(A_{\ideal_1,\ideal_2})\cap S_{\sC}$ corresponding to cover relations $(\ideal_1,\ideal_2)\gtrdot (\ideal_3,\ideal_4)$ in the partial ordering on $\sC$. Then by Lemmas \ref{lem:unique_non_commuting} and \ref{lem: i1=i3}, we have that $\delta_{i_1}$ acts trivially on $(\ideal_1,\ideal_2)\in\sC$ if and only if $s_{m_j}\notin \labl(A_{\ideal_1,\ideal_2})$, so that
    \[
    \{\deis(\ideal_1,\ideal_2)~|~ (\ideal_1,\ideal_2)\in \sC:~\deis(\ideal_1,\ideal_2)\neq 0\} = \{\deis(\ideal_1,\ideal_2)\in\sC~|~\labl(A_{\ideal_1,\ideal_2})\ni s_{m_j}\}
    \] 
    can be identified with the sub-cube of $\sC$ consisting of elements $(\ideal_1,\ideal_2)$ with $s_{m_j}\in \labl(A_{\ideal_1,\ideal_2})$, which is an $(m-1)$-cube. Again, $\delta_{i_1}(\sC)$ is a cube in $\sP_{\is}^+$ because every cover relation between elements of $\sC$ in which $\delta_{i_1}$ acts nontrivially on both elements corresponds to a cover relation between the corresponding elements of $\delta_{i_1}(\sC)$ by definition of $\sP_{\is}^+$.
\end{proof}

In fact, it is an immediate consequence of Corollary \ref{cor:Pi*+} that every cube in $\sP_{\is}^+$ arises in this way, since this corollary implies that the Hasse diagram of $\sP_{\is}^+$ is isomorphic to a connected component of the Hasse diagram of $\sP_{\is}$. With this lemma, we can use $\deis$ to lift the bijections of Lemma \ref{lem: toric-bijections} to the sets $T^{A_1,A_2,+}_{\ideal_1,\ideal_2}$. Recall that as before we will write $A_{\ideal_1,\ideal_2}(S)$ to denote the set $\{\pb\in A_{\ideal_1,\ideal_2}\mid \pb\in S\}$ of moves in $A_{\ideal_1,\ideal_2}$ with labels in $S$. 

\begin{lem}
    \label{lem: toric_bijections+}
    Let $\sC\subset \sP_{\is}$ be an $m$-cube with set of labels $S$. For any two distinct elements $(\ideal_1,\ideal_2), (\ideal_3,\ideal_4)\in \sC$ such that $\delta_{i_1}$ acts nontrivially on both, consider the sets $A_{21} = A_{\ideal_2,\ideal_1}(S)$, $A_{12}=A_{\ideal_1,\ideal_2}(S)$, $A_{43}=A_{\ideal_4,\ideal_3}(S)$, and $A_{34} = A_{\ideal_3,\ideal_4}(S)$. Then there is a bijection
    \[
    \kappa_{A_{21},A_{12}}^{A_{43},A_{34},+}: T^{A_{21},A_{12},+}_{\ideal_1,\ideal_2} \overset{\sim}{\longrightarrow} T^{A_{43},A_{34},+}_{\ideal_3,\ideal_4}
    \]
    such that for any $(\emb_1^+,\emb_2^+)\in T^{A_{21},A_{12},+}_{\ideal_1,\ideal_2}$, we have $(\emb_1^+,\emb_2^+) \overset{t}{=} \kappa_{A_{21},A_{12}}^{A_{43},A_{34},+}(\emb_1^+,\emb_2^+)$, where $\overset{t}{=}$ denotes toric equality as in Definition \ref{df:toric_equality}.
\end{lem}

\begin{proof}
    As in the proof of Lemma \ref{lem: toric-bijections}, we may immediately reduce to the case where $(\ideal_3,\ideal_4)$ can be obtained from $(\ideal_1,\ideal_2)$ by a single move $\pb$ with label $s_m=\labl(\pb)\in S$, corresponding to a cover relation of $\sP_{\is}$. Furthermore, we will similarly simplify notation by writing $T_{12}^+ = T^{A_{21},A_{12},+}_{\ideal_1,\ideal_2}$ and $T_{34}^+ = T^{A_{43},A_{34},+}_{\ideal_3,\ideal_4}$, while writing $\kappa_{12}^{34,+}=\kappa_{A_{21},A_{12}}^{A_{43},A_{34},+}$ for the remainder of this proof.
    
    We now consider two cases depending on whether $s_m$ commutes with $s_{i_1}$. If $s_m$ commutes with $s_{i_1}$, then both the construction and proof of Lemma \ref{lem: toric-bijections} can be applied directly, with $(\emb_1^+,\emb_2^+)$ in place of $(\emb_1,\emb_2)$, to construct $\kappa_{12}^{34,+}$ and show that it satisfies the required properties. 

    Now suppose that $s_m$ does not commute with $s_{i_1}$. In this case since we are assuming that $\delta_{i_1}$ acts nontrivially on both $(\ideal_1,\ideal_2)$ and $(\ideal_3,\ideal_4)$, then by Lemmas \ref{lem: derivation_plucker_coordinate},  \ref{lem:unique_non_commuting}, and \ref{lem: move_action}, we have that
    \[
    \begin{cases}
        \ideal_1^+ = \ideal_1,~\ideal_2^+=\ideal_2\sqcup\pb_2,~\ideal_3^+=\ideal_3\sqcup\pb_3 ~\text{and}~\ideal_4^+=\ideal_4, 
        &\text{if $(\ideal_1,\ideal_2) \lessdot (\ideal_3,\ideal_4)$,} \\
        \ideal_1^+ = \ideal_1\sqcup\pb_1,~\ideal_2^+ = \ideal_2,~\ideal_3^+=\ideal_3~\text{and}~\ideal_4^+=\ideal_4\sqcup\pb_4, 
        &\text{if $(\ideal_1,\ideal_2)\gtrdot(\ideal_3,\ideal_4)$,}
    \end{cases}
    \]
    where in each case, $\pb_i$ denotes the minimal element of $\minposet\setminus\ideal_i$ with label $i_1$.  
    
    First, consider the case $(\ideal_1,\ideal_2) \lessdot (\ideal_3,\ideal_4)$ corresponding to the move $\pb\in A_{\ideal_3,\ideal_4}$ with label $\labl(\pb)=s_m$. This means that $(\ideal_1,\ideal_2) = (\ideal_3\setminus\pb, \ideal_4\sqcup\pb')$, so that we obtain (since $\ideal_1^+=\ideal_1$ and $\ideal_4^+=\ideal_4$)
    \[
    \ideal_3^+\setminus\ideal_1^+ = \{\pb_3, \pb\} \qand \ideal_2^+\setminus \ideal_4^+ =\{\pb_2,\pb'\}.
    \]
    Then given a pair of embeddings $(\emb_1^+,\emb_2^+)\in T_{12}^+$ of $\ideal_1^+$ and $\ideal_2^+$, we define 
    \[
    \emb_3^+|_{\ideal_1^+} = \emb_1^+
    ,\quad \emb_3^+(\pb_3)=\emb_2^+(\pb_2) 
    \qand \emb_3^+(\pb)=\emb_2^+(\pb'), \qwhile \emb_4^+=\emb_2^+|_{\ideal_4^+},
    \tag{$\dag\dag$}
    \label{aux_eq:def_kappa_plus_case1}
    \]
    and set $\kappa_{12}^{34,+}(\emb_1^+,\emb_2^+) = (\emb_3^+, \emb_4^+)$. The image lies in $T_{34}^+$, since $\emb_3^+|_{\ideal_3}$ and $\emb_4^+|_{\ideal_4}$ are clearly embeddings of $\ideal_3$ and $\ideal_4$ and $(\emb_3^+|_{\ideal_3}, \emb_4^+|_{\ideal_4})\in T_{34}$ by the definitions of $A_{\ideal_2,\ideal_1}$ and $\first[\pb]_{\ideal_1,\ideal_2}$ along with Corollary \ref{cor:mcube-subsets} as in the proof of Lemma \ref{lem: toric-bijections}.
    
    In the second case, we have $(\ideal_1,\ideal_2) \gtrdot (\ideal_3,\ideal_4)$ corresponding to the move $\pb\in A_{\ideal_1,\ideal_2}$ with label $s_m$.
    This means that $(\ideal_3,\ideal_4) = (\ideal_1\setminus\pb, \ideal_2\sqcup\pb)$, so that we obtain (since $\ideal_3^+=\ideal_3$ and $\ideal_2^+=\ideal_2$)
    \[
    \ideal_1^+\setminus\ideal_3^+ = \{\pb,\pb_1\}\qand \ideal_4^+\setminus\ideal_2^+=\{\pb_4,\pb'\}.
    \]
    Then given a pair of embeddings $(\emb_1^+,\emb_2^+)\in T_{12}^+$ of $\ideal_1^+$ and $\ideal_2^+$, we define 
    \[
    \emb_3^+ = \emb_1^+|_{\ideal_3^+}, \qwhile \emb_4^+|_{\ideal_2^+}=\emb_2^+, \quad \emb_4^+(\pb_4)=\emb_1^+(\pb_1) \qand \emb_4^+(\pb')=\emb_1^+(\pb),
    \tag{$\ddag\ddag$}
    \label{aux_eq:def_kappa_plus_case2}
    \]
    and set $\kappa_{12}^{34,+}(\emb_1^+,\emb_2^+)=(\emb_3^+,\emb_4^+)$. Again, the image clearly lies in $T_{34}^+$.
    
    Showing that $\kappa_{12}^{34,+}$ defines a bijection and that $\kappa_{12}^{34,+}(\emb_1^+,\emb_2^+) \overset{t}{=} (\emb_1^+,\emb_2^+)$ is completely analogous to the proof of Lemma \ref{lem: toric-bijections}.
\end{proof}
Before we can compute the toric restrictions of the $\delta_{i_1}(\cD_{\is})$, we need one final technical detail in the situation of Lemma \ref{lem: deriv_cube} where a $1$-cube in $\sP_{\is}$ becomes a $0$-cube, since as we saw before in the proof of Theorem \ref{thm: phi_cancellation}, $0$-cubes correspond to the toric terms which do not cancel. 
\begin{lem}
    \label{lem: cube_demotion}
    Let $(\ideal_1,\ideal_2) \gtrdot (\ideal_3,\ideal_4)\in\sP_{\is}$ be a $1$-cube corresponding to a move $\pb\in A_{\ideal_1,\ideal_2}$ with label $s_m$ which does not commute with $s_{i_1}$ and suppose that $i_1=i_3$. Then $T^{\varnothing,s_m,+}_{\ideal_1,\ideal_2}=\varnothing$.
\end{lem}

\begin{proof}
    Suppose $(\emb_1^+,\emb_2^+)\in T^{\varnothing,s_m,+}_{\ideal_1,\ideal_2}$ and let $\pb\in\ideal_1$ which moves to $\pb'\in\minposet\setminus\ideal_2$ be as in the statement. By the proof of Lemma \ref{lem: i1=i3}, we have that $(\ideal_1^+,\ideal_2^+)=\delta_{i_1}(\ideal_1,\ideal_2)=(\ideal_1\sqcup\pb_{i_1},\ideal_2)$, where $\pb_{i_1}$ is the minimal element of $\minposet\setminus\ideal_1$ with label $s_{i_1}$. Let $\emb_i=\emb_i^+|_{\ideal_i}$. Then by definition of $T^{\varnothing,s_m,+}_{\ideal_1,\ideal_2}$, we have $(\emb_1,\emb_2)\in T^{\varnothing,s_m}_{\ideal_1,\ideal_2}$, so that extending $\emb_2$ by $\emb_2(\pb')=\emb_1(\pb)$ defines a valid embedding of $\ideal_2\sqcup \pb$. Since $\ideal_2\supset \ideal_{i_2}$, we must have in particular that $\emb_1(\pb_1)\in \minposet\setminus \ideal_{i_2}$.
    
    Since $s_m$ does not commute with $s_{i_1}$, then $\pb_{i_1}$ and $\pb$ must be comparable in $\minposet$. Furthermore, since $\pb_{i_1}$ is maximal in $\ideal_1^+$, we must have that $\emb_1^+(\pb_{i_1}) > \emb_1^+(\pb) = \emb_1(\pb)$. Combined with $\emb_1(\pb)\in\minposet\setminus\ideal_{i_2}$, we thus have that $\emb_1^+(\pb_{i_1})\in \minposet\setminus\ideal_{i_2}$. However, this is a contradiction since $i_1=i_3$ implies that $\minposet\setminus\ideal_{i_2}$ has no elements with label $s_{i_1}=s_{i_3}$.
\end{proof}

We now come to the main theorem in this section computing the toric expression of $\delta_{i_1}(\cD_{\is})$.
\begin{thm}
    \label{thm:derivation_cancellation}
    For each $\is\neq k\in[n]$, the restriction of $\delta_{i_1}(\cD_{\is})$ to the torus $\opendunim$ is given by
    \[
    \delta_{i_1}(\cD_{\is})|_{\opendunim} = a_{\ideal_{i_1}}
    a_{\ideal_{i_2}} \left(\tsum[_{\pb:\,\indx(\pb)={i_1}}] a_{\pb}\right), 
    \]
    where the sum is over all $\pb\in\minposet$ with index $i_1$ or equivalently $\labl(\pb)=s_{i_1}$. 
\end{thm}
\begin{proof}
    This proof is entirely analogous to the proof of Theorem \ref{thm: phi_cancellation} using the results proved so far in this section. Recall that we defined $L_d\subset\sP_{\is}$ to be the subset of elements of $\sP_{\is}$ which can be obtained from the top element $(\idealij[1],\idealij[2])$ by a minimal length sequence of $d\ge 0$ moves
    \[
    \pb_1\in A_{\idealij[1],\idealij[2]},~\ldots,~
    \pb_d\in A_{\idealij[1]\setminus\{\pb_1,\dots,\pb_{d-1}\}, \idealij[2]\sqcup\{\pb_1',\dots,\pb_{d-1}'\}}.
    \]
    We begin by applying $\deis$ to $\cD_{\is}$:
    \[
    \deis(\cD_{\is}) 
    = \sum_{d\ge 0}(-1)^d \sum_{(\ideal_1,\ideal_2)\in L_d} \deis(p_{\ideal_1}p_{\ideal_2})
    = \sum_{d\ge 0}(-1)^d \sum_{(\ideal_1^+,\ideal_2^+)\in L_d^+} p_{\ideal_1^+}p_{\ideal_2^+},
    \]
    where we write $L_d^+ = \deis(L_d)$ and used Lemma \ref{lem:derivation-one-term} in the last equality. Since the toric sets $T^{A_1,A_2,+}_{\ideal_1,\ideal_2}$ partition the toric terms of $p_{w_1^+}p_{w_2^+}$, we obtain the following expression for the restriction
    \[
    \delta_{i_1}(\cD_{\is})|_{\opendunim} = \sum_{d\ge 0}(-1)^d\sum_{(\ideal_1^+,\ideal_2^+)\in L_d^+}~\sum_{A_1,A_2}~\sum_{(\emb_1^+,\emb_1^+)\in T^{A_1,A_2,+}_{\ideal_1,\ideal_2'}} a_{\emb_1^+}a_{\emb_2^+},
    \]
    where the third sum is over all $A_1\subset A_{\ideal_1,\ideal_2}$ and $A_2\subset A_{\ideal_2,\ideal_1}$. 
    
    For each pair $A_1,A_2$, by Lemma \ref{lem: unique-mcube}, there is a unique $(|A_1|+|A_2|)$-cube in $\sP_{\is}$ containing $(\ideal_1,\ideal_2)$, and this cube has set of labels $\{\labl(\pb)\mid \pb\in A_1\cup A_2\}$. For any cube $\sC\subset\sP_{\is}$, let $d_{\sC}$ denote the minimum number of moves in $\sP_{\is}$ between the top element $(\idealij[1],\idealij[2])$ of $\sP_{\is}$ and the top element of $\sC$ and let $S_{\sC}$ denote the set of labels of $\sC$. Then we may rewrite the sum above by summing over all cubes contained in $\sP_{\is}$. Before rewriting it like this, we first note that by Lemma \ref{lem: toric_bijections+}, we have for all cubes $\sC\subset\sP_\is$ and all pairs $(\ideal_1,\ideal_2), (\ideal_3,\ideal_4)\in \sC$ such that $\deis$ acts nontrivially on both $(\ideal_1,\ideal_2)$ and $(\ideal_3,\ideal_4)$ that 
    \[
    \sum_{(\emb_1^+,\emb_2^+)\in T^{A_1,A_2,+}_{\ideal_1,\ideal_2}}a_{\emb_1^+}a_{\emb_2^+} = \sum_{(\emb_3^+,\emb_4^+)\in T^{A_3,A_4,+}_{\ideal_3,\ideal_4}}a_{\emb_3^+}a_{\emb_4^+}
    \]
    where $A_1 = A_{\ideal_2,\ideal_1}(S_{\sC})$, $A_2=A_{\ideal_1,\ideal_2}(S_{\sC})$, $A_3=A_{\ideal_4,\ideal_3}(S_{\sC})$ and $A_4= A_{\ideal_3,\ideal_4}(S_{\sC})$. Thus, the sum of monomials corresponding to pairs of embeddings in the toric set $T^{A_{\ideal_2,\ideal_1}(S_{\sC}), A_{\ideal_1,\ideal_2}(S_{\sC}),+}_{\ideal_1,\ideal_2}$ is independent of the particular element $(\ideal_1,\ideal_2)\in \sC$ and instead only depends on whether $\delta_{i_1}$ acts nontrivially on $(\ideal_1,\ideal_2)$ or not. Thus, for each cube $\sC$, we will simply use the notation $T_{\sC}^{S_{\sC},+}$ to denote any one of the torically equivalent toric sets $T^{A_{\ideal_2,\ideal_1}(S_{\sC}), A_{\ideal_1,\ideal_2}(S_{\sC}),+}_{\ideal_1,\ideal_2}$ where $\delta_{i_1}$ acts on $(\ideal_1,\ideal_2)$ nontrivially or the empty set if $\delta_{i_1}$ acts trivially on every $(\ideal_1,\ideal_2)\in\sC$. We may now rewrite the sum as

    \begin{align*}
      \deis(\cD_{\is})|_{\opendunim} 
      &= \sum_{\sC\subset \sP_{\is}} \sum_{(\ideal_1,\ideal_2)\in\sC:\,\deis(\ideal_1,\ideal_2)\neq 0}(-1)^{d_{\sC}+|A_{\ideal_2,\ideal_1}(S_{\sC})|}\sum_{(\emb_1^+,\emb_2^+)\in T^{S_\sC,+}_{\sC}} a_{\emb_1^+}a_{\emb_2^+}\\
      &= \sum_{\sC\subset \sP_{\is}} \sum_{(\emb_1^+,\emb_2^+)\in T^{S_\sC,+}_{\sC}} a_{\emb_1^+}a_{\emb_2^+}\sum_{(\ideal_1,\ideal_2)\in\sC:\,\deis(\ideal_1,\ideal_2)\neq 0}(-1)^{d_{\sC}+|A_{\ideal_2,\ideal_1}(S_{\sC})|}.
    \end{align*}    
    By Lemma \ref{lem: deriv_cube}, the innermost sum above is the alternating sum over the levels of an $m'=|S_{\sC}|$-cube, an $m'=(|S_{\sC}|-1)$-cube, or an empty cube, which again by the binomial theorem we may evaluate as
    \[
    \sum_{j=0}^{m'} (-1)^j \binom{m'}{j} = (1+(-1))^{m'} = \begin{cases} 0, & \text{if $m'>0$,} \\ 1, & \text{if $m'=0$.}\end{cases}
    \]
    The sum above is nonzero only when $m'=0$, which can occur in two cases by Lemma \ref{lem: deriv_cube}: either when $\sC=\sC_0$ is a $0$-cube where $\deis$ acts nontrivially; or else when $\sC=\sC_1$ is a $1$-cube where $\deis$ acts nontrivially on the top element of $\sC$ and trivially on the bottom element of $\sC$, that is, the case where the unique label $s_m$ of $\sC$ does not commute with $s_{i_1}$. Hence, we can split the nonzero terms into two sums depending on these cases:
    \[
    \deis(\cD_{\is})|_{\opendunim} = \left(\sum_{\sC=\sC_0} \sum_{(\emb_1^+,\emb_2^+)\in T_{\sC}^{\varnothing,+}} a_{\emb_1^+}a_{\emb_2^+}\right) + \left(\sum_{\sC=\sC_1} \sum_{(\emb_1^+,\emb_2^+)\in T_{\sC}^{s_m,+}} a_{\emb_1^+}a_{\emb_2^+}\right).
    \]
    However, the second sum over $\sC=\sC_1$ is zero since $T_{\sC}^{s_m,+}=\varnothing$ by Lemma \ref{lem: cube_demotion}. Furthermore, by Corollary \ref{cor: toric_sets+}, there is only one nonzero term in the first sum coming from $\sC_0 = \{(\idealij[1],\idealij[2])\}$, so we obtain
    \[
    \deis(\cD_{\is})|_{\opendunim} 
    = \sum_{(\emb_{i_1}^+, \emb_{i_2}^+)\in T_{\idealij[1],\idealij[2]}^{\varnothing,\varnothing,+}} a_{\emb_{i_1}^+}a_{\emb_{i_2}^+} 
    = a_{\ideal_{i_1}}a_{\ideal_{i_2}}\left(\tsum[_{\pb:\,\indx(\pb)={i_1}}] a_{\pb}\right)
    \]
    where the rightmost sum is over $\pb\in\minposet$ with index $i_1$. 
\end{proof}

%% file: 6main_proof.tex
\section{Proof of main theorem}
\label{sec:main_proof}

In the previous Section \ref{sec:numerators}, we explicitly avoided $\is=k$ because the corresponding term $\pot_k$ behaves slightly differently than the terms $\pot_{\is}$ where $\is\neq k$ in our construction. Before we can present the proof of our main theorem, we handle this special term which we often call the ``quantum term'' directly in the first subsection. 

\addtocontents{toc}{\SkipTocEntry}
\subsection{The quantum term}\label{sec:quantum}
Comparing Theorem \ref{thm:toric_monomial_phi} and Theorem \ref{thm:derivation_cancellation} with the terms of toric restriction of the Lie-theoretic superpotential as given in \eqref{eq:toric_expression_pot_Lie_3}, we can see that the terms $\pot_{\is}$ for $\is\in[n]\setminus\{k\}$ of $\pot_{\can}$ deal with the terms $\dChfdual[_{i_1}](u_-)$ and it is clear that $\pot_k=\p_{\yng(1)}/\p_{\varnothing}$ takes care of the term $\dChfdual[_k](u_-)$. Hence, the remaining term in the Lie-theoretic superpotential is $\dChedual[_k](u_+^{-1})$ where $k$ is such that $\P=\P_k$. We will start with the toric expression for this term as stated in equation \eqref{eq:toric_expression_pot_Lie_3}:
\[
\dChedual[_k](u_+^{-1}) = q\frac{\sum_{\iota:\idealPrime\hookrightarrow\minposet} a_{\iota(\idealPrime)}}{a_{\minposet}}.
\]
In our case, this term is the only one containing the quantum parameter, so it is often referred to as the \emph{quantum term}. Note that this is the only term of the Lie-theoretic potential (or more precisely its restriction to $\opendunim$) that depends on the formal quantum paramter $q\in\C^*$.

Both the numerator and the denominator of the toric expression above are given by summing over the products of toric coordinates corresponding to embeddings of a poset into the minuscule poset $\minposet$: for the numerator it is the subposet $\idealPrime$, whereas for the denominator it is the order ideal $\minposet$ itself. Hence, we directly see using \eqref{eq:toric_expression_for_Plucker_coord_in_posets} and \eqref{eq:phi_k_and_p_minposet} that the denominator equals $\phiGLS[k]=\p_{\minposet}$ on the torus and therefore on all of $\dunimP$ since $\dunimP$ is irreducible and $\opendunim\subset\dunimP$ is dense. On the other hand, for the numerator to correspond to a Pl\"ucker coordinate, the subset $\idealPrime$ must be an order ideal, or equivalently $\wPrime$ must be a minimal coset representative, a fact which we foreshadowed already in Remark \ref{rem:wPrime_is_min_coset_rep}. We will establish this by showing a more general result, which does require $k$ to be a minuscule vertex of $\dG$.

For $j\in[n]$, recall that we wrote $\woj\in\weylj$ for the longest elements of the parabolic Weyl subgroups. Write $\cosetsj$ for the set of minimal coset representatives of the cosets in $\weyl/\weylj$, with $\wJ\in\cosetsj$ the representative of $\wo\weylj$. Moreover, let $\wJpprime\in\cosetsj$ be the representative of $\woj s_j\weylj$, and let $\wJprime = \wJ(\wJpprime)^{-1}\in\weyl$. Then:
\begin{lem}\label{lem:wJprime_min_coset_rep}
    The element $\wJprime=\wJ(\wJpprime)^{-1}\in\weyl$ lies in $\cosetsj$.
\end{lem}
\begin{proof}
We begin by excluding the trivial case where $\wJpprime=\wJ$, since $\wJprime=e\in\cosetsj$. (This only occurs for type $\LGA_n$ with $j\in\{1,n\}$; cf.~\cite[Remark 7.3]{Spacek_LP_LG_models}.)

Working from contradiction, assume that $\wJprime\notin\cosetsj$, i.e.~there exists an $i\in[n]\setminus\{j\}$ such that $\wJppprime=\wJprime s_i$ has $\ell(\wJppprime)=\ell(\wJprime)-1$. On the other hand, since $\woj\in\weylj$ is the longest element of the parabolic Weyl subgroup, we have $\ell(s_i\woj)=\ell(\woj)-1$, so that $s_i\woj s_j(-\dfwt[j]) = \wJpppprime(-\dfwt[j])$ with $\ell(\wJpppprime)\le\ell(\wJpprime)$. 
However, now we find that
\[
\dfwt[\DS(j)]=\wJ(-\dfwt[j])=\wJprime\wJpprime(-\dfwt[j])=\wJppprime\wJpppprime(-\dfwt[j]).
\]
In other words, $\wJppprime\wJpppprime$ and $\wJ$ are both coset representatives of $\wo\weylj$, with lengths related by $\ell(\wJppprime\wJpppprime)\le\ell(\wJ)-1$, which contradicts the minimality of the coset representative $\wJ$.
\end{proof}
\begin{rem}
The above proof generalizes straightforwardly to parabolic Weyl subgroups $\weylj[J]$ with $J\subset[n]$ having $|J|>1$.
\end{rem}
Since $\weylp=\weylj[k]$, $\wP=\wJ[k]$, $\wop=\woj[k]$, $\wPrime=\wJprime[k]$ and $\wPPrime=\wJpprime[k]$, we combine Lemma \ref{lem:wJprime_min_coset_rep} with the preceding discussion to obtain the following corollary.
\begin{cor}\label{cor:quantum_term}
The element $\wPrime=\wP(\wPPrime)^{-1}\in\weyl$ lies in $\cosets$ and hence $\idealPrime$ is an order ideal of $\minposet$. In particular, the complement of the ideal $\ideal''\subideal\minposet$ defined by $\wPPrime$ is isomorphic to the ideal $\ideal'\subideal\minposet$ defined by $\wPrime$. 
\end{cor}

Since equation \eqref{eq:toric_expression_for_Plucker_coord_in_posets} tells us that $\p_{\minposet}(u_-)=a_{\minposet}$ and $\p_{\idealPrime}(u_-)=\sum_{\iota:\idealPrime\hookrightarrow\minposet} a_{\iota(\idealPrime)}$ for $u_-\in\opendunim$, we directly conclude the following:

\begin{prop}
\label{prop:quantum_term}
For any fixed $q\in\C^*$, the following equality holds on the open dense torus $\opendunim$ and hence on $\dunimP$
\[
\dChedual[_k](u_+^{-1}) = q\frac{\p_{\idealPrime}(u_-)}{\p_{\minposet}(u_-)},
\]
where $u_+\in\dunip$ is determined by $u_-\in\dunimP$ and $q\in\C^*$ as in Corollary \ref{cor:unique_u+_for_u-_for_fixed_q}.
\end{prop} 

We now have descriptions for all the terms of the superpotential in Pl\"ucker coordinates, so we can put these together and prove that the Lie-theoretic and canonical mirror models are isomorphic type-independently.

\addtocontents{toc}{\SkipTocEntry}
\subsection{The canonical mirror variety}
\label{sec:mirror_isom}

In order to relate our LG models to Rietsch's LG models, we first establish that the mirror spaces $\mX_\can$ and $\mX_\Lie$ are isomorphic. Our strategy will follow the outline used for the proof of \cite[Theorem 3.1]{Spacek_Wang_exceptional}, although we will need to make several adjustments to generalize the statement to all cominuscule types simultaneously without relying on type-specific computations. 

We will construct the isomorphism via the composition
\[
\mX_\Lie = \dRichard_{\wop,\wo} \xrightarrow{\varphi} \dunimP \xrightarrow{\pi} \cmX = \udP_k \backslash \udG \supset \mX_\can,
\]
where $\pi$ is the quotient map given by $u_-\mapsto \udP_k u_-$ and where $\varphi: \mX_\Lie \rightarrow \dunimP$ is the isomorphism obtained via restricting the composition of isomorphisms $\mX_\Lie\times\invdtorus \xrightarrow{\RisoZ} \decomps \xrightarrow{\sim}  \dunimP \times \invdtorus$ to the first factor via $\mX_\Lie \cong \mX_\Lie \times \{\mathrm{id}\}$. See \cite[Remark 3.1]{Spacek_LP_LG_models} for a discussion of the isomorphism $\RisoZ$ with our current conventions or \cite[Section 4]{Rietsch_Mirror_Construction} for the original definition of $\RisoZ$; the isomorphism $\decomps \xrightarrow{\sim}  \dunimP \times \invdtorus$ is discussed in Lemma \ref{lem:decomps_iso_dunimp_times_invdtorus}. 

Since $\varphi$ is an isomorphism, it suffices to show that $\pi$ is an isomorphism $\dunimP\to\mX_\can$ to allow us to conclude $\mX_\Lie\cong\mX_\can$. Proposition 8.5 of \cite{GLS_Kac_Moody_groups_and_cluster_algebras} implies that $\dunimP$ is affine while we know from the fact that $\mX_\can = \cmX\setminus D_{\ac}$ is the complement in $\cmX$ of an anticanonical (hence ample, since $\cmX$ is Fano) divisor that $\mX_\can$ is also affine. Thus, it suffices to prove at the level of coordinate rings that the induced map $\pi^*:\C[\mX_\can] \rightarrow \C[\dunimP]$ is an isomorphism. 

To see that $\pi^*$ is injective, we observe that the dimensions of $\dunimP$ and $\mX_\can$ are both  equal to the dimension of $\cmX$ as well as the fact that $\pi$ is an open map since it is the quotient by the action of $\udP$. Hence, $\pi(\dunimP)\cap\mX_\can$ is an open, dense subset of $\mX_\can$ so that $\pi: \dunimP \dashrightarrow\mX_\can$ is dominant and hence $\pi^*$ is injective. 

On the other hand, to see that $\pi^*$ is surjective, we first recall Proposition \ref{prop:coord_ring_dunimP_in_posets}, which gives a presentation of the coordinate ring of $\dunimP$ as 
\[
\C[\dunimP] = \C\bigl[\p_{\ideal}~\big|~\ideal\subideal\idealPPrime\bigr] \bigl[\phiGLS[j]^{-1}~\big|~j\in[n]\bigr]
\]
where we recall that $\idealPPrime\subideal\minposet$ is the order ideal corresponding to $\wPPrime\in\cosets$ the minimal coset representative of $\wop s_k\weylp$.

Each generator of this coordinate ring is a Pl\"ucker coordinate $p_{\ideal}$ which is a priori also a well-defined function on $\mX_\can$. Furthermore, we proved in Theorem \ref{thm: phi_cancellation} that each of the generalized minors $\phiGLS[j]$ is equal to the polynomial in Pl\"ucker coordinates $\cD_{\is}$ on $\mX_\can$. Furthermore, the zero set $D_j=\{\cD_{\is}=0\}$ is contained in $D_{\ac}$ which was removed from $\cmX$ to obtain $\mX_\can$. Thus, each $\phiGLS[j]$ is a well-defined function on $\mX_\can$ which is nonzero everywhere on $\mX_\can$ and hence invertible. Therefore $\pi^*$ is surjective, so we conclude that $\pi^*$ and hence $\pi$ is an isomorphism. Hence:

\begin{thm}
    \label{thm: mirror_space}
    The spaces $\mX_\Lie$ and $\mX_\can$ are isomorphic via the composition
    \[
    \mX_\Lie \xrightarrow{\varphi} \dunimP \xrightarrow{\pi} \mX_\can.
    \]
\end{thm}

This isomorphism justifies many of the computations we have performed up to this point, especially the restrictions of various functions to the torus $\opendunim$.
\begin{rem}\label{rem:equality_of_functions_on_dunimP}
Since the quotient map $\pi:\dunimP\rightarrow\mX_\can$ is an isomorphism, the dense torus $\opendunim \subset \dunimP$ maps to an isomorphic dense torus $\dP_k \backslash \opendunim\subset \mX_\can$ which can thus be endowed with the same toric coordinates $\{a_{\pb}\mid \pb\in\minposet\}$ as $\opendunim$. This allows for the simultaneous interpretations of the toric expressions obtained in Theorems \ref{thm:toric_monomial_phi} and \ref{thm:derivation_cancellation} on both the tori $\opendunim$ and $\dP_k\backslash \opendunim$.
\end{rem}

\addtocontents{toc}{\SkipTocEntry}
\subsection{The canonical superpotential}
We finally have all the results necessary to prove our main Theorem \ref{thm:main_theorem}. 

\begin{proof}[Proof of Thm. \ref{thm:main_theorem}]
In light of Theorem \ref{thm: mirror_space} and Remark \ref{rem:equality_of_functions_on_dunimP}, parts (i), (ii), (iii), and (iv) of Theorem \ref{thm:main_theorem} immediately follow from equations \ref{eq:phi_0}, \ref{eq:phi_k_and_p_minposet}, and Theorem \ref{thm: phi_cancellation}; Remark \ref{rem:anticanonical}; Theorem \ref{thm:derivation_cancellation}; and Proposition \ref{prop:quantum_term} respectively. Furthermore, using these tools, we will now establish part (v) as well.

We verified in Theorem \ref{thm: mirror_space} that $\mX_\can\cong\mX_\Lie$, so it remains to check that the superpotentials $\pot_\can$ and $\pot_\Lie$ agree under this isomorphism, i.e. that for any value of the formal quantum parameter $q$, the two potentials agree. As with many computations we have performed, we will check this statement by restricting the pullbacks $\varphi^*(\pot_\Lie)$ and $\pi^*(\pot_\can)$ of both superpotentials to the torus $\opendunim$, and thus comparing the restriction of our potential $\pi^*(\pot_\can)|_{\opendunim}$ to the restriction 
\[
\varphi^*(\pot_{\Lie})|_{\opendunim} = \tsum[_{i\in[n]}]\Bigl(\tsum[_{\pb\in\minposet:\, 
\indx(\pb)=i }]a_{\pb}\Bigr) + q\frac{\sum_{\iota:\idealPrime\hookrightarrow\minposet} a_{\iota(\idealPrime)}}{a_{\minposet}}
\]
determined in equation \eqref{eq:toric_expression_pot_Lie_2}.

First, for $\pot_0$ we have
\[
\Yboxdim{3pt}
\delta_k(\cD_0) = \delta_k(p_{\varnothing}) = p_{\yng(1)} 
\quad\Rightarrow\quad 
\pot_0|_{\opendunim} 
= \frac{p_{\yng(1)}}{p_{\varnothing}}\bigg|_{\opendunim} 
= \sum_{\indx(\pb) = k} a_{\pb},
\]
where the sum is over all elements $\pb\in\minposet$ with index $k$. Next, for any $\is\in [n]\setminus \{k\}$, we obtain from Theorems \ref{thm: phi_cancellation} and \ref{thm:derivation_cancellation} the equality
\[
\pot_{\is}|_{\opendunim} = \frac{\delta_{i_1}(\cD_{\is})}{\cD_{\is}}\bigg|_{\opendunim} = \sum_{\indx(\pb)=i_1} a_{\pb},
\]
where the sum is over all elements $\pb\in\minposet$ with index $i_1$. Since we never have $i_1=k$ for $\is\neq k$, we thus have so far that
\[
\pot_0|_{\opendunim} + \sum_{\is\in [n]\setminus \{k\}} \pot_{\is}|_{\opendunim} = \sum_{\indx(\pb)=k} a_{\pb} + \sum_{\is\in [n]\setminus \{k\}} \sum_{\indx(\pb)=i_1} a_{\pb} = \sum_{i\in [n]} \sum_{\indx(\pb)=i} a_{\pb},
\]
where the final equality follows from the fact that $\sigma_k$ is an involution on $[n]\setminus \{k\}$. Thus we obtain exactly the first term 
\[
\tsum[_{i\in[n]}]\Bigl(\tsum[_{\pb\in\minposet: \indx(\pb)=i }]a_{\pb}\Bigr)
\]
of the Lie-theoretic potential via the terms $\pot_0 + \sum_{\is\in [n]\setminus \{k\}} \pot_{\is}$. 

The remaining quantum term is exactly equal to $\pot_k$ for any value of the quantum parameter $q$ by Proposition \ref{prop:quantum_term}. This establishes the equality between $\pot_\Lie$ and $\pot_\can$ for any choice of the quantum parameter $q$, completing the proof of the isomorphism of Landau-Ginzburg models.
\end{proof}

%% file: appAexceptional.tex
\section{Cubic and quartic terms}
\label{app:exceptional}
While there are infinitely many minuscule homogeneous spaces $\cmX=\dP_k\backslash\dG$ that have $\is\in [n]$ such that $\cis\le 2$, only two homogeneous spaces allow $\is$ with $\cis > 2$. Namely, the Cayley plane $\OP^2=\dP_6\backslash\LGE_6^\AD$ has $\cis=3$ for $\is=4$, and the Freudenthal variety $\dP_7\backslash\LGE_7^\AD$ has $\cis=3$ for $\is\in\{3,5\}$ and $\cis=4$ for $\is=4$. In these four cases, Corollary \ref{cor:leading_plucker_for_phi_i} tells us that the toric expression of $\phiGLS[\is]$ on $\opendunim$ is a single term contained in a cubic resp.~quartic Pl\"ucker monomial. While the same techniques as used in Sections \ref{sec:denominators} and \ref{sec:numerators} apply to the cancellations of the corresponding Pl\"ucker polynomials $\cD_\is$, the fact that the posets $\sP_{\is}$ consists of triples resp.~quadruples of order ideals makes the arguments and notation quite a bit more technical. Since there are only four such cases, we will instead deal with them explicitly by giving the corresponding posets $\sP_\is$ and Pl\"ucker polynomials $\cD_\is$ and $\deis(\cD_\is)$. We will then argue in Appendix \ref{app:type_dependent_LGs} that the resulting potentials agree with the potentials in \cite{Spacek_Wang_exceptional}, confirming that these are LG models isomorphic to the Lie-theoretic model of \cite{Rietsch_Mirror_Construction}.

\addtocontents{toc}{\SkipTocEntry}
\subsection*{The Cayley plane}
For $\is=4$ in type $\LGE_6$ with $k=6$, we find the sequence $(i_j)=(4,5,3)$ as calculated in the proof of Lemma \ref{lem:special_weyl_elt_action}. Hence, we find
\[
\Yboxdim{7pt}
(\idealij[1],\idealij[2],\idealij[3]) = 
\left(
\raisebox{10.25pt}{\tiny\young(65)}
,~
\raisebox{3.5pt}{\tiny\young(65431,::243)}
,~
\raisebox{-10pt}{\tiny\young(65431,::243,:::542,:::654)}
\right)
\qwhere
\minposet = \raisebox{-10pt}{\tiny\young(65431,::243,:::542,:::65431)}.
\]
Moving boxes from $\idealij$ to $\idealij[j']$ for $j<j'$, we obtain the following poset $\sP_4$ on the left; and applying $\deis[4]$ yields in turn $\sP_4^+$ on the right:

\input{p4_e6_denom_num}


Taking the corresponding alternating sums yield
\ali{
\cD_4 & = \Yboxdim{3pt}
p_{\yng(2)}(p_{\raisebox{-2.7pt}{\young(~~~~~,::~~~)}}p_{\raisebox{-8.1pt}{\young(~~~~~,::~~~,:::~~~,:::~~~)}} - p_{\young(~~~~~,::~~)}p_{\young(~~~~~,::~~~,:::~~~,:::~~~~)} + p_{\young(~~~~,::~~)}p_{\young(~~~~~,::~~~,:::~~~,:::~~~~~)})
 + p_{\yng(1)}(-p_{\young(~~~~~,::~~~,:::~)}p_{\young(~~~~~,::~~~,:::~~~,:::~~~)} + p_{\young(~~~~~,::~~,:::~)}p_{\young(~~~~~,::~~~,:::~~~,:::~~~~)} - p_{\young(~~~~,::~~,:::~)}p_{\young(~~~~~,::~~~,:::~~~,:::~~~~~)})\\
&\quad \Yboxdim{3pt}
+ p_{\varnothing}(p_{\young(~~~~~,::~~~,:::~,:::~)}p_{\young(~~~~~,::~~~,:::~~~,:::~~~)} - p_{\young(~~~~~,::~~,:::~,:::~)}p_{\young(~~~~~,::~~~,:::~~~,:::~~~~)} + p_{\young(~~~~,::~~,:::~,:::~)}p_{\young(~~~~~,::~~~,:::~~~,:::~~~~~)}) \\
\deis[4](\cD_4) & = \Yboxdim{3pt}
p_{\young(~~~)}(p_{\raisebox{-2.7pt}{\young(~~~~~,::~~~)}}p_{\raisebox{-8.1pt}{\young(~~~~~,::~~~,:::~~~,:::~~~)}} - p_{\young(~~~~~,::~~)}p_{\young(~~~~~,::~~~,:::~~~,:::~~~~)} + p_{\young(~~~~,::~~)}p_{\young(~~~~~,::~~~,:::~~~,:::~~~~~)})
+ p_{\yng(1)}(-p_{\young(~~~~~,::~~~,:::~~)}p_{\young(~~~~~,::~~~,:::~~~,:::~~~)}) + p_{\varnothing}(p_{\young(~~~~~,::~~~,:::~~,:::~)}p_{\young(~~~~~,::~~~,:::~~~,:::~~~)})
}

\addtocontents{toc}{\SkipTocEntry}
\subsection*{The Freudenthal variety}
For $\is=3$ we have $\cis=3$ in type $\LGE_7$ with $k=7$, and we found in the proof of Lemma \ref{lem:special_weyl_elt_action} the sequence $(i_j)=(5,6,2)$. Hence, we have
\[
\Yboxdim{7pt}
(\idealij[1],\idealij[2],\idealij[3]) = 
\left(
\raisebox{13.5pt}{\tiny\young(76)}
,~
\raisebox{0pt}{\tiny\young(765431,:::243,::::542)}
,~
\raisebox{-13.5pt}{\tiny\young(765431,:::243,::::542,::::65431,::::76543)}
\right)
\qwhere
\minposet = \raisebox{-3em}{\tiny\young(765431,:::243,::::542,::::65431,::::76543,:::::::24,::::::::5,::::::::6,::::::::7)}
\]

\newcommand{\pt}{\raisebox{-.32pt}{\scalebox{.45}{$\blacksquare$}}}
In this case, the posets $\sP_3$ and $\deis[5](\sP_3)$ are isomorphic, so we draw only the latter and designate with $\young(\pt)$ the element added by the action of $\deis[5]$; removing it yields $\sP_3$:

\input{p3_e7_denom_num}

In fact, here we see two elements $(\ideal_1,\ideal_2,\ideal_3)\in\sP_3$ on which the derivation $\deis[5]$ acts nontrivially on two of the order ideals instead of just one, namely the third and fourth elements from the top moving right. 
This means that the derivation applied to the corresponding Pl\"ucker monomials yields two terms each; however, one pair of these terms cancels out, so we only indicated the pair which does not cancel.

We find the corresponding Pl\"ucker polynomials
\ali{
\cD_3 &= \Yboxdim{3pt} 
p_{\young(~~)}\Bigl(
p_{\young(~~~~~~,:::~~~,::::~~~)}p_{\young(~~~~~~,:::~~~,::::~~~,::::~~~~~,::::~~~~~)} - p_{\young(~~~~~~,:::~~~,::::~~)}p_{\young(~~~~~~,:::~~~,::::~~~,::::~~~~~,::::~~~~~,:::::::~)} + p_{\young(~~~~~~,:::~~~,::::~)}p_{\young(~~~~~~,:::~~~,::::~~~,::::~~~~~,::::~~~~~,:::::::~~)} - p_{\young(~~~~~~,:::~~~)}p_{\young(~~~~~~,:::~~~,::::~~~,::::~~~~~,::::~~~~~,:::::::~~,::::::::~)}
\Bigr)\\
&\quad\Yboxdim{3pt}
-p_{\young(~)}\Bigl(
p_{\young(~~~~~~,:::~~~,::::~~~,::::~)}p_{\young(~~~~~~,:::~~~,::::~~~,::::~~~~~,::::~~~~~)} - p_{\young(~~~~~~,:::~~~,::::~~,::::~)}p_{\young(~~~~~~,:::~~~,::::~~~,::::~~~~~,::::~~~~~,:::::::~)} + p_{\young(~~~~~~,:::~~~,::::~,::::~)}p_{\young(~~~~~~,:::~~~,::::~~~,::::~~~~~,::::~~~~~,:::::::~~)} - p_{\young(~~~~~~,:::~~~)}p_{\young(~~~~~~,:::~~~,::::~~~,::::~~~~~,::::~~~~~,:::::::~~,::::::::~,::::::::~)}
\Bigr)\\
&\quad+\Yboxdim{3pt}
p_\varnothing\Bigl(
p_{\young(~~~~~~,:::~~~,::::~~~,::::~,::::~)}p_{\young(~~~~~~,:::~~~,::::~~~,::::~~~~~,::::~~~~~)} - p_{\young(~~~~~~,:::~~~,::::~~,::::~,::::~)}p_{\young(~~~~~~,:::~~~,::::~~~,::::~~~~~,::::~~~~~,:::::::~)} + p_{\young(~~~~~~,:::~~~,::::~,::::~,::::~)}p_{\young(~~~~~~,:::~~~,::::~~~,::::~~~~~,::::~~~~~,:::::::~~)} - p_{\young(~~~~~~,:::~~~)}p_{\young(~~~~~~,:::~~~,::::~~~,::::~~~~~,::::~~~~~,:::::::~~,::::::::~,::::::::~,::::::::~)}
\Bigr)\\
\deis[5](\cD_3) &= \Yboxdim{3pt}
p_{\young(~~~)}\Bigl(
p_{\young(~~~~~~,:::~~~,::::~~~)}p_{\young(~~~~~~,:::~~~,::::~~~,::::~~~~~,::::~~~~~)} - p_{\young(~~~~~~,:::~~~,::::~~)}p_{\young(~~~~~~,:::~~~,::::~~~,::::~~~~~,::::~~~~~,:::::::~)} + p_{\young(~~~~~~,:::~~~,::::~)}p_{\young(~~~~~~,:::~~~,::::~~~,::::~~~~~,::::~~~~~,:::::::~~)} - p_{\young(~~~~~~,:::~~~)}p_{\young(~~~~~~,:::~~~,::::~~~,::::~~~~~,::::~~~~~,:::::::~~,::::::::~)}
\Bigr)\\
&\quad\Yboxdim{3pt}
-p_{\young(~)}\Bigl(
p_{\young(~~~~~~,:::~~~,::::~~~,::::~~)}p_{\young(~~~~~~,:::~~~,::::~~~,::::~~~~~,::::~~~~~)} - p_{\young(~~~~~~,:::~~~,::::~~,::::~~)}p_{\young(~~~~~~,:::~~~,::::~~~,::::~~~~~,::::~~~~~,:::::::~)} + p_{\young(~~~~~~,:::~~~,::::~,::::~)}p_{\young(~~~~~~,:::~~~,::::~~~,::::~~~~~,::::~~~~~,:::::::~~,::::::::~)} - p_{\young(~~~~~~,:::~~~,::::~)}p_{\young(~~~~~~,:::~~~,::::~~~,::::~~~~~,::::~~~~~,:::::::~~,::::::::~,::::::::~)}
\Bigr)\\
&\quad\Yboxdim{3pt}
+p_\varnothing\Bigl(
p_{\young(~~~~~~,:::~~~,::::~~~,::::~~,::::~)}p_{\young(~~~~~~,:::~~~,::::~~~,::::~~~~~,::::~~~~~)} - p_{\young(~~~~~~,:::~~~,::::~~,::::~~,::::~)}p_{\young(~~~~~~,:::~~~,::::~~~,::::~~~~~,::::~~~~~,:::::::~)} + p_{\young(~~~~~~,:::~~~,::::~,::::~,::::~)}p_{\young(~~~~~~,:::~~~,::::~~~,::::~~~~~,::::~~~~~,:::::::~~,::::::::~)} - p_{\young(~~~~~~,:::~~~,::::~)}p_{\young(~~~~~~,:::~~~,::::~~~,::::~~~~~,::::~~~~~,:::::::~~,::::::::~,::::::::~,::::::::~)}
\Bigr)
}

Moving to the next case, for $\is=5$, we have $\cis=3$ and the sequence $(i_j)=(3,2,6)$ in the proof of Lemma \ref{lem:special_weyl_elt_action}, and hence
\[
\Yboxdim{7pt}
(\idealij[1],\idealij[2],\idealij[3]) = 
\left(
\raisebox{10pt}{\tiny\young(7654,:::2)}~,
\raisebox{-9pt}{\tiny\young(765431,:::243,::::54,::::65,::::76)}~,
\raisebox{-22pt}{\tiny\young(765431,:::243,::::542,::::65431,::::76543,:::::::24,::::::::5)}
\right)
\qwhere
\minposet = \raisebox{-3em}{\tiny\young(765431,:::243,::::542,::::65431,::::76543,:::::::24,::::::::5,::::::::6,::::::::7)}
\]
We again have isomorphic posets $\sP_5$ and $\deis[3](\sP_5)$ so we show the latter as well as the elements that need to be removed to obtain $\sP_5$ by $\young(\pt)$:
\input{p5_e7_denom_num.tex}
These yield the Pl\"ucker polynomials
\ali{
\cD_5 &= \Yboxdim{3pt}\Bigl(
p_{\young(~~~~,:::~)}p_{\young(~~~~~~,:::~~~,::::~~,::::~~,::::~~)} - p_{\young(~~~~)}p_{\young(~~~~~~,:::~~~,::::~~~,::::~~,::::~~)} + p_{\young(~~~)}p_{\young(~~~~~~,:::~~~,::::~~~,::::~~~,::::~~)} - p_{\young(~~)}p_{\young(~~~~~~,:::~~~,::::~~~,::::~~~,::::~~~)}
\Bigr)p_{\young(~~~~~~,:::~~~,::::~~~,::::~~~~~,::::~~~~~,:::::::~~,::::::::~)}\\
&\quad\Yboxdim{3pt} 
-\Bigl(p_{\young(~~~~,:::~)}p_{\young(~~~~~~,:::~~~,::::~~,::::~~,::::~)} - p_{\young(~~~~)}p_{\young(~~~~~~,:::~~~,::::~~~,::::~~,::::~)} + p_{\young(~~~)}p_{\young(~~~~~~,:::~~~,::::~~~,::::~~~,::::~)} - p_{\young(~)}p_{\young(~~~~~~,:::~~~,::::~~~,::::~~~,::::~~~)}
\Bigr)p_{\young(~~~~~~,:::~~~,::::~~~,::::~~~~~,::::~~~~~,:::::::~~,::::::::~,::::::::~)}\\
&\quad \Yboxdim{3pt} 
+\Bigl(p_{\young(~~~~,:::~)}p_{\young(~~~~~~,:::~~~,::::~~,::::~~)} - p_{\young(~~~~)}p_{\young(~~~~~~,:::~~~,::::~~~,::::~~)} + p_{\young(~~~)}p_{\young(~~~~~~,:::~~~,::::~~~,::::~~~)} - p_{\varnothing}p_{\young(~~~~~~,:::~~~,::::~~~,::::~~~,::::~~~)}
\Bigr)p_{\young(~~~~~~,:::~~~,::::~~~,::::~~~~~,::::~~~~~,:::::::~~,::::::::~,::::::::~,::::::::~)}
\\
\deis[3](\cD_5) &= \Yboxdim{3pt} 
\Bigl(p_{\young(~~~~~,:::~)}p_{\young(~~~~~~,:::~~~,::::~~,::::~~,::::~~)} - p_{\young(~~~~~)}p_{\young(~~~~~~,:::~~~,::::~~~,::::~~,::::~~)} + p_{\young(~~~)}p_{\young(~~~~~~,:::~~~,::::~~~,::::~~~~,::::~~)} - p_{\young(~~)}p_{\young(~~~~~~,:::~~~,::::~~~,::::~~~~,::::~~~)}
\Bigr)p_{\young(~~~~~~,:::~~~,::::~~~,::::~~~~~,::::~~~~~,:::::::~~,::::::::~)}\\
&\quad \Yboxdim{3pt} 
-\Bigl(p_{\young(~~~~~,:::~)}p_{\young(~~~~~~,:::~~~,::::~~,::::~~,::::~)} - p_{\young(~~~~~)}p_{\young(~~~~~~,:::~~~,::::~~~,::::~~,::::~)} + p_{\young(~~~)}p_{\young(~~~~~~,:::~~~,::::~~~,::::~~~~,::::~)} - p_{\young(~)}p_{\young(~~~~~~,:::~~~,::::~~~,::::~~~~,::::~~~)}
\Bigr)p_{\young(~~~~~~,:::~~~,::::~~~,::::~~~~~,::::~~~~~,:::::::~~,::::::::~,::::::::~)}\\
&\quad \Yboxdim{3pt} 
+\Bigl(p_{\young(~~~~~,:::~)}p_{\young(~~~~~~,:::~~~,::::~~,::::~~)} - p_{\young(~~~~~)}p_{\young(~~~~~~,:::~~~,::::~~~,::::~~)} + p_{\young(~~~)}p_{\young(~~~~~~,:::~~~,::::~~~,::::~~~~)} - p_{\varnothing}p_{\young(~~~~~~,:::~~~,::::~~~,::::~~~~,::::~~~)}
\Bigr)p_{\young(~~~~~~,:::~~~,::::~~~,::::~~~~~,::::~~~~~,:::::::~~,::::::::~,::::::::~,::::::::~)}
}

Then for the final case, $\is=4$, we have $\cis=4$, $(i_j) = (4,5,3,5)$ and hence
\[
\Yboxdim{7pt}
(\idealij) = \left(
\raisebox{12pt}{\tiny\young(765)}~,
\raisebox{7pt}{\tiny\young(765431,:::243)}~,
\raisebox{-12pt}{\tiny\young(765431,:::243,::::542,::::654,::::765)}~,
\raisebox{-18pt}{\tiny\young(765431,:::243,::::542,::::65431,::::76543,:::::::24)}
\right)
\qwhere
\minposet = \raisebox{-3em}{\tiny\young(765431,:::243,::::542,::::65431,::::76543,:::::::24,::::::::5,::::::::6,::::::::7)}
\]
Hence, we obtain the following poset for $\sP_4$:

\input{p4_e7_alt}

Furthermore, we have $i_1=4$, so we further obtain the following poset for $\sP_4^+$. Note that similarly to the case of $\is=3$, the derivation $\delta_{4}$ acts nontrivially on multiple Pl\"ucker coordinates in certain terms.  This means that the derivation applied to these Pl\"ucker monomials yields two terms each; however, one pair of these terms cancels out, so we only indicate the pair which does not cancel. \newpage

\input{p4_e7_num}

%% file: p4_e6_denom_num.tex
\[
\Yboxdim{3pt}
\hfill\begin{tikzpicture}[scale=.45,style=very thick,baseline=0.25em]
    \node (0) at (0,4){\tiny$
    \raisebox{2.7pt}{\young(~~)} 
    ~~
    \raisebox{0pt}{\young(~~~~~,::~~~)} 
    ~~
    \raisebox{-5.4pt}{\young(~~~~~,::~~~,:::~~~,:::~~~)}
    $};
    \node (1) at (2,2){\tiny$
    \raisebox{2.7pt}{\yng(2)} 
    ~~
    \raisebox{0pt}{\young(~~~~~,::~~)} 
    ~~
    \raisebox{-5.4pt}{\young(~~~~~,::~~~,:::~~~,:::~~~~)}
    $};
    \node (2) at (4,0){\tiny$
    \raisebox{2.7pt}{\yng(2)} 
    ~~
    \raisebox{0pt}{\young(~~~~,::~~)} 
    ~~
    \raisebox{-5.4pt}{\young(~~~~~,::~~~,:::~~~,:::~~~~~)}
    $};
    \node (3) at (2,-2){\tiny$
    \raisebox{2.7pt}{\yng(1)} 
    ~~
    \raisebox{-2.7pt}{\young(~~~~,::~~,:::~)} 
    ~~
    \raisebox{-5.4pt}{\young(~~~~~,::~~~,:::~~~,:::~~~~~)}
    $};
    \node (4) at (0,-4){\tiny$
    \raisebox{-2.2pt}{$\varnothing$}
    ~~
    \raisebox{-5.4pt}{\young(~~~~,::~~,:::~,:::~)} 
    ~~
    \raisebox{-5.4pt}{\young(~~~~~,::~~~,:::~~~,:::~~~~~)}
    $};
    \node (5) at (-2,-2){\tiny$
    \raisebox{-2.2pt}{$\varnothing$}
    ~~
    \raisebox{-5.4pt}{\young(~~~~~,::~~,:::~,:::~)} 
    ~~
    \raisebox{-5.4pt}{\young(~~~~~,::~~~,:::~~~,:::~~~~)}
    $};
    \node (6) at (-4,0){\tiny$
    \raisebox{-2.2pt}{$\varnothing$}
    ~~
    \raisebox{-5.4pt}{\young(~~~~~,::~~~,:::~,:::~)} 
    ~~
    \raisebox{-5.4pt}{\young(~~~~~,::~~~,:::~~~,:::~~~)}
    $};
    \node (7) at (-2,2){\tiny$
    \raisebox{2.7pt}{\yng(1)} 
    ~~
    \raisebox{-2.7pt}{\young(~~~~~,::~~~,:::~)} 
    ~~
    \raisebox{-5.4pt}{\young(~~~~~,::~~~,:::~~~,:::~~~)}
    $};
    \node (8) at (0,0){\tiny$
    \raisebox{2.7pt}{\yng(1)} 
    ~~
    \raisebox{-2.7pt}{\young(~~~~~,::~~,:::~)} 
    ~~
    \raisebox{-5.4pt}{\young(~~~~~,::~~~,:::~~~,:::~~~~)}
    $};
    \draw (0)--(1)--(2)--(3)--(4)--(5)--(6)--(7)--(0);
    \draw (1)--(8)--(5);
    \draw (7)--(8)--(3);
\end{tikzpicture}
\qquad
\Yboxdim{3pt}
\hfill\begin{tikzpicture}[scale=.45,style=very thick,baseline=0.25em]
    \node (0) at (0,4){\tiny$
    \raisebox{2.7pt}{\young(~~~)} 
    ~~
    \raisebox{0pt}{\young(~~~~~,::~~~)} 
    ~~
    \raisebox{-5.4pt}{\young(~~~~~,::~~~,:::~~~,:::~~~)}
    $};
    \node (1) at (2,2){\tiny$
    \raisebox{2.7pt}{\young(~~~)} 
    ~~
    \raisebox{0pt}{\young(~~~~~,::~~)} 
    ~~
    \raisebox{-5.4pt}{\young(~~~~~,::~~~,:::~~~,:::~~~~)}
    $};
    \node (2) at (4,0){\tiny$
    \raisebox{2.7pt}{\young(~~~)} 
    ~~
    \raisebox{0pt}{\young(~~~~,::~~)} 
    ~~
    \raisebox{-5.4pt}{\young(~~~~~,::~~~,:::~~~,:::~~~~~)}
    $};
    \node (6) at (-4,0){\tiny$
    \raisebox{-2.2pt}{$\varnothing$}
    ~~
    \raisebox{-5.4pt}{\young(~~~~~,::~~~,:::~~,:::~)} 
    ~~
    \raisebox{-5.4pt}{\young(~~~~~,::~~~,:::~~~,:::~~~)}
    $};
    \node (7) at (-2,2){\tiny$
    \raisebox{2.7pt}{\yng(1)} 
    ~~
    \raisebox{-2.7pt}{\young(~~~~~,::~~~,:::~~)} 
    ~~
    \raisebox{-5.4pt}{\young(~~~~~,::~~~,:::~~~,:::~~~)}
    $};
    \draw (6)--(7)--(0)--(1)--(2);
\end{tikzpicture}
\]

%% file: p3_e7_denom_num.tex
\[
\Yboxdim{3pt}
\begin{tikzpicture}[scale=.6,style=very thick,baseline=0.25em]
    \node (0) at (0,4){\tiny\young(~~\pt)~~\young(~~~~~~,:::~~~,::::~~~)~~\young(~~~~~~,:::~~~,::::~~~,::::~~~~~,::::~~~~~)};
    \node (1) at (2,2){\tiny\young(~~\pt)~~\young(~~~~~~,:::~~~,::::~~)~~\young(~~~~~~,:::~~~,::::~~~,::::~~~~~,::::~~~~~,:::::::~)};
    \node (2) at (4,0){\tiny\young(~~\pt)~~\young(~~~~~~,:::~~~,::::~)~~\young(~~~~~~,:::~~~,::::~~~,::::~~~~~,::::~~~~~,:::::::~~)};
    \node (3) at (2,-2){\tiny\young(~)~~\young(~~~~~~,:::~~~,::::~,::::~)~~\young(~~~~~~,:::~~~,::::~~~,::::~~~~~,::::~~~~~,:::::::~~,::::::::\pt)};
    \node (4) at (0,-4){\tiny$\varnothing$~~\young(~~~~~~,:::~~~,::::~,::::~,::::~)~~\young(~~~~~~,:::~~~,::::~~~,::::~~~~~,::::~~~~~,:::::::~~,::::::::\pt)};
    \node (5) at (-2,-2){\tiny$\varnothing$~~\young(~~~~~~,:::~~~,::::~~,::::~\pt,::::~)~~\young(~~~~~~,:::~~~,::::~~~,::::~~~~~,::::~~~~~,:::::::~)};
    \node (6) at (-4,0){\tiny$\varnothing$~~\young(~~~~~~,:::~~~,::::~~~,::::~\pt,::::~)~~\young(~~~~~~,:::~~~,::::~~~,::::~~~~~,::::~~~~~)};
    \node (7) at (-2,2){\tiny\young(~)~~\young(~~~~~~,:::~~~,::::~~~,::::~\pt)~~\young(~~~~~~,:::~~~,::::~~~,::::~~~~~,::::~~~~~)};
    \node (8) at (0,0){\tiny\young(~)~~\young(~~~~~~,:::~~~,::::~~,::::~\pt)~~\young(~~~~~~,:::~~~,::::~~~,::::~~~~~,::::~~~~~,:::::::~)};
    \node (9) at (6,-2){\tiny\young(~~\pt)~~\young(~~~~~~,:::~~~,::::~)~~\young(~~~~~~,:::~~~,::::~~~,::::~~~~~,::::~~~~~,:::::::~~,::::::::~)};
    \node (10) at (4,-4){\tiny\young(~)~~\young(~~~~~~,:::~~~,::::\pt)\young(~~~~~~,:::~~~,::::~~~,::::~~~~~,::::~~~~~,:::::::~~,::::::::~,::::::::~)};
    \node (11) at (2,-6){\tiny$\varnothing$~~\young(~~~~~~,:::~~~,::::\pt)~~\young(~~~~~~,:::~~~,::::~~~,::::~~~~~,::::~~~~~,:::::::~~,::::::::~,::::::::~,::::::::~)};
    \draw (0)--(1)--(2)--(3)--(4)--(5)--(6)--(7)--(0);
    \draw (2)--(9)--(10)--(11)--(4);
    \draw (1)--(8)--(5);
    \draw (7)--(8)--(3)--(10);
\end{tikzpicture}
\]

%% file: p5_e7_denom_num.tex
\[
\Yboxdim{3pt}
\begin{tikzpicture}[scale=.7,style=very thick,baseline=0.25em]
    \node (0) at (0,4){\tiny\young(~~~~\pt,:::~)~~\young(~~~~~~,:::~~~,::::~~,::::~~,::::~~)~~\young(~~~~~~,:::~~~,::::~~~,::::~~~~~,::::~~~~~,:::::::~~,::::::::~)};
    \node (1) at (2,2){\tiny\young(~~~~\pt)~~\young(~~~~~~,:::~~~,::::~~~,::::~~,::::~~)~~\young(~~~~~~,:::~~~,::::~~~,::::~~~~~,::::~~~~~,:::::::~~,::::::::~)};
    \node (2) at (4,0){\tiny\young(~~~)~~\young(~~~~~~,:::~~~,::::~~~,::::~~~\pt,::::~~)~~\young(~~~~~~,:::~~~,::::~~~,::::~~~~~,::::~~~~~,:::::::~~,::::::::~)};
    \node (3) at (2,-2){\tiny\young(~~~)~~\young(~~~~~~,:::~~~,::::~~~,::::~~~\pt,::::~)~~\young(~~~~~~,:::~~~,::::~~~,::::~~~~~,::::~~~~~,:::::::~~,::::::::~,::::::::~)};
    \node (4) at (0,-4){\tiny\young(~~~)~~\young(~~~~~~,:::~~~,::::~~~,::::~~~\pt)~~\young(~~~~~~,:::~~~,::::~~~,::::~~~~~,::::~~~~~,:::::::~~,::::::::~,::::::::~,::::::::~)};
    \node (5) at (-2,-2){\tiny\young(~~~~\pt)~~\young(~~~~~~,:::~~~,::::~~~,::::~~)~~\young(~~~~~~,:::~~~,::::~~~,::::~~~~~,::::~~~~~,:::::::~~,::::::::~,::::::::~,::::::::~)};
    \node (6) at (-4,0){\tiny\young(~~~~\pt,:::~)~~\young(~~~~~~,:::~~~,::::~~,::::~~)~~\young(~~~~~~,:::~~~,::::~~~,::::~~~~~,::::~~~~~,:::::::~~,::::::::~,::::::::~,::::::::~)};
    \node (7) at (-2,2){\tiny\young(~~~~\pt,:::~)~~\young(~~~~~~,:::~~~,::::~~,::::~~,::::~)~~\young(~~~~~~,:::~~~,::::~~~,::::~~~~~,::::~~~~~,:::::::~~,::::::::~,::::::::~)};
    \node (8) at (0,0){\tiny\young(~~~~\pt)~~\young(~~~~~~,:::~~~,::::~~~,::::~~,::::~)~~\young(~~~~~~,:::~~~,::::~~~,::::~~~~~,::::~~~~~,:::::::~~,::::::::~,::::::::~)};
    \node (9) at (6,-2){\tiny\young(~~)~~\young(~~~~~~,:::~~~,::::~~~,::::~~~\pt,::::~~~)~~\young(~~~~~~,:::~~~,::::~~~,::::~~~~~,::::~~~~~,:::::::~~,::::::::~)};
    \node (10) at (4,-4){\tiny\young(~)~~\young(~~~~~~,:::~~~,::::~~~,::::~~~\pt,::::~~~)~~\young(~~~~~~,:::~~~,::::~~~,::::~~~~~,::::~~~~~,:::::::~~,::::::::~,::::::::~)};
    \node (11) at (2,-6){\tiny$\varnothing$~~\young(~~~~~~,:::~~~,::::~~~,::::~~~\pt,::::~~~)~~\young(~~~~~~,:::~~~,::::~~~,::::~~~~~,::::~~~~~,:::::::~~,::::::::~,::::::::~,::::::::~)};
    \draw (0)--(1)--(2)--(3)--(4)--(5)--(6)--(7)--(0);
    \draw (2)--(9)--(10)--(11)--(4);
    \draw (1)--(8)--(5);
    \draw (7)--(8)--(3)--(10);
\end{tikzpicture}
\]

%% file: p4_e7_alt.tex
\Yboxdim{3pt}
\begin{tikzpicture}[scale=1,style=very thick,baseline=0.25em]
\node (0) at (0,0) {\tiny\young(~~~)~~\young(~~~~~~,:::~~~)~~\young(~~~~~~,:::~~~,::::~~~,::::~~~,::::~~~)~~\young(~~~~~~,:::~~~,::::~~~,::::~~~~~,::::~~~~~,:::::::~~)};
\node (1) at (3.5,0) {\tiny\young(~~~)~~\young(~~~~~~,:::~~~)~~\young(~~~~~~,:::~~~,::::~~~,::::~~~,::::~~)~~\young(~~~~~~,:::~~~,::::~~~,::::~~~~~,::::~~~~~,:::::::~~,::::::::~)};
\node (2) at (7,0) {\tiny\young(~~~)~~\young(~~~~~~,:::~~~)~~\young(~~~~~~,:::~~~,::::~~~,::::~~~,::::~)~~\young(~~~~~~,:::~~~,::::~~~,::::~~~~~,::::~~~~~,:::::::~~,::::::::~,::::::::~)};
\node (3) at (10.5,0) {\tiny\young(~~~)~~\young(~~~~~~,:::~~~)~~\young(~~~~~~,:::~~~,::::~~~,::::~~~)~~\young(~~~~~~,:::~~~,::::~~~,::::~~~~~,::::~~~~~,:::::::~~,::::::::~,::::::::~,::::::::~)};
\node (4) at (1,-1.5) {\tiny\young(~~~)~~\young(~~~~~~,:::~~)~~\young(~~~~~~,:::~~~,::::~~~,::::~~~~,::::~~~)~~\young(~~~~~~,:::~~~,::::~~~,::::~~~~~,::::~~~~~,:::::::~~)};
\node (5) at (4.5,-1.5) {\tiny\young(~~~)~~\young(~~~~~~,:::~~)~~\young(~~~~~~,:::~~~,::::~~~,::::~~~~,::::~~)~~\young(~~~~~~,:::~~~,::::~~~,::::~~~~~,::::~~~~~,:::::::~~,::::::::~)};
\node (6) at (8,-1.5) {\tiny\young(~~~)~~\young(~~~~~~,:::~~)~~\young(~~~~~~,:::~~~,::::~~~,::::~~~~,::::~)~~\young(~~~~~~,:::~~~,::::~~~,::::~~~~~,::::~~~~~,:::::::~~,::::::::~,::::::::~)};
\node (7) at (11.5,-1.5) {\tiny\young(~~~)~~\young(~~~~~~,:::~~)~~\young(~~~~~~,:::~~~,::::~~~,::::~~~~)~~\young(~~~~~~,:::~~~,::::~~~,::::~~~~~,::::~~~~~,:::::::~~,::::::::~,::::::::~,::::::::~)};
\node (8) at (2,-3) {\tiny\young(~~~)~~\young(~~~~~,:::~~)~~\young(~~~~~~,:::~~~,::::~~~,::::~~~~~,::::~~~)~~\young(~~~~~~,:::~~~,::::~~~,::::~~~~~,::::~~~~~,:::::::~~)};
\node (9) at (5.5,-3) {\tiny\young(~~~)~~\young(~~~~~,:::~~)~~\young(~~~~~~,:::~~~,::::~~~,::::~~~~~,::::~~)~~\young(~~~~~~,:::~~~,::::~~~,::::~~~~~,::::~~~~~,:::::::~~,::::::::~)};
\node (10) at (9,-3) {\tiny\young(~~~)~~\young(~~~~~,:::~~)~~\young(~~~~~~,:::~~~,::::~~~,::::~~~~~,::::~)~~\young(~~~~~~,:::~~~,::::~~~,::::~~~~~,::::~~~~~,:::::::~~,::::::::~,::::::::~)};
\node (11) at (12.5,-3) {\tiny\young(~~~)~~\young(~~~~~,:::~~)~~\young(~~~~~~,:::~~~,::::~~~,::::~~~~~)~~\young(~~~~~~,:::~~~,::::~~~,::::~~~~~,::::~~~~~,:::::::~~,::::::::~,::::::::~,::::::::~)};
\draw 
(0)--(1)--(2)--(3)
(4)--(5)--(6)--(7)
(8)--(9)--(10)--(11)
(0)--(4)--(8)
(1)--(5)--(9)
(2)--(6)--(10)
(3)--(7)--(11);
\node (-1) at (-1.5,-3) {(1)};
\draw[dashed]
(4)--(-1);

\node (22) at ($(0)+(0,-6.00)$) {\tiny\young(~~)~~\young(~~~~~~,:::~~~,::::~)~~\young(~~~~~~,:::~~~,::::~~~,::::~~~,::::~~~)~~\young(~~~~~~,:::~~~,::::~~~,::::~~~~~,::::~~~~~,:::::::~~)};
\node (23) at ($(1)+(0,-6.00)$){\tiny\young(~~)~~\young(~~~~~~,:::~~~,::::~)~~\young(~~~~~~,:::~~~,::::~~~,::::~~~,::::~~)~~\young(~~~~~~,:::~~~,::::~~~,::::~~~~~,::::~~~~~,:::::::~~,::::::::~)};
\node (24) at ($(2)+(0,-6.00)$){\tiny\young(~~)~~\young(~~~~~~,:::~~~,::::~)~~\young(~~~~~~,:::~~~,::::~~~,::::~~~,::::~)~~\young(~~~~~~,:::~~~,::::~~~,::::~~~~~,::::~~~~~,:::::::~~,::::::::~,::::::::~)};
\node (25) at ($(3)+(0,-6.00)$){\tiny\young(~~)~~\young(~~~~~~,:::~~~,::::~)~~\young(~~~~~~,:::~~~,::::~~~,::::~~~)~~\young(~~~~~~,:::~~~,::::~~~,::::~~~~~,::::~~~~~,:::::::~~,::::::::~,::::::::~,::::::::~)};
\node (26) at ($(4)+(0,-6.00)$){\tiny\young(~~)~~\young(~~~~~~,:::~~,::::~)~~\young(~~~~~~,:::~~~,::::~~~,::::~~~~,::::~~~)~~\young(~~~~~~,:::~~~,::::~~~,::::~~~~~,::::~~~~~,:::::::~~)};
\node (27) at ($(5)+(0,-6.00)$){\tiny\young(~~)~~\young(~~~~~~,:::~~,::::~)~~\young(~~~~~~,:::~~~,::::~~~,::::~~~~,::::~~)~~\young(~~~~~~,:::~~~,::::~~~,::::~~~~~,::::~~~~~,:::::::~~,::::::::~)};
\node (28) at ($(6)+(0,-6.00)$){\tiny\young(~~)~~\young(~~~~~~,:::~~,::::~)~~\young(~~~~~~,:::~~~,::::~~~,::::~~~~,::::~)~~\young(~~~~~~,:::~~~,::::~~~,::::~~~~~,::::~~~~~,:::::::~~,::::::::~,::::::::~)};
\node (29) at ($(7)+(0,-6.00)$){\tiny\young(~~)~~\young(~~~~~~,:::~~,::::~)~~\young(~~~~~~,:::~~~,::::~~~,::::~~~~)~~\young(~~~~~~,:::~~~,::::~~~,::::~~~~~,::::~~~~~,:::::::~~,::::::::~,::::::::~,::::::::~)};
\node (30) at ($(8)+(0,-6.00)$){\tiny\young(~~)~~\young(~~~~~,:::~~,::::~)~~\young(~~~~~~,:::~~~,::::~~~,::::~~~~~,::::~~~)~~\young(~~~~~~,:::~~~,::::~~~,::::~~~~~,::::~~~~~,:::::::~~)};
\node (31) at ($(9)+(0,-6.00)$){\tiny\young(~~)~~\young(~~~~~,:::~~,::::~)~~\young(~~~~~~,:::~~~,::::~~~,::::~~~~~,::::~~)~~\young(~~~~~~,:::~~~,::::~~~,::::~~~~~,::::~~~~~,:::::::~~,::::::::~)};
\node (32) at ($(10)+(0,-6.00)$){\tiny\young(~~)~~\young(~~~~~,:::~~,::::~)~~\young(~~~~~~,:::~~~,::::~~~,::::~~~~~,::::~)~~\young(~~~~~~,:::~~~,::::~~~,::::~~~~~,::::~~~~~,:::::::~~,::::::::~,::::::::~)};
\node (33) at ($(11)+(0,-6.00)$){\tiny\young(~~)~~\young(~~~~~,:::~~,::::~)~~\young(~~~~~~,:::~~~,::::~~~,::::~~~~~)~~\young(~~~~~~,:::~~~,::::~~~,::::~~~~~,::::~~~~~,:::::::~~,::::::::~,::::::::~,::::::::~)};
\draw 
(22)--(23)--(24)--(25)
(26)--(27)--(28)--(29)
(30)--(31)--(32)--(33)
(22)--(26)--(30)
(23)--(27)--(31)
(24)--(28)--(32)
(25)--(29)--(33);
\foreach \x in {0,...,11}
\draw (\x)--({\the\numexpr\x+22\relax});
\node (-2) at (-1.5,-9) {(2)};
\draw[dashed]
(26)--(-2);

\node (44) at ($(22)+(0,-6.00)$) {\young(~)~~\young(~~~~~~,:::~~~,::::~,::::~)~~\young(~~~~~~,:::~~~,::::~~~,::::~~~,::::~~~)~~\young(~~~~~~,:::~~~,::::~~~,::::~~~~~,::::~~~~~,:::::::~~)};
\node (45) at ($(23)+(0,-6.00)$) {\young(~)~~\young(~~~~~~,:::~~~,::::~,::::~)~~\young(~~~~~~,:::~~~,::::~~~,::::~~~,::::~~)~~\young(~~~~~~,:::~~~,::::~~~,::::~~~~~,::::~~~~~,:::::::~~,::::::::~)};
\node (46) at ($(24)+(0,-6.00)$) {\young(~)~~\young(~~~~~~,:::~~~,::::~,::::~)~~\young(~~~~~~,:::~~~,::::~~~,::::~~~,::::~)~~\young(~~~~~~,:::~~~,::::~~~,::::~~~~~,::::~~~~~,:::::::~~,::::::::~,::::::::~)};
\node (47) at ($(25)+(0,-6.00)$) {\young(~)~~\young(~~~~~~,:::~~~,::::~,::::~)~~\young(~~~~~~,:::~~~,::::~~~,::::~~~)~~\young(~~~~~~,:::~~~,::::~~~,::::~~~~~,::::~~~~~,:::::::~~,::::::::~,::::::::~,::::::::~)};
\node (48) at ($(26)+(0,-6.00)$) {\young(~)~~\young(~~~~~~,:::~~,::::~,::::~)~~\young(~~~~~~,:::~~~,::::~~~,::::~~~~,::::~~~)~~\young(~~~~~~,:::~~~,::::~~~,::::~~~~~,::::~~~~~,:::::::~~)};
\node (49) at ($(27)+(0,-6.00)$) {\young(~)~~\young(~~~~~~,:::~~,::::~,::::~)~~\young(~~~~~~,:::~~~,::::~~~,::::~~~~,::::~~)~~\young(~~~~~~,:::~~~,::::~~~,::::~~~~~,::::~~~~~,:::::::~~,::::::::~)};
\node (50) at ($(28)+(0,-6.00)$) {\young(~)~~\young(~~~~~~,:::~~,::::~,::::~)~~\young(~~~~~~,:::~~~,::::~~~,::::~~~~,::::~)~~\young(~~~~~~,:::~~~,::::~~~,::::~~~~~,::::~~~~~,:::::::~~,::::::::~,::::::::~)};
\node (51) at ($(29)+(0,-6.00)$) {\young(~)~~\young(~~~~~~,:::~~,::::~,::::~)~~\young(~~~~~~,:::~~~,::::~~~,::::~~~~)~~\young(~~~~~~,:::~~~,::::~~~,::::~~~~~,::::~~~~~,:::::::~~,::::::::~,::::::::~,::::::::~)};
\node (52) at ($(30)+(0,-6.00)$) {\young(~)~~\young(~~~~~,:::~~,::::~,::::~)~~\young(~~~~~~,:::~~~,::::~~~,::::~~~~~,::::~~~)~~\young(~~~~~~,:::~~~,::::~~~,::::~~~~~,::::~~~~~,:::::::~~)};
\node (53) at ($(31)+(0,-6.00)$) {\young(~)~~\young(~~~~~,:::~~,::::~,::::~)~~\young(~~~~~~,:::~~~,::::~~~,::::~~~~~,::::~~)~~\young(~~~~~~,:::~~~,::::~~~,::::~~~~~,::::~~~~~,:::::::~~,::::::::~)};
\node (54) at ($(32)+(0,-6.00)$) {\young(~)~~\young(~~~~~,:::~~,::::~,::::~)~~\young(~~~~~~,:::~~~,::::~~~,::::~~~~~,::::~)~~\young(~~~~~~,:::~~~,::::~~~,::::~~~~~,::::~~~~~,:::::::~~,::::::::~,::::::::~)};
\node (55) at ($(33)+(0,-6.00)$) {\young(~)~~\young(~~~~~,:::~~,::::~,::::~)~~\young(~~~~~~,:::~~~,::::~~~,::::~~~~~)~~\young(~~~~~~,:::~~~,::::~~~,::::~~~~~,::::~~~~~,:::::::~~,::::::::~,::::::::~,::::::::~)};
\draw 
(44)--(45)--(46)--(47)
(48)--(49)--(50)--(51)
(52)--(53)--(54)--(55)
(44)--(48)--(52)
(45)--(49)--(53)
(46)--(50)--(54)
(47)--(51)--(55);
\foreach \x in {22,...,33}
\draw (\x)--({\the\numexpr\x+22\relax});
\node (-3) at (-1.5,-15) {(3)};
\draw[dashed]
(48)--(-3);

\node (66) at ($(44)+(0,-6.00)$) {$\varnothing$~~\young(~~~~~~,:::~~~,::::~,::::~,::::~)~~\young(~~~~~~,:::~~~,::::~~~,::::~~~,::::~~~)~~\young(~~~~~~,:::~~~,::::~~~,::::~~~~~,::::~~~~~,:::::::~~)};
\node (67) at ($(45)+(0,-6.00)$) {$\varnothing$~~\young(~~~~~~,:::~~~,::::~,::::~,::::~)~~\young(~~~~~~,:::~~~,::::~~~,::::~~~,::::~~)~~\young(~~~~~~,:::~~~,::::~~~,::::~~~~~,::::~~~~~,:::::::~~,::::::::~)};
\node (68) at ($(46)+(0,-6.00)$) {$\varnothing$~~\young(~~~~~~,:::~~~,::::~,::::~,::::~)~~\young(~~~~~~,:::~~~,::::~~~,::::~~~,::::~)~~\young(~~~~~~,:::~~~,::::~~~,::::~~~~~,::::~~~~~,:::::::~~,::::::::~,::::::::~)};
\node (69) at ($(47)+(0,-6.00)$) {$\varnothing$~~\young(~~~~~~,:::~~~,::::~,::::~,::::~)~~\young(~~~~~~,:::~~~,::::~~~,::::~~~)~~\young(~~~~~~,:::~~~,::::~~~,::::~~~~~,::::~~~~~,:::::::~~,::::::::~,::::::::~,::::::::~)};
\node (70) at ($(48)+(0,-6.00)$) {$\varnothing$~~\young(~~~~~~,:::~~,::::~,::::~,::::~)~~\young(~~~~~~,:::~~~,::::~~~,::::~~~~,::::~~~)~~\young(~~~~~~,:::~~~,::::~~~,::::~~~~~,::::~~~~~,:::::::~~)};
\node (71) at ($(49)+(0,-6.00)$) {$\varnothing$~~\young(~~~~~~,:::~~,::::~,::::~,::::~)~~\young(~~~~~~,:::~~~,::::~~~,::::~~~~,::::~~)~~\young(~~~~~~,:::~~~,::::~~~,::::~~~~~,::::~~~~~,:::::::~~,::::::::~)};
\node (72) at ($(50)+(0,-6.00)$) {$\varnothing$~~\young(~~~~~~,:::~~,::::~,::::~,::::~)~~\young(~~~~~~,:::~~~,::::~~~,::::~~~~,::::~)~~\young(~~~~~~,:::~~~,::::~~~,::::~~~~~,::::~~~~~,:::::::~~,::::::::~,::::::::~)};
\node (73) at ($(51)+(0,-6.00)$) {$\varnothing$~~\young(~~~~~~,:::~~,::::~,::::~,::::~)~~\young(~~~~~~,:::~~~,::::~~~,::::~~~~)~~\young(~~~~~~,:::~~~,::::~~~,::::~~~~~,::::~~~~~,:::::::~~,::::::::~,::::::::~,::::::::~)};
\node (74) at ($(52)+(0,-6.00)$) {$\varnothing$~~\young(~~~~~,:::~~,::::~,::::~,::::~)~~\young(~~~~~~,:::~~~,::::~~~,::::~~~~~,::::~~~)~~\young(~~~~~~,:::~~~,::::~~~,::::~~~~~,::::~~~~~,:::::::~~)};
\node (75) at ($(53)+(0,-6.00)$) {$\varnothing$~~\young(~~~~~,:::~~,::::~,::::~,::::~)~~\young(~~~~~~,:::~~~,::::~~~,::::~~~~~,::::~~)~~\young(~~~~~~,:::~~~,::::~~~,::::~~~~~,::::~~~~~,:::::::~~,::::::::~)};
\node (76) at ($(54)+(0,-6.00)$) {$\varnothing$~~\young(~~~~~,:::~~,::::~,::::~,::::~)~~\young(~~~~~~,:::~~~,::::~~~,::::~~~~~,::::~)~~\young(~~~~~~,:::~~~,::::~~~,::::~~~~~,::::~~~~~,:::::::~~,::::::::~,::::::::~)};
\node (77) at ($(55)+(0,-6.00)$) {$\varnothing$~~\young(~~~~~,:::~~,::::~,::::~,::::~)~~\young(~~~~~~,:::~~~,::::~~~,::::~~~~~)~~\young(~~~~~~,:::~~~,::::~~~,::::~~~~~,::::~~~~~,:::::::~~,::::::::~,::::::::~,::::::::~)};
\draw
(66)--(67)--(68)--(69)
(70)--(71)--(72)--(73)
(74)--(75)--(76)--(77)
(66)--(70)--(74)
(67)--(71)--(75)
(68)--(72)--(76)
(69)--(73)--(77);
\foreach \x in {44,...,55}
\draw (\x)--({\the\numexpr\x+22\relax});
\node (-4) at (-1.5,-21) {(4)};
\draw[dashed]
(70)--(-4);
\end{tikzpicture}

\begin{tikzpicture}[scale=2.25,style=very thick,baseline=0.25em]
\node (12) at (0,1) {\tiny\young(~~~)~~\young(~~~~~~,:::~)~~\young(~~~~~~,:::~~~,::::~~~,::::~~~~,::::~~~~)~~\young(~~~~~~,:::~~~,::::~~~,::::~~~~~,::::~~~~~,:::::::~~)};
\node (13) at (.5,0) {\tiny\young(~~~)~~\young(~~~~~~)~~\young(~~~~~~,:::~~~,::::~~~,::::~~~~,::::~~~~,:::::::~)~~\young(~~~~~~,:::~~~,::::~~~,::::~~~~~,::::~~~~~,:::::::~~)}; 
\node (14) at (1,1.5) {\tiny\young(~~~)~~\young(~~~~~,:::~)~~\young(~~~~~~,:::~~~,::::~~~,::::~~~~~,::::~~~~)~~\young(~~~~~~,:::~~~,::::~~~,::::~~~~~,::::~~~~~,:::::::~~)};
\node (15) at (1.5,0.5) {\tiny\young(~~~)~~\young(~~~~~)~~\young(~~~~~~,:::~~~,::::~~~,::::~~~~~,::::~~~~,:::::::~)~~\young(~~~~~~,:::~~~,::::~~~,::::~~~~~,::::~~~~~,:::::::~~)}; 
\node (16) at (2,2) {\tiny\young(~~~)~~\young(~~~~,:::~)~~\young(~~~~~~,:::~~~,::::~~~,::::~~~~~,::::~~~~~)~~\young(~~~~~~,:::~~~,::::~~~,::::~~~~~,::::~~~~~,:::::::~~)};
\node (17) at (2.5,1) {\tiny\young(~~~)~~\young(~~~~)~~\young(~~~~~~,:::~~~,::::~~~,::::~~~~~,::::~~~~~,:::::::~)~~\young(~~~~~~,:::~~~,::::~~~,::::~~~~~,::::~~~~~,:::::::~~)};
\node (18) at (3.5,1.5) {\tiny\young(~~~)~~\young(~~~)~~\young(~~~~~~,:::~~~,::::~~~,::::~~~~~,::::~~~~~,:::::::~~)~~\young(~~~~~~,:::~~~,::::~~~,::::~~~~~,::::~~~~~,:::::::~~)};
\node (19) at (4,1) {\tiny\young(~~~)~~\young(~~)~~\young(~~~~~~,:::~~~,::::~~~,::::~~~~~,::::~~~~~,:::::::~~,::::::::~)~~\young(~~~~~~,:::~~~,::::~~~,::::~~~~~,::::~~~~~,:::::::~~)};
\node (20) at (5,1.5) {\tiny\young(~~~)~~\young(~)~~\young(~~~~~~,:::~~~,::::~~~,::::~~~~~,::::~~~~~,:::::::~~,::::::::~,::::::::~)~~\young(~~~~~~,:::~~~,::::~~~,::::~~~~~,::::~~~~~,:::::::~~)};
\node (21) at (5.5,1) {\tiny\young(~~~)~~$\varnothing$~~\young(~~~~~~,:::~~~,::::~~~,::::~~~~~,::::~~~~~,:::::::~~,::::::::~,::::::::~,::::::::~)~~\young(~~~~~~,:::~~~,::::~~~,::::~~~~~,::::~~~~~,:::::::~~)};
\draw 
(12)--(14)--(16)--(17)--(18)--(19)--(20)--(21)
(12)--(13)--(15)--(17)
(14)--(15);
\node (-1) at (-.5,1.5) {(1)};
\draw[dashed] (12)--(-1);

\node(34) at ($(12)+(0,-2.5)$) {\young(~~)~~\young(~~~~~~,:::~)~~\young(~~~~~~,:::~~~,::::~~~,::::~~~~,::::~~~~)~~\young(~~~~~~,:::~~~,::::~~~,::::~~~~~,::::~~~~~,:::::::~~,::::::::~)};
\node (35) at ($(13)+(0,-2.5)$) 
{\young(~~)~~\young(~~~~~~)~~\young(~~~~~~,:::~~~,::::~~~,::::~~~~,::::~~~~,:::::::~)~~\young(~~~~~~,:::~~~,::::~~~,::::~~~~~,::::~~~~~,:::::::~~,::::::::~)};
\node (36) at ($(14)+(0,-2.5)$) {\young(~~)~~\young(~~~~~,:::~)~~\young(~~~~~~,:::~~~,::::~~~,::::~~~~~,::::~~~~)~~\young(~~~~~~,:::~~~,::::~~~,::::~~~~~,::::~~~~~,:::::::~~,::::::::~)};
\node (37) at ($(15)+(0,-2.5)$) {\young(~~)~~\young(~~~~~)~~\young(~~~~~~,:::~~~,::::~~~,::::~~~~~,::::~~~~,:::::::~)~~\young(~~~~~~,:::~~~,::::~~~,::::~~~~~,::::~~~~~,:::::::~~,::::::::~)};
\node (38) at ($(16)+(0,-2.5)$) {\young(~~)~~\young(~~~~,:::~)~~\young(~~~~~~,:::~~~,::::~~~,::::~~~~~,::::~~~~~)~~\young(~~~~~~,:::~~~,::::~~~,::::~~~~~,::::~~~~~,:::::::~~,::::::::~)};
\node (39) at ($(17)+(0,-2.5)$) {\young(~~)~~\young(~~~~)~~\young(~~~~~~,:::~~~,::::~~~,::::~~~~~,::::~~~~~,:::::::~)~~\young(~~~~~~,:::~~~,::::~~~,::::~~~~~,::::~~~~~,:::::::~~,::::::::~)};
\node (40) at ($(18)+(0,-2.5)$) {\young(~~)~~\young(~~~)~~\young(~~~~~~,:::~~~,::::~~~,::::~~~~~,::::~~~~~,:::::::~~)~~\young(~~~~~~,:::~~~,::::~~~,::::~~~~~,::::~~~~~,:::::::~~,::::::::~)};
\node (41) at ($(19)+(0,-2.5)$) {\young(~~)~~\young(~~)~~\young(~~~~~~,:::~~~,::::~~~,::::~~~~~,::::~~~~~,:::::::~~,::::::::~)~~\young(~~~~~~,:::~~~,::::~~~,::::~~~~~,::::~~~~~,:::::::~~,::::::::~)};
\node (42) at ($(20)+(0,-2.5)$) {\young(~~)~~\young(~)~~\young(~~~~~~,:::~~~,::::~~~,::::~~~~~,::::~~~~~,:::::::~~,::::::::~,::::::::~)~~\young(~~~~~~,:::~~~,::::~~~,::::~~~~~,::::~~~~~,:::::::~~,::::::::~)};
\node (43) at ($(21)+(0,-2.5)$) {\young(~~)~~$\varnothing$~~\young(~~~~~~,:::~~~,::::~~~,::::~~~~~,::::~~~~~,:::::::~~,::::::::~,::::::::~,::::::::~)~~\young(~~~~~~,:::~~~,::::~~~,::::~~~~~,::::~~~~~,:::::::~~,::::::::~)};

\draw 
(34)--(36)--(38)--(39)--(40)--(41)--(42)--(43)
(34)--(35)--(37)--(39)
(36)--(37);
\node (-2) at (-.5,-1) {(2)};
\draw[dashed] (34)--(-2);
\foreach \x in {12,...,21}
\draw (\x)--({\the\numexpr\x+22\relax});

\node (56) at ($(34)+(0,-2.5)$) {\young(~)~~\young(~~~~~~,:::~)~~\young(~~~~~~,:::~~~,::::~~~,::::~~~~,::::~~~~)~~\young(~~~~~~,:::~~~,::::~~~,::::~~~~~,::::~~~~~,:::::::~~,::::::::~,::::::::~)};
\node (57) at ($(35)+(0,-2.5)$)
{\young(~)~~\young(~~~~~~)~~\young(~~~~~~,:::~~~,::::~~~,::::~~~~,::::~~~~,:::::::~)~~\young(~~~~~~,:::~~~,::::~~~,::::~~~~~,::::~~~~~,:::::::~~,::::::::~,::::::::~)};
\node (58) at ($(36)+(0,-2.5)$) {\young(~)~~\young(~~~~~,:::~)~~\young(~~~~~~,:::~~~,::::~~~,::::~~~~~,::::~~~~)~~\young(~~~~~~,:::~~~,::::~~~,::::~~~~~,::::~~~~~,:::::::~~,::::::::~,::::::::~)};
\node (59) at ($(37)+(0,-2.5)$) {\young(~)~~\young(~~~~~)~~\young(~~~~~~,:::~~~,::::~~~,::::~~~~~,::::~~~~,:::::::~)~~\young(~~~~~~,:::~~~,::::~~~,::::~~~~~,::::~~~~~,:::::::~~,::::::::~,::::::::~)};
\node (60) at ($(38)+(0,-2.5)$) {\young(~)~~\young(~~~~,:::~)~~\young(~~~~~~,:::~~~,::::~~~,::::~~~~~,::::~~~~~)~~\young(~~~~~~,:::~~~,::::~~~,::::~~~~~,::::~~~~~,:::::::~~,::::::::~,::::::::~)};
\node (61) at ($(39)+(0,-2.5)$) {\young(~)~~\young(~~~~)~~\young(~~~~~~,:::~~~,::::~~~,::::~~~~~,::::~~~~~,:::::::~)~~\young(~~~~~~,:::~~~,::::~~~,::::~~~~~,::::~~~~~,:::::::~~,::::::::~,::::::::~)};
\node (62) at ($(40)+(0,-2.5)$) {\young(~)~~\young(~~~)~~\young(~~~~~~,:::~~~,::::~~~,::::~~~~~,::::~~~~~,:::::::~~)~~\young(~~~~~~,:::~~~,::::~~~,::::~~~~~,::::~~~~~,:::::::~~,::::::::~,::::::::~)};
\node (63) at ($(41)+(0,-2.5)$) {\young(~)~~\young(~~)~~\young(~~~~~~,:::~~~,::::~~~,::::~~~~~,::::~~~~~,:::::::~~,::::::::~)~~\young(~~~~~~,:::~~~,::::~~~,::::~~~~~,::::~~~~~,:::::::~~,::::::::~,::::::::~)};
\node (64) at ($(42)+(0,-2.5)$) {\young(~)~~\young(~)~~\young(~~~~~~,:::~~~,::::~~~,::::~~~~~,::::~~~~~,:::::::~~,::::::::~,::::::::~)~~\young(~~~~~~,:::~~~,::::~~~,::::~~~~~,::::~~~~~,:::::::~~,::::::::~,::::::::~)};
\node (65) at ($(43)+(0,-2.5)$) {\young(~)~~$\varnothing$~~\young(~~~~~~,:::~~~,::::~~~,::::~~~~~,::::~~~~~,:::::::~~,::::::::~,::::::::~,::::::::~)~~\young(~~~~~~,:::~~~,::::~~~,::::~~~~~,::::~~~~~,:::::::~~,::::::::~,::::::::~)};
\draw 
(56)--(58)--(60)--(61)--(62)--(63)--(64)--(65)
(56)--(57)--(59)--(61)
(58)--(59);
\node (-3) at (-.5,-3.5) {(3)};
\draw[dashed] (56)--(-3);
\foreach \x in {34,...,43}
\draw (\x)--({\the\numexpr\x+22\relax});

\node (78) at ($(56)+(0,-2.5)$) {$\varnothing$~~\young(~~~~~~,:::~)~~\young(~~~~~~,:::~~~,::::~~~,::::~~~~,::::~~~~)~~\young(~~~~~~,:::~~~,::::~~~,::::~~~~~,::::~~~~~,:::::::~~,::::::::~,::::::::~,::::::::~)};
\node (79) at ($(57)+(0,-2.5)$) {$\varnothing$~~\young(~~~~~~)~~\young(~~~~~~,:::~~~,::::~~~,::::~~~~,::::~~~~,:::::::~)~~\young(~~~~~~,:::~~~,::::~~~,::::~~~~~,::::~~~~~,:::::::~~,::::::::~,::::::::~,::::::::~)};

\node (80) at ($(58)+(0,-2.5)$) {$\varnothing$~~\young(~~~~~,:::~)~~\young(~~~~~~,:::~~~,::::~~~,::::~~~~~,::::~~~~)~~\young(~~~~~~,:::~~~,::::~~~,::::~~~~~,::::~~~~~,:::::::~~,::::::::~,::::::::~,::::::::~)};
\node (81) at ($(59)+(0,-2.5)$) {$\varnothing$~~\young(~~~~~)~~\young(~~~~~~,:::~~~,::::~~~,::::~~~~~,::::~~~~,:::::::~)~~\young(~~~~~~,:::~~~,::::~~~,::::~~~~~,::::~~~~~,:::::::~~,::::::::~,::::::::~,::::::::~)};
\node (82) at ($(60)+(0,-2.5)$) {$\varnothing$~~\young(~~~~,:::~)~~\young(~~~~~~,:::~~~,::::~~~,::::~~~~~,::::~~~~~)~~\young(~~~~~~,:::~~~,::::~~~,::::~~~~~,::::~~~~~,:::::::~~,::::::::~,::::::::~,::::::::~)};
\node (83) at ($(61)+(0,-2.5)$) {$\varnothing$~~\young(~~~~)~~\young(~~~~~~,:::~~~,::::~~~,::::~~~~~,::::~~~~~,:::::::~)~~\young(~~~~~~,:::~~~,::::~~~,::::~~~~~,::::~~~~~,:::::::~~,::::::::~,::::::::~,::::::::~)};
\node (84) at ($(62)+(0,-2.5)$) {$\varnothing$~~\young(~~~)~~\young(~~~~~~,:::~~~,::::~~~,::::~~~~~,::::~~~~~,:::::::~~)~~\young(~~~~~~,:::~~~,::::~~~,::::~~~~~,::::~~~~~,:::::::~~,::::::::~,::::::::~,::::::::~)};
\node (85) at ($(63)+(0,-2.5)$) {$\varnothing$~~\young(~~)~~\young(~~~~~~,:::~~~,::::~~~,::::~~~~~,::::~~~~~,:::::::~~,::::::::~)~~\young(~~~~~~,:::~~~,::::~~~,::::~~~~~,::::~~~~~,:::::::~~,::::::::~,::::::::~,::::::::~)};
\node (86) at ($(64)+(0,-2.5)$) {$\varnothing$~~\young(~)~~\young(~~~~~~,:::~~~,::::~~~,::::~~~~~,::::~~~~~,:::::::~~,::::::::~,::::::::~)~~\young(~~~~~~,:::~~~,::::~~~,::::~~~~~,::::~~~~~,:::::::~~,::::::::~,::::::::~,::::::::~)};
\node (87) at ($(65)+(0,-2.5)$) {$\varnothing~~\varnothing$~~\young(~~~~~~,:::~~~,::::~~~,::::~~~~~,::::~~~~~,:::::::~~,::::::::~,::::::::~,::::::::~)~~\young(~~~~~~,:::~~~,::::~~~,::::~~~~~,::::~~~~~,:::::::~~,::::::::~,::::::::~,::::::::~)};
\draw 
(78)--(80)--(82)--(83)--(84)--(85)--(86)--(87)
(78)--(79)--(81)--(83)
(80)--(81);
\node (-4) at (-.5,-6) {(4)};
\draw[dashed] (78)--(-4);
\foreach \x in {56,...,65}
\draw (\x)--({\the\numexpr\x+22\relax});

\end{tikzpicture}

\ali{
\cD_4 &= 
p_{\young(~~~)}
(p_{\young(~~~~~~,:::~~~)}
(p_{\young(~~~~~~,:::~~~,::::~~~,::::~~~,::::~~~)}p_{\young(~~~~~~,:::~~~,::::~~~,::::~~~~~,::::~~~~~,:::::::~~)}
-p_{\young(~~~~~~,:::~~~,::::~~~,::::~~~,::::~~)}p_{\young(~~~~~~,:::~~~,::::~~~,::::~~~~~,::::~~~~~,:::::::~~,::::::::~)}
+p_{\young(~~~~~~,:::~~~,::::~~~,::::~~~,::::~)}p_{\young(~~~~~~,:::~~~,::::~~~,::::~~~~~,::::~~~~~,:::::::~~,::::::::~,::::::::~)}
-p_{\young(~~~~~~,:::~~~,::::~~~,::::~~~)}p_{\young(~~~~~~,:::~~~,::::~~~,::::~~~~~,::::~~~~~,:::::::~~,::::::::~,::::::::~,::::::::~)})\\
&\quad+p_{\young(~~~~~~,:::~~)}
(-p_{\young(~~~~~~,:::~~~,::::~~~,::::~~~~,::::~~~)}p_{\young(~~~~~~,:::~~~,::::~~~,::::~~~~~,::::~~~~~,:::::::~~)}
+p_{\young(~~~~~~,:::~~~,::::~~~,::::~~~~,::::~~)}p_{\young(~~~~~~,:::~~~,::::~~~,::::~~~~~,::::~~~~~,:::::::~~,::::::::~)}
-p_{\young(~~~~~~,:::~~~,::::~~~,::::~~~~,::::~)}p_{\young(~~~~~~,:::~~~,::::~~~,::::~~~~~,::::~~~~~,:::::::~~,::::::::~,::::::::~)}
+p_{\young(~~~~~~,:::~~~,::::~~~,::::~~~~)}p_{\young(~~~~~~,:::~~~,::::~~~,::::~~~~~,::::~~~~~,:::::::~~,::::::::~,::::::::~,::::::::~)})\\
&\quad+p_{\young(~~~~~,:::~~)}
(p_{\young(~~~~~~,:::~~~,::::~~~,::::~~~~~,::::~~~)}p_{\young(~~~~~~,:::~~~,::::~~~,::::~~~~~,::::~~~~~,:::::::~~)}
-p_{\young(~~~~~~,:::~~~,::::~~~,::::~~~~~,::::~~)}p_{\young(~~~~~~,:::~~~,::::~~~,::::~~~~~,::::~~~~~,:::::::~~,::::::::~)}
+p_{\young(~~~~~~,:::~~~,::::~~~,::::~~~~~,::::~)}p_{\young(~~~~~~,:::~~~,::::~~~,::::~~~~~,::::~~~~~,:::::::~~,::::::::~,::::::::~)}
-p_{\young(~~~~~~,:::~~~,::::~~~,::::~~~~~)}p_{\young(~~~~~~,:::~~~,::::~~~,::::~~~~~,::::~~~~~,:::::::~~,::::::::~,::::::::~,::::::::~)})\\
&\quad+(p_{\young(~~~~~~,:::~)}p_{\young(~~~~~~,:::~~~,::::~~~,::::~~~~,::::~~~~)}
-p_{\young(~~~~~,:::~)}p_{\young(~~~~~~,:::~~~,::::~~~,::::~~~~~,::::~~~~)}
-p_{\young(~~~~~~)}p_{\young(~~~~~~,:::~~~,::::~~~,::::~~~~,::::~~~~,:::::::~)}
+p_{\young(~~~~~)}p_{\young(~~~~~~,:::~~~,::::~~~,::::~~~~~,::::~~~~,:::::::~)}
+p_{\young(~~~~,:::~)}p_{\young(~~~~~~,:::~~~,::::~~~,::::~~~~~,::::~~~~~)}\\
&\qquad-p_{\young(~~~~)}p_{\young(~~~~~~,:::~~~,::::~~~,::::~~~~~,::::~~~~~,:::::::~)}
+p_{\young(~~~)}p_{\young(~~~~~~,:::~~~,::::~~~,::::~~~~~,::::~~~~~,:::::::~~)}
-p_{\young(~~)}p_{\young(~~~~~~,:::~~~,::::~~~,::::~~~~~,::::~~~~~,:::::::~~,::::::::~)}
+p_{\young(~)}p_{\young(~~~~~~,:::~~~,::::~~~,::::~~~~~,::::~~~~~,:::::::~~,::::::::~,::::::::~)}
-p_{\varnothing}p_{\young(~~~~~~,:::~~~,::::~~~,::::~~~~~,::::~~~~~,:::::::~~,::::::::~,::::::::~,::::::::~)})
p_{\young(~~~~~~,:::~~~,::::~~~,::::~~~~~,::::~~~~~,:::::::~~)})\\
&+p_{\young(~~)}
(p_{\young(~~~~~~,:::~~~,::::~)}
(-p_{\young(~~~~~~,:::~~~,::::~~~,::::~~~,::::~~~)}p_{\young(~~~~~~,:::~~~,::::~~~,::::~~~~~,::::~~~~~,:::::::~~)}
+p_{\young(~~~~~~,:::~~~,::::~~~,::::~~~,::::~~)}p_{\young(~~~~~~,:::~~~,::::~~~,::::~~~~~,::::~~~~~,:::::::~~,::::::::~)}
-p_{\young(~~~~~~,:::~~~,::::~~~,::::~~~,::::~)}p_{\young(~~~~~~,:::~~~,::::~~~,::::~~~~~,::::~~~~~,:::::::~~,::::::::~,::::::::~)}
+p_{\young(~~~~~~,:::~~~,::::~~~,::::~~~)}p_{\young(~~~~~~,:::~~~,::::~~~,::::~~~~~,::::~~~~~,:::::::~~,::::::::~,::::::::~,::::::::~)})\\
&\quad+p_{\young(~~~~~~,:::~~,::::~)}
(p_{\young(~~~~~~,:::~~~,::::~~~,::::~~~~,::::~~~)}p_{\young(~~~~~~,:::~~~,::::~~~,::::~~~~~,::::~~~~~,:::::::~~)}
-p_{\young(~~~~~~,:::~~~,::::~~~,::::~~~~,::::~~)}p_{\young(~~~~~~,:::~~~,::::~~~,::::~~~~~,::::~~~~~,:::::::~~,::::::::~)}
+p_{\young(~~~~~~,:::~~~,::::~~~,::::~~~~,::::~)}p_{\young(~~~~~~,:::~~~,::::~~~,::::~~~~~,::::~~~~~,:::::::~~,::::::::~,::::::::~)}
-p_{\young(~~~~~~,:::~~~,::::~~~,::::~~~~)}p_{\young(~~~~~~,:::~~~,::::~~~,::::~~~~~,::::~~~~~,:::::::~~,::::::::~,::::::::~,::::::::~)})\\
&\quad+p_{\young(~~~~~,:::~~,::::~)}
(-p_{\young(~~~~~~,:::~~~,::::~~~,::::~~~~~,::::~~~)}p_{\young(~~~~~~,:::~~~,::::~~~,::::~~~~~,::::~~~~~,:::::::~~)}
+p_{\young(~~~~~~,:::~~~,::::~~~,::::~~~~~,::::~~)}p_{\young(~~~~~~,:::~~~,::::~~~,::::~~~~~,::::~~~~~,:::::::~~,::::::::~)}
-p_{\young(~~~~~~,:::~~~,::::~~~,::::~~~~~,::::~)}p_{\young(~~~~~~,:::~~~,::::~~~,::::~~~~~,::::~~~~~,:::::::~~,::::::::~,::::::::~)}
+p_{\young(~~~~~~,:::~~~,::::~~~,::::~~~~~)}p_{\young(~~~~~~,:::~~~,::::~~~,::::~~~~~,::::~~~~~,:::::::~~,::::::::~,::::::::~,::::::::~)})\\
&\quad+(-p_{\young(~~~~~~,:::~)}p_{\young(~~~~~~,:::~~~,::::~~~,::::~~~~,::::~~~~)}
+p_{\young(~~~~~,:::~)}p_{\young(~~~~~~,:::~~~,::::~~~,::::~~~~~,::::~~~~)}
+p_{\young(~~~~~~)}p_{\young(~~~~~~,:::~~~,::::~~~,::::~~~~,::::~~~~,:::::::~)}
-p_{\young(~~~~~)}p_{\young(~~~~~~,:::~~~,::::~~~,::::~~~~~,::::~~~~,:::::::~)}
-p_{\young(~~~~,:::~)}p_{\young(~~~~~~,:::~~~,::::~~~,::::~~~~~,::::~~~~~)}\\
&\qquad+p_{\young(~~~~)}p_{\young(~~~~~~,:::~~~,::::~~~,::::~~~~~,::::~~~~~,:::::::~)}
-p_{\young(~~~)}p_{\young(~~~~~~,:::~~~,::::~~~,::::~~~~~,::::~~~~~,:::::::~~)}
+p_{\young(~~)}p_{\young(~~~~~~,:::~~~,::::~~~,::::~~~~~,::::~~~~~,:::::::~~,::::::::~)}
-p_{\young(~)}p_{\young(~~~~~~,:::~~~,::::~~~,::::~~~~~,::::~~~~~,:::::::~~,::::::::~,::::::::~)}
+p_{\varnothing}p_{\young(~~~~~~,:::~~~,::::~~~,::::~~~~~,::::~~~~~,:::::::~~,::::::::~,::::::::~,::::::::~)})
p_{\young(~~~~~~,:::~~~,::::~~~,::::~~~~~,::::~~~~~,:::::::~~,::::::::~)})\\
&+p_{\young(~)}
(p_{\young(~~~~~~,:::~~~,::::~,::::~)}
(p_{\young(~~~~~~,:::~~~,::::~~~,::::~~~,::::~~~)}p_{\young(~~~~~~,:::~~~,::::~~~,::::~~~~~,::::~~~~~,:::::::~~)}
-p_{\young(~~~~~~,:::~~~,::::~~~,::::~~~,::::~~)}p_{\young(~~~~~~,:::~~~,::::~~~,::::~~~~~,::::~~~~~,:::::::~~,::::::::~)}
+p_{\young(~~~~~~,:::~~~,::::~~~,::::~~~,::::~)}p_{\young(~~~~~~,:::~~~,::::~~~,::::~~~~~,::::~~~~~,:::::::~~,::::::::~,::::::::~)}
-p_{\young(~~~~~~,:::~~~,::::~~~,::::~~~)}p_{\young(~~~~~~,:::~~~,::::~~~,::::~~~~~,::::~~~~~,:::::::~~,::::::::~,::::::::~,::::::::~)})\\
&\quad+p_{\young(~~~~~~,:::~~,::::~,::::~)}
(-p_{\young(~~~~~~,:::~~~,::::~~~,::::~~~~,::::~~~)}p_{\young(~~~~~~,:::~~~,::::~~~,::::~~~~~,::::~~~~~,:::::::~~)}
+p_{\young(~~~~~~,:::~~~,::::~~~,::::~~~~,::::~~)}p_{\young(~~~~~~,:::~~~,::::~~~,::::~~~~~,::::~~~~~,:::::::~~,::::::::~)}
-p_{\young(~~~~~~,:::~~~,::::~~~,::::~~~~,::::~)}p_{\young(~~~~~~,:::~~~,::::~~~,::::~~~~~,::::~~~~~,:::::::~~,::::::::~,::::::::~)}
+p_{\young(~~~~~~,:::~~~,::::~~~,::::~~~~)}p_{\young(~~~~~~,:::~~~,::::~~~,::::~~~~~,::::~~~~~,:::::::~~,::::::::~,::::::::~,::::::::~)})\\
&\quad+p_{\young(~~~~~,:::~~,::::~,::::~)}
(p_{\young(~~~~~~,:::~~~,::::~~~,::::~~~~~,::::~~~)}p_{\young(~~~~~~,:::~~~,::::~~~,::::~~~~~,::::~~~~~,:::::::~~)}
-p_{\young(~~~~~~,:::~~~,::::~~~,::::~~~~~,::::~~)}p_{\young(~~~~~~,:::~~~,::::~~~,::::~~~~~,::::~~~~~,:::::::~~,::::::::~)}
+p_{\young(~~~~~~,:::~~~,::::~~~,::::~~~~~,::::~)}p_{\young(~~~~~~,:::~~~,::::~~~,::::~~~~~,::::~~~~~,:::::::~~,::::::::~,::::::::~)}
-p_{\young(~~~~~~,:::~~~,::::~~~,::::~~~~~)}p_{\young(~~~~~~,:::~~~,::::~~~,::::~~~~~,::::~~~~~,:::::::~~,::::::::~,::::::::~,::::::::~)})\\
&\quad+(p_{\young(~~~~~~,:::~)}p_{\young(~~~~~~,:::~~~,::::~~~,::::~~~~,::::~~~~)}
-p_{\young(~~~~~,:::~)}p_{\young(~~~~~~,:::~~~,::::~~~,::::~~~~~,::::~~~~)}
-p_{\young(~~~~~~)}p_{\young(~~~~~~,:::~~~,::::~~~,::::~~~~,::::~~~~,:::::::~)}
+p_{\young(~~~~~)}p_{\young(~~~~~~,:::~~~,::::~~~,::::~~~~~,::::~~~~,:::::::~)}
+p_{\young(~~~~,:::~)}p_{\young(~~~~~~,:::~~~,::::~~~,::::~~~~~,::::~~~~~)}\\
&\qquad-p_{\young(~~~~)}p_{\young(~~~~~~,:::~~~,::::~~~,::::~~~~~,::::~~~~~,:::::::~)}
+p_{\young(~~~)}p_{\young(~~~~~~,:::~~~,::::~~~,::::~~~~~,::::~~~~~,:::::::~~)}
-p_{\young(~~)}p_{\young(~~~~~~,:::~~~,::::~~~,::::~~~~~,::::~~~~~,:::::::~~,::::::::~)}
+p_{\young(~)}p_{\young(~~~~~~,:::~~~,::::~~~,::::~~~~~,::::~~~~~,:::::::~~,::::::::~,::::::::~)}
-p_{\varnothing}p_{\young(~~~~~~,:::~~~,::::~~~,::::~~~~~,::::~~~~~,:::::::~~,::::::::~,::::::::~,::::::::~)})
p_{\young(~~~~~~,:::~~~,::::~~~,::::~~~~~,::::~~~~~,:::::::~~,::::::::~,::::::::~)})\\
&+p_{\varnothing}
(p_{\young(~~~~~~,:::~~~,::::~,::::~,::::~)}
(-p_{\young(~~~~~~,:::~~~,::::~~~,::::~~~,::::~~~)}p_{\young(~~~~~~,:::~~~,::::~~~,::::~~~~~,::::~~~~~,:::::::~~)}
+p_{\young(~~~~~~,:::~~~,::::~~~,::::~~~,::::~~)}p_{\young(~~~~~~,:::~~~,::::~~~,::::~~~~~,::::~~~~~,:::::::~~,::::::::~)}
-p_{\young(~~~~~~,:::~~~,::::~~~,::::~~~,::::~)}p_{\young(~~~~~~,:::~~~,::::~~~,::::~~~~~,::::~~~~~,:::::::~~,::::::::~,::::::::~)}
+p_{\young(~~~~~~,:::~~~,::::~~~,::::~~~)}p_{\young(~~~~~~,:::~~~,::::~~~,::::~~~~~,::::~~~~~,:::::::~~,::::::::~,::::::::~,::::::::~)})\\
&\quad+p_{\young(~~~~~~,:::~~,::::~,::::~,::::~)}
(p_{\young(~~~~~~,:::~~~,::::~~~,::::~~~~,::::~~~)}p_{\young(~~~~~~,:::~~~,::::~~~,::::~~~~~,::::~~~~~,:::::::~~)}
-p_{\young(~~~~~~,:::~~~,::::~~~,::::~~~~,::::~~)}p_{\young(~~~~~~,:::~~~,::::~~~,::::~~~~~,::::~~~~~,:::::::~~,::::::::~)}
+p_{\young(~~~~~~,:::~~~,::::~~~,::::~~~~,::::~)}p_{\young(~~~~~~,:::~~~,::::~~~,::::~~~~~,::::~~~~~,:::::::~~,::::::::~,::::::::~)}
-p_{\young(~~~~~~,:::~~~,::::~~~,::::~~~~)}p_{\young(~~~~~~,:::~~~,::::~~~,::::~~~~~,::::~~~~~,:::::::~~,::::::::~,::::::::~,::::::::~)})\\
&\quad+p_{\young(~~~~~,:::~~,::::~,::::~,::::~)}
(-p_{\young(~~~~~~,:::~~~,::::~~~,::::~~~~~,::::~~~)}p_{\young(~~~~~~,:::~~~,::::~~~,::::~~~~~,::::~~~~~,:::::::~~)}
+p_{\young(~~~~~~,:::~~~,::::~~~,::::~~~~~,::::~~)}p_{\young(~~~~~~,:::~~~,::::~~~,::::~~~~~,::::~~~~~,:::::::~~,::::::::~)}
-p_{\young(~~~~~~,:::~~~,::::~~~,::::~~~~~,::::~)}p_{\young(~~~~~~,:::~~~,::::~~~,::::~~~~~,::::~~~~~,:::::::~~,::::::::~,::::::::~)}
+p_{\young(~~~~~~,:::~~~,::::~~~,::::~~~~~)}p_{\young(~~~~~~,:::~~~,::::~~~,::::~~~~~,::::~~~~~,:::::::~~,::::::::~,::::::::~,::::::::~)})\\
&\quad+(-p_{\young(~~~~~~,:::~)}p_{\young(~~~~~~,:::~~~,::::~~~,::::~~~~,::::~~~~)}
+p_{\young(~~~~~,:::~)}p_{\young(~~~~~~,:::~~~,::::~~~,::::~~~~~,::::~~~~)}
+p_{\young(~~~~~~)}p_{\young(~~~~~~,:::~~~,::::~~~,::::~~~~,::::~~~~,:::::::~)}
-p_{\young(~~~~~)}p_{\young(~~~~~~,:::~~~,::::~~~,::::~~~~~,::::~~~~,:::::::~)}
-p_{\young(~~~~,:::~)}p_{\young(~~~~~~,:::~~~,::::~~~,::::~~~~~,::::~~~~~)}\\
&\qquad+p_{\young(~~~~)}p_{\young(~~~~~~,:::~~~,::::~~~,::::~~~~~,::::~~~~~,:::::::~)}
-p_{\young(~~~)}p_{\young(~~~~~~,:::~~~,::::~~~,::::~~~~~,::::~~~~~,:::::::~~)}
+p_{\young(~~)}p_{\young(~~~~~~,:::~~~,::::~~~,::::~~~~~,::::~~~~~,:::::::~~,::::::::~)}
-p_{\young(~)}p_{\young(~~~~~~,:::~~~,::::~~~,::::~~~~~,::::~~~~~,:::::::~~,::::::::~,::::::::~)}
+p_{\varnothing}p_{\young(~~~~~~,:::~~~,::::~~~,::::~~~~~,::::~~~~~,:::::::~~,::::::::~,::::::::~,::::::::~)})
p_{\young(~~~~~~,:::~~~,::::~~~,::::~~~~~,::::~~~~~,:::::::~~,::::::::~,::::::::~,::::::::~)})
}

Factoring the expression above yields one much closer to the one presented in \cite{Spacek_Wang_exceptional}.
\ali{
\cD_4&= (p_{\young(~~~)}p_{\young(~~~~~~,:::~~~)} - p_{\young(~~)}p_{\young(~~~~~~,:::~~~,::::~)} + p_{\young(~)}p_{\young(~~~~~~,:::~~~,::::~,::::~)} - p_{\varnothing}p_{\young(~~~~~~,:::~~~,::::~,::::~,::::~)})*\\
&\qquad(p_{\young(~~~~~~,:::~~~,::::~~~,::::~~~,::::~~~)}p_{\young(~~~~~~,:::~~~,::::~~~,::::~~~~~,::::~~~~~,:::::::~~)}
-p_{\young(~~~~~~,:::~~~,::::~~~,::::~~~,::::~~)}p_{\young(~~~~~~,:::~~~,::::~~~,::::~~~~~,::::~~~~~,:::::::~~,::::::::~)}
+p_{\young(~~~~~~,:::~~~,::::~~~,::::~~~,::::~)}p_{\young(~~~~~~,:::~~~,::::~~~,::::~~~~~,::::~~~~~,:::::::~~,::::::::~,::::::::~)}
-p_{\young(~~~~~~,:::~~~,::::~~~,::::~~~)}p_{\young(~~~~~~,:::~~~,::::~~~,::::~~~~~,::::~~~~~,:::::::~~,::::::::~,::::::::~,::::::::~)})\\
&\quad-(p_{\young(~~~)}p_{\young(~~~~~~,:::~~)}-p_{\young(~~)}p_{\young(~~~~~~,:::~~,::::~)}+p_{\young(~)}p_{\young(~~~~~~,:::~~,::::~,::::~)}-p_{\varnothing}p_{\young(~~~~~~,:::~~,::::~,::::~,::::~)})*\\
&\qquad(p_{\young(~~~~~~,:::~~~,::::~~~,::::~~~~,::::~~~)}p_{\young(~~~~~~,:::~~~,::::~~~,::::~~~~~,::::~~~~~,:::::::~~)}
-p_{\young(~~~~~~,:::~~~,::::~~~,::::~~~~,::::~~)}p_{\young(~~~~~~,:::~~~,::::~~~,::::~~~~~,::::~~~~~,:::::::~~,::::::::~)}
+p_{\young(~~~~~~,:::~~~,::::~~~,::::~~~~,::::~)}p_{\young(~~~~~~,:::~~~,::::~~~,::::~~~~~,::::~~~~~,:::::::~~,::::::::~,::::::::~)}
-p_{\young(~~~~~~,:::~~~,::::~~~,::::~~~~)}p_{\young(~~~~~~,:::~~~,::::~~~,::::~~~~~,::::~~~~~,:::::::~~,::::::::~,::::::::~,::::::::~)})\\
&\quad+(p_{\young(~~~)}p_{\young(~~~~~,:::~~)}-p_{\young(~~)}p_{\young(~~~~~,:::~~,::::~)}+p_{\young(~)}p_{\young(~~~~~,:::~~,::::~,::::~)}-p_{\varnothing}p_{\young(~~~~~,:::~~,::::~,::::~,::::~)})*\\
&\qquad(p_{\young(~~~~~~,:::~~~,::::~~~,::::~~~~~,::::~~~)}p_{\young(~~~~~~,:::~~~,::::~~~,::::~~~~~,::::~~~~~,:::::::~~)}
-p_{\young(~~~~~~,:::~~~,::::~~~,::::~~~~~,::::~~)}p_{\young(~~~~~~,:::~~~,::::~~~,::::~~~~~,::::~~~~~,:::::::~~,::::::::~)}
+p_{\young(~~~~~~,:::~~~,::::~~~,::::~~~~~,::::~)}p_{\young(~~~~~~,:::~~~,::::~~~,::::~~~~~,::::~~~~~,:::::::~~,::::::::~,::::::::~)}
-p_{\young(~~~~~~,:::~~~,::::~~~,::::~~~~~)}p_{\young(~~~~~~,:::~~~,::::~~~,::::~~~~~,::::~~~~~,:::::::~~,::::::::~,::::::::~,::::::::~)})\\
&\quad+(p_{\young(~~~)}p_{\young(~~~~~~,:::~~~,::::~~~,::::~~~~~,::::~~~~~,:::::::~~,::::::::~,::::::::~,::::::::~)}-p_{\young(~~)}p_{\young(~~~~~~,:::~~~,::::~~~,::::~~~~~,::::~~~~~,:::::::~~,::::::::~,::::::::~,::::::::~)}+p_{\young(~)}p_{\young(~~~~~~,:::~~~,::::~~~,::::~~~~~,::::~~~~~,:::::::~~,::::::::~,::::::::~,::::::::~)}-p_{\varnothing}p_{\young(~~~~~~,:::~~~,::::~~~,::::~~~~~,::::~~~~~,:::::::~~,::::::::~,::::::::~,::::::::~)})*\\
&\qquad\,
(p_{\young(~~~~~~,:::~)}p_{\young(~~~~~~,:::~~~,::::~~~,::::~~~~,::::~~~~)}
-p_{\young(~~~~~,:::~)}p_{\young(~~~~~~,:::~~~,::::~~~,::::~~~~~,::::~~~~)}
-p_{\young(~~~~~~)}p_{\young(~~~~~~,:::~~~,::::~~~,::::~~~~,::::~~~~,:::::::~)}
+p_{\young(~~~~~)}p_{\young(~~~~~~,:::~~~,::::~~~,::::~~~~~,::::~~~~,:::::::~)}
+p_{\young(~~~~,:::~)}p_{\young(~~~~~~,:::~~~,::::~~~,::::~~~~~,::::~~~~~)}\\
&\qquad-p_{\young(~~~~)}p_{\young(~~~~~~,:::~~~,::::~~~,::::~~~~~,::::~~~~~,:::::::~)}
+p_{\young(~~~)}p_{\young(~~~~~~,:::~~~,::::~~~,::::~~~~~,::::~~~~~,:::::::~~)}
-p_{\young(~~)}p_{\young(~~~~~~,:::~~~,::::~~~,::::~~~~~,::::~~~~~,:::::::~~,::::::::~)}
+p_{\young(~)}p_{\young(~~~~~~,:::~~~,::::~~~,::::~~~~~,::::~~~~~,:::::::~~,::::::::~,::::::::~)}
-p_{\varnothing}p_{\young(~~~~~~,:::~~~,::::~~~,::::~~~~~,::::~~~~~,:::::::~~,::::::::~,::::::::~,::::::::~)})\\
}

%% file: p4_e7_num.tex
\Yboxdim{3pt}
\begin{tikzpicture}[scale=1,style=very thick,baseline=0.25em]
\node (0) at (0,0) {\tiny\young(~~~~)~~\young(~~~~~~,:::~~~)~~\young(~~~~~~,:::~~~,::::~~~,::::~~~,::::~~~)~~\young(~~~~~~,:::~~~,::::~~~,::::~~~~~,::::~~~~~,:::::::~~)};
\node (1) at (3.5,0) {\tiny\young(~~~~)~~\young(~~~~~~,:::~~~)~~\young(~~~~~~,:::~~~,::::~~~,::::~~~,::::~~)~~\young(~~~~~~,:::~~~,::::~~~,::::~~~~~,::::~~~~~,:::::::~~,::::::::~)};
\node (2) at (7,0) {\tiny\young(~~~~)~~\young(~~~~~~,:::~~~)~~\young(~~~~~~,:::~~~,::::~~~,::::~~~,::::~)~~\young(~~~~~~,:::~~~,::::~~~,::::~~~~~,::::~~~~~,:::::::~~,::::::::~,::::::::~)};
\node (3) at (10.5,0) {\tiny\young(~~~~)~~\young(~~~~~~,:::~~~)~~\young(~~~~~~,:::~~~,::::~~~,::::~~~)~~\young(~~~~~~,:::~~~,::::~~~,::::~~~~~,::::~~~~~,:::::::~~,::::::::~,::::::::~,::::::::~)};
\node (4) at (1,-1.5) {\tiny\young(~~~~)~~\young(~~~~~~,:::~~)~~\young(~~~~~~,:::~~~,::::~~~,::::~~~~,::::~~~)~~\young(~~~~~~,:::~~~,::::~~~,::::~~~~~,::::~~~~~,:::::::~~)};
\node (5) at (4.5,-1.5) {\tiny\young(~~~~)~~\young(~~~~~~,:::~~)~~\young(~~~~~~,:::~~~,::::~~~,::::~~~~,::::~~)~~\young(~~~~~~,:::~~~,::::~~~,::::~~~~~,::::~~~~~,:::::::~~,::::::::~)};
\node (6) at (8,-1.5) {\tiny\young(~~~~)~~\young(~~~~~~,:::~~)~~\young(~~~~~~,:::~~~,::::~~~,::::~~~~,::::~)~~\young(~~~~~~,:::~~~,::::~~~,::::~~~~~,::::~~~~~,:::::::~~,::::::::~,::::::::~)};
\node (7) at (11.5,-1.5) {\tiny\young(~~~~)~~\young(~~~~~~,:::~~)~~\young(~~~~~~,:::~~~,::::~~~,::::~~~~)~~\young(~~~~~~,:::~~~,::::~~~,::::~~~~~,::::~~~~~,:::::::~~,::::::::~,::::::::~,::::::::~)};
\node (8) at (2,-3) {\tiny\young(~~~~)~~\young(~~~~~,:::~~)~~\young(~~~~~~,:::~~~,::::~~~,::::~~~~~,::::~~~)~~\young(~~~~~~,:::~~~,::::~~~,::::~~~~~,::::~~~~~,:::::::~~)};
\node (9) at (5.5,-3) {\tiny\young(~~~~)~~\young(~~~~~,:::~~)~~\young(~~~~~~,:::~~~,::::~~~,::::~~~~~,::::~~)~~\young(~~~~~~,:::~~~,::::~~~,::::~~~~~,::::~~~~~,:::::::~~,::::::::~)};
\node (10) at (9,-3) {\tiny\young(~~~~)~~\young(~~~~~,:::~~)~~\young(~~~~~~,:::~~~,::::~~~,::::~~~~~,::::~)~~\young(~~~~~~,:::~~~,::::~~~,::::~~~~~,::::~~~~~,:::::::~~,::::::::~,::::::::~)};
\node (11) at (12.5,-3) {\tiny\young(~~~~)~~\young(~~~~~,:::~~)~~\young(~~~~~~,:::~~~,::::~~~,::::~~~~~)~~\young(~~~~~~,:::~~~,::::~~~,::::~~~~~,::::~~~~~,:::::::~~,::::::::~,::::::::~,::::::::~)};
\draw 
(0)--(1)--(2)--(3)
(4)--(5)--(6)--(7)
(8)--(9)--(10)--(11)
(0)--(4)--(8)
(1)--(5)--(9)
(2)--(6)--(10)
(3)--(7)--(11);
\node (-1) at (-1.5,-3) {(1)};
\draw[dashed]
(4)--(-1);

\node (22) at ($(0)+(0,-6.00)$) {\tiny\young(~~)~~\young(~~~~~~,:::~~~,::::~~)~~\young(~~~~~~,:::~~~,::::~~~,::::~~~,::::~~~)~~\young(~~~~~~,:::~~~,::::~~~,::::~~~~~,::::~~~~~,:::::::~~)};
\node (23) at ($(1)+(0,-6.00)$){\tiny\young(~~)~~\young(~~~~~~,:::~~~,::::~~)~~\young(~~~~~~,:::~~~,::::~~~,::::~~~,::::~~)~~\young(~~~~~~,:::~~~,::::~~~,::::~~~~~,::::~~~~~,:::::::~~,::::::::~)};
\node (24) at ($(2)+(0,-6.00)$){\tiny\young(~~)~~\young(~~~~~~,:::~~~,::::~~)~~\young(~~~~~~,:::~~~,::::~~~,::::~~~,::::~)~~\young(~~~~~~,:::~~~,::::~~~,::::~~~~~,::::~~~~~,:::::::~~,::::::::~,::::::::~)};
\node (25) at ($(3)+(0,-6.00)$){\tiny\young(~~)~~\young(~~~~~~,:::~~~,::::~~)~~\young(~~~~~~,:::~~~,::::~~~,::::~~~)~~\young(~~~~~~,:::~~~,::::~~~,::::~~~~~,::::~~~~~,:::::::~~,::::::::~,::::::::~,::::::::~)};
\node (26) at ($(4)+(0,-6.00)$){\tiny\young(~~)~~\young(~~~~~~,:::~~,::::~)~~\young(~~~~~~,:::~~~,::::~~~,::::~~~~,::::~~~~)~~\young(~~~~~~,:::~~~,::::~~~,::::~~~~~,::::~~~~~,:::::::~~)};
\node (30) at ($(8)+(0,-6.00)$){\tiny\young(~~)~~\young(~~~~~,:::~~,::::~)~~\young(~~~~~~,:::~~~,::::~~~,::::~~~~~,::::~~~~)~~\young(~~~~~~,:::~~~,::::~~~,::::~~~~~,::::~~~~~,:::::::~~)};
\draw 
(22)--(23)--(24)--(25)
(22)--(26)--(30);
\foreach \x in {0,...,4,8}
\draw (\x)--({\the\numexpr\x+22\relax});
\node (-2) at (-1.5,-9) {(2)};
\draw[dashed]
(26)--(-2);

\node (44) at ($(22)+(0,-6.00)$) {\young(~)~~\young(~~~~~~,:::~~~,::::~,::::~)~~\young(~~~~~~,:::~~~,::::~~~,::::~~~,::::~~~)~~\young(~~~~~~,:::~~~,::::~~~,::::~~~~~,::::~~~~~,:::::::~~)};
\node (45) at ($(23)+(0,-6.00)$) {\young(~)~~\young(~~~~~~,:::~~~,::::~,::::~)~~\young(~~~~~~,:::~~~,::::~~~,::::~~~,::::~~)~~\young(~~~~~~,:::~~~,::::~~~,::::~~~~~,::::~~~~~,:::::::~~,::::::::~)};
\node (46) at ($(24)+(0,-6.00)$) {\young(~)~~\young(~~~~~~,:::~~~,::::~,::::~)~~\young(~~~~~~,:::~~~,::::~~~,::::~~~,::::~)~~\young(~~~~~~,:::~~~,::::~~~,::::~~~~~,::::~~~~~,:::::::~~,::::::::~,::::::::~)};
\node (47) at ($(25)+(0,-6.00)$) {\young(~)~~\young(~~~~~~,:::~~~,::::~,::::~)~~\young(~~~~~~,:::~~~,::::~~~,::::~~~)~~\young(~~~~~~,:::~~~,::::~~~,::::~~~~~,::::~~~~~,:::::::~~,::::::::~,::::::::~,::::::::~)};
\node (48) at ($(26)+(0,-6.00)$) {\young(~)~~\young(~~~~~~,:::~~,::::~,::::~)~~\young(~~~~~~,:::~~~,::::~~~,::::~~~~,::::~~~)~~\young(~~~~~~,:::~~~,::::~~~,::::~~~~~,::::~~~~~,:::::::~~)};
\node (52) at ($(30)+(0,-6.00)$) {\young(~)~~\young(~~~~~,:::~~,::::~,::::~)~~\young(~~~~~~,:::~~~,::::~~~,::::~~~~~,::::~~~)~~\young(~~~~~~,:::~~~,::::~~~,::::~~~~~,::::~~~~~,:::::::~~)};
\draw 
(44)--(45)--(46)--(47)
(44)--(48)--(52);
\foreach \x in {22,...,26,30}
\draw (\x)--({\the\numexpr\x+22\relax});
\node (-3) at (-1.5,-15) {(3)};
\draw[dashed]
(48)--(-3);

\node (66) at ($(44)+(0,-6.00)$) {$\varnothing$~~\young(~~~~~~,:::~~~,::::~,::::~,::::~)~~\young(~~~~~~,:::~~~,::::~~~,::::~~~,::::~~~)~~\young(~~~~~~,:::~~~,::::~~~,::::~~~~~,::::~~~~~,:::::::~~)};
\node (67) at ($(45)+(0,-6.00)$) {$\varnothing$~~\young(~~~~~~,:::~~~,::::~,::::~,::::~)~~\young(~~~~~~,:::~~~,::::~~~,::::~~~,::::~~)~~\young(~~~~~~,:::~~~,::::~~~,::::~~~~~,::::~~~~~,:::::::~~,::::::::~)};
\node (68) at ($(46)+(0,-6.00)$) {$\varnothing$~~\young(~~~~~~,:::~~~,::::~,::::~,::::~)~~\young(~~~~~~,:::~~~,::::~~~,::::~~~,::::~)~~\young(~~~~~~,:::~~~,::::~~~,::::~~~~~,::::~~~~~,:::::::~~,::::::::~,::::::::~)};
\node (69) at ($(47)+(0,-6.00)$) {$\varnothing$~~\young(~~~~~~,:::~~~,::::~,::::~,::::~)~~\young(~~~~~~,:::~~~,::::~~~,::::~~~)~~\young(~~~~~~,:::~~~,::::~~~,::::~~~~~,::::~~~~~,:::::::~~,::::::::~,::::::::~,::::::::~)};
\node (70) at ($(48)+(0,-6.00)$) {$\varnothing$~~\young(~~~~~~,:::~~,::::~,::::~,::::~)~~\young(~~~~~~,:::~~~,::::~~~,::::~~~~,::::~~~)~~\young(~~~~~~,:::~~~,::::~~~,::::~~~~~,::::~~~~~,:::::::~~)};
\node (74) at ($(52)+(0,-6.00)$) {$\varnothing$~~\young(~~~~~,:::~~,::::~,::::~,::::~)~~\young(~~~~~~,:::~~~,::::~~~,::::~~~~~,::::~~~)~~\young(~~~~~~,:::~~~,::::~~~,::::~~~~~,::::~~~~~,:::::::~~)};
\draw
(66)--(67)--(68)--(69)
(66)--(70)--(74);
\foreach \x in {44,...,48,52}
\draw (\x)--({\the\numexpr\x+22\relax});
\node (-4) at (-1.5,-21) {(4)};
\draw[dashed]
(70)--(-4);
\end{tikzpicture}

\begin{tikzpicture}[scale=2.25,style=very thick,baseline=0.25em]
\node (12) at (0,1) {\tiny\young(~~~~)~~\young(~~~~~~,:::~)~~\young(~~~~~~,:::~~~,::::~~~,::::~~~~,::::~~~~)~~\young(~~~~~~,:::~~~,::::~~~,::::~~~~~,::::~~~~~,:::::::~~)};
\node (13) at (.5,0) {\tiny\young(~~~~)~~\young(~~~~~~)~~\young(~~~~~~,:::~~~,::::~~~,::::~~~~,::::~~~~,:::::::~)~~\young(~~~~~~,:::~~~,::::~~~,::::~~~~~,::::~~~~~,:::::::~~)}; 
\node (14) at (1,1.5) {\tiny\young(~~~~)~~\young(~~~~~,:::~)~~\young(~~~~~~,:::~~~,::::~~~,::::~~~~~,::::~~~~)~~\young(~~~~~~,:::~~~,::::~~~,::::~~~~~,::::~~~~~,:::::::~~)};
\node (15) at (1.5,0.5) {\tiny\young(~~~~)~~\young(~~~~~)~~\young(~~~~~~,:::~~~,::::~~~,::::~~~~~,::::~~~~,:::::::~)~~\young(~~~~~~,:::~~~,::::~~~,::::~~~~~,::::~~~~~,:::::::~~)}; 
\node (16) at (2,2) {\tiny\young(~~~~)~~\young(~~~~,:::~)~~\young(~~~~~~,:::~~~,::::~~~,::::~~~~~,::::~~~~~)~~\young(~~~~~~,:::~~~,::::~~~,::::~~~~~,::::~~~~~,:::::::~~)};
\node (17) at (2.5,1) {\tiny\young(~~~~)~~\young(~~~~)~~\young(~~~~~~,:::~~~,::::~~~,::::~~~~~,::::~~~~~,:::::::~)~~\young(~~~~~~,:::~~~,::::~~~,::::~~~~~,::::~~~~~,:::::::~~)};
\node (18) at (3.5,1.5) {\tiny\young(~~~~)~~\young(~~~)~~\young(~~~~~~,:::~~~,::::~~~,::::~~~~~,::::~~~~~,:::::::~~)~~\young(~~~~~~,:::~~~,::::~~~,::::~~~~~,::::~~~~~,:::::::~~)};
\node (19) at (4,1) {\tiny\young(~~~~)~~\young(~~)~~\young(~~~~~~,:::~~~,::::~~~,::::~~~~~,::::~~~~~,:::::::~~,::::::::~)~~\young(~~~~~~,:::~~~,::::~~~,::::~~~~~,::::~~~~~,:::::::~~)};
\node (20) at (5,1.5) {\tiny\young(~~~~)~~\young(~)~~\young(~~~~~~,:::~~~,::::~~~,::::~~~~~,::::~~~~~,:::::::~~,::::::::~,::::::::~)~~\young(~~~~~~,:::~~~,::::~~~,::::~~~~~,::::~~~~~,:::::::~~)};
\node (21) at (5.5,1) {\tiny\young(~~~~)~~$\varnothing$~~\young(~~~~~~,:::~~~,::::~~~,::::~~~~~,::::~~~~~,:::::::~~,::::::::~,::::::::~,::::::::~)~~\young(~~~~~~,:::~~~,::::~~~,::::~~~~~,::::~~~~~,:::::::~~)};
\draw 
(12)--(14)--(16)--(17)--(18)--(19)--(20)--(21)
(12)--(13)--(15)--(17)
(14)--(15);
\node (-1) at (-.5,1.5) {(1)};
\draw[dashed] (12)--(-1);

\node(34) at ($(12)+(0,-2)$) {\young(~~)~~\young(~~~~~~,:::~~)~~\young(~~~~~~,:::~~~,::::~~~,::::~~~~,::::~~~~)~~\young(~~~~~~,:::~~~,::::~~~,::::~~~~~,::::~~~~~,:::::::~~,::::::::~)};
\node (35) at ($(14)+(0,-2)$) {\young(~~)~~\young(~~~~~,:::~~)~~\young(~~~~~~,:::~~~,::::~~~,::::~~~~~,::::~~~~)~~\young(~~~~~~,:::~~~,::::~~~,::::~~~~~,::::~~~~~,:::::::~~,::::::::~)};

\draw 
(34)--(35);
\node (-2) at (-.5,-.5) {(2)};
\draw[dashed] (34)--(-2);
\draw
(12)--(34)
(14)--(35);

\node (56) at ($(34)+(0,-2)$) {\young(~)~~\young(~~~~~~,:::~~)~~\young(~~~~~~,:::~~~,::::~~~,::::~~~~,::::~~~~)~~\young(~~~~~~,:::~~~,::::~~~,::::~~~~~,::::~~~~~,:::::::~~,::::::::~,::::::::~)};
\node (57) at ($(35)+(0,-2)$) {\young(~)~~\young(~~~~~,:::~~)~~\young(~~~~~~,:::~~~,::::~~~,::::~~~~~,::::~~~~)~~\young(~~~~~~,:::~~~,::::~~~,::::~~~~~,::::~~~~~,:::::::~~,::::::::~,::::::::~)};
\draw 
(56)--(57);
\node (-3) at (-.5,-2.5) {(3)};
\draw[dashed] (56)--(-3);
\foreach \x in {34,35}
\draw (\x)--({\the\numexpr\x+22\relax});

\node (78) at ($(56)+(0,-2)$) {$\varnothing$~~\young(~~~~~~,:::~~)~~\young(~~~~~~,:::~~~,::::~~~,::::~~~~,::::~~~~)~~\young(~~~~~~,:::~~~,::::~~~,::::~~~~~,::::~~~~~,:::::::~~,::::::::~,::::::::~,::::::::~)};
\node (79) at ($(57)+(0,-2)$) {$\varnothing$~~\young(~~~~~,:::~~)~~\young(~~~~~~,:::~~~,::::~~~,::::~~~~~,::::~~~~)~~\young(~~~~~~,:::~~~,::::~~~,::::~~~~~,::::~~~~~,:::::::~~,::::::::~,::::::::~,::::::::~)};
\draw 
(78)--(79);
\node (-4) at (-.5,-4.5) {(4)};
\draw[dashed] (78)--(-4);
\foreach \x in {56,57}
\draw (\x)--({\the\numexpr\x+22\relax});

\end{tikzpicture}

\ali{
\deis[4](\cD_4) &= 
p_{\young(~~~~)}
(p_{\young(~~~~~~,:::~~~)}
(p_{\young(~~~~~~,:::~~~,::::~~~,::::~~~,::::~~~)}p_{\young(~~~~~~,:::~~~,::::~~~,::::~~~~~,::::~~~~~,:::::::~~)}
-p_{\young(~~~~~~,:::~~~,::::~~~,::::~~~,::::~~)}p_{\young(~~~~~~,:::~~~,::::~~~,::::~~~~~,::::~~~~~,:::::::~~,::::::::~)}
+p_{\young(~~~~~~,:::~~~,::::~~~,::::~~~,::::~)}p_{\young(~~~~~~,:::~~~,::::~~~,::::~~~~~,::::~~~~~,:::::::~~,::::::::~,::::::::~)}
-p_{\young(~~~~~~,:::~~~,::::~~~,::::~~~)}p_{\young(~~~~~~,:::~~~,::::~~~,::::~~~~~,::::~~~~~,:::::::~~,::::::::~,::::::::~,::::::::~)})\\
&\quad+p_{\young(~~~~~~,:::~~)}
(-p_{\young(~~~~~~,:::~~~,::::~~~,::::~~~~,::::~~~)}p_{\young(~~~~~~,:::~~~,::::~~~,::::~~~~~,::::~~~~~,:::::::~~)}
+p_{\young(~~~~~~,:::~~~,::::~~~,::::~~~~,::::~~)}p_{\young(~~~~~~,:::~~~,::::~~~,::::~~~~~,::::~~~~~,:::::::~~,::::::::~)}
-p_{\young(~~~~~~,:::~~~,::::~~~,::::~~~~,::::~)}p_{\young(~~~~~~,:::~~~,::::~~~,::::~~~~~,::::~~~~~,:::::::~~,::::::::~,::::::::~)}
+p_{\young(~~~~~~,:::~~~,::::~~~,::::~~~~)}p_{\young(~~~~~~,:::~~~,::::~~~,::::~~~~~,::::~~~~~,:::::::~~,::::::::~,::::::::~,::::::::~)})\\
&\quad+p_{\young(~~~~~,:::~~)}
(p_{\young(~~~~~~,:::~~~,::::~~~,::::~~~~~,::::~~~)}p_{\young(~~~~~~,:::~~~,::::~~~,::::~~~~~,::::~~~~~,:::::::~~)}
-p_{\young(~~~~~~,:::~~~,::::~~~,::::~~~~~,::::~~)}p_{\young(~~~~~~,:::~~~,::::~~~,::::~~~~~,::::~~~~~,:::::::~~,::::::::~)}
+p_{\young(~~~~~~,:::~~~,::::~~~,::::~~~~~,::::~)}p_{\young(~~~~~~,:::~~~,::::~~~,::::~~~~~,::::~~~~~,:::::::~~,::::::::~,::::::::~)}
-p_{\young(~~~~~~,:::~~~,::::~~~,::::~~~~~)}p_{\young(~~~~~~,:::~~~,::::~~~,::::~~~~~,::::~~~~~,:::::::~~,::::::::~,::::::::~,::::::::~)})\\
&\quad+(p_{\young(~~~~~~,:::~)}p_{\young(~~~~~~,:::~~~,::::~~~,::::~~~~,::::~~~~)}
-p_{\young(~~~~~,:::~)}p_{\young(~~~~~~,:::~~~,::::~~~,::::~~~~~,::::~~~~)}
-p_{\young(~~~~~~)}p_{\young(~~~~~~,:::~~~,::::~~~,::::~~~~,::::~~~~,:::::::~)}
+p_{\young(~~~~~)}p_{\young(~~~~~~,:::~~~,::::~~~,::::~~~~~,::::~~~~,:::::::~)}
+p_{\young(~~~~,:::~)}p_{\young(~~~~~~,:::~~~,::::~~~,::::~~~~~,::::~~~~~)}\\
&\qquad-p_{\young(~~~~)}p_{\young(~~~~~~,:::~~~,::::~~~,::::~~~~~,::::~~~~~,:::::::~)}
+p_{\young(~~~)}p_{\young(~~~~~~,:::~~~,::::~~~,::::~~~~~,::::~~~~~,:::::::~~)}
-p_{\young(~~)}p_{\young(~~~~~~,:::~~~,::::~~~,::::~~~~~,::::~~~~~,:::::::~~,::::::::~)}
+p_{\young(~)}p_{\young(~~~~~~,:::~~~,::::~~~,::::~~~~~,::::~~~~~,:::::::~~,::::::::~,::::::::~)}
-p_{\varnothing}p_{\young(~~~~~~,:::~~~,::::~~~,::::~~~~~,::::~~~~~,:::::::~~,::::::::~,::::::::~,::::::::~)})
p_{\young(~~~~~~,:::~~~,::::~~~,::::~~~~~,::::~~~~~,:::::::~~)})\\
&+p_{\young(~~)}
(p_{\young(~~~~~~,:::~~~,::::~~)}
(-p_{\young(~~~~~~,:::~~~,::::~~~,::::~~~,::::~~~)}p_{\young(~~~~~~,:::~~~,::::~~~,::::~~~~~,::::~~~~~,:::::::~~)}
+p_{\young(~~~~~~,:::~~~,::::~~~,::::~~~,::::~~)}p_{\young(~~~~~~,:::~~~,::::~~~,::::~~~~~,::::~~~~~,:::::::~~,::::::::~)}
-p_{\young(~~~~~~,:::~~~,::::~~~,::::~~~,::::~)}p_{\young(~~~~~~,:::~~~,::::~~~,::::~~~~~,::::~~~~~,:::::::~~,::::::::~,::::::::~)}
+p_{\young(~~~~~~,:::~~~,::::~~~,::::~~~)}p_{\young(~~~~~~,:::~~~,::::~~~,::::~~~~~,::::~~~~~,:::::::~~,::::::::~,::::::::~,::::::::~)})\\
&\quad+\Bigl[p_{\young(~~~~~~,:::~~,::::~)}
(p_{\young(~~~~~~,:::~~~,::::~~~,::::~~~~,::::~~~~)}p_{\young(~~~~~~,:::~~~,::::~~~,::::~~~~~,::::~~~~~,:::::::~~)})\Bigr] +\Bigl[p_{\young(~~~~~,:::~~,::::~)}(-p_{\young(~~~~~~,:::~~~,::::~~~,::::~~~~~,::::~~~~)}p_{\young(~~~~~~,:::~~~,::::~~~,::::~~~~~,::::~~~~~,:::::::~~)})\Bigr]\\
&\quad+(-p_{\young(~~~~~~,:::~~)}p_{\young(~~~~~~,:::~~~,::::~~~,::::~~~~,::::~~~~)}
+p_{\young(~~~~~,:::~~)}p_{\young(~~~~~~,:::~~~,::::~~~,::::~~~~~,::::~~~~)}
)
p_{\young(~~~~~~,:::~~~,::::~~~,::::~~~~~,::::~~~~~,:::::::~~,::::::::~)})\\
&+p_{\young(~)}
(p_{\young(~~~~~~,:::~~~,::::~~,::::~)}
(p_{\young(~~~~~~,:::~~~,::::~~~,::::~~~,::::~~~)}p_{\young(~~~~~~,:::~~~,::::~~~,::::~~~~~,::::~~~~~,:::::::~~)}
-p_{\young(~~~~~~,:::~~~,::::~~~,::::~~~,::::~~)}p_{\young(~~~~~~,:::~~~,::::~~~,::::~~~~~,::::~~~~~,:::::::~~,::::::::~)}
+p_{\young(~~~~~~,:::~~~,::::~~~,::::~~~,::::~)}p_{\young(~~~~~~,:::~~~,::::~~~,::::~~~~~,::::~~~~~,:::::::~~,::::::::~,::::::::~)}
-p_{\young(~~~~~~,:::~~~,::::~~~,::::~~~)}p_{\young(~~~~~~,:::~~~,::::~~~,::::~~~~~,::::~~~~~,:::::::~~,::::::::~,::::::::~,::::::::~)})\\
&\quad+\Bigl[p_{\young(~~~~~~,:::~~,::::~,::::~)}
(-p_{\young(~~~~~~,:::~~~,::::~~~,::::~~~~,::::~~~~)}p_{\young(~~~~~~,:::~~~,::::~~~,::::~~~~~,::::~~~~~,:::::::~~)})\Bigr]+\Bigl[p_{\young(~~~~~,:::~~,::::~,::::~)}
(p_{\young(~~~~~~,:::~~~,::::~~~,::::~~~~~,::::~~~~)}p_{\young(~~~~~~,:::~~~,::::~~~,::::~~~~~,::::~~~~~,:::::::~~)})\Bigr]\\
&\quad+(p_{\young(~~~~~~,:::~~)}p_{\young(~~~~~~,:::~~~,::::~~~,::::~~~~,::::~~~~)}
-p_{\young(~~~~~,:::~~)}p_{\young(~~~~~~,:::~~~,::::~~~,::::~~~~~,::::~~~~)})
p_{\young(~~~~~~,:::~~~,::::~~~,::::~~~~~,::::~~~~~,:::::::~~,::::::::~,::::::::~)})\\
&+p_{\varnothing}
(p_{\young(~~~~~~,:::~~~,::::~~,::::~,::::~)}
(-p_{\young(~~~~~~,:::~~~,::::~~~,::::~~~,::::~~~)}p_{\young(~~~~~~,:::~~~,::::~~~,::::~~~~~,::::~~~~~,:::::::~~)}
+p_{\young(~~~~~~,:::~~~,::::~~~,::::~~~,::::~~)}p_{\young(~~~~~~,:::~~~,::::~~~,::::~~~~~,::::~~~~~,:::::::~~,::::::::~)}
-p_{\young(~~~~~~,:::~~~,::::~~~,::::~~~,::::~)}p_{\young(~~~~~~,:::~~~,::::~~~,::::~~~~~,::::~~~~~,:::::::~~,::::::::~,::::::::~,::::::::~)}
+p_{\young(~~~~~~,:::~~~,::::~~~,::::~~~)}p_{\young(~~~~~~,:::~~~,::::~~~,::::~~~~~,::::~~~~~,:::::::~~,::::::::~,::::::::~)})\\
&\quad+\Bigl[p_{\young(~~~~~~,:::~~,::::~,::::~,::::~)}
(p_{\young(~~~~~~,:::~~~,::::~~~,::::~~~~,::::~~~~)}p_{\young(~~~~~~,:::~~~,::::~~~,::::~~~~~,::::~~~~~,:::::::~~,::::::::~,::::::::~,::::::::~)})\Bigr]+\Bigl[p_{\young(~~~~~,:::~~,::::~,::::~,::::~)}
(-p_{\young(~~~~~~,:::~~~,::::~~~,::::~~~~~,::::~~~~)}p_{\young(~~~~~~,:::~~~,::::~~~,::::~~~~~,::::~~~~~,:::::::~~,::::::::~,::::::::~,::::::::~)})\Bigr]\\
&\quad+(-p_{\young(~~~~~~,:::~~)}p_{\young(~~~~~~,:::~~~,::::~~~,::::~~~~,::::~~~~)}
+p_{\young(~~~~~,:::~~)}p_{\young(~~~~~~,:::~~~,::::~~~,::::~~~~~,::::~~~~)}
)
p_{\young(~~~~~~,:::~~~,::::~~~,::::~~~~~,::::~~~~~,:::::::~~,::::::::~,::::::::~,::::::::~)})
}

%% file: appBtype_dependent_LGs.tex
\section{Relation to previous type-dependent LG models}\label{app:type_dependent_LGs}
We claimed that our models naturally generalize the previous models given in \cite{Marsh_Rietsch_Grassmannians, Pech_Rietsch_Odd_Quadrics, Pech_Rietsch_Williams_Quadrics, Pech_Rietsch_Lagrangian_Grassmannians, Spacek_Wang_exceptional, Spacek_Wang_OG} and we now describe how to recover each of these type-dependent models from our formalism, using the Dynkin diagram combinatorics as calculated in the proof of Lemma \ref{lem:special_weyl_elt_action}; for representative examples of the minuscule posets $\minposet$, see Example \ref{ex:list_of_minposets}.

\addtocontents{toc}{\SkipTocEntry}
\subsection*{Grassmannians}
In \cite{Marsh_Rietsch_Grassmannians}, an LG model is defined for $\Gr(n-k,n)=\SL_n/\P_{n-k}=\LGA_n^{\SC}/\P_{n-k}$ using the standard Pl\"ucker embedding of $\cmX$ into $\PS\bigl(\bigwedge^k(\C^n)^*\bigr)$ with the resulting Pl\"ucker coordinates labeled by Young diagrams inside the $(n-k)\times k$-rectangle; it is straightforward to see that this labeling agrees with our labeling by order ideals after forgetting the simple reflection labels. The LG model in \cite{Marsh_Rietsch_Grassmannians} is defined using the Young diagrams $\mu_i$ and $\hat\mu_i$, defined as follows: 
\[
\mu_i = \begin{cases}
    \text{the first $i$ rows,} & \text{for $i\in\{1,\ldots,n-k\}$,} \\
    \text{the first $n-i$ columns,} & \text{for $i\in\{n-k,\ldots,n\}$.}
\end{cases}
\]
(Note that $\mu_n=\varnothing$ and that both descriptions yield the full $(n-k)\times k$-rectangle for $\mu_{n-k}$.) On the other hand, $\hat\mu_i$ is the unique diagram that can be obtained by adding a box to $\mu_i$ for $i\neq n-k$, and $\hat\mu_{n-k}$ is the $(n-k-1)\times(k-1)$-subrectangle. With this, the superpotential is defined as
\[
\pot_{\text{MR}} = \sum_{i\neq n-k} \frac{\p_{\hat\mu_i}}{\p_{\mu_i}} + q\frac{\p_{\hat\mu_{n-k}}}{\p_{\mu_{n-k}}}.
\]
In our formalism, we find that $\minposet$ is the $(n-k)\times k$-rectangle with the bottom-left element labeled $s_1$, the top-left element labeled $s_{n-k}$ and the top-right element labeled $s_n$ (labels increasing over diagonals from left to right). In the proof of Lemma \ref{lem:special_weyl_elt_action}, we see that Grassmannians always have $\cis=1$ for any $\is\in[n]$. (Note that we considered $\Gr(k,n)$ instead of $\Gr(n-k,n)$.) Moreover, we have that $i_1=n-k-\is$ for $\is\in\{1,\ldots,n-k-1\}$ and $i_1=n-(\is-n+k)=2n-k-\is$ for $\is\in\{n-k+1,\ldots,n\}$, so we find that $\idealij[1]$ consists of of the first $i_1$ rows in the first case and the first $i_1$ columns in the second case. Hence we find $\idealij[1]=\mu_{i_1}$ for all $\is\neq n-k$ (forgetting the simple reflection labeling), and therefore we conclude $\deis(\idealij[1])=\hat\mu_{i_1}$ as well. For $\idealPrime$, we note that $\idealPPrime$ is the union of the first column and first row---being the maximal order ideal containing exactly one occurrence of an element labeled $s_{n-k}$---so its complement $\idealPrime$ is the $(n-k-1)\times(k-1)$-rectangle, i.e.~$\idealPrime=\hat\mu_{n-k}$. Together with $\minposet=\mu_{n-k}$, we conclude that the potential $\pot_\can$ is identically equal to $\pot_{\text{MR}}$.

\addtocontents{toc}{\SkipTocEntry}
\subsection*{Quadrics}
In \cite{Pech_Rietsch_Williams_Quadrics}, LG models for odd and even quadrics are given using the (generalized) Pl\"ucker embeddings of $\cmX$. The models for odd and even quadrics are fairly similar, but since $\cmX$ is isomorphic to a projective plane for odd quadrics whereas for even quadrics it is isomorphic to another even quadric, the labeling of Pl\"ucker coordinates in each case is slightly different.

For the odd quadric $Q_{2m-1}=\SO_{2m+1}/\P_1=\LGB_m^\SC/\P_1$, the coordinates on $\cmX=\CP^{2m-1}=\dP_1\backslash\LGC_m^\AD$ are labeled $\p_i$ for $i\in\{0,\ldots,2m-1\}$, and it is straightforward to see that $\p_i$ coincides with the Pl\"ucker coordinate $\p_{[i]}$ labeled by the order ideal $[i]$ consisting of $i$ elements, where the minuscule poset $\minposet=[2m-1]$ is the row of $2m-1$ boxes labeled $s_1,\ldots,s_{m-1},s_m,s_{m-1},\ldots,s_1$. The potential is given there as
\[
\pot_{\text{PRW},\text{odd}} = \frac{\p_1}{\p_0} + \sum_{j=1}^{m-1}\frac{\p_{j+1}\p_{2m-1-j}}{\cD_{\text{PRW},j}} + q\frac{\p_1}{\p_{2m-1}},
\qwhere
\cD_{\text{PRW},j} = \sum_{i=0}^j (-1)^i \p_{j-i}\p_{2m-1-j+i}.
\]
In our formalism, we see that $\cis=2$ except for the minuscule vertex $\is=1$; hence all terms are ratios of quadratic homogeneous polynomials in the Pl\"ucker coordinate, except for the quantum term and the term $\p_{\yng(1)}/\p_\varnothing = \p_1/\p_0$. Since $\idealPPrime$ is the maximal order ideal containing a single occurrence of the label $s_1$, we see that it has a single box as complement, so that the quantum terms agree as well: $q\p_{\yng(1)}/\p_{\minposet} = q\p_1/\p_{2m-1}$.

For the remaining terms ($\is\in\{2,\ldots,m\}$), we start by noting that $i_1=\is$ and $i_2=\is-1$. Hence, we find that $\idealij[1]=[\is-1]$ while $\idealij[2]=[2m-1-(\is-1)]=[2m-\is]$. Inspecting the labels, we see that from $(\idealij[1],\idealij[2])=([\is-1],[2m-\is])$ we can move a single element at a time until we reach $(\varnothing,\minposet)=([0],[2m-1])$; hence, we see
\[
\cD_{\is} = \sum_{j=0}^{\is-1} (-1)^j\p_{[\is-1-j]}\p_{[2m-\is+j]} = \cD_{\text{PRW},\is-1}.
\]
Next, it is clear that $\deis([j])=0$ for all $j\notin\{\is-1,2m-\is-1\}$, the only nonzero action of $\deis$ on $\cD_\is$ is on the first factor of the first term, i.e.~
\[
\deis(\cD_\is) = \deis(\p_{\idealij[1]})\p_{\idealij[2]} = \p_{[\is]}\p_{[2m-\is]},
\]
and hence we find that $\pot_\can=\pot_{\text{PRW},\text{odd}}$ for odd quadrics.

For the even quadric $Q_{2m-2}=\Spin_{2m}/\P_1=\LGD_m^\SC/\P_1$, the Pl\"ucker coordinates $(\p_0:\ldots:\p_{m-1}:\p_{m-1}':\p_m:\ldots:\p_{2m-2})$ on $\cmX=\dP_1\backslash\PSO_{2m}=\dP_1\backslash\LGD_m^\AD$ are defined using the same method, so we only need to identify the minimal coset representatives as defined in \cite[Section 3.1]{Pech_Rietsch_Williams_Quadrics} with the corresponding order ideals. Recalling Example \ref{ex:list_of_minposets}, $\minposet$ consist of $2m-2$ elements: it starts with a chain of $m-2$ elements labeled $s_1,\ldots,s_{m-2}$; the last element is covered by two elements labeled $s_{m-1}$ and $s_{m}$ respectively; both of these are covered by an element labeled $s_{m-2}$ followed by a chain of $m-3$ elements labeled $s_{m-3},\ldots,s_1$. With this in mind, define for $i\in\{0,\ldots,m-2\}$ the order ideal $[i]$ as the chain of the first $i$ elements, define $[m-1]$ to be the chain $s_1<\ldots<s_{m-2}<s_{m-1}$ and $[m-1]'$ as the chain $s_1<\ldots<s_{m-2}<s_m$, while we define $[m]=[m-1]\cup[m-1]'$, and finally for $i\in\{m+1,\ldots,2m-2\}$ we define $[i]$ to be the order ideal $s_1<\ldots<s_{m-2}<s_{m-1},s_m<s_{m-1}<\ldots<s_{2m-1-i}$ consisting of $i$ elements. This convention will let $\p_i=\p_{[i]}$ as well as $\p_{m-1}'=\p_{[m-1]'}$.

Now, comparing the potentials, the one in \cite{Pech_Rietsch_Williams_Quadrics} is given as
\[
\pot_{\text{PRW},\text{even}} = \frac{\p_1}{\p_0}+\frac{\p_m}{\p_{m-1}}+\frac{\p_m}{\p_{m-1}'} + \sum_{j=1}^{m-3}\frac{\p_{j+1}\p_{2m-2-j}}{\cD_{\text{PRW},j}} + q\frac{\p_1}{\p_{2m-2}},
\]
where $\cD_{\text{PRW},j}=\sum_{i=0}^j (-1)^i \p_{j-i}\p_{2m-2-j+i}$ is the ($2m-2$)-dimensional analogue. In our formalism, we find $\cis=1$ for $\is\in\{1,m-1,m\}$ and $\cis=2$ otherwise. We directly see $\p_{[1]}/\p_{[0]}=\p_1/\p_0$ and the quantum term is analogous to the case of the odd quadric. Recalling the proof of Lemma \ref{lem:special_weyl_elt_action}, we see that $\is\in\{m-1,m\}$ gives $i_1\in\{m-1,m\}$ (depending on $m-1$ even resp.~odd) so we find that $\idealij[1]$ is either $[m-1]$ or $[m-1]'$, both of which have $\deis([m-1])=\deis([m-1]')=[m]$ (note that $i_1$ depends on $\is$, so the equality contains two distinct actions of $\deis$). Hence, we conclude $\pot_{m-1}+\pot_m = \p_{[m]}/\p_{[m-1]} + \p_{[m]}/\p_{[m-1]'} = \p_m/\p_{m-1} + \p_m/\p_{m-1}'$.

For $\is\notin\{1,m-1,m\}$, we find $i_1=\is$ and $i_2=\is-1$, so that $\idealij[1]=[\is-1]$ and $\idealij[2]=[2m-2-(\is-1)]=[2m-\is-1]$. Concluding that $\cD_\is=\cD_{\text{PRW},\is-1}$ and $\deis(\cD_\is)=\p_{[\is]}\p_{[2m-\is-1]}$ is entirely analogous to the case of the odd quadric. With this, we arrive at $\pot_\can=\pot_{\text{PRW},\text{even}}$ for even quadrics as well.

\addtocontents{toc}{\SkipTocEntry}
\subsection*{Lagrangian Grassmannians}
The models presented in \cite{Pech_Rietsch_Lagrangian_Grassmannians} for the Lagrangian Grassmannians $\LG(m,2m)=\Sp_{2m}/\P_m = \LGC_m^{\SC}/\P_m$ use the same embedding of $\cmX=\dP_m\backslash\PSO_{2m+1}=\dP_m\backslash\LGB_m^\AD$ in $\PS\bigl(V(\dfwt[m])^*\bigr)$ as is used for the models $(\mX_\can,\pot_\can)$. (Note that the adjoint and simply-connected groups are reversed compared to that reference.) Hence, it suffices to identify the labelings of the corresponding Pl\"ucker coordinates and show that the superpotentials agree.

In \cite{Pech_Rietsch_Lagrangian_Grassmannians}, the Pl\"ucker coordinates $\p_\la$ are labeled by strict partitions $(\la)=(\la_1,\ldots,\la_j)$ with $1\le\la_{i+1}<\la_{i}\le m$; here the empty partition corresponds to the lowest weight vector and the maximal partition $(m,m-1,\ldots,1)$ corresponds to the highest weight vector. On the other hand, in the order ideal formalism, $\minposet$ is the ``decreasing staircase'' consisting of $m$ columns with $m,m-1,\ldots,1$ elements respectively, aligned along the bottom elements; the elements are labeled $1$ to $m$ along diagonals starting from the bottom-left element to the main diagonal. (Cf.~Example \ref{ex:list_of_minposets}.) For comparison, we will denote by $[\la]=[\la_1,\la_2,\ldots,\la_j]$ the order ideal consisting of $j$ columns with sizes $\la_i$ satisfying $1\le \la_{i+1} < \la_i\le m$ (ignoring empty columns). Thus we identify the Pl\"ucker coordinate $p_{(\la)}$ of \cite{Pech_Rietsch_Lagrangian_Grassmannians} with our Pl\"ucker coordinate $p_{[\la]}$. The reason for this translation between Young diagrams is because we view our order ideals as arising from the corresponding minuscule poset which we orient to place the smallest element in the upper left and the largest element in the lower right. 

The potential in \cite{Pech_Rietsch_Lagrangian_Grassmannians} is defined using a set of specific partitions (note that the notation used here slightly deviates from the source): First, they define $(\rho_l)=(l,l-1,\ldots,1)$ as the length $l$ staircase partition, and take $(\rho_l^+)=(l+1,l-1,l-2,\ldots,1)$ to be the unique strict partition obtained from $(\rho_l)$ for $l<m$ by adding a single box. Next, for $J\subset[l]$, they define $(\rho_l)^J$ and $(\rho_l^+)^J$ to be the partitions obtained by deleting the parts $(\rho_l)_j$ from $(\rho_l)$ and $(\rho_l^+)_j$ from $(\rho_l^+)$ for each $j\in J$. Next, they define the partition $(\mu_l)=(m,m-1,\ldots,m+1-l)$ to be the maximal partition with $l$ parts, and set $(\mu_l^J)=(m,m-1,\ldots,m+1-l,l+1-j_1,\ldots,l+1-j_i)$ and $(\mu_l^J)^+=(m,m-1,\ldots,m+1-l,l+1-j_1+\de_{j_1,1},\ldots,l+1-j_i)$, where $J=\{j_1<\ldots<j_i\}$ and $\de_{j_1,1}=1$ for $j_1=1$ and $0$ otherwise. Whenever these constructions do \emph{not} result in a strict partition $(\mu)$, they set $\p_{(\mu)}=0$ for notational ease. With this, they define
\[
\pot_{\text{PR}} = \frac{\p_{(\rho_0^+)}}{\p_{(\rho_0)}}+\sum_{l=1}^{m-1}\frac{\cD_{\text{PR},l}^+}{\cD_{\text{PR},l}} + q\frac{\p_{(\rho_{m-1})}}{\p_{(\rho_m)}}, \qwhere \cD_{\text{PR},l} = \sum_{J\subset[l]} (-1)^{s(J)}\p_{(\rho_l)^J\p_{(\mu_l^J)}}
\]
with $s(J)=\sum_{j\in J}j$ and with $\cD_{\text{PR},l}^+$ obtained from $\cD_{\text{PR},l}$ by replacing $(\rho_l)^J$ by $(\rho_l^+)^J$ and $(\mu_l^J)$ by $(\mu_l^J)^+$.

Returning to the order ideal formalism, we directly see that $(\rho_0)=\varnothing$ and $(\rho_1)=\yng(1)$, so we find that the first term agrees with the potential presented here. Moreoever, $\idealPrime$ is the complement of a single column, i.e.~the staircase of length $m-1$, so that $q\p_{\idealPrime}/\p_{\minposet}=q\p_{(\rho_{m-1})}/\p_{(\rho_m)}$ for the quantum term as well. 

For the remaining terms ($\is\neq m$), we find from the proof of Lemma \ref{lem:special_weyl_elt_action} that $\cis=2$, and that $i_1=m-\is$ and $i_2=m$. Hence, we find that $\idealij[1]=[\rho_{\is}]$ (the staircase order ideal of length $\is$) and $\idealij[2]=[\mu_{\is}]$ (the order ideal consisting of the first $\is$ columns of $\minposet$). This immediately tells us that the initial term of $\cD_\is$ agrees with the $J=\varnothing$ term of $\cD_{\text{PR},\is}$; for the remaining terms, we note that written as Young diagrams, the pair of partitions $\bigl((\rho_l)^J,(\mu_l)^J\bigr)$ can be thought of as \emph{moving} $s(J)$ boxes in the rows indexed by $J$ from $(\rho_l)$ to $(\mu_l)$ whenever such a move is valid, so it follows that the identification $(\la)\mapsto[\la]$ maps the set $\{(\rho_{\is})^J,(\mu_{\is}^J)~|~J\subset[\is]\}$ to $\sP_{\is}$, and hence we find $\cD_{\is}=\cD_{\text{PR},\is}$ for $\is\in[m-1]$. For the numerators, we note that $(\rho_\is^+)$ is obtained by adding a single box, and that this box in $[\rho_\is^+]$ has label $m-\is=i_1$, so that $\deis([\rho_\is])=[\rho_\is^+]$ and we conclude that the initial terms of $\deis(\cD_\is)$ and $\cD_{\text{PR},\is}$ agree; for the remaining terms, we find an analogous interpretation of moving boxes for the numerators as for the denominators, and we conclude that the potentials $\pot_\can=\pot_{\text{PR}}$ agree.

\addtocontents{toc}{\SkipTocEntry}
\subsection*{Orthogonal Grassmannians}
The models presented in \cite{Spacek_Wang_OG} were obtained by applying the model $(\mX_\can,\pot_\can)$ presented here to the specific case of $\OG(n,2n)=\Spin_{2n}/\P_n=\LGD_n^\SC/\P_n$, and hence require no further verifications.

\addtocontents{toc}{\SkipTocEntry}
\subsection*{The exceptional family}
The models presented in \cite{Spacek_Wang_exceptional} formed the direct inspiration for the models $(\mX_\can,\pot_\can)$ presented here, so relating the two formalisms is straightforward. Firstly, comparing Definition \ref{df:Plucker_coords_posets} here with Definition 2.12 of \cite{Spacek_Wang_exceptional}, we see that the latter is a special case of the former. Hence, we only need to translate the labeling used in both papers, but this is straightforward as well: given an order ideal $\ideal\subideal\minposet$, consider the corresponding minimal coset representative $\Pi(\ideal)\in\cosets$ (as defined in Remark \ref{rem:order_ideals_to_min_coset_reps}) and follow the action of its simple reflections through the Hasse diagrams as given in equations (2.25) and (2.27) of \cite[Section 2.4]{Spacek_Wang_exceptional} to find the corresponding weight vector $v_i^{(j)}$; then $\p_\ideal = \p_i^{(j)}$. Establishing that the models agree thus reduces to showing that the potentials are equal.

For the Cayley plane $\OP^2=\LGE_6^\SC/\P_6$, we have written out the values of $\cis$ and the sequences $(i_j)$ in equation \eqref{eq:table_of_ij_for_Cayley} of the proof of Lemma \ref{lem:special_weyl_elt_action}. Hence, we find the following order ideals $\idealij$:
\[
\begin{array}{c||c|c|c|c|c}
    \is  & 1 & 2 & 3 & 5 & 4 \\\hline
    \cis & 1 & 2 & 2 & 2 & 3 \\\hline
    \raisebox{.4em}{$(\idealij)$}
    \vphantom{\yng(1,1,1,1,1)}& 
    \young(~~~~,::~~,:::~,:::~) & 
    \raisebox{7.9pt}{\yng(1)}~\raisebox{2.5pt}{\young(~~~~~,::~~~,:::~~~)} & 
    \raisebox{5.4pt}{\young(~~~,::~)}~\young(~~~~~,::~~~,:::~~,:::~~) &
    \raisebox{8pt}{\yng(5)}~\young(~~~~~,::~~~,:::~~~,:::~~~~) &
    \raisebox{8pt}{\yng(2)}~\raisebox{5.4pt}{\young(~~~~~,::~~~)}~\young(~~~~~,::~~~,:::~~~,:::~~~)
\end{array}
\qwhere
\minposet = \Yboxdim{7pt}
\raisebox{-10pt}{\scriptsize\young(65431,::243,:::542,:::65431)}.
\]
The posets $\sP_\is$ are now obtained by moving boxes; instead of showing the posets, we directly translate the resulting $\cD_\is$ and $\deis(\cD_\is)$ into the $\p_i^{(j)}$ notation as used in \cite{Spacek_Wang_exceptional}:
\ali{
\pot_0 &= \frac{\p_1}{\p_0},\qquad 
\pot_1 = \frac{\p_9'}{\p_8}, \qquad
\pot_6 = q\frac{\p_5''}{\p_{16}},\qquad
\pot_2 = \frac{\p_2\p_{11}''-\p_0\p_{13}}{\p_1\p_{11}''-\p_0\p_{12}}, \\
\pot_3 &= \frac{\p_5'\p_{12}'-\p_4''\p_{13}+\p_2\p_{15}}{\p_4'\p_{12}' - \p_3\p_{13} + \p_2\p_{14} - \p_1\p_{15}+\p_0\p_{16}}, \qquad
\pot_5 = \frac{\p_6''\p_{15}-\p_5'\p_{16}}{\p_5\p_{15}-\p_4\p_{16}}, \\
\pot_4 &= 
\frac
{p_{3}(p_{8}''p_{14} - p_{7}''p_{15} + p_{6}'p_{16}) + p_{1}(-p_{10}''p_{14}) + p_{0}(p_{11}'p_{14})}
{p_2(p_{8}''p_{14} - p_{7}''p_{15} + p_{6}'p_{16}) + p_{1}(-p_{9}''p_{14} + p_{8}'p_{15} - p_{7}'p_{16})  + p_{0}(p_{10}'p_{14} - p_{9}'p_{15} + p_{8}p_{16})}.
}
Comparing these expressions with the potential as given in equations (3.1), (3.2) and (3.3) in \cite[Section 3.1]{Spacek_Wang_exceptional}, we find that the above expressions for $\pot_\is$ already agree for $\is\in\{0,1,2,5,6\}$. For the remaining terms, we require a number of Pl\"ucker relations as given in Remark 3.14 of \cite{Spacek_Wang_exceptional}, but the simplification is straightforward.

The story for the Freudenthal variety $\LGE_7^\SC/\P_7$ is analogous, but somewhat more elaborate. For the order ideals $\idealij$ we find the following using equation \eqref{eq:table_of_ij_for_Freudenthal} from the proof of Lemma \ref{lem:special_weyl_elt_action}
\[
\begin{array}{c||c|c|c|c|c|c}
    \is  & 1 & 2 & 6 & \hspace{.66em}3\hspace{.66em} & \hspace{.66em}5\hspace{.66em} & \hspace{.66em}4\hspace{.66em} \\\hline
    \cis & 2 & 2 & 2 & 3 & 3 & 4 \\\hline
    \raisebox{.8em}{$(\idealij)$} & 
    \raisebox{18.6pt}{\yng(1)}~ \raisebox{10.5pt}{\young(~~~~~~,:::~~~,::::~~~,::::~~~~~)} &
    \raisebox{18.6pt}{\yng(6)}~ \raisebox{5.2pt}{\young(~~~~~~,:::~~~,::::~~~,::::~~~~,::::~~~~,:::::::~)} &
    \raisebox{8pt}{\young(~~~~~,:::~~,::::~,::::~,::::~)}~ \young(~~~~~~,:::~~~,::::~~~,::::~~~~~,::::~~~~~,:::::::~~,::::::::~,::::::::~)
    \vphantom{\young(~,~,~,~,~,~,~,~,~)}&
    \multicolumn{3}{c}{\raisebox{.8em}{see Appendix \ref{app:exceptional}}}
\end{array}
\qwhere
\minposet = 
\Yboxdim{7pt}
\raisebox{-28pt}{\scriptsize\young(765431,:::243,::::542,::::65431,::::76543,:::::::24,::::::::5,::::::::6,::::::::7)}
\]
Hence, we obtain the following terms of the potential, written in the $p_i^{(j)}$ notation of \cite{Spacek_Wang_exceptional}:
\ali{
\pot_1 &= \frac{\p_2\p_{17}-\p_0\p_{19}'}{\p_1\p_{17}-\p_0\p_{18}},
\qquad
\pot_6 = \frac{\p_{11}\p_{26} - \p_{10}'\p_{27}}{\p_{10}\p_{26} - \p_9\p_{27}},
\qquad
\pot_7 = q\frac{\p_{10}}{\p_{17}}, \\
\pot_2 &= \frac{\p_7''\p_{21}'' - \p_6''\p_{22}'' + \p_5'\p_{23}}{\p_6''\p_{21}'' - \p_5''\p_{22}'' + \p_4\p_{23} - \p_3\p_{24} + \p_2\p_{25} - \p_1\p_{26} + \p_0\p_{27}},
}
and the remaining terms can be obtained from Appendix \ref{app:exceptional}:
\ali{
\pot_3 &= \textstyle
\frac{
p_{3}(p_{12}''p_{22}' - p_{11}''p_{23} + p_{10}''p_{24} - p_{9}''p_{25})
-p_{1}(p_{14}p_{22}' - p_{13}'p_{23} + p_{11}'p_{25} -   p_{10}''p_{26})
+p_0(p_{15}'p_{22}' - p_{14}''p_{23} + p_{12}p_{25} - p_{10}''p_{27})
}{
p_{2}
(p_{12}''p_{22}' - p_{11}''p_{23} + p_{10}''p_{24} - p_{9}''p_{25})
-p_{1}
(p_{13}''p_{22}' - p_{12}'p_{23} + p_{11}'p_{24} - p_{9}''p_{26})
+p_0
(p_{14}'p_{22}' - p_{13}p_{23} + p_{12}p_{24} - p_{9}''p_{27})
},
\\
\pot_5 &= \textstyle
\frac{
(p_{6}'p_{15}'' - p_{5}''p_{16}'' + p_{3}p_{18}' - p_{2}p_{19}''
)p_{25}
+(-p_{6}'p_{14}'' + p_{5}''p_{15}' - p_{3}p_{17}' + p_{1}p_{19}''
)p_{26}
+(p_{6}'p_{13}' - p_{5}''p_{14} + p_{3}p_{16} - p_{0}p_{19}''
)p_{27}
}{
(p_{5}'p_{15}'' - p_{4}p_{16}'' + p_{3}p_{17}'' - p_{2}p_{18}''
)p_{25}
-(p_{5}'p_{14}'' - p_{4}p_{15}' + p_{3}p_{16}' - p_{1}p_{18}''
)p_{26}
+(p_{5}'p_{13}' - p_{4}p_{14} + p_{3}p_{15} - p_{0}p_{18}'')p_{27}
}.
}
We suppress the translated expression for $\pot_4$ due to its length. Comparing these expressions with the potential as given in equation (3.4), (3.5) and (3.6) in \cite[Section 3.2]{Spacek_Wang_exceptional}, we see that the $\pot_\is$ already agree for $\is\in\{1,2,3,6,7\}$ (up to adding minus signs in the numerator a denominator of $\is=2$ and reorganizing the terms for $\is=3$). The remaining terms agree after simplifying using the Pl\"ucker relations as given in Remark 3.18 of \cite{Spacek_Wang_exceptional}.